%% file: 9_23_16_arxiv_resubmit.tex
\documentclass[12pt]{article}
\usepackage[cp1251]{inputenc}
\usepackage{amsmath,amsthm,mathrsfs}
\usepackage[english]{babel}

\usepackage{amsfonts,amstext,amssymb,verbatim,epsfig}
\usepackage{pgf,tikz}
\usepackage{float}
\usetikzlibrary{arrows}
\usepackage{hyperref}
\usepackage{cite}
\usepackage{epstopdf}
\usepackage{color}
\usepackage{transparent}
\usepackage{subfigure}

\def\R{{\mathbb{R}}}
\def\N{{\mathbb{N}}}
\def\Z{{\mathbb{Z}}}

\newcommand{\E}{\mathbb{E}}
\renewcommand{\P}{\mathbb{P}}
\def\e{{\varepsilon}}

\usepackage{xcolor}
\usepackage[normalem]{ulem}

\hypersetup{colorlinks=true,pdfborder=000,  citecolor  = blue, citebordercolor= {magenta}}
\hypersetup{linkcolor=purple}
\makeatletter
\let\reftagform@=\tagform@
\def\tagform@#1{\maketag@@@{(\ignorespaces\textcolor{purple}{#1}\unskip\@@italiccorr)}}
\renewcommand{\eqref}[1]{\textup{\reftagform@{\ref{#1}}}}
\makeatother

\makeatletter

\DeclareUrlCommand\ULurl@@{%
  \def\UrlLeft{\uline\bgroup}%
  \def\UrlRight{\egroup}}
\def\ULurl@#1{\hyper@linkurl{\ULurl@@{#1}}{#1}}
\DeclareRobustCommand*\ULurl{\hyper@normalise\ULurl@}
\makeatother

\def\<{\langle}
\def\>{\rangle}

\newcommand{\Pro}{\ensuremath{\mathbb{P}}}
\newcommand{\indi}{\ensuremath{\boldsymbol 1}}

\usepackage[left=1.13in,top=1.3in,right=1.13in,bottom=1.3in]{geometry}
\DeclareMathOperator*{\esssup}{ess\,sup}

\theoremstyle{usual}
\newtheorem{theorem}{Theorem}[section]
\newtheorem{corollary}[theorem]{Corollary}
\newtheorem{lemma}[theorem]{Lemma}
\newtheorem{proposition}[theorem]{Proposition}

\newtheoremstyle{likedef}
  {}%
  {}%
  {}%
  {\parindent}%
  {\bfseries}%
  {.}%
  {.5em}%
  {}%

\theoremstyle{likedef}
\newtheorem{definition}[theorem]{Definition}
\newtheorem{remark}[theorem]{Remark}
\newtheorem{example}[theorem]{Example}

\newtheorem{question}{Question}
\newtheorem{questionend}{Question}
\numberwithin{equation}{section}

\newtheorem{thm}{Theorem}
\newtheorem{cor}[thm]{Corollary}
\newtheorem{prop}[thm]{Proposition}
\newtheorem{claim}{Claim}

\begin{document}

\title{50 years of first passage percolation}
\author{Antonio Auffinger \thanks{The research of A. A. is supported by NSF grant DMS-1597864.}  \\ \small{Northwestern University} \and Michael Damron \thanks{The research of M. D. is supported by NSF grant DMS-1419230 and an NSF CAREER award.} \\ \small{Georgia Tech} \and Jack Hanson \thanks{The research of J. H. is supported by the AMS-Simons travel grant and NSF grant DMS-1612921.} \\ \small{CUNY}}\maketitle

\footnotetext{MSC2000: Primary 60K35, 82B43.}
\footnotetext{Keywords: First-passage percolation, shape fluctuations, Busemann functions, Richardson's growth model, graph of infection.}

\begin{abstract}

We celebrate the 50th anniversary of one the most classical models in probability theory. In this survey, we describe the main results of first passage percolation, paying special attention to the recent burst of advances of the past 5 years. The purpose of these notes is twofold. In the first chapters, we give self-contained proofs of seminal results obtained in the '80s and '90s on limit shapes and geodesics, while covering the state of the art of these questions.  Second, aside from these classical results, we discuss recent perspectives and directions including (1) the connection between Busemann functions and geodesics, (2) the proof of sublinear variance under $2+\log$ moments of passage times and (3) the role of growth and competition models. We also provide a collection of (old and new) open questions, hoping to solve them before the 100th birthday.  \end{abstract}

\tableofcontents

\newpage
\section{Introduction}

\subsection{The model of first passage percolation and its history}

First passage percolation (FPP) was originally introduced by Hammersley and Welsh \cite{HW} in 1965 as a model of fluid flow through a random medium. It has been a stage of research for probabilists since its origins but despite all efforts through the past decades, most of the predictions about its important statistics remain to be understood. Most of the beauty of the model lies in its simple definition (as a random metric space) and the property that several of its fascinating conjectures do not require much effort to be stated. During these $50$ years, FPP brought attention of theoretical physicists, biologists, and computer scientists and also gave birth to some of the most classical tools in mathematics, the sub-additive ergodic theorem as one of the main examples. Here, we will focus on the model defined on the lattice $\mathbb Z^d$; some variants will be discussed in Section \ref{sec:alltheotherthings}.

The model is defined as follows. We place a non-negative random variable $\tau_e$, called the {\it passage time} of the edge $e$, at each nearest-neighbor edge in $\Z^d$. The collection $(\tau_e)$ is assumed to be independent, identically distributed with common distribution $F$ and probability measure~$\nu$. 
The random variable $\tau_e$ is interpreted as the time or the cost needed to traverse edge~$e$.

 A {\it path} $\Gamma$ is a finite or infinite sequence of edges $e_1, e_2, \ldots$ in $\Z^d$ such that for each $n \geq 1$, $e_n$ and $e_{n+1}$ share exactly one endpoint. For any finite path $\Gamma$ we define the {\it passage time} of $\Gamma$ to be
\[
T(\Gamma)=\sum_{e \in \Gamma} \tau_e.
\]
Given two points $x,y \in \mathbb{R}^d$ one then sets
\begin{equation}\label{definition:passagetime}
T(x,y) = \inf_{\Gamma} T(\Gamma),
\end{equation}
where the infimum is over all finite paths $\Gamma$ that contain both $x'$ and $y'$, and $x'$ is the unique vertex in $\mathbb{Z}^d$ such that $x \in x' + [0,1)^d$ (similarly for $y'$). The random variable $T(x,y)$ will be called the passage time between points $x$ and $y$. In the original interpretation of the model, $T(x,y)$ represents the time that a fluid with source in $x$ takes to reach a location $y$. 

For each $t \geq 0$ let 
\[
B(t) = \{y \in \mathbb{R}^d~:~T(0,y)\leq t\} .
\]
In the case that $F(0)=0$, the pair $(\mathbb{Z}^d,T(\cdot, \cdot))$ is a metric space and $B(t) \cap \mathbb {Z}^d$ is the (random) ball of radius $t$ around the origin. The ultimate goal of first passage percolation is to understand this metric as the observer moves away from $\mathbb Z^d$ or as we make the Euclidean length of the edges small. A variety of questions comes to mind almost immediately. We write $|\cdot|$ for the $\ell^1$ norm in $\mathbb R^d$.

\begin{enumerate}
\item What is the typical distance between two points that are far from each other in the lattice? Or in other words, what can we say about $T(x,y)$ as $|x-y| \to \infty$? Does it converge? What is the rate of convergence?
\item How does a ball of large radius look? Do we have a scaling limit and fluctuation theory for the set $B(t)$? 
\item What is the geometry of geodesics (time-minimizing paths) between two distant points? How different they are from straight lines?  
\item What role does the distribution of the passage times play in describing the metric?
\end{enumerate}

In this manuscript, we will discuss progress on these and related questions. The purpose is twofold. First, we hope that this set of notes will serve as a quick guide for readers who are not necessarily experts in the field. We will try to provide not only the main results but also the main techniques and a large collection of open problems. Second, the field had a  burst of activity in the past $5$ years and the most complete survey is more than a decade old. We hope that these notes will fill this gap. At least, we hope to share some of the beautiful mathematical ideas and constructions that arise through FPP and which have enchanted many throughout these years.

\begin{figure}[ht]
\centering
\subfigure[$t=0$]{%
  \includegraphics[scale=0.2]{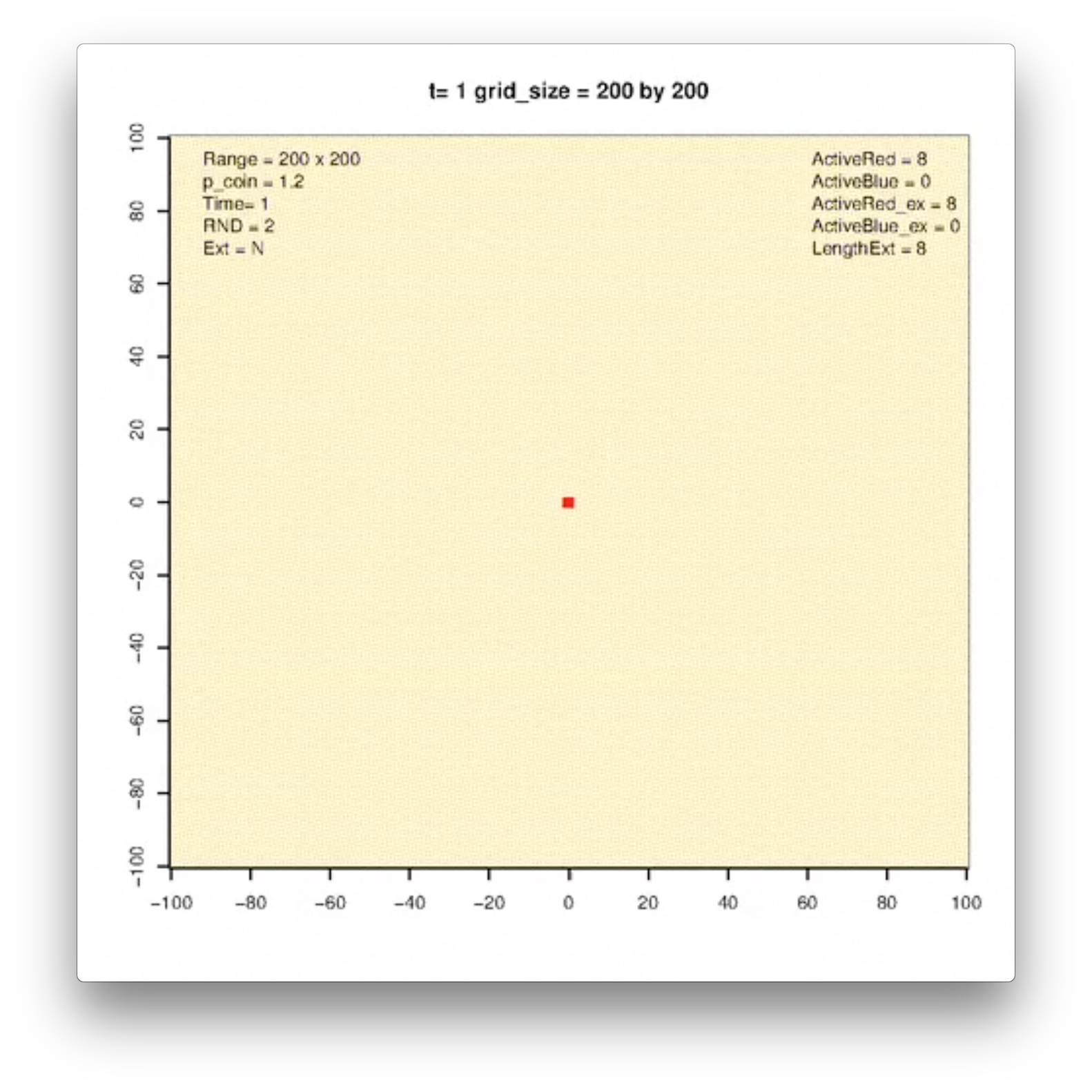}
  \label{fig:subfigure1}}
\quad
\subfigure[$t=70$]{%
  \includegraphics[scale=0.2]{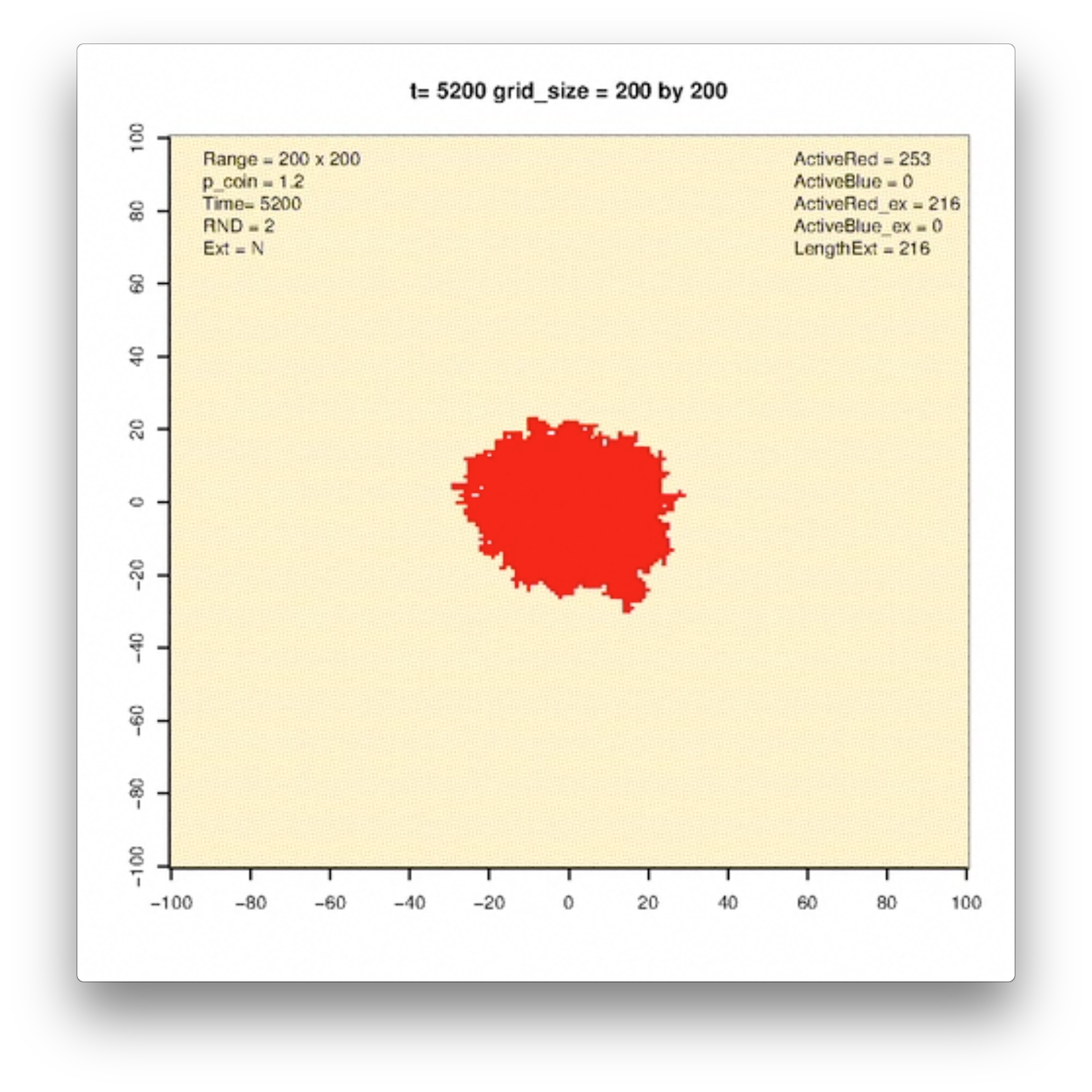}
  \label{fig:subfigure2}}
\subfigure[$t=135$]{%
  \includegraphics[scale=0.2]{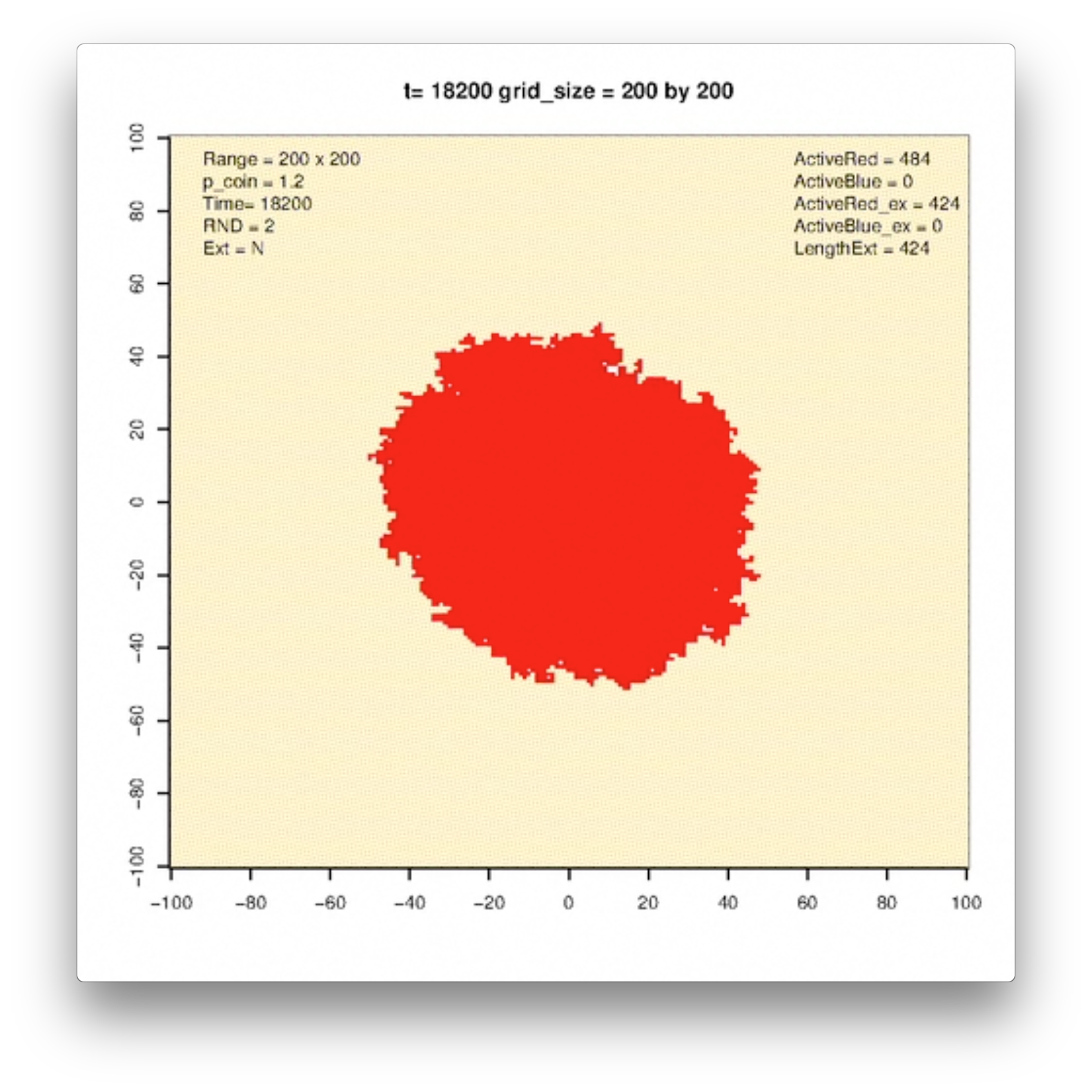}
  \label{fig:subfigure3}}
\quad
\subfigure[$t=250$]{%
  \includegraphics[scale=0.2]{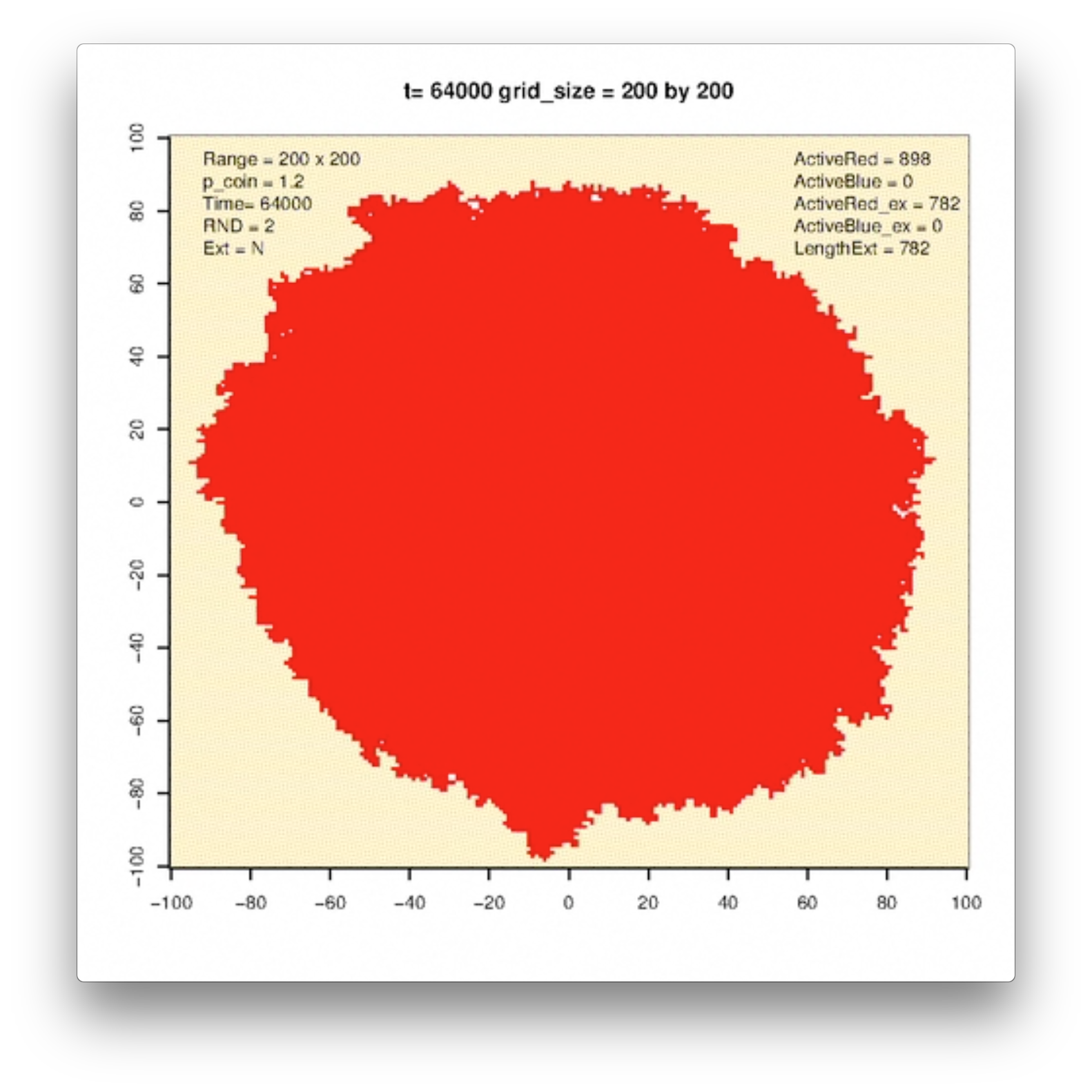}
  \label{fig:subfigure4}}

\caption{Simulation of the ball $B(t)$ at $t=0, t= 70, t= 135$ and $t=250$. The passage times $\tau_{e}$ have exponential mean $1$ distribution. Simulations by Si Tang.}
\label{fig:figure}
\end{figure}

Let's now go back to questions $1$ to $4$. The original paper of Hammersley and Welsh \cite{HW} considered question $1$ for a class of passage times in $\mathbb Z^2$. If we write $e_1$ for the first coordinate vector, they showed that $T(0,ne_1)$ grows linearly in $n$. Their result was extended in the famous paper of Kingman \cite{Burkholder, Kingman, Kingman73}. It was also the building block for the classical ``shape theorem'' of Richardson \cite{Richardson}, improved by Cox and Durrett \cite{CoxDurrett} and Kesten \cite{Kesten} that gives the analogue of the law of large numbers for the random ball $B(t)$.  It roughly says that $B(t)$ grows linearly in $t$ and, when normalized, it converges to a deterministic subset $\mathcal B_\nu$ of $\mathbb R^d$, called the limit shape. The set $\mathcal B_\nu$ is not universal and depends on the distribution of the passage times.  Section \ref{sec:limitshape} is devoted to explaining the  shape theorem and certain properties of the limit shape $\mathcal B_\nu$.

In Section \ref{sec:fluctuations} we discuss the variance and the order of fluctuations of the passage time $T$.  In two dimensions, it is expected that under certain assumptions on $\mathcal B_\nu$ the fluctuations are governed by the predictions of physicists, including Kardar, Parisi, Zhang \cite{KPZ,K, Krug, DD}. In higher dimensions, the picture is less clear and some of the predictions disagree. After stating what is conjectured, we focus our attention on presenting a simple proof of sub linear variance valid under minimal assumptions on the passage time. 

Section \ref{sec:geodesics} is devoted to the study of geodesics. We briefly discuss the existence of finite geodesics between any two points, then move to the study of geodesic rays. We present results on coalescence, directional properties and we sketch the proof of the absence of geodesic lines (or bigeodesics) in the upper half plane. The important connection between geodesic lines and ground states of the two-dimensional Ising ferromagnet is also presented.

In Section \ref{sec:Busemann} we describe the modern role of Busemann functions. We explain a beautiful argument by Hoffman for the existence of $2$ or more geodesic rays. 
We then focus on Busemann function limits  and the relation to limiting geodesic graphs.  Section \ref{sec:growth}  tries to cover the vast relation between FPP, growth processes and infection models. We focus on questions of coexistence of multiple species and the limiting interface.  

Section \ref{sec:alltheotherthings} is our attempt to show the reader what this survey is not about. In the literature, there are thousands of pages of related (equally fascinating) questions and models, similar to or inspired by FPP. We collect a few of these directions and try to point the right references. In particular, we briefly discuss FPP on different graphs, the maximum flow problem, the exactly solvable models for last passage percolation and the positive temperature version of the model. Section~\ref{sec:open} is just a recollection of the open questions spread throughout this manuscript for easy reference.

\subsection{Acknowledgments} We thank the American Institute of Mathematics  and its staff for helping us organize a workshop on this subject. A.A. thanks the hospitality of the Mathematics department at Indiana University, and its visiting professor program, where part of this survey was completed. He also thanks Elizabeth Housworth who, introduced him to the japanese eraser and chalk dust cleaner. 
M.D. thanks the hospitality of the Mathematics Department at Northwestern University. We thank Phil Sosoe for several fruitful discussions on this topic, especially on concentration estimates. We thank Si Tang for making the simulations used in Figure 1 available to us, Sven Erick Alm and Maria Deijfen for the simulations and Figures 2, 4, and 5 and Xuan Wang for discussions on non-random fluctuations. We also thank Daniel Ahlberg, Wai-Kit Lam, and Phil Sosoe for spotting typos and giving comments.


\newpage

\section{The time constant and the limit shape}\label{sec:limitshape}

\subsection{Subadditivity and the time constant} 
The first-order of growth of the passage time $T(0,ne_1)$ is described by the following theorem.
\begin{theorem}[Theorem 2.18 in \cite{KestenAspects}] \label{thm:timeconstant} Assume that 
\begin{equation}\label{mincondition} \E \min[t_1, \ldots, t_{2d}] < \infty 
\end{equation} where $t_i's$ are i.i.d copies of $\tau_e$. Then there exists a constant $\mu(e_1) \in [0,\infty)$ (called the time constant)  such that 
\begin{equation*}
\lim_{n \rightarrow \infty} \frac{T(0,ne_1)}{n} = \mu(e_1) \quad \text{ a.s. and in } L^1.
\end{equation*}
\end{theorem}

The proof of Theorem \ref{thm:timeconstant} is a classic application of the subadditive ergodic theorem that we now state. The version that we write here is due to Liggett \cite{LiggettST} and suffices for our purposes. Several versions, with different hypotheses, including the one with Kingman's original assumptions can be found in \cite{LiggettST}. 

\begin{theorem}[Subadditive Ergodic Theorem \cite{LiggettST}] \label{subadditivity}
Let $(X_{m,n})_{0\leq m <n}$ be a family of random variables that satisfies:
\begin{enumerate}
\item[$(a)$] $X_{0,n} \leq X_{0,m}+X_{m,n}$, for all $0< m < n$.
\item[$(b)$] The distribution of the sequences $(X_{m, m+k})_{k\geq 1}$ and $(X_{m+1,m+k+1})_{k\geq1}$ is the same for all $m\geq 0$.
\item[$(c)$] For each $k \geq 1$, the sequence $(X_{nk, (n+1)k})_{n\geq 0}$ is stationary.
\item[$(d)$] $\E X_{0,1} < \infty$ and $\E X_{0,n} > -cn$ for some finite constant $c$.
\end{enumerate}
Then 
\begin{equation}\label{eq:ergodicthm}
\lim_{n \to \infty} \frac{X_{0,n}}{n} \text{ exists  a.s. and in } L^1.
\end{equation}
Furthermore, if the stationary sequence in $(c)$ is also ergodic, then the limit in \eqref{eq:ergodicthm} is constant almost surely and equal to
$$ \lim_{n\to \infty } \frac{\E X_{0,n}}{n} = \inf_{n} \frac{1}{n} \E X_{0,n}.$$
\end{theorem}

One can find different proofs of the subadditive ergodic theorem in the literature and some of them are in standard books of probability theory, for instance \cite[Section 6.6]{Durrettbook}. We will visit and discuss the original proof given by Kingman in Section \ref{sec:visit}. We now show how Theorem~\ref{thm:timeconstant} is an easy consequence of Theorem \ref{subadditivity}.

\begin{figure}
\scalebox{1.0}{ 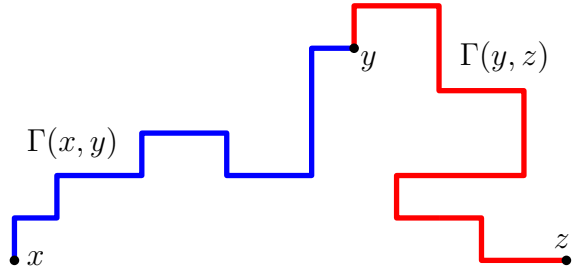}
 \centering
 \def \svgwidth{3000pt}
 \caption{The geodesics $\Gamma(x,y)$ and $\Gamma(y,z)$ and their concatenation. As a geodesic from $x$ to $z$ does not need to pass through $y$, we have $T(x,z)\leq T(x,y) + T(y,z)$.  }
 \label{fig:easy}
\end{figure}

\begin{proof}[Proof of Theorem \ref{thm:timeconstant}] We apply the sub-additive ergodic theorem to $X_{m,n} = T(me_1, ne_1)$. It is not difficult to verify conditions $(a)$ to $(d)$. To check $(a)$ (which is just the triangle inequality) note that a path from $0$ to $ne_1$ does not necessarily need to go through $me_1$ while a concatenation of paths from $0$ to $me_1$ and from $me_1$ to $ne_1$ gives us a path from $0$ to $ne_1$. Therefore,  see Figure \ref{fig:easy}, $$T(0,ne_1) \leq T(0,me_1)+T(me_1,ne_1).$$ Items $(b)$, $(c)$ and ergodicity follow directly from the fact that the environment is i.i.d., thus it is invariant under horizontal shifts of $\Z^d$. $\E X_{0,n} > -cn$ holds as passage times are non-negative. The rest of item $(d)$ follows from assumption \eqref{mincondition} as there are $2d$ disjoint deterministic paths $\Gamma_1, \ldots, \Gamma_{2d}$ in $\mathbb Z^d$ joining $0$ to $e_1$ and therefore
$$T(0,e_1) \leq \min\{T(\Gamma_1), \ldots, T(\Gamma_{2d})\}.$$ Order them in such a way that $\Gamma_1$ is the path with the largest number of edges and denote by $B$ the number of edges of $\Gamma_1$. Then 
$$\Pro(T(0,e_1)>s)\leq \prod_{i=1}^{2d} \Pro(T(\Gamma_{i})>s) \leq \Pro(T(\Gamma_1)>s)^{2d},$$
and
$$ \Pro (T(\Gamma_1)>s) \leq B\Pro (\tau_e > s/B).$$
The second inequality comes from the fact that if $T(\Gamma_1)>s$ then at least one of the edges in $\Gamma_1$ must have a passage time larger than $s/B$. Combining the previous two inequalities and setting $Y= \min[t_1, \ldots, t_{2d}]$ we obtain 
\begin{equation}\label{eq:moments}
 \Pro (T(0,e_1)>s) \leq B^{2d}\Pro(\tau_e>s/B)^{2d} = B^{2d} \Pro(Y>s/B)
 \end{equation} which proves the desired result.
\end{proof}

In fact, the following lemma comes directly from \eqref{eq:moments} and the observation that any path from $0$ to $x$ must contain at least one edge incident to $0$.
\begin{lemma}\label{lemma:moments}Let $k\geq 1$. Then $\E \min[t_1, \ldots, t_{2d}]^{k} <\infty$ if and only if  $\E T(0,x)^k<\infty$ for all $x \in \mathbb Z^d$.  
\end{lemma}

The convergence in probability of the normalized passage time was first proved in two dimensions in the original paper of Hammersley and Welsh \cite{HW} under the assumption of finite mean of the random variable $\tau_e$. Still under the assumption of $\mathbb E \tau_e < \infty$, this result was strengthened to almost sure and $L^1$ convergence by Kingman, using his sub-additivity theorem. 

The condition \eqref{mincondition} is necessary to have almost sure or $L^1$ convergence.  Indeed, if \eqref{mincondition} does not hold, denoting by $\tau_1^n, \ldots, \tau_{2d}^n$ the edge weights of edges incident to the vertex $2ne_1$, then for any $c>0$,  the events
$$ A_n := \{ \min [\tau_1^n, \ldots, \tau_{2d}^n] > cn  \} $$ 
are independent and satisfy 
$$ \sum_n \Pro (A_n) = \infty.$$
Thus, an application of Borel-Cantelli Lemma shows that with probability one $\limsup_n T(0,ne_1)/n > c$. 

Without assuming \eqref{mincondition}, Kesten \cite[Page 137]{KestenAspects}, Cox-Durrett \cite{CoxDurrett} and Wierman \cite{Wierman} establish the existence of a constant $\tilde \mu(e_1)$ such that $$\lim_{n \to \infty} T(0,ne_1)/n= \tilde \mu(e_1) \text{ in probability.}$$  Clearly, if \eqref{mincondition} holds, then $\tilde \mu(e_1) = \mu(e_1)$.

We now gather more information on the time constant $\mu(e_1)$. It is clear that $\mu(e_1)$ satisfies
$$ 0 \leq \mu(e_1) \leq \E \tau_e,$$
by considering the direct path from $0$ to $ne_1$ and using the law of the large numbers.

That strict inequality does not always hold is seen by taking  $\tau_e=1$ almost surely. However, if the distribution $F$ of the passage times has at least two points in its support, we can prove: 
\begin{theorem}[Hammersley-Welsh \cite{HW}] If $F$ is not a trivial distribution, we have $\mu(e_1)~<~\E \tau_e$.
\end{theorem}
\begin{proof} Choose $a < b$ such that $0<F(a)\leq F(b) < 1$. Pick $n >  2a/(b-a)$. Let $e(1), \ldots, e(n)$ be the $n$ edges from $0$ to $ne_1$ and $f(1), \ldots, f(n+2)$ the edges from $0$ to $e_2$, $e_2$ to $e_2 +ne_1$, and $ e_2 +ne_1$ to $ne_1$. On the positive probability event 
$$ \{ \tau_{e(i)} \geq b, \tau_{f(j)}\leq a, \quad \forall i,j \}$$
one has $T(0,ne_1)< \tau_{e(1)} + \ldots + \tau_{e(n)}$. Thus, $\E T(0,ne_1) < n \E \tau_e$ and $$\mu(e_1) \leq \frac{\E T(0,ne_1)}{n}< \E \tau_e.$$
\end{proof}

Let's now look at lower bounds for the time constant. If $F(0)>0$; that is, if we have edges that are cost-free to cross, one may wonder if $\mu$ is equal to $0$ and the growth of the time constant is in fact not linear. This issue is handled in the next theorem. Let $p_c(d)$ be the critical probability for bond percolation in $\mathbb Z^d$.  

\begin{theorem}[Theorem 6.1 in \cite{KestenAspects}] \label{thm:percolate} For FPP on $\mathbb Z^d$, 
$\mu(e_1) > 0$ if and only if $F(0)<p_c(d)$. \end{theorem}

We now make a very simple but important remark. One can extend Theorem \ref{thm:timeconstant} with a similar proof to arbitrary directions with rational coordinates. Then, we define a homogeneous function $\mu:\mathbb Q^d\rightarrow \R$ such that, for any $x \in \mathbb Q^d$,
\begin{equation}\label{def:mu}
\lim_{n \to \infty} \frac{T(0,nx)}{n} = \mu(x) \quad \text{ a.s. and in } L^1.
\end{equation}

The reader should see \eqref{def:mu} as the analogue of a law of large numbers. In Section \ref{sec:fluctuations}, we will discuss the fluctuations of $T(0,nx)$ around $n\mu(x)$, and we will discover that, in general, a central limit theorem with Gaussian fluctuations does not hold. Before moving to the study of fluctuations, we continue to gather more information on $\mu(x)$. It is not difficult to establish the following properties of $\mu$ for $x,y \in \mathbb Q^{d}$ and $c \in \mathbb Q$:

\begin{enumerate}
\item $\mu(x+y) \leq \mu(x) + \mu(y).$
\item $\mu(cx)=|c|\mu(x)$
\item $\mu$ is invariant under symmetries of $\mathbb Z^{d}$ that fix the origin.
\item $\mu$ is uniformly continuous and Lipschitz on bounded subsets of $\mathbb Q^{d}$, so it has a unique continuous extension to $\mathbb R^{d}.$
\end{enumerate}

Furthermore, these properties imply the following:

\begin{theorem}
$\mu(e_1)>0$ if and only if $\mu(x)>0$ for all $x \in \mathbb R^d \setminus \{ 0 \}$.
\end{theorem}

Subadditivity is a nice argument to show the existence of the time constant. However it gives no insight, nor closed expression for $\mu$.  The determination of $\mu(e_1)$ as a function of $F$ is a fundamental and difficult problem in first passage percolation. 
\begin{question}\label{q1}
Find a non-trivial explicit distribution for  which we can actually determine $\mu(e_1)$. 
\end{question}

Although fundamental and old, the question above is perhaps not as interesting as one addressing geometric properties of the limit shape discussed later in this section. Furthermore, it is not even clear what one should mean by ``determine'' $\mu(e_{1})$.  In FPP, one does not know any distribution that allows exact computations, and maybe there are none. In Section \ref{sec:alltheotherthings}, we discuss similar solvable models where explicit computations are possible. A more interesting related question at this point is: 

\vspace{0.2cm}
 
{\it Given distributions $F$ and $\tilde F$, how different are their respective time constants $\mu(e_{1})$ and $\tilde \mu(e_{1})$? }

\vspace{0.2cm}

In general, we lack strong information about how the limit shape changes under small perturbations of the edge-weight distribution. If one could derive strong results in this direction, perhaps the establishment of various conjectures about the limit shape (e.g., curvature) could be made easier, or reduced to finding some special class of distributions for which the properties are explicitly derivable. The best current results on stability, dating back over thirty years, say simply that the time constant is a continuous function of the edge-weight distribution.

\begin{theorem}[Cox-Kesten \cite{CoxKesten81}, Kesten \cite{KestenAspects}]
\label{thm:contcox}
  The time constant $\mu(e_1)$ is continuous under weak convergence of i.i.d. distribution. That is, if $(F_n)_n$ is a sequence of distribution functions for the edge weight $\tau_e$ with $F_n \Rightarrow F$, and if $(\mu_n(e_1)),$ $\mu(e_1)$ denote the respective time constants, then
\[\lim_n \mu_n(e_1) = \mu(e_1)\ . \] 
\end{theorem}

While the preceding result is not strong enough to preserve curvature, it does guarantee a certain semicontinuity property of the set of extreme points of $\mathcal{B}$. In \cite{DH}, this was used to establish the existence of  limit shapes with arbitrarily many extreme points for some nonatomic edge weight distributions; improvements to Theorem \ref{thm:contcox} could be useful for similar constructions. One existing improvement of Theorem  \ref{thm:contcox} is the recent work of Garet, Marchand, Procaccia and Th\'eret \cite{GMPT}, which establishes an analogous continuity result in the case that the edge weights are allowed to assume the value $\tau_e = +\infty$.

On the other hand, when comparing distributions $F,\, \tilde F$ which obey certain stochastic orderings, much more can be said\cite[Theorem 3]{CoxKesten81} (see also \cite[Section 6.4]{HW} and \cite{Cox80}). When $\tilde F(t) \leq F(t)$ for all $t \in \mathbb R$, it is possible to provide an easy answer to this question, as we can straightfowardly couple the passage times to obtain $ \mu(e_{1}) \leq \tilde \mu(e_{1})$.
The question above was considered by several authors \cite{KestenAspects, vdbk, Marchand}. A gorgeous answer came with the work of van den Berg and Kesten \cite{vdbk} proving the strict inequality  $\mu(e_{1}) < \tilde \mu(e_{1})$ if $F$ is strictly more variable than $\tilde F$. Here we follow an extension of the van den Berg - Kesten comparison theorem provided by Marchand \cite{Marchand}.

\begin{definition} Let $F$ and $\tilde F$ be two distributions on $\mathbb R$. We say that $F$ is more variable than $ \tilde F$ if for every concave increasing function $\phi:\R \to \mathbb R$, 
\[ \int \phi \; d  F \leq \int \phi \; d \tilde F \]
when the two integrals exist.  If in addition $F\neq \tilde F$ then we say that  $F$ is strictly more variable than $\tilde F$.
\end{definition}

\begin{example}
If $\tilde F$ stochastically dominates $ F$ then $F$ is more variable than $\tilde F$. \end{example}
\begin{example}
For any non-negative random variable $X$ and any constant $M$ the distribution of $X \wedge M$ is more variable than the distribution of $X$. 
\end{example}

\begin{example}
For $0<p<1$ and $\rho$ a probability measure, consider the probability measures $\tilde \nu = p\delta_{a} + (1-p)\rho$ and $\nu = (p/2)\delta_{a-\epsilon} + (p/2)\delta_{a+\epsilon} + (1-p)\rho$. Writing $F_{\pi}$ for the probability distribution associated to a probability measure $\pi$,  then $F_{\nu}$ is strictly more variable than $F_{\tilde \nu}$. 
\end{example}

\begin{theorem}[van den Berg-Kesten \cite{vdbk}, Marchand \cite{Marchand}]\label{thm:Marchand}
Assume that $d=2$ and suppose that $F(0)< p_c$. If $F$ is strictly more variable than $\tilde F$ then $ \mu(e_{1}) < \tilde \mu(e_{1})$.
\end{theorem}

For $x\neq e_{1}$ a version of $\mu(x) < \tilde \mu(x)$ can also be found in \cite{Marchand,vdbk}.
Van den Berg-Kesten-Marchand's comparision theorem also comes with a nice remark that gets rid of some na\"ive intuition. Suppose that the distribution of the passage time has unbounded support. One may (wrongly) imagine that there exists a threshold $M$ such that an optimal path between $0$ and $ne_{1}$, $n$ large,  never takes a linear fraction of edges with weights above $M$. However, the theorem above implies that, if one truncates the passage time at level $M$, the time constant $\mu_M(e_1)$ of the truncated model is strictly less than the original $\mu(e_1)$. Thus, it is more efficient for the model to use a certain positive proportion of edges with very large passage time than to try to always avoid them.

This remark leads also to the following question, which seems to be open. 

\begin{question} Suppose that the support of the distribution of $\tau_e$ equals $\mathbb R^+$. 
Let $$X_n = \max \{ \tau_e: e \text{ is in a geodesic from } 0 \text{ to } ne_1 \}.$$ How does  $X_n$ scale with $n$?
\end{question}

The assumption $F(0)< p_c$ in Theorem \ref{thm:Marchand} cannot be removed because of Theorem \ref{thm:percolate}. Now, let $r$ be the infimum of the support of $F$. In higher dimensions, a  version of the theorem above is known \cite{vdbk} if $F(r)<\vec p_c(d)$, where $\vec p_c(d)$ is the critical probability for directed edge percolation.   The only missing case at this point is:

\begin{question}\label{q:q3}
Extend the comparision theorem to the case $d>2$, $r>0$ and $\vec p_c(d)\leq F(r)$.
\end{question}

\begin{remark} Marchand's original approach does not work in higher dimensions, as she uses large deviations for supercritical oriented percolation available only in dimension $2$ (see \cite[Section 7]{Marchand}) at the time of her paper (see the estimate in \cite[Page 1014]{Marchand}). New estimates were obtained in $d>2$ in \cite[Proposition 2]{GM4} recently. These combined with Marchand's arguments may provide a solution to Question \ref{q:q3}.
\end{remark}
\subsection{The time constant through a homogenization problem}\label{sec:3243}

Another way to interpret the time constant was recently explored  by Krishnan \cite{Arjun} in FPP and by Georgiou, Rassoul-Agha and Sepp\"{a}l\"{a}inen \cite{GRS} in last-passage percolation. We briefly describe it here. 

The idea is to interpret the passage time as a solution of an optimal-control problem. Define
   \begin{equation*}
     A := \{\pm e_1,\ldots,\pm e_d\}.
   \end{equation*} We think of $A$ as the collection of possible directions to exit a vertex. We now write $\tau(v,\alpha)$ to refer to the weight $\tau_e$ at $v \in \Z^d$ along the direction $\alpha \in A$. It is now possible to check that
 $$ T(0,x) = \inf_{\alpha \in A} \{ T(0,x+\alpha) + \tau(x,\alpha) \}. $$
This suggests that one could think of the problem as a homogenization problem for metric Hamilton-Jacobi equations in $\mathbb R^d$. The advantage of such a perspective is to allow us to use the work of Lions-Souganidis  \cite{LS, LS2}  on certain homogenization problems in $\R^d$ to give a different characterization of the time constant. To see this we need a few definitions.

\begin{definition}
  For a function $f:\Z^d \to \R$, let
  \[ Df (x,{\alpha}) = f(x+\alpha) - f(x) \]
    be its discrete derivative at $x$ in the direction $\alpha \in A$. 
\end{definition}

We write $\Omega$ for the probability space and $\sigma_v:\Omega \rightarrow \Omega$ for the shift that translates the random variables $\tau_e$ by $v$.  
Define \begin{equation*}
  \mathcal F:= \left\{ 
    f:\Z^d \times \Omega \to \R ~\left|~ 
  \begin{split}  
      & Df(x + v,\alpha) (\omega) = Df(x, \alpha)(\sigma_v \omega), ~\forall v, x \in \mathbb Z^d, \forall \omega \in \Omega \\
      & \E[Df(x,\alpha)] = 0 ~\forall x \in \Z^d \text{ and } \alpha \in A 
\end{split} 
\right\} \right. .
\end{equation*}

For $p \in \mathbb R^d$, $x \in \mathbb Z^d$ and $f \in \mathcal F$ define
\begin{equation}
  H(f,p,x) = \sup_{\alpha \in A} \left\{ \frac{-Df(x,\alpha) - ( p, \alpha )}{\tau(x,\alpha)} \right\} 
  \label{eq:discrete-hamiltonian-1},
\end{equation}
where $( \cdot, \cdot )$ is the standard inner product in $\mathbb R^d$.
The following is the main result of \cite{Arjun}.

\begin{theorem}
 Assume that the passage times are bounded and bounded away from $0$, that is, there exist $m,M$ such that $0 < m < \tau_e <M$ almost surely. Then $\mu(x)$ solves the following Hamilton-Jacobi equation
  \begin{equation*} 
    \begin{split}
      \bar H(D\mu(x)) & = 1,\\
      \mu(0)     & = 0
    \end{split}   \end{equation*}
where $\bar H(p)$ is a convex, coercive, Lipschitz continuous function given by
    \begin{equation*}
   \bar H(p) = \inf_{f \in \mathcal F} \esssup_{\omega \in \Omega} \sup_{x \in \Z^d} H(f,p,x)(\omega) . 
    \end{equation*}
Furthermore, $\bar H(p)$ is the dual norm of $\mu(x)$ on $\mathbb R^d$, defined by  
$$ \bar H(p) = \sup_{\mu(x)=1} (p,x ).$$
\end{theorem}

Although the theorem above gives a different characterization for the time constant,  it has not yet been used to get a better understanding of $\mu(e_1)$ as a function of $F$. An algorithm for finding a minimizer for the above variational formula is explained in part II of the  Ph.D. thesis of Krishnan \cite{KrishnanPhD}. It would be extremely nice to extend these ideas to say more geometric information on the limit ball (see questions in the next subsections). 

\subsection{The limiting ball: Cox-Durrett shape theorem}\label{subsec:limitshape}

For each unit vector $x \in \mathbb R^d$ we define the time constant $\mu(x)$ in direction $x$ through \eqref{def:mu}. In this section, we will see how the function $\mu$ describes the first order approximation of the random ball $B(t)$ as $t$ goes to infinity. The main result of the section is the world famous shape theorem, Theorem \ref{thm:limitshape}. 

Let $\mathcal{M}$ be the set of Borel probability measures on $[0,\infty)$ satisfying
\begin{equation}\label{eq:conditionmomentsLS}
\E \min \{ t_1^d, \ldots, t_{2d}^d\} <\infty,
\end{equation}
where $t_i, i=1, \ldots 2d$, are independent copies of $\tau_e$ and with 
\begin{equation}\label{eq:conditionmomentsLS2}
F(0)<p_c(d),
\end{equation} where $p_c(d)$ is the threshold for bond percolation in $\Z^d$. If $S$ is a subset of $\mathbb{R}^d$ and $r\in \mathbb{R}$ we write $rS = \{rs~:~s \in S\}$.

\begin{theorem}[Cox and Durrett \cite{CoxDurrett}]\label{thm:limitshape}
For each $\nu \in \mathcal{M}$, there exists a deterministic, convex, compact set $\mathcal B_\nu$ in $\mathbb{R}^d$ such that for each $\varepsilon>0$,
\begin{equation}\label{eq:limitshapeeq}
\mathbb{P}\left( (1-\varepsilon)\mathcal B_\nu \subset \frac{B(t)}{t} \subset (1+\varepsilon) \mathcal B_\nu \text{ for all large } t \right) = 1\ .
\end{equation}
Furthermore, $\mathcal B_{\nu}$ has non-empty interior and is symmetric about the axes of $\mathbb R^{d}.$
\end{theorem}

\begin{remark}
If \eqref{eq:conditionmomentsLS2} does not hold, edges with zero passage time percolate, creating several instantaneous 'highways'. In this case, Theorem \ref{thm:percolate} says that the time constant $\mu(e_1)=0$. As for the limit shape, one has $\mathcal B_\nu=\mathbb R^d$ \cite[Theorem 1.10]{KestenAspects}. Precisely, under  assumption  \eqref{eq:conditionmomentsLS},  we have $F(0) \geq p_c(d)$ if and only if for every $M >0$
$$ \{ x:|x| \leq M \} \subset \frac{B(t)}{t}  \quad \text{ eventually w.p. } 1.$$
\end{remark}

\begin{remark}\label{rem:itblowsup!} For $x \in \mathbb Z^d$, let $m(x)$ be the minimum of  $\tau_e$ over all edges $e$ incident to $x$. 
If \eqref{eq:conditionmomentsLS} fails, then for any $C>0$, $$ \sum_{x \in \mathbb Z^d} \Pro \bigg( m(x) > C|x| \bigg) = \infty.$$
Now note that $T(0,x) \geq m(x)$ and the random variables $\{ m(x): x \in (2\mathbb Z)^{d}\}$ are independent. Then we can apply the Borel-Cantelli Lemma to obtain
$$\frac{T(0,x)}{|x|} > C \text{ for infinitely many } x \in \mathbb Z^d, \text{ a.s.} $$ Thus, if \eqref{eq:conditionmomentsLS} fails, \eqref{eq:limitshapeeq} also does not hold.
\end{remark}

\begin{remark}
With Theorem \ref{thm:limitshape}, Cox and Durrett provided sufficient and necessary conditions for the  shape theorem to hold in FPP. The first shape theorem, however, was proven in the seminal work of Richardson \cite{Richardson} in 1973. We will come back to discuss Richardson's model in Section \ref{sec:growth}.
\end{remark}

\begin{figure}[h]
\centering
 \includegraphics[scale=0.77]{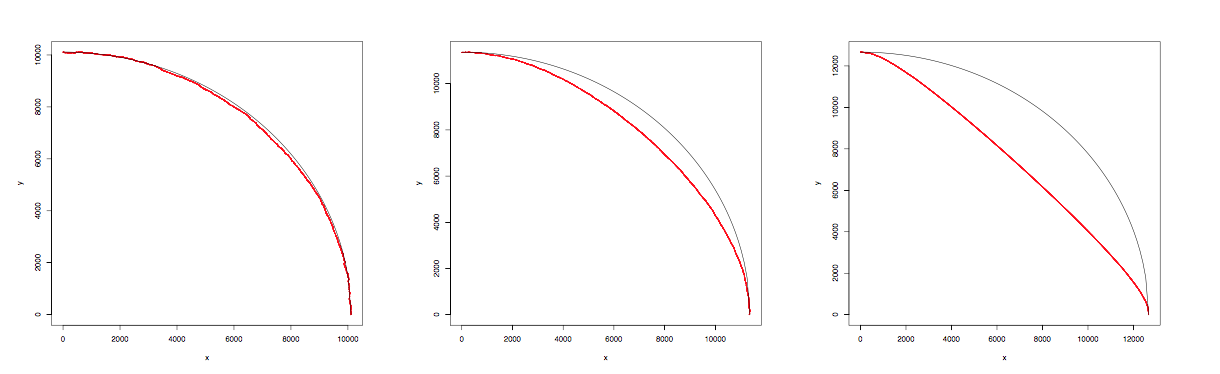}
 \caption{In red, simulation of the $2$ dimensional ball $B(t)$, intersected with the first quadrant, when $t=15000$. The passage times $\tau_{e}$ are distributed according to an exponential random variable of mean $1$ plus a constant $C$.  Here, $C=0$ (left), $C=0.5$ (middle), and $C=4$ (right). In blue, the shape of a circle of radius $T(0,te_{1})$. Simulations and figures are courtesy of S.-E. Alm and M. Deijfen.}
 \label{FigureAlm}
\end{figure}

The idea of the proof of Theorem \ref{thm:limitshape} is to first use subadditivity to demonstrate the linear growth of $B(t)$ in a fixed rational direction. This implies that, with probability one, we have the right growth rate in a countable dense set of directions simultaneously. To obtain the full result from this, we need some bound which allows us to interpolate between these directions, insuring that the convergence occurs along all rays with probability one.

There are different ways to implement this interpolation step. Here we sketch a method which first appeared in \cite{hansonthesis} for establishing the interpolative bound just mentioned. This method has the advantage of applying (with modifications) to the stationary ergodic case of FPP, and is inspired by \cite{B90}. This proof method was further used for the Busemann shape theorem (see Lemma~\ref{thm:Busemannlimitshape}) in \cite{DHanson}. 

\begin{lemma}[Difference estimate] \label{lem:estimate}
Let $\nu \in \mathcal{M}$. Then there exists a constant $\kappa < \infty$ such that, for any $x \in \mathbb{Z}^d$,
\begin{equation}\label{eq: maximal_inequality}
\mathbb{P}\left(\sup_{\substack{z \in \mathbb{Z}^d\\ z \neq x}}
\frac{T(x,z)}{\|x - z\|_1} < \kappa \right) > 0\ .
\end{equation}
\end{lemma}
\begin{proof}[Idea of proof]
This proof follows the line of a similar estimate in \cite{CoxDurrett}. If $\mathbb{E} \tau_e^d < \infty$, the proof follows by noting that there exist $2d$ edge-disjoint paths $\{r_i\}$ between $x$ and $z$, each of length order $\|x-z\|_1$. For the event $\{T(x,z) > \kappa \|x - z\|_1\}$ to occur, each of these paths must have $T(r_i) > \kappa \|x-z\|_1$. Using standard estimates for i.i.d. sums, this probability is small for $\kappa$ sufficiently large--in fact, we can get an estimate which is summable in $z$. Using the Borel-Cantelli Lemma (and adjusting the constant $\kappa$ upwards if necessary) allows us to complete the proof.

In the general case, we consider a sparse lattice $C\mathbb{Z}^d$ such that between ``neighboring'' vertices of this lattice, there exist $2d$ disjoint paths lying inside cells of side length of order $C$. In particular, paths corresponding to well-separated pairs of ``neighboring'' vertices are disjoint, and their passage times are independent. The smallest passage time among these $2d$ paths is clearly an upper bound for the passage time between ``neighboring'' vertices; we treat this as the passage time of a ``renormalized'' edge of the sparse lattice. Mimicking the argument of the preceding paragraph, then extending the bound to the rest of $\mathbb{Z}^d$, completes the proof.
\end{proof}
\begin{proof}[Proof of Theorem \ref{thm:limitshape}]
We will call an $x$ in $\mathbb{Z}^d$ for which the event appearing in \eqref{eq: maximal_inequality} occurs a ``good'' vertex. We can immediately leverage the information in that lemma to show
\begin{claim}\label{clam: ergodic_ratio}
Let $\zeta \in \mathbb{Z}^d\setminus \{ 0 \}$. For a given realization of edge-weights, denote by $(n_k)$ the sequence of natural numbers such that $n_k \zeta$ is a good vertex. Then with probability one, the sequence $(n_k)$ is infinite and $\lim_k (n_{k+1}/n_k) = 1$.
\end{claim}

To see that the claim is correct, note that the ergodic theorem implies that the sequence $(n_{k})$ is infinite almost surely. Let $B_m$ denote the event that $m\zeta$ is a good vertex.   Then
\[
\frac{k}{n_k} = \frac{1}{n_k} \sum_{i=1}^{n_k} \mathbf{1}_{B_i};
\]
the right side converges to the probability in \eqref{eq: maximal_inequality} by the ergodic theorem. Thus,
\[
\frac{n_{k+1}}{n_k} = \left( \frac{n_{k+1}}{k+1} \right) \left( \frac{k}{n_k} \right) \left( \frac{k+1}{k} \right) \to 1
\]
almost surely. This proves the claim.

Let $\Xi_1$ denote the event that $\lim_n T(0, nq) / n = \mu(q)$ for all $q$ having rational coordinates; let $\Xi_2$ denote the event that for every $\zeta \in \mathbb{Z}^d$, the sequence $(n_k)$ defined in Claim~\ref{clam: ergodic_ratio} is infinite and that the ratio of successive terms tends to one. From here, the proof of Theorem~\ref{thm:limitshape} proceeds by contradiction. Assume the Shape Theorem does not hold. Then there exists $\delta>0$ and a collection of edge-weight configurations $D_\delta$ with $\mathbb{P}(D_\delta)>0$ such that, for every outcome in $D_\delta$, there are infinitely many vertices $x \in \mathbb{Z}^d$ with
\begin{equation}\label{eq: toward_contradiction}
|T(0,x) - \mu(x)| > \delta\|x\|_1.
\end{equation}

Since $\mathbb{P}(\Xi_1)=1=\mathbb{P}(\Xi_2)$, the event $D_\delta \cap \Xi_1 \cap \Xi_2$ contains some outcome $\omega$; we claim that $\omega$ has contradictory properties. On outcome $\omega$, there must exist a sequence $(x_i) \subset \mathbb{Z}^d$ satisfying the condition in \eqref{eq: toward_contradiction}. We can assume that $x_i / \|x_i\|_1$ converges to some $y$ with $\|y\|_1=1$ by compactness of the unit sphere. Let $\delta'>0$ be arbitrary; we will fix its value at the end of the proof. We first choose some large $N$ such that $\|x_n/\|x_n\|_1-y\|_1 < \delta'$ and such that $$|\mu(x_n) - \|x_n\|_1\mu(y)| < \delta \|x_n\|_1/2$$ for $n > N$. Then we have for $n>N$ (using our assumption \eqref{eq: toward_contradiction}:
\begin{equation}\label{eq: toward_contradiction_2}
|T(0,x_n) - \|x_n\|_1\mu(y)| > \delta \|x_n\|_1/2.
\end{equation}

Next, we set up a sequence of approximating good vertices. We find some $z \in \mathbb{R}^d$, $\|z\|_1=1$ such that $\|z-y\|_1 < \delta'$, with the additional property that $z=x/M$ for some $x \in \mathbb{Z}^d$ and some positive integer $M$. This can be done because vectors with rational coordinates are dense in the unit sphere. On $\omega$, there must exist a sequence $(n_k)$ such that $n_k Mz$ is a good vertex and such that $n_{k+1}/n_k$ tends to one. For any $n$, there exists a value of $k$ such that $$n_{k+1}M \geq \|x_n\|_1 \geq n_k M;$$ denote this value by $k(n)$. Finally, fix $K>0$ such that $n_{k+1} < (1+\delta')n_k$ and $$|T(0,n_kMz)/(n_kM) - \mu(z) | <\delta'$$ for all $k > K$. We now let $n > N$ be large enough that $k(n) > K$.

Before completing the calculation here, it is worth considering where the contradiction will arise. We have (essentially by assumption) that $T(0,ny) - n\mu(y)$ is of order $n$ for infinitely many $n$. Since $\mu$ is a norm, $\mu(y)$ and $\mu(z)$ are arbitrarily close and since infinitely many of the $\{nz\}$ are good vertices, $T(0,ny)$ and $T(0,nz)$ are arbitrarily close. Thus $|T(0,nz) - n\mu(z)|$ is large -- but this is counter to the properties assumed for $\omega$.

To turn the above into a rigorous estimate, write $k$ for $k(n)$ and expand
\begin{align*}
\left| \frac{T(0,x_n)}{\|x_n\|_1} - \mu(y) \right| &\leq \left| \frac{T(0,x_n) - T(0,n_kMz)}{\|x_n\|_1}\right| + \frac{T(0,n_kMz)}{n_kM} \left( 1- \frac{n_kM}{\|x_n\|_1}\right) \\
&+ \left| \frac{T(0,n_kMz)}{n_kM} - \mu(z) \right| + |\mu(z)-\mu(y)|.
\end{align*}
There are four terms on the right side of the above, which we number from left to right and bound individually in terms of $\delta'$.

\medskip
\noindent
{\bf Term 1.} Since $n>N$ and $k>K$, one has $n_kM \leq \|x_n\|_1 \leq (1+\delta')n_{k}M$, that $\|n_{k}My-n_kMz\|_1 \leq \delta'n_kM$, and that $\|x_n/\|x_n\|_1 - y \|_1 < \delta'$. Therefore, $\|x_n - n_kMz\|_1 \leq 2\delta'\|x_n\|_1$. Using the fact that $n_kMz$ is a good vertex yields
\[
|T(0,x_n) - T(0,n_kMz)| \leq \kappa\|x_n-n_kMz\|_1 \leq 2\kappa \delta'\|x_n\|_1.
\]

\medskip
\noindent
{\bf Term 2.} The relationship between $n_kM$ and $\|x_n\|_1$ given in the Term 1 estimates yield an upper bound for the second factor of Term 2. By the fact that $k>K$, we can bound the first factor. The overall bound is
\[
[\mu(z)+\delta']\left( 1 - (1+\delta')^{-1} \right).
\]

\medskip
\noindent
{\bf Term 3.} By the fact that $k$ is chosen greater than $K$, this term is bounded above by $\delta'$.

\medskip
\noindent
{\bf Term 4.} If $\mu$ is identically zero, this term is trivially zero. If $\mu$ is not identically zero, it is a norm on $\mathbb{R}^d$ and is thus bounded by the $\| \cdot \|_{1}$ norm:
\[
c_L\|\cdot\|_1 \leq \mu(\cdot) \leq c_U\|\cdot\|_1.
\]
Since $\|z-y\|_1 < \delta'$, Term 4 is bounded above by $\delta'$ times a constant depending only on $\mu$. 

We have therefore bounded the left side of \eqref{eq: toward_contradiction_2} by an expression of the form $f(\delta')\|x_n\|_1$, where $f$ tends to zero as $\delta' \to 0$. Since $\delta'$ was arbitrary, we can choose it such that $f(\delta')\|x_n\|_1$ is smaller than the right side of \eqref{eq: toward_contradiction_2}. This contradiction proves the theorem.

\end{proof}

\subsection{Other limit shapes}\label{sec:shell}
In this section, we briefly discuss a few extensions of Theorem \ref{thm:limitshape}.
\subsubsection{Shell passage times}
As we saw in Remark \ref{rem:itblowsup!}, when \eqref{eq:conditionmomentsLS} fails we have with probability one

\begin{equation*} \limsup_{v \to \infty} \frac{T(0,v)}{|v|}  =\infty.
\end{equation*}
Nevertheless, without any moment condition, one can define a modified passage time $\hat T(u,v)$ such that the family of random variables $\hat T(u,v) - T (u,v)$ is tight and one has a limit shape for the modified $\hat T$. This was first done by Cox-Durrett in dimension $2$ and later extended by Kesten to all dimensions. Their construction goes as follows.  Let  $M \in \mathbb R_+$ be large enough so that $F([0,M])$ is very close to $1$. The collection of edges $e$ such that $\tau_{e} \leq M$ is a super-critical percolation process, so if we denote by $\mathcal C_M$ its infinite cluster, each point $u \in \mathbb Z^d$ is a.s. surrounded by a small contour (or shell) $S(u) \subset \mathcal C_M$. They define $\hat T(u,v) = T (S(u), S(v))$ for $u, v \in \mathbb Z^d$. The times $\hat T(0, u)$ have good moment properties; thus their limit shape can be defined using the a.s. and $L^{1}$ limit given by the previous arguments. The details of this construction in $d>2$ requires certain topological properties of the exterior boundary of a subset of $\mathbb Z^{d}$. These properties were derived by Kesten and generalized in the work of Tim\'{a}r \cite{Timar}.

\subsubsection{FPP in the super-critical percolation cluster}
Another direction where  shape theorems have been proven is where we allow passage times to be infinite. This is equivalent to considering FPP on a super-critical Bernoulli percolation performed independently. When the edge weights are either $1$ or $\infty$, the passage time is also known as the chemical distance. 

In this setting, the benchmark is the work of Gar\'et-Marchand \cite{GM2}, where the analogues of Theorems \ref{thm:percolate}, \ref{thm:limitshape} were proven under a moment condtion $\E \tau_{e}^{\alpha(d)} \mathbf 1_{\{\tau_{e}<\infty\}} <\infty$, where $\alpha(d)= 2(d^{2}+d-1)+\epsilon$; see hypothesis $H_{\alpha}$ on page 4 in \cite{GM2}. Their results are also valid for stationary ergodic passage times, in the spirit of the work of Boivin \cite{B90}. 

In the i.i.d. case, in two independent works, Cerf-Th\'{e}ret \cite{CT0} and Mourrat \cite{Mourrat} recently removed all moment assumptions of \cite{GM2}, by proving a weak shape theorem. In his paper, Mourrat considers a model of a random walk in a random potential, but he discusses how his theorems easily extend to our setting (See Section 11 of \cite{Mourrat}). We describe their results below, as they are a nice compromise between the results of Gar\'et-Marchand and shell passage times of Cox-Durrett and Kesten. Let $\mathcal C_\infty$ be the infinite cluster for the Bernoulli percolation. For all $x \in \mathbb Z^d$, let $x^* \in \mathbb Z^d$ be the random point of $\mathcal C_\infty$ such that $\| x - x^* \|_1$ is minimal, with a deterministic rule to break any possible ties. Define $$T^*(x,y) = T(x^*, y^*)$$ and, for $t \geq 0$,

$$B_t^* = \{ z + u | z \in \mathbb Z^d, T^*(0,z)\leq t, \; u \in [-1/2,1/2]^d\}. $$

The time constant satisfies:

\begin{theorem}[Theorem 4 \cite{CT0}, Theorem 1.2 \cite{Mourrat}] Suppose that $F([0,\infty)) > p_c(d)$. Then there exists $\mu^*(x)$ such that  for all $x \in \mathbb Z^d$,
$$ \lim_n \frac{T^*(0,nx)}{n} = \mu^*(x) \; \text{ in probability, }$$
and
 $$  \lim_n \frac{T(0,nx)}{n} = Z \quad \text{ in law, }$$ 
where the distribution measure of $Z$ is given by $\theta^2 \delta_{\mu^*(x)} + (1-\theta^2) \delta_{\infty} $
and $\theta = \mathbb P (0\in \mathcal C_\infty)$.
\end{theorem}

The limit shape obeys:
\begin{theorem}[Theorem 5 \cite{CT0}, Theorem 1.2 \cite{Mourrat}] Suppose that $F([0,\infty)) > p_c(d)$ and $F(\{ 0\}) < p_c(d)$. Then there exists a compact set $B^*$ such that almost surely 
$$ \lim_t \mathcal \lambda^d\left( \frac{B_t^*}{t} \bigtriangleup B^* \right) = 0,$$
where $\mathcal \lambda^d$ denotes the Lebesgue measure in $\mathbb R^d$ and $A \bigtriangleup B$ is the symmetric difference of sets $A$ and $B$. 
\end{theorem}

\subsection{Properties of the limit shape}

\subsubsection{Flat edges for limit shapes}\label{sec: flat_edge} In this section, we address one of the questions presented in the introduction.
\begin{equation} \label{QuestionLS}
\textit{ Which compact convex sets } C \textit{ are realizable as limit shapes?}
\end{equation}

The question above is \textbf{ completely open} in the i.i.d. setting. A partial expected answer is given by the following conjecture.

\begin{question}\label{q:strictconvexity}
Show that if $F$ is a continuous distribution then the limit shape is strictly convex.
\end{question}

Surprisingly, not even the following is known:

\begin{question}\label{q:square}
Show that the $d$-dimensional cube is not a possible limit shape for a FPP model with independent, identically distributed passage times.
\end{question}
\begin{remark}\label{rem: HM}
Interestingly, question \eqref{QuestionLS} is  solved by H\"aggstr\"om and Meester \cite{HM} in the case of stationary (not necessarily i.i.d.) passage times. They establish that any non-empty compact, convex set $C$ that is symmetric about the coordinate axes is a limit shape for some FPP model with weights distributed according to a stationary (under translations of $\mathbb Z^d$) and ergodic measure.  This is in sharp contrast with the i.i.d. case explained above.
\end{remark}

However, there is one class of weights where the limit shape is known in some directions. This collection was introduced by Durrett and Liggett \cite{DurrettLiggett} and further studied by Marchand \cite{Marchand}, Zhang \cite{Zhang, Zhang2} and by Auffinger and Damron \cite{AD12}. Its main feature is the presence of a flat edge for the limit shape, as we describe below. We will stick to dimension $2$ in what follows.

Write $supp(\nu)$ for the support of $\nu$ where $\nu[0,x]=F(x)$ is the probability distribution of $\tau_e$. Let $\mathcal{M}_p$ be the set of measures $\nu$ that satisfy the following:
\begin{enumerate}
\item $supp(\nu) \subseteq [1,\infty)$ and
\item $\nu(\{1\})=p\geq\vec p_c$,
\end{enumerate}
where $\vec p_c$ is the critical parameter for oriented percolation on $\mathbb{Z}^2$ (see, e.g., \cite{Durrett}). In \cite{DurrettLiggett}, it was shown that if $\nu \in \mathcal{M}_p$ then $\mathcal B_\nu$ has some flat edges. The precise location of these edges was found in \cite{Marchand}. To describe this, write $\mathcal{B}_1$ for the closed $\ell^1$ unit ball:
\[
\mathcal{B}_1 = \{(x,y)\in \mathbb{R}^2~:~ |x|+|y|\leq 1\}
\]
and write $int ~\mathcal{B}_1$ for its interior. For $p > \vec p_c$ let $\alpha_p$ be the asymptotic speed of oriented percolation \cite{Durrett}, define the points
\begin{equation}\label{eq:NP}
M_p = \left(\frac{1}{2} - \frac{\alpha_p}{\sqrt 2}, \frac{1}{2} + \frac{\alpha_p}{\sqrt 2}\right) \text{ and } N_p = \left(\frac{1}{2} + \frac{\alpha_p}{\sqrt 2}, \frac{1}{2} - \frac{\alpha_p}{\sqrt 2}\right)
\end{equation}
and let $[M_p,N_p]$ be the line segment in $\mathbb{R}^2$ with endpoints $M_p$ and $N_p$. For symmetry reasons, the following theorem is stated only for the first quadrant.

\begin{theorem}[Durrett-Liggett \cite{DurrettLiggett},  Marchand \cite{Marchand}]\label{thm:marchand1}
Let $\nu \in \mathcal{M}_p$.
\begin{enumerate}
\item $\mathcal B_\nu \subset \mathcal{B}_1$.
\item If $p < \vec p_c$ then $\mathcal B_\nu \subset int~\mathcal{B}_1$.
\item If $p > \vec p_c$ then $\mathcal B_\nu \cap [0,\infty)^2 \cap \partial \mathcal{B}_1 = [M_p,N_p]$.
\item If $p = \vec p_c$ then $\mathcal B_\nu \cap [0,\infty)^2 \cap \partial \mathcal{B}_1 = (1/2,1/2)$.
\end{enumerate}
\end{theorem}

The angles corresponding to points in the line segment $[M_p,N_p]$ are said to be in the \emph{percolation cone}; see Figure~4 below.

\begin{figure}[h]
\centering
 \includegraphics[scale=0.65]{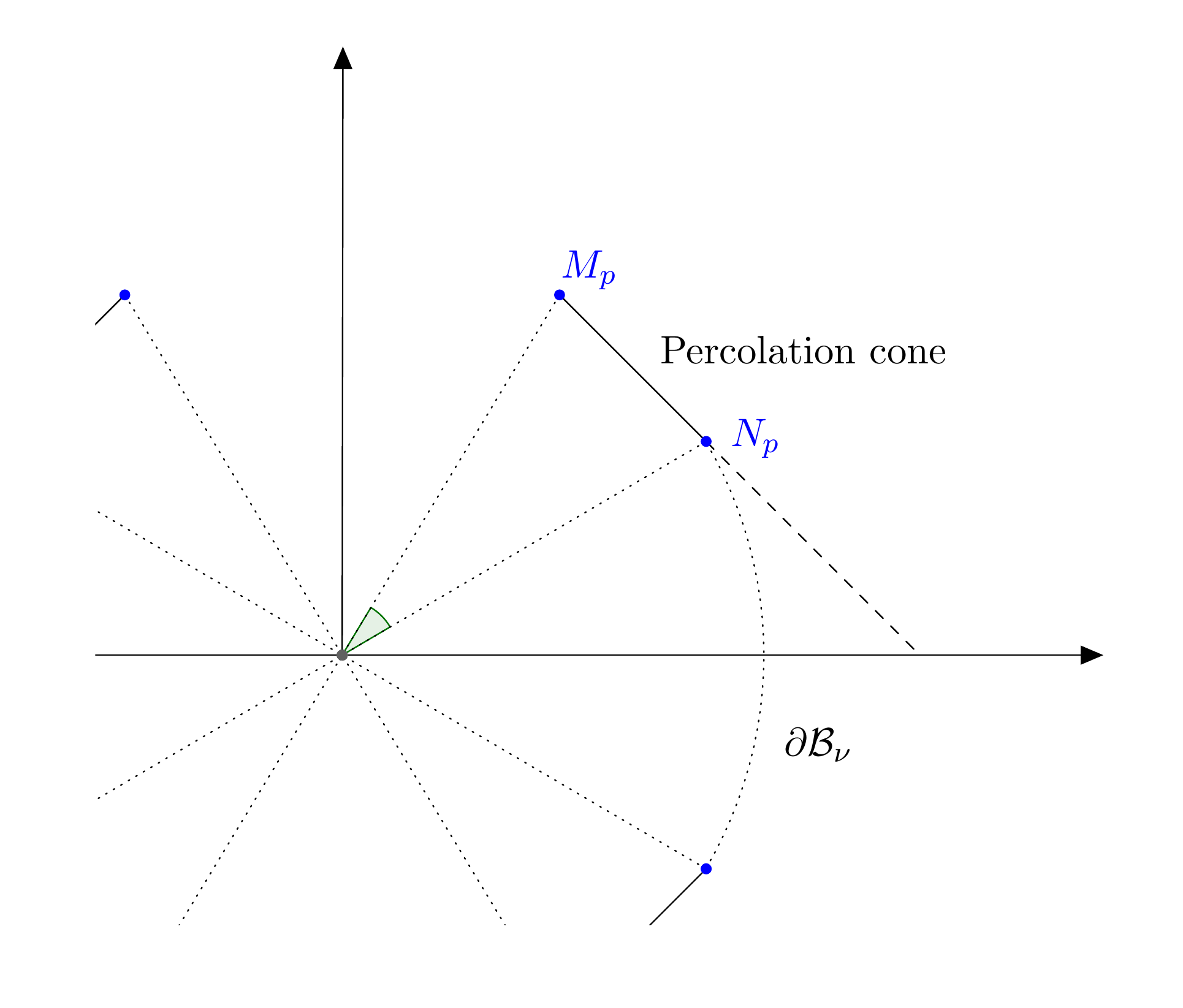}
 \caption{Pictorial description of the limit shape $\mathcal B_\nu$. If $\nu \in \mathcal{M}_p$, the shape has a flat edge between the points $M_p$ and $N_p$. The limit shape is differentiable at $M_p$ and $N_p$. Outside the percolation cone the limit shape is unknown.}
 \label{FigureCone}
\end{figure}

Let $\beta_p := 1/2 + \alpha_p/\sqrt{2}$, that is, define $\beta_p$ as the $x$ coordinate of $N_p$. Convexity and symmetry of the limit shape imply that $1/\mu(e_1) \geq \beta_p$. A non-trivial statement about the edge of the percolation cone came in 2002 when Marchand \cite[Theorem 1.4]{Marchand} proved that this inequality is in fact strict:

$$1/\mu(e_1) > 1/2 + \alpha_p/\sqrt{2}.$$ In other words, Marchand's result says that the line that goes through $N_p$ and is orthogonal to the $x$-axis is not a tangent line of $\partial \mathcal B_\nu$. The following theorem builds on Marchand's result and technique and says that at the edge of the percolation cone, one cannot have a corner.

\begin{theorem}[Auffinger-Damron \cite{AD12}]\label{thm: diffll}
Let $\nu \in \mathcal{M}_p$ for $p\in[\vec p_c,1)$. The boundary $\partial \mathcal B_\nu$ is differentiable at $N_p$.
\end{theorem}

\begin{remark}
Theorem~\ref{thm: diffll} is stated for the single point $N_p$  but, due to symmetry, it is clearly valid for $M_p$ and for the reflections of these two points about the coordinate axes.   
\end{remark}

The theorem above shows that any measure in $\mathcal M_p$ has a non-polygonal limit shape. The question of finding a single distribution where the limit shape is non-polygonal was raised by H. Kesten \cite{KestenAspects}. The first example of a non-polygonal limit shape was discovered by Damron-Hochman \cite{DH}. 

The flat part of the percolation cone ends at $M_p$ and $N_p$; however, that does not exclude the limit shape from having further flat spots. This is not expected though.

\begin{question}
Show that for any measure $\nu \in \mathcal M_p$ the boundary of the limit shape is not flat outside the percolation cone. 
\end{question}

A similar, but perhaps more ambitious question is to show:

\begin{question}
Show that if the limit shape of a measure $\nu$ has a flat piece then $\nu \in \mathcal M_p$, for $p > p_{c}$ and the flat piece is delimited by the percolation cone. 
\end{question}

More approachable may be the following two open questions:

\begin{question}\label{oq:3} 
Show that for any measure $\nu$, in direction $e_1$ the boundary of the boundary of the limit shape does not contain any segment parallel to the $e_{2}$ axis.  
\end{question}

A solution to \ref{oq:3} immediately implies a solution to Question \ref{q:square}.

\begin{question}\label{oq:4}
Find a non-trivial example of a measure $\nu$ such that in direction $e_1$ the limit shape boundary is not parallel to the $e_{2}$ axis. 
\end{question}

One may reasonably guess that to solve Question \ref{oq:4}, it suffices to consider small perturbations of the trivial measure $\nu= \delta_1$ (every edge has non-random passage time equal to $1$) as the limit shape in this case is just the $\ell^{1}$ ball, a diamond.  This observation was explored by  Basdevant and co-authors by considering FPP in the highly super-critical bond percolation cluster. 
Precisely, let $\tau_e$ be equal to $1$ with probability $p \in (0,1)$ and infinite with probability $1-p$.  
\begin{theorem}[Basdevant et al. \cite{BNGG}] For all $0 \leq \lambda \leq 1$, on the event that the origin is connected to infinity, almost surely we have
\begin{equation*}
\lim_{n \rightarrow \infty} \frac{T(0, n(1,\lambda(1-p))}{n} = 1 + (1-p)\frac{1+\lambda^2}{2} + O\big((1-p)^2\big).
\end{equation*}

\end{theorem} 
The theorem above roughly says that when $p$ is close to $1$, the four corners of the $\ell_1$ ball are replaced by curves that resemble parabolas.

\subsection{The subadditive ergodic theorem revisited}\label{sec:visit}

Recall that the main tool to prove  the existence of the time constant (Theorem \ref{thm:timeconstant}) was the subadditive ergodic theorem (Theorem \ref{subadditivity}). The fact that Theorem \ref{subadditivity} requires few assumptions makes it very powerful and widens its scope far beyond FPP. This strength also comes with two main drawbacks. First, it only provides the existence of the limiting object. The characterization as $\inf \mathbb E X_{0,n}/n$ is not always helpful to obtain further information on the time constant. Second, if one tries to obtain different results (fluctuations, concentration, large deviations) for ergodic processes under the same assumptions, one fails miserably. Nearly anything can happen.

The goal of this section is to dissect a major tool in Kingman's original proof of the subadditive ergodic theorem. Kingman's proof provides a non-trivial decomposition of the subadditive process. One term of this decomposition has the same properties as the Busemann function in FPP, an object  that we will study in detail in Section~\ref{sec:Busemann}. Later on, we will pursue this direction and  provide extra assumptions that allow one to derive further information about the subadditive process $X_{0,n}$. In this chapter, we stick with the first task.

An array $X = (X_{m,n})$ satisfying the assumptions of Theorem \ref{subadditivity} but also $X_{0,n}=X_{0,m}+X_{m,n}$ is called an additive ergodic sequence.  The proof of Kingman's theorem depends on the following decomposition:

\begin{theorem}\label{thm:decomposition}
If $X_{m,n}$ is a subadditive ergodic sequence then there exist arrays $Y_{m,n}$ and $Z_{m,n}$ such that 
\begin{enumerate}
\item[$(a)$] $Y_{m,n}$ is an additive ergodic sequence with $\E Y_{0,1} = \mu = \inf_{n} \E X_{0,n}/n$.   
\item[$(b)$] $Z_{m,n}$ is a non-negative subadditive ergodic sequence with time constant equal to $0$. 
\item[$(c)$] $X_{m,n} = Z_{m,n} + Y_{m,n}.$
\end{enumerate}
\end{theorem}

The decomposition above is not necessarily unique. Here is an easy counter-example. Let $(y_k)_{k\geq 1}$ be a sequence of independent standard Gaussians and set $S_{m,n} = y_{m+1} + \ldots + y_n$. Then $S$ is additive, with time constant $0$. If we put
$$X_{m,n}= \max(S_{m,n},0),$$ 
we see that $X$ is subadditive, with $\mathbb E X_{0,n} =(n/2\pi)^{1/2}$, so $\mu=0$. Two decompositions of $X$ are given by $Y=0$, $Z=X$ and $Y=S$, $Z=\max(-S,0)$.

To see some implications of the above decomposition, suppose that the following limit exists almost surely:
\begin{equation}\label{Busemanningeneral}
B_{m,n} := \lim_{N} (X_{m,N} - X_{n,N})
\end{equation}
and satisfies $\E |B_{0,n}| \leq Cn$ for some positive constant $C$.
We claim that $B_{m,n}$ is  an additive process in the decomposition above. To see this, note that $$B_{0,n} = \lim_{N} (X_{0,N} - X_{n,N}) =  \lim_{N} (X_{0,N} - X_{m,N} + X_{m,N} - X_{n,N}) = B_{0,m} + B_{m,n}$$
and $(b)$ and $(c)$ from the definition in Theorem \ref{subadditivity} follow from the stationarity of the process $(X_{m,n})$. For instance,
\begin{eqnarray*} 
B_{m,n} =  \lim_{N} (X_{m,N} - X_{n,N}) &=&  \lim_{N} (X_{m,N-1} - X_{n,N-1})\\ 
&\stackrel{d}{=}&  \lim_{N} (X_{m+1,N} - X_{n+1,N})=B_{m+1,n+1},
\end{eqnarray*}
where we used the fact that if $X_N \stackrel{d}{=}Y_N$ and $X_N$, $Y_N$ converge in distribution to $X$ and $Y$ respectively then $X \stackrel{d}{=}Y$. The moment bounds follows from assumption.

Now set 

$$Z_{m,n} = X_{m,n} - B_{m,n}.$$

\noindent For any integer $N > n$, subadditivity gives $X_{m,n}  \geq X_{m,N} - X_{n,N}$ almost surely. Thus, we also have $Z_{m,n} \geq 0$ almost surely. As $B$ is additive, $Z_{m,n}$ is sub-additive and thus it satisfies~$(b)$. 

The importance of the limit in \eqref{Busemanningeneral} will become clear in Section \ref{sec:Busemann}. At this point, it would be interesting to determine when we can actually use \eqref{Busemanningeneral}.
\begin{question}\label{q10}
Find conditions that guarantee the existence of the limit \eqref{Busemanningeneral}.
\end{question}

The way that Kingman avoided the problem of existence of the limit in \eqref{Busemanningeneral} was to construct weak averaged limits. It will be beneficial to explain his idea here. Let $\Omega$ be the space of all subadditive functions. A subadditive ergodic process is a random element on $\Omega$, inducing a probability measure $\mathbb P$ on this space. For any $x \in \Omega$ we define the shift $\theta x$ as the element in $\Omega$ given by $$\theta x(m,n)= x(m+1,n+1).$$ 
Now let $f$ be an element of $L_1(\Omega, \mathbb P)$ and define $T: L_1(\Omega, \mathbb P) \to L_1(\Omega, \mathbb P)$ as the bounded linear operator given by
$$ (Tf)(x)= f(\theta x).$$
Kingman's magic was to construct a function $f$ in $L_1(\Omega, \mathbb P)$ such that 
$$f +Tf+\ldots + T^{n-1}f \leq X_{0,n}$$ and $\mathbb E f = g$. Once this function is constructed, the reader can easily check that the representation follows by taking
\begin{equation}\label{eq:Kingmanisgood}
Y_{m,n}= T^m f + T^{m+1}f + \ldots + T^{n-1}f
\end{equation}
 and $Z_{m,n}=X_{m,n}-Y_{m,n}$. Note that both sides in \eqref{eq:Kingmanisgood} are random variables in $\Omega$.
To construct such $f$ he considered the process: 
\begin{equation}\label{eq:metoohahaha}
f_k = \frac{1}{k} \sum_{i=1}^{k} (X_{0,i}-X_{1,i}) 
\end{equation}
and for each $n\geq1$ its iterates
\begin{equation}\label{eq:sohumble}
 f_k + Tf_k + \ldots + T^{n-1}f_k = \frac{1}{k} \sum_{i=1}^{k+n-1} (X_{a,i} - X_{b,i})
 \end{equation}
where $a=\max(i-k,0)$ and $b=\min(i,n)$. The reader can now see the connection: as $k$ goes to infinity, \eqref{eq:metoohahaha} and \eqref{eq:sohumble} play the role of $B_{0,1}$ and $B_{0,n}$, respectively. The existence of the $k$ limit turns out to follow from the Bourbaki-Alaoglu theorem. We will come back to this point in Section \ref{sec:Busemann}. When the limit in \eqref{Busemanningeneral} exists, we will call it a {\it generalized Busemann function} for the subadditive process $X$.

\subsection{Pointed Gromov-Hausdorff convergence}
FPP is a model of a random (pseudo)metric space. There is a classical way of defining convergence of a sequence of metric spaces, using the Gromov-Hausdorff distance on the space of metric spaces. We follow this route in this section and rephrase the limit shape in this context. The reader is invited to check \cite{Burago, Bridson} for a detailed explanation and historic motivation of the basic topics we touch here. Although this approach brings a different perspective, these methods have not yet provided significant new progress in FPP on $\mathbb Z^{d}$. However, they were successfully used to extend the limit shape to FPP in Cayley graphs with polynomial growth, where the subadditive ergodic theorem does not immediately apply \cite{Tessera, TesseraB}. We will need a few definitions before we start, and most of the following is taken from \cite{Bridson}.

\begin{definition}
A subset $S$ of a metric space $X$ is said to be {\it$\epsilon$-dense} if every point of $X$ lies in the $\epsilon$-neighborhood of $S$. 
\end{definition}
\begin{definition}
An $\epsilon$-relation between two (pseudo)metric spaces $X_{1}$ and $X_{2}$ is a subset $R~\subset~X_{1} \times~X_{2}$ such that
\begin{enumerate}
\item For $i=1,2$, the projection of $R$ to $X_{i}$ is $\epsilon$-dense.
\item If $(x_{1},x_{2}), (x_{1}',x_{2}') \in R$ then $|d_{X_{1}}(x_{1},x_{1}') - d_{X_{2}}(x_{2},x_{2}')| < \epsilon$.
\end{enumerate}
\end{definition}
If there exists an $\epsilon$-relation between the  metric spaces $X_{1}$ and $X_{2}$, we say that $X_{1}$ and $X_{2}$ are $\epsilon$-related and use the notation $X_{1}\sim_{\epsilon}X_{2}$. When the projection of an $\epsilon$-relation is onto in both of its coordinates, we say that the relation is surjective and we write $X_{1} \simeq_{\epsilon} X_{2}$. It is an exercise to show that if $X_{1} \sim_{\epsilon} X_{2}$ then $X_{1} \simeq_{3\epsilon} X_{2}$.

The {\it Gromov-Hausdorff} distance between $X_{1}$ and $X_{2}$ is defined as:
$$ D_{H}(X_{1},X_{2}):=\inf \{ \epsilon | X_{1} \simeq_{\epsilon} X_{2}\}. $$
If there is no $\epsilon$ such that $X_{1} \simeq_{\epsilon} X_{2}$, then $D_{H}(X_{1},X_{2})$ is infinite. It is not difficult to see that $D_{H}$ satisfies the triangle inequality on the space of metric spaces and thus it is a pseudo-metric which may take the value infinity.

\begin{definition}
We say that a sequence of (pseudo)metric spaces $X_{n}$ converges to $X$ in the Gromov-Hausdorff metric, and write $X_{n} \to X$ if and only if $D_{H}(X_{n},X) \to 0$ as $n\to \infty$.
\end{definition}

Gromov-Hausdorff convergence works well in contexts where one deals with sequences of compact metric spaces, but it is a less satisfactory concept when applied to non-compact spaces. One disadvantage is that the distance between a compact space and an unbounded set is always infinity. Since the spaces that we care about are not compact, we will need the following alternative definition of convergence. The magic here is that the intuitive sense of convergence comes from observations from a fixed point.

\begin{definition}
A pointed space $(X, d, x)$ is a pair of a metric space $(X,d)$ and a point $x \in X$. The point $x$ is called the basepoint of the pointed space $(X,d,x)$.
\end{definition}

\begin{definition}
A sequence of pointed spaces $(X_{n}, d_{n}, x_{n})$ converges to a pointed space $(X,d,x)$ if for every $r>0$ the sequence of closed balls $B(x_{n},r)$ (with induced metrics) converges to the closed ball $B(x,r)$ in the Gromov-Hausdorff metric.
\end{definition}

One of the nice features of pointed Gromov-Hausdorff convergence is that it preserves several properties of the sequence of metric spaces in the limit. We will comment on this at the end of the section.

Now, let's go back to FPP. Let $(\Omega,\mathcal F, \mathbb P)$ be a probability space where all the $\tau_{e}$'s are defined. For each $\omega \in \Omega$ we define a sequence of pseudometric spaces: 

$$ \bigg(X_{n}, d_{n}(x,y)\bigg) = \left(\frac{1}{n}\mathbb Z^{d}, \frac{T(nx,ny)}{n} \right).$$

The pseudometric space $X_{n}$ is just the original lattice rescaled by $n$ with the normalized pseudometric; that is, $d_{n}(0,x) = T(0,nx)/n$. The origin $0$ is a point of $X_{n}$ for all $n \geq 1$. Now recall the construction from Section \ref{subsec:limitshape}. Given an edge distribution on $\mathbb Z^{d}$, there exists a norm $\mu$ on $\mathbb R^{d}$ where the unit ball in that norm is the limit shape of the FPP model. The pair $(\mathbb R^{d}, \mu)$ is a normed vector space with distance $d(x,y) = \mu(y-x)$ for $x,y \in \mathbb R^{d}$. We will  assume that the passage times have finite exponential moments. This assumption is to make sure the metric satisfies a concentration bound given by Lemma \ref{lem:Pokahontas} below.

The limit shape theorem translates to the following statement. 
\begin{theorem}\label{thm:GromoHauss} Assume that $F(0)<p_c(d)$ and $\int e^{\alpha x} \; d\nu < \infty$ for some $\alpha>0$. Almost surely, the sequence $(X_{n},d_{n},0)$ converges in the pointed Gromov-Hausdorff sense to $(\mathbb R^{d}, \mu, 0).$
\end{theorem}
\begin{proof} 
Fix $r>0$ rational. We first  show that almost surely the balls $B_{n}:=B_{n}(0,r) \subset X_{n}$ converge in the Gromov-Hausdorff sense to the ball $$ B:=B(0,r)= \{ x \in \mathbb R^{d} | \mu(x) \leq r \}.$$ For this, it suffices to show that for any $\epsilon >0$ there exists $n_{0} \in \mathbb N$ so that, for any $n \geq n_{0}$ there is an $\epsilon$-relation between $B_{n}$ and $B$. 

We construct such a relation as follows. Fix $\epsilon>0$. Given $0<\epsilon' < \epsilon$, use Theorem \ref{thm:limitshape} to choose $n_{0}=n_{0}(\omega)$ so that for $n \geq n_{0}$ 

$$ \mathbb P \bigg( (1-\epsilon') B \subseteq \bar B_{n} \subseteq (1+\epsilon')B \text{ for all } n \geq n_{0} \bigg) = 1,$$
where $\bar B_{n} = B_{n} + [-1/(2n),1/(2n)]^{d}$.
The set $R$ is defined as  the union of two sets, $R = R_{1} \cup R_{2}$ where
$$ R_{1}:=\bigg\{ (x_{1},x_{2}) \in B_{n} \times B : x_2=x_{1}/(1+\epsilon') \bigg\},$$
$$ R_{2}:=\bigg\{(x_{1},x_{2}) \in B_{n} \times B : (1-\epsilon')x_{2}\in x_{1} + [-1/(2n),1/(2n)]^{d} \bigg \}.$$ 
Note that $R$ is surjective. Indeed, since  $B_{n} \subseteq (1+\epsilon')B$, every element of $B_{n}$ is related to some element of $B$ through $R_{1}$ while $(1-\epsilon')B \subseteq \bar B_{n}$ implies that every element of $B$ is related to at least one element of $B_{n}$ through $R_{2}$.
Now take $(x_{1},x_{2})$ and $(x_{1}',x_{2}')$ in $R$. We have 
\begin{equation} \label{eq:triangle}
\begin{split}
|d_{n}(x_{1},x_{1}') - d(x_{2},x_{2}')| &= |T(nx_{1}, nx_{1}')/n - \mu(x_{2}-x_{2}')| \\
&\leq |T(nx_{1}, nx_{1}')/n - \mu(x_{1}-x_{1}')| + |\mu(x_{1}-x_{1}')-\mu(x_{2}-x_{2}')|.
\end{split}
\end{equation}
Let's first look at the second term in the right side of  \eqref{eq:triangle}. If $(x_{1}, x_{2})$ and $(x_{1}',x_{2}')$ are both in $R_{1}$ we have 
\begin{equation}\label{eq:bothinr1}
\mu(x_{1}-x_{1}')-\mu(x_{2}-x_{2}')= (1+\epsilon')\mu(x_{2}-x_{2}') - \mu(x_{2}-x_{2}') = \mu(x_{2}-x_{2}') \epsilon' < 2r \epsilon'.
\end{equation}
If both pairs of points are in $R_{2}$, then $x_{2} =(1-\epsilon')^{-1}(x_{1}+z_{n})$, $x_{2}' =(1-\epsilon')^{-1}(x_{1}'+z_{n}')$ for some $z_{n}, z_{n}'$ in $[-1/(2n),1/(2n)]^{d}$. We thus obtain
\begin{equation}\label{eq:replacingr2}
\mu(x_{2}-x_{2}') = (1-\epsilon')^{-1}\mu(x_{1}-x_{1}'+z_{n}-z_{n}').
\end{equation}
 Since $\mu$ is a norm, we can find $n_{1} \in \mathbb N$ so that for any $n \geq \max(n_{0},n_{1})$ and any $|w|<2\sqrt{d}/n$, $|\mu(z+w)-\mu(z)|\leq \epsilon'$. 
As $|z_{n}-z_{n}'|\leq 2\sqrt{d}/n$, we have by \eqref{eq:replacingr2} and the triangle inequality
\begin{equation}\label{eq:bothinr2}
|\mu(x_{1}-x_{1}')-\mu(x_{2}-x_{2}')|\leq ((1-\epsilon')^{{-1}}-1)\mu(x_{1}-x_{1}') + \epsilon'(1-\epsilon')^{-1} \leq 2\epsilon, \end{equation}
for $\epsilon'$ small. If $(x_{1}, x_{2}) \in R_1$ and $(x_{1}',x_{2}') \in R_{2}$, then similarly to \eqref{eq:bothinr1} and \eqref{eq:bothinr2} we obtain
  \begin{equation}\label{eq:bothinr3}
|\mu(x_{1}-x_{1}')-\mu(x_{2}-x_{2}')|\leq  2\epsilon, 
\end{equation}
for sufficiently small $\epsilon'$. Thus a combination of \eqref{eq:bothinr1}, \eqref{eq:bothinr2} and \eqref{eq:bothinr3} tells us that if we choose $\epsilon'$ small enough
\begin{equation}\label{eq:secondtermin}
|\mu(x_{1}-x_{1}')-\mu(x_{2}-x_{2}')| < \epsilon/2 \quad  \text{ for all } n \geq \max\{ n_0,n_1\}.
\end{equation}

The first term of \eqref{eq:triangle} is controlled by the following concentration bound.

\begin{lemma}\label{lem:Pokahontas}
Given $\epsilon >0$ and $R>0$ there exists $C_1>0$ such that for any $n\geq 1$
$$\mathbb P \bigg( \exists \; x,y \in [-Rn,Rn]^d  \text{ with } T(x,y) > n (\mu(y-x)+\epsilon/2) \bigg) \leq \exp(-n^{C_1}).   $$
\end{lemma}
\begin{proof} See Theorem \ref{thm: subdiffusive}.
\end{proof}

Combining \eqref{eq:triangle}, \eqref{eq:secondtermin} and Lemma \ref{lem:Pokahontas}, we see that 
$$ \mathbb P\bigg( R \text{ is an } \epsilon\text{-relation  between } B \text{ and } B_n \bigg) \geq 1 -  \exp(-n^{C_1}), $$ and thus by taking a countable sequence of $\epsilon_{n} \to 0$ and using Borel-Cantelli we obtain the desired result for each $r \in \mathbb Q$. However, if $R$ is a $\epsilon$-relation between $B_n(0,r)$ and $B(0,r)$ then (the restriction of) $R$ is also an   $\epsilon$-relation between $B_n(0,r')$ and $B(0,r')$ for any $r' < r$. This last observation suffices to end the proof of Theorem~\ref{thm:GromoHauss}.
\end{proof}

\subsection{Strict convexity of the limit shape}\label{sec:curvaturesection}
In this section, we explore in more detail the conjecture that, under mild assumptions on  $F$, the limit shape (see Question 4) is strictly convex. We also introduce the definition of uniform positive  curvature,  a  concept related to strict convexity. In Sections $3$ and $4$, we will discuss important results of Newman where this unproven property of uniform positive curvature will play a major role. Strict convexity also plays an important part in certain questions regarding the evolution of multi-type stochastic competition models discussed in Section $6$.

Recall that we call a subset $\mathcal B$ of $\mathbb R^{d}$ strictly convex if every line segment connecting any two points of $\mathcal B$ is entirely contained, except for its endpoints, in the interior of $\mathcal B$.

Let $u$ be a unit vector of $\mathbb R^d$ and let $H_0$ be a hyperplane such that $u+H_0$ is supporting hyperplane for $\mu(u) \mathcal B_\nu$ at $u$ (this means that $u+H_{0}$ contains $u$ and $\mu(u) \mathcal B_{\nu}$ intersects only one of the two halfspaces determined by $u+H_{0}$).
We introduce an exponent that captures the nature of the boundary of $\mathcal B_\nu$ in direction~$u$, called the \textit{curvature exponent}, as follows.

\begin{definition}[Curvature Exponent]\label{def:curvatureexponent}
Assume that $\partial \mathcal B_{\nu}$ is differentiable. The curvature exponent $\kappa(u)$ in the direction $u$ is a real number such that there exist positive constants $c$, $C$ and $\e$ such that for any $z \in H_0$ with $|z|<\e$, one has 
\begin{equation}\label{eq: curvatureassumption}
c|z|^{\kappa(u)} \leq \mu(u+z)-\mu(u) \leq C|z|^{\kappa(u)}.
\end{equation}
\end{definition}

\begin{definition}[Uniformly curved]\label{def:unifcurved} We say that $\mathcal B_\nu$ is uniformly curved if for every unit vector $u\in \R^d$, 
\[
\kappa(u) \geq 2,
\]
with constants in \eqref{eq: curvatureassumption} that are uniform in $u$.
\end{definition}
In the case that $\partial \mathcal B_{\nu}$ is not necessarily differentiable, Newman \cite{Newman} gave a general definition of uniform curvature: there exists $C>0$ such that for all $z_1,z_2 \in \partial \mathcal{B}_\nu$ and $z = (1-\lambda)z_1 + \lambda z_2$ with $\lambda \in [0,1]$,
\[
1-\mu(z) \geq C \min\{\mu(z-z_1),\mu(z-z_2)\}^2.
\]
Either of these definitions is suitable for the results in this survey.

\begin{figure}[h]
\centering
 \includegraphics[scale=1]{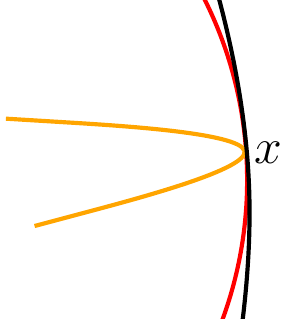}
 \caption{Representation of the limit shape (in red) in direction $x$ with curvature exponent $\kappa(x)=2$. The limit shape curve fits two tangent parabolas, one inside (orange), other outside (black). }
 \label{FigureCurvature}
\end{figure}

In two dimensions, the exponent $\kappa(u)$ tells us that it is possible to trace two curves of the form $y=x^{\kappa(u)}$ that are tangent to the limit shape, one inside and the other outside of $\mathcal B_\nu$ (see Figure~\ref{FigureCurvature}).  For instance, a Euclidean ball is uniformly curved with $\kappa(u)=2$ for every $u$. The $\ell^{1}$ ball is not uniformly curved as outside its corners one has $\kappa(u) = 1$.
Unfortunately, uniform curvature has not been proved for the limit shape of any FPP model with i.i.d. passage times. Any advance in the direction of the following question would be a major contribution.

\begin{question}
Show that for continuous distributions of passage times, the limit shape is uniformly curved.
\end{question}

The importance of the notion of curvature will be revealed in the next two sections. One characterization of curvature of the limit shape is to establish that $\mathcal B$ is a strict convex set of $\mathbb R^{d}$. A conditional proof of strict convexity was obtained by Lalley \cite{Lalley}. The two hypotheses of Lalley's result, however, seem to be out of reach at this moment. Hypothesis two may not be valid as for instance the Tracy-Widom distribution does not have mean $0$. We describe them now. 

For $u$ a fixed nonzero vector in $\mathbb R^{d}$ let $L_{u}$ be the ray through $u$ emanating from the origin.

\textbf{(H1).} For any convex cone $\mathcal A$ of $\mathbb  R^{d}$ containing the vector $u$ in its interior, and for each $\delta>0$ there exists $R=R(\delta, \mathcal A)<\infty$ such that the following is true: For each point $v \in \mathbb Z^{d} \cap \mathcal A$ at distance $\leq 2$ from the line $L_{u}$, the probability that the time-minimizing path from the origin to $v$ is contained in $\mathcal A \cup \{ x: \|x\| \leq R\} $ is at least $1-\delta$.

\vspace{0.5cm}
The second assumption requires a fluctuation theorem for the normalized passage times.
\vspace{0.5cm}

\textbf{(H2).} There exists a mean-zero probability distribution $G_{u}$ on the real line and a scalar sequence $a(n) \to \infty$ such that as $n \to \infty$
$$\frac{T(0,nu)-n\mu(u)}{a(n)} \stackrel{d}{\to} G_{u}.$$

\begin{theorem}[Theorem 1,\cite{Lalley}] Let $u$ and $v$ be linearly independent vectors in $\mathbb R^{d}$ and assume hypothesis \textbf{(H1)} and \textbf{(H2)} for both $u$ and $v$. Then for each $\lambda \in (0,1)$,
$$\mu(\lambda u + (1-\lambda)v) < \lambda \mu(u) + (1-\lambda)\mu(v).$$
\end{theorem} 

\subsection{Simulations}

Although we are celebrating the fiftieth anniversary of the model, simulation studies on first passage percolation were somewhat limited until very recently. Initial work is due to Richardson \cite{Richardson} in 1973, where the  model with exponentially distributed weights was analyzed. In \cite{Richardson} the limit shape $\mathcal B_{\nu}$ seemed to be curved, with a shape resembling a circle. As one could imagine, these simulations were restricted due to limitations in computer power. Further investigation (also in the Eden model) came in the work of Zabolitzky and Stauffer \cite{Zabolitsky}, in 1986, and by Durrett and Liggett in 1981. In particular, the numbers obtained in \cite{Zabolitsky} indicate the predicted fluctuation exponents for $\xi=2/3$ and $\chi=1/3$ by theoretical physicists \cite{K,KPZ,Krug} (see next two sections for the study of these exponents). 

A major contribution was done recently in the beautiful and extensive work of Alm and Deifjen \cite{AlmDeifjen} for FPP in two dimensions. Running $19$ years of CPU time, in a cluster of 28 Linux machines, they investigate the value of the time constant and the limit shape for several continuous distributions. Their results are consistent with most of the famous conjectures of the model. Their numerical simulations show strict convexity of the limit shape, with a limit shape different from a circle in all cases. The exponents for the standard deviation of hitting times and for the fluctuations of hitting points on lines also matched the predicted values 1/3 and 2/3, respectively. 

The paper of Alm and Deifjen also brought new findings to the table. It seems that the time constant depends primarily on the mean of the minimum edge weight adjacent to the origin,  that is, $$\mathbb E_{4}\tau_{e}:= \mathbb E \min\{t_{1}, t_{2}, t_{3},t_{4}\},$$ where the $t_{i}'s$ are independent copies of $\tau_{e}$, at least for continuous distributions which are not too concentrated. They reported that the time constants along the axis and along the diagonal for all simulated distributions have an almost perfect linear relation to $\mathbb E_4 \tau_e$. The distributions are not scaled to have the same mean, but most of them have $\inf \mathrm{supp}\; \nu = 0$. They also suggest that if $F(0)=0$ then $\mu \geq \mathbb E_{4} \tau_{e}$.

\begin{question}[Alm-Deifjen] 
Assume $F(0)=0$. Show that $\mu \geq \mathbb E_{4} \tau_{e}$.
\end{question}

\begin{figure}[h]
\centering
 \includegraphics[scale=0.7]{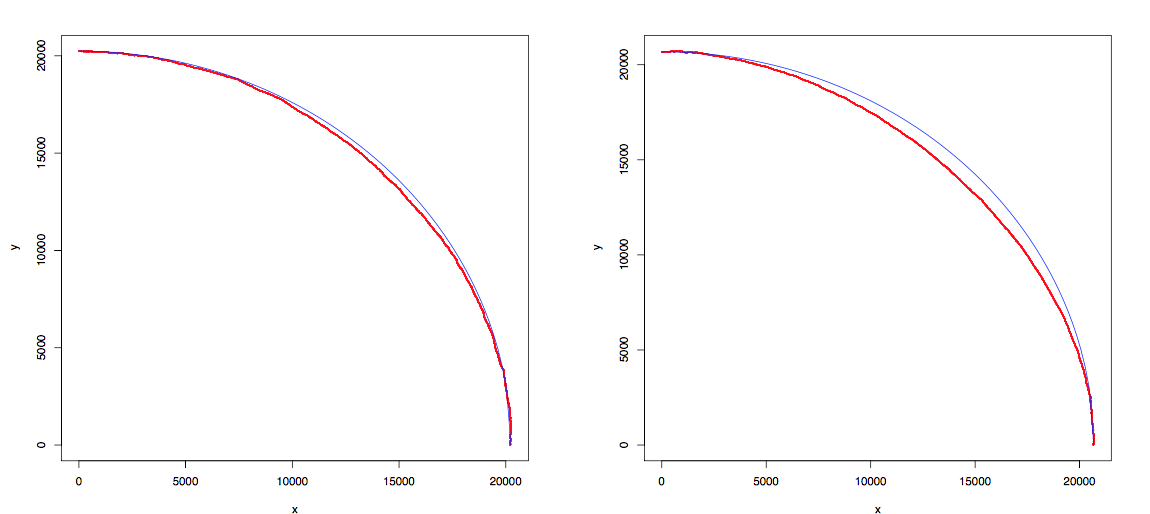}
 \caption{In red, simulation of the ball $B(t)$ for passage times distributed according to a Gamma random variable $\Gamma(k,k)$ with parameters $k=1$ (left) and $k=2$ (right) with $t=20000$.  In blue, the shape of a circle of radius $T(0,te_{1})$. Simulations and figures are courtesy of S.-E. Alm and M. Deijfen.}
 \label{FigureGamma1}
\end{figure}

\begin{figure}[h]
\centering
 \includegraphics[scale=0.7]{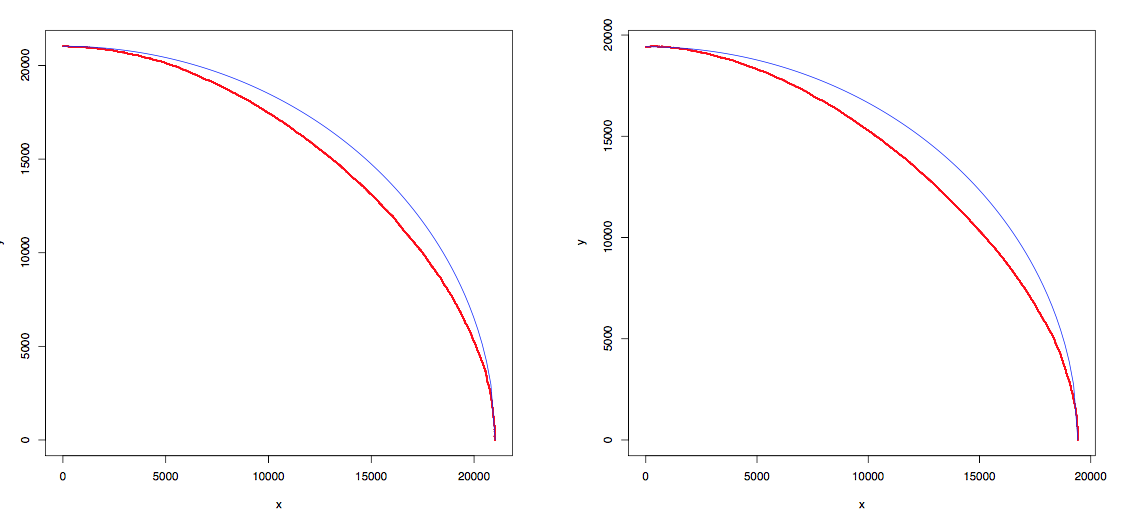}
 \caption{In red, simulation of the ball $B(t)$ for passage times distributed according to a Gamma random variable $\Gamma(k,k)$ with parameters $k=3$ (left) and $k=4$ (right) with $t=20000$.  In blue, the shape of a circle of radius $T(0,te_{1})$. Simulations and figures are courtesy of S.-E. Alm and M. Deijfen.}
 \label{FigureGamma2}
\end{figure}


\newpage

\section{Fluctuations and concentration bounds}\label{sec:fluctuations}
The passage time between $0$ and a vertex $x \in \mathbb{Z}^d$ can be approximated (almost surely) as
\[
T(0,x) = \mu(x) + o(\|x\|_1),
\]
due to the shape theorem. Quantifying the error term is the main subject of this section. It has been traditionally analyzed in two pieces:
\[
o(\|x\|_1) = \underbrace{T(0,x) - \mathbb{E}T(0,x)}_{random~fluctuations} + \underbrace{\mathbb{E}T(0,x) - \mu(x)}_{non-random~fluctuations}.
\]
The reason this splitting is useful is that the first term can typically be treated using techniques from concentration of measure, whereas the second term is analyzed using, in part, bounds on the first.

\subsection{Variance bounds}\label{sec:variancebounds}
The most basic control on the random fluctuation term is a variance bound. It has been predicted in the physics literature by simulation \cite{Zabolitsky, Wolf} and some scaling theory \cite{KPZ, Krug} (see also in Kesten \cite{Kesten}) that there is a dimension-dependent exponent $\chi = \chi(d)$ such that
\begin{equation}\label{eq:variancescaling}
\mathrm{Var}~T(0,x) \sim \|x\|_1^{2\chi}\ .
\end{equation}
This exponent is expected to be universal; it should not depend on the underlying environment, as long as $F$ satisfies some mild moment conditions and the limit shape has no flat edges. The meaning of ``$\sim$'' has not been made clear. For instance, it could be that as $x \to \infty$, the ratio of both sides converges to a constant, is bounded away from $0$ and $\infty$, or even that the variance has the expression $\|x\|_1^{2\chi + o(1)}$, as in the current case of Bernoulli percolation exponents. Regardless, the following dependence on $d$ is predicted for $\chi$:
\begin{table}[h]
\centering
\begin{tabular}{|c | c |}
\hline
$d$ & $\chi$\\
\hline
\hline
1 & 1/2 \\
\hline
2 & 1/3 \\
\hline
3 & ? \\
\hline
$\cdot$ & $\cdot$ \\
\hline
$d_c$ & 0 ? \\
\hline
\end{tabular}
\end{table}
In $d=1$, the passage time $T(0,x)$ is just a sum of $\|x\|_1$ i.i.d. random variables, and so one has $\chi(1)=1/2$ under any reasonable definition of ``$\sim$''. For $d \geq 2$, the infimum in the definition of $T(0,x)$ is predicted to produce sub-diffusive fluctuations, giving $\chi < 1/2$. It is clear that $\chi$ should decrease with dimension, but there is not agreement on whether $\chi(d) = 0$ for all $d$ at least equal to some $d_c$, and some even debate whether $\chi \to 0$.

The history of rigorous variance bounds begins with Kesten's work \cite[Theorem~5.16]{KestenAspects}, showing that $\mathrm{Var}~T(0,ne_1) \leq C(n/\log^{1/p} n)^2$ for $p = 9d+3$. Although this bound is only logarithmically better than a trivial bound (say in the case of bounded weights), the proof is far from trivial. In '93, Kesten introduced the ``method of bounded differences'' to FPP, and with this he was able to prove the best current bounds for $\chi$:
\[
0 \leq \chi(d) \leq 1/2 \text{ for all } d \geq 1\ .
\]
We will begin by giving a sketch of his argument using the Efron-Stein inequality.
\begin{theorem}[Kesten \cite{Kesten}]\label{thm: kesten_variance}
Assume $\mathbb{E}\tau_e^2<\infty$, $\mathbb{P}(\tau_e=0) < p_c(d)$ and that the distribution of $\tau_e$ is not concentrated at one point. There exist $C_1,C_2>0$ such that for all non-zero $x \in \mathbb{Z}^d$,
\[
C_1 \leq \mathrm{Var}~T(0,x) \leq C_2 \|x\|_1\ .
\]
\end{theorem}

The proof will use the following inequality for functions of independent random variables. We use the notation $x_+ = \max\{0,x\}$.
\begin{lemma}[Efron-Stein's inequality]
Let $X_1, X_2, \ldots$ be independent and let $X_i'$ be an independent copy of $X_i$, for $i \geq 1$. If $f$ is an $L^2$ function of $(X_1, X_2, \ldots)$ then
\[
\mathrm{Var}~f \leq \sum_{i=1}^\infty \mathbb{E}[(Z_i-Z)_+]^2\ ,
\]
where $Z = f(X_1, X_2, \ldots)$ and
\[
Z_i = f(X_1, \ldots, X_{i-1},X_i', X_{i+1}, \ldots)\ .
\]
\end{lemma}
\begin{proof}
Letting $\Sigma_i = \sigma(X_1, \ldots, X_i)$, the proof consists of writing $\mathrm{Var}~f = \sum_{i=1}^\infty \mathbb{E}\Delta_i^2$, where $\Delta_i = \mathbb{E}[f\mid \Sigma_i] - \mathbb{E}[f \mid \Sigma_{i-1}]$ and then applying Jensen's inequality to $\Delta_i^2$, along with symmetry. See \cite{Efron,Steele}.
\end{proof}

\begin{proof}[Proof of Theorem~\ref{thm: kesten_variance}]
We now apply Efron-Stein to the passage time, noting that the condition $\mathbb{E}\tau_e^2<\infty$ implies that $T$ has two moments. Then
\[
\mathrm{Var}~T(0,x) \leq \sum_{i=1}^\infty \mathbb{E}[(T_i(0,x)-T(0,x))_+]^2\ ,
\]
where $e_1, e_2, \ldots$ is any enumeration of the edges and $T_i$ is the passage time in the edge-weight configuration $(\tau_e)$ but with the weight $\tau_{e_i}$ replaced by an independent copy $\tau_{e_i}'$. Note that $T_i(0,x) > T(0,x)$ only when both $e_i$ is in \textsc{Geo}$(0,x)$, the intersection of all geodesics from $0$ to $x$ in the original edge-weight configuration $(\tau_e)$, and $\tau_{e_{i}}' > \tau_{e_{i}}$. Furthermore, in this case, $T_i(0,x) - T(0,x) \leq \tau_{e_i}'$. Therefore we obtain the bound
\[
\sum_{i=1}^\infty \mathbb{E}(\tau_{e_i}')^2 \mathbf{1}_{\{e \in \textsc{Geo}(0,x)\}}\ .
\]
By independence, this equals
\[
\mathbb{E}\tau_{e}^2 \mathbb{E}\#\textsc{Geo}(0,x)\ .
\]
Therefore we can conclude the upper bound for $\mathrm{Var}~T(0,x)$ with the following lemma.
\begin{lemma}\label{lem: geo_bound}
Assume $F(0)<p_c(d)$. There exists $C>0$ such that for all $x \in \mathbb{Z}^d$,
\[
\mathbb{E}\#\textsc{Geo}(0,x) \leq C \mathbb{E}T(0,x)\ .
\]
\end{lemma}
\begin{proof}
Note that if $\tau_e \in [a,b]$ almost surely, where $0<a<b<\infty$, then the statement is easy to prove. Indeed, letting $\Gamma$ be a geodesic from $0$ to $x$, one has
\[
a\#\Gamma \leq T(\Gamma) = T(0,x) \leq b\|x\|_1\ ,
\]
giving $$\#\textsc{Geo}(0,x) \leq (b/a)\|x\|_1.$$ In the general case, one can modify the idea from \cite[Cor.~1.4]{DHS1}. Setting $G(0,x)$ to be the maximal number of edges in any self-avoiding geodesic from $0$ to $x$ and 
\[
Y_x = G(0,x) \mathbf{1}_{\{T(0,x) < a G(0,x)\}}\ ,
\]
then \cite[Prop.~1.3]{DHS1} gives existence of $a>0$ such that for some $C_2>0$, one has $$\mathbb{P}(Y_x \geq n) \leq e^{-C_2n}$$ for all $x \in \mathbb{Z}^d$ and $n \geq 1$. Then write
\[
\#\textsc{Geo}(0,x) \leq G(0,x) \leq a^{-1} T(0,x) + Y_x\ .
\]
Taking expectations gives the result.
\end{proof}

The lower bound in Kesten's theorem is easier. Setting $\Sigma$ to be the sigma-algebra generated by the $2d$ edge-weights for edges adjacent to $0$, one has $$\mathrm{Var}~T(0,x) \geq \mathbb{E} \left( \mathbb{E}[T(0,x) \mid \Sigma] - \mathbb{E}T(0,x)\right)^2 = \mathrm{Var} \bigg[ \mathbb E[ T(0,x) | \Sigma] \bigg].$$

Let $t_{1}', \ldots, t_{2d}'$ be independent copies of the passage times of edges adjacent to $0$ and set $\Sigma' = \sigma(t_{1}', \ldots, t_{2d}')$. Let $T'(0,x)$ be the passage time from $0$ to $x$ where we replaced the edge weights of adjacent edges to the origin by $t_{i}'$s. Then the right side of the display above is equal to 
$$ \frac{1}{2} \E \bigg[ \E [T(0,x)|\Sigma] - \E[T'(0,x)|\Sigma']\bigg]^{2}.$$

Now pick $a<b$ such that $\mathbb{P}(\tau_e<a)>0$ and $\mathbb{P}(\tau_e>b)>0$. By considering the events $A = \{t_{\{0,y\}} < a \text{ for all } y \in \mathbb{Z}^d \text{ with } \|y\|_1=1\}$ and $B$ the same event but with $<a$ replaced by $>b$ and with $t_{\{0,y\}}$'s replaced by $t_{i}'$s, one has the lower bound 
\[
\E \bigg[ \big[ \E [T(0,x) |\Sigma] - \E[T'(0,x)|\Sigma']\big]\indi_{A \cap B}\bigg]^{2} \geq C (b-a)^2>0
\]
independently of $x$.
\end{proof}

We end by restating three open questions discussed at the beginning of the section.

\begin{question}
Show that for any $d\geq 2$, under suitable conditions on $F$, 
$\chi<1/2$. 
If $d=2$, show that $\chi=1/3$.
\end{question}

\begin{question}
Determine whether or not $$\lim_{d \to \infty} \chi(d) = 0.$$
\end{question}

\begin{question} For suitable $d$, show that $\chi >0$.
\end{question}

\subsection{Log improvement to upper bound for $d \geq 2$}\label{sec:logimprovementsUB}
The main tools used in the proof of Kesten's bound were (a) $\mathbb{E}\#\textsc{Geo}(0,x) \leq C\|x\|_1$ and (b) $T_i(0,x) -T(0,x) \leq \tau_{e_i}'$. To improve the variance upper bound we will need to use one more piece of information: for most edges $e$, the probability that $e$ is in a geodesic from $0$ to $x$ is small (in $x$). Another way to say this is that each edge has small \emph{influence} on the variable $T(0,x)$. This statement requires $d\geq 2$ since in $d=1$, each edge has high influence (there is only one path from $0$ to $x$ in that case).

The first proof of sublinear variance for $T(0,x)$ was due to Benjamini-Kalai-Schramm \cite{BKS} in '03 and applied only to $\tau_e$ that are Bernoulli: there exist $0<a<b<\infty$ such that $\tau_e$ takes values $a$ or $b$ with probability $1/2$. This specific distribution was needed to take advantage of Talagrand's influence inequality on the hypercube \cite{Talagrand_Russo}. The second proof was due to Bena\"im-Rossignol \cite{BenaimRossignol} in '08 and applied to distributions in the \emph{nearly-Gamma} class: those distributions that satisfy a log-Sobolev inequality similar to that for the Gamma distribution. Their methods were based on entropy and used an inequality due to Falik-Samorodnitsky\cite{FS} which replaced Talagrand's inequality (a similar inequality was derived by Rossignol \cite{R06}).

The most recent proof is due to Damron-Hanson-Sosoe \cite{DHS2} in '14 and applies to all distributions with $2+\log$ moments. Their method follows that of Bena\"im-Rossignol, but replaces the representation of an edge-weight as a push-forward of a Gaussian variable with a representation using a Bernoulli encoding. That is, each edge-weight is encoded as an infinite sequence of $0/1$-valued random variables, and the Gross two-point inequality \cite{Gross75} is used to bound the entropy.

Below we state the sublinear variance bound from \cite{DHS2}, but we will sketch the proof only in the simplest case (uniform weights), following \cite{BenaimRossignol}, and indicating where complications arise in extending the argument.

\begin{theorem}\label{thm:DHS222}
For $d \geq 2$, suppose $\mathbb{P}(\tau_e=0)<p_c$ and $\mathbb{E}\tau_e^2(\log \tau_e)_+<\infty$. There exists $C>0$ such that for all $x \in \mathbb{Z}^d$ with $\|x\|_1>1$,
\[
\mathrm{Var}~T(0,x) \leq C\frac{\|x\|_1}{\log \|x\|_1}\ .
\]
\end{theorem}

To prove the theorem above, we will use Falik-Samorodnitsky's inequality. Recall that the entropy of a nonnegative random variable $X$ is defined as
\[
\mathrm{Ent}~X = \mathbb{E}X \log X - \mathbb{E}X \log \mathbb{E}X\ .
\]
Now, let $\nu_n := \nu \times \ldots \times \nu$ be the uniform measure on $[0,1]^{n}$. In what follows, expectation is with respect to $\nu_{n}$. For $1\leq k \leq n$, define   $\Sigma_k$ to be the sigma-algebra generated by the first $k$ coordinates in $\mathbb{R}^n$, with $\Sigma_{0}$ the trivial sigma-algebra. Let $f:\mathbb{R}^n \to \mathbb{R}$ be such that $\mathbb{E} f^2<\infty$. Last, define the martingale difference
\[
\Delta_k = \mathbb{E} [f \mid \Sigma_k] - \mathbb{E} [f \mid \Sigma_{k-1}]\ .
\]


\begin{theorem} \label{thm:FSineq}(Falik-Samorodnitsky) Let $f:\mathbb{R}^n \to \mathbb{R}$ be nonconstant and  such that $\mathbb{E}f^2<\infty$. Then

\begin{equation}\label{eq:FSinequality}
\mathrm{Var}~f ~ \log \left[ \frac{\mathrm{Var}~f}{\sum_{k=1}^n \left( \mathbb{E} |\Delta_k| \right)^2 } \right] \leq \sum_{k=1}^n \mathrm{Ent}~\Delta_k^2\ .
\end{equation}
\end{theorem}

Before we prove the above inequality, a few words of comment are needed. 
First, note that by the martingale decomposition of the variance 
$$  \mathrm{Var}~f = \sum_{k=1}^{n} \mathbb{E} \Delta_k^2\ . $$
Applying Jensen's inequality, 
\begin{equation}\label{eq:JensenFS}
  \mathrm{Var}~f \geq\sum_{k=1}^n \left( \mathbb{E} |\Delta_k| \right)^2, 
\end{equation}
so the term inside the logarithm in \eqref{eq:FSinequality} is greater than equal to $1$. 

The inequality \eqref{eq:FSinequality} is useful to obtain $\log$-sublinear bounds if one can show that the right side of  \eqref{eq:JensenFS} is lower order of $\mathrm{Var}~f$. We will see that this is the case in our setting as the ratio can be shown to be at least of order $n^{\alpha}, \alpha >0$.  Cutting the story short, the $\log \|x \|_{1}$ improvement in Theorem \ref{thm:DHS222} comes from the appearance of the $\log$ term in the left side of Falik-Samorodnitsky's inequality.

Last, Equation \eqref{eq:FSinequality} first appeared in a paper of Falik-Samorodnitsky \cite[Equation (3)]{FS} as a functional version of an edge-isoperimetric inequality for Boolean functions. We provide a  different proof than the one given in the Appendix of \cite{FS}.

We start the proof of \eqref{eq:FSinequality} with the following lemma.

\begin{lemma}\label{lem:helloswimmers} Let $f$ be a nonnegative function on a probability space $(\Omega, \mathcal F, \mathbb P)$ such that $\mathrm{Ent} f^{2} <\infty$. Then,
\begin{equation}\label{eq:LemmaFSineq}
 \mathbb E f^{2}  \log \frac{\mathbb E f^{2}}{(\mathbb Ef)^{2}} \leq \mathrm{Ent} (f^{2}).
\end{equation}
\end{lemma}
If $f$ is identically zero we interpret  the left side of \eqref{eq:LemmaFSineq} to be $0$.
\begin{proof}
 As  the inequality is preserved if we multiply $f$ by any positive constant, we can assume that $\mathbb E f^{2} = 1$.  In this case, the inequality reads
$$  - \log (\mathbb Ef)^{2} \leq  \mathbb{E}f^2 \log f^2,  $$ 
or, as $f \geq 0$,
$$  0  \leq  \mathbb{E}f^2 \log (f \mathbb E f).  $$ 
However, on the event $f>0$, we can use the fact that $ 1- x \leq \log x^{-1}$ with $x= (f \mathbb E f)^{-1}$ to obtain
$$ 0 =  \mathbb E f^{2} -1 = \mathbb E \bigg[ f^{2} \bigg[ 1- \frac{1}{f \mathbb E f}\bigg];f>0\bigg ] \leq \mathbb{E}f^2 \log (f \mathbb E f). $$
Therefore, \eqref{eq:LemmaFSineq} holds. 

\end{proof}

\begin{proof}[Proof of Theorem \ref{thm:FSineq}]
We use Lemma \ref{lem:helloswimmers} on each $\Delta_{k}$:
$$ \sum_{k=1}^{n}\mathrm{Ent} (\Delta_{k}^{2}) \geq   \sum_{k=1}^{n} \mathbb E (\Delta_{k}^{2}) \log \frac{\mathbb E (\Delta_{k}^{2}) }{ (\mathbb E |\Delta_{k}|)^{2} } = - (\mathrm{Var}~f ) \sum_{k=1}^{n} \frac{\mathbb E (\Delta_{k}^{2})}{\mathrm{Var}~f} \log \frac{ (\mathbb E |\Delta_{k}|)^{2} }{  \mathbb E (\Delta_{k}^{2})}.   $$
As $\mathrm{Var}~f  = \sum_{k=1}^{n}\mathbb E (\Delta_{k}^{2})$, we can apply Jensen's inequality with the function $x \mapsto -\log x$ to get the lower bound
$$ - (\mathrm{Var}~f ) \log \sum_{k=1}^{n} \frac{\mathbb E (\Delta_{k}^{2})}{\mathrm{Var}~f}  \frac{ (\mathbb E |\Delta_{k}|)^{2} }{  \mathbb E (\Delta_{k}^{2})}.  $$
The above equation is exactly the left-side of \eqref{eq:FSinequality}. 
\end{proof}

With Falik-Samorodnitsky's inequality in our hands, we turn back to the proof of Theorem \ref{thm:DHS222}. As mentioned before, we will prove this theorem under the assumption that the passage times are uniformly distributed in the interval $[0,1]$. This assumption allow us to use the fact that this probability  measure satisfies a log-Sobolev inequality. Precisely, for any $n\geq 1$, under uniform measure on $[0,1]^{n}$, there exists $C>0$ such that for any $f:[0,1]^{n} \to \mathbb{R}$  that is smooth, one has
\[
\mathrm{Ent}(f^2) \leq C \mathbb{E} \|\nabla f \|_{2}^2\ .
\]

The above inequality combined with Theorem \ref{thm:FSineq} leads to 

\begin{equation}\label{eq: FS_LSI}
\mathrm{Var}~f ~ \log \left[ \frac{\mathrm{Var}~f}{\sum_{k=1}^n \left( \mathbb{E} |\Delta_k| \right)^2 } \right] \leq C  \sum_{k=1}^{n} \mathbb E \| \nabla \Delta_{k} \|_{2}^{2}  = C \mathbb E \| \nabla f\|_{2}^2\ ,
\end{equation}
where in the last equality we used the fact that $\partial_{i} \Delta_{k} = 0$ if $i \neq k$ and $\sum_{k=1}^{n} \mathbb E |\partial_{k} \Delta_{k}|^{2} =   \mathbb E \|\nabla f\|_{2}^{2}.$

\begin{proof}[Proof of Theorem \ref{thm:DHS222}]

We now apply \eqref{eq: FS_LSI} to the passage time $T(0,x)$, noting that since we assumed that our edge-weights are uniform, $T(0,x)$ is bounded, and so we can extend the above inequality with $n \to \infty$ as
\begin{equation}\label{eq: FS_LSI_new}
\mathrm{Var}~T(0,x) ~\log \left[ \frac{\mathrm{Var}~T(0,x)}{\sum_{k=1}^\infty \left( \mathbb{E}|\Delta_k|\right)^2} \right] \leq C\sum_{i=1}^\infty \mathbb{E} \left( \frac{\partial}{\partial \tau_{e_i}} T(0,x) \right)^2\ .
\end{equation}
The derivative on the right is in the sense of distributions (since the passage time is not a smooth function of the edge-weights) and it is relative to the $i$-th edge-weight, where we have enumerated the edges in the lattice as $e_1, e_2, \ldots$. When the edge-weights are not bounded, one needs to argue that $n$ can be taken to $\infty$ more carefully, imposing that $T$ has at least $2+\epsilon$ moments (and this is guaranteed by existence of $2+\log$ moments for $\tau_e$), to exploit uniform integrability.

Next we use the fact that
\[
\frac{\partial}{\partial \tau_{e_i}} T(0,x) = \mathbf{1}_{\{e_i \in \textsc{Geo}(0,x)\}}\ ,
\]
where we recall that $\textsc{Geo}(0,x)$ is the intersection of all geodesics from $0$ to $x$. This holds Lebesgue almost surely, which is ok for us since the weights are uniform. So we obtain an upper bound for the right side of \eqref{eq: FS_LSI_new} of
\[
C \mathbb{E}\#\textsc{Geo}(0,x) \leq C\|x\|_1\ .
\]
Here we have used Lemma~\ref{lem: geo_bound}. Note that this is the same bound we obtained from Efron-Stein (in Kesten's method in the last section) but now the advantage is that we have an extra factor of $\log [ \cdots ]$ on the left of \eqref{eq: FS_LSI_new}.

We are now left to show that $\sum_{k=1}^\infty \left( \mathbb{E}|\Delta_k| \right)^2$ is at most $\|x\|_1^a$ for some $a<1$. If we succeed in this, then \eqref{eq: FS_LSI_new} implies sub-linear variance. Indeed, in that case, either $\mathrm{Var}~T(0,x)$ is already $\leq C\|x\|_1^{1-\epsilon}$ for some $\epsilon>0$, or it is not, in which case, the term $\log [ \cdots ]$ is at least order $\log \|x\|_1$, and we divide it to the other side to complete the proof.

Unfortunately, it is not known how to show the  bound on $\sum_{k=1}^\infty \left( \mathbb{E}|\Delta_k| \right)^2$. The reason is that it is at most of order $\sum_{k=1}^\infty (\mathbb{P}(e_k \in \textsc{Geo}(0,x)))^2$, and the only information we have on these probabilities is
\[
\sum_{k=1}^\infty \mathbb{P}(e_k \in \textsc{Geo}(0,x)) = \mathbb{E}\#\textsc{Geo}(0,x) \asymp \|x\|_1\ .
\]
If the geodesic prefers to take certain nearly deterministic edges (say in a small tube centered on an $\ell^2$-geodesic from $0$ to $x$), then the sum of squares can be of order $\|x\|_1$. The work of Benjamini-Kalai-Schramm \cite{BKS} introduced an averaging trick to get around this. The main realization is that if the system is translation-invariant, we can give the appropriate inequality. First, we can restrict attention to a box around $0$ of size $C\|x\|_1$ for a large constant $C$. If all the probabilities are equal, they must be of order $\|x\|_{1}^{1-d}$. Plugging this in gives the correct bound. 

So we consider an averaged passage time (this form of averaging was used in \cite{Sodin,AlexZygouras})
\[
F_m = \frac{1}{Z_m} \sum_{\|z\|_1 \leq m} T(z,z+x)\ ,
\]
where the sum is over integer sites only and $m = \left\lfloor \|x\|_1^{1/4} \right\rfloor$. $Z_m$ is the number of terms in the sum. By Jensen's inequality, we can still obtain the same upper bound in \eqref{eq: FS_LSI}, using $f=F_m$:
\[
\sum_{i=1}^\infty \mathbb{E} \left( \frac{\partial}{\partial e_i} \frac{1}{Z_m} \sum_{\|z\|_1 \leq m} T(z,z+x) \right)^2 \leq \frac{1}{Z_m} \sum_{\|z\|_1 \leq m} \sum_{i=1}^\infty \mathbb{P}(e_i \in \textsc{Geo}(z,z+x))\ ,
\]
which is bounded by $C\|x\|_1$. Furthermore, $F_m$ is not too different from $T(0,x)$: by the triangle inequality,
\[
\|F_m - T(0,x)\|_2 \leq \frac{1}{Z_m} \sum_{\|z\|_1 \leq m} \|T(z,z+x)-T(0,x)\|_2 \leq 2 \max_{\|z\|_1 \leq m} \|T(0,z)\|_2\ .
\]
By our bound on the weights,
\[
\|F_m - T(0,x) \|_2 \leq 2\|x\|_1^{1/4}\ ,
\]
which is $o(\|x\|^{1/2}/\log \| x\|_{1})$, so it suffices to bound $\mathrm{Var}~F_m$. By the arguments in the beginning of the proof, we need only show that $\sum_{k=1}^\infty \left( \mathbb{E}|\Delta_k| \right)^2 \leq \|x\|_1^a$ for some $a < 1$, where $\Delta_k$ is the martingale difference associated to $F_m$.

By applying Jensen's inequality (as in the proof of Efron-Stein),
\[
\mathbb{E}|\Delta_k| \leq \frac{2}{Z_m} \sum_{\|z\|_1 \leq m} \mathbb{E} |T_i(z,z+x)-T(z,z+x)|_+ \leq  \frac{2}{Z_m} \sum_{\|z\|_1 \leq m} \mathbb{P}(e_k \in \textsc{Geo}(z,z+x))\ .
\]
By translation invariance, we obtain the upper bound
\[
\frac{2}{Z_m} \sum_{\|z\|_1 \leq m} \mathbb{P}(e_k-z \in \textsc{Geo}(0,x)) \leq Cm/Z_m\ .
\]
Here we have used an extension of Lemma~\ref{lem: geo_bound} found in \cite{DHS2}. So
\[
\sum_{k=1}^\infty \left( \mathbb{E}|\Delta_k|\right)^2 \leq \frac{2Cm}{Z_m^2} \sum_{\|z\|_1 \leq m} \sum_{k=1}^\infty \mathbb{P}(e_k \in \textsc{Geo}(z,z+x)) \leq \frac{Cm}{Z_m} \|x\|_1\ .
\]
This is bounded by $Cm^{1-d} \|x\|_1 \leq C\|x\|_1^{3/4}$ and completes the proof.
\end{proof}

In the general case (assuming only $\mathbb{E}\tau_e^2(\log \tau_e)_+<\infty$), the difference in the proof is in the entropy bound. By writing each $\tau_e$ as the push-forward of an infinite sequence  $\omega_e = (\omega_{e,1}, \omega_{e,2}, \ldots)$ of Bernoulli random variables, one can apply the Gross two-point entropy bound for
\[
\sum_{k=1}^\infty \mathrm{Ent}~\Delta_k^2 \leq C \sum_{k=1}^\infty \sum_{e} \mathbb{E}\left( \Delta_{e,k} F_m \right)\ ,
\]
where $\Delta_{e,k}$ is the discrete derivative operator relative to $\omega_{e,k}$. To give the upper bound $C\|x\|_1$ for the right side, one needs a careful analysis of these discrete derivatives and tools from the theory of greedy lattice animals. See \cite{DHS1} for details.

\subsection{Log improvement to lower bound for $d=2$}\label{sec:loweboundNP}

The following theorem by Newman-Piza in '95 represents the state of the art for the lower bound on the variance of passage times. It improves on Kesten's lower bound by a factor of $\log n$.  
\begin{theorem}[Newman-Piza \cite{NewmanPiza}]\label{thm:NP}
 Let $\lambda = \inf \{s : \mathbb{P}(\tau_e \leq s)>0\}$. Assume $\mathbb{E}\tau_e^2<\infty$ and $\mathrm{Var}~\tau_e>0$. Assume in addition that one of the following two conditions is satisfied:
\[
\lambda = 0 \text{ and } \mathbb{P}(\tau_e=0) < p_c
\]
\[
\lambda > 0 \text{ and } \mathbb{P}(\tau_e=\lambda)<\vec p_c\ .
\]
If $d=2$ then there is a constant $B>0$ such that 
\begin{equation}\label{eq:NPlowerboundformula}
\mathrm{Var}~T(0,nx) \geq B\log n
\end{equation}
for $n \geq 1$ and all unit vectors $x \in \mathbb{R}^2$.
\end{theorem}
In the case when $\tau_e$ is exponential mean $1$, the $\log n$ lower bound was also obtained, using different methods, by Pemantle and Peres \cite{PP}. Theorem \ref{thm:NP} was extended to distributions in $\mathcal{M}_p$ (whose limit shape has a flat edge -- see Section \ref{sec: flat_edge}) first in the $e_{1}$ direction by Zhang \cite{Zhang} and then for all directions outside the percolation cone by Auffinger and Damron \cite{AD12} (see also Kubota \cite[Corollary 1.4]{Kubota}, who reduced the moment condition of \cite{AD12}). It is important to note, however, that inside the percolation cone, the variance of the passage time is of order constant \cite{Zhang}, and this is a strong version of $\chi = 0$ in those directions.

Newman-Piza used a martingale method to obtain the desired lower bound in \eqref{eq:NPlowerboundformula}. Their technique goes as follows. Enumerate the edges $(e_i)$ of $\mathbb{Z}^2$ in a spiral order starting from the origin and define a filtration $(\Sigma_i)$, where $\Sigma_i$ is the sigma-algebra generated by the weights $\tau_{e_1}, \ldots, \tau_{e_i}$ and $\Sigma_0$ is the trivial sigma-algebra. Writing $T=T(0,nx)$, we may use $L^2$-orthogonality of martingales to find
\[
\mathrm{Var}~T = \sum_{i=1}^\infty \E \left[\E[T~|~\Sigma_i] - \E[T~|~\Sigma_{i-1}]\right]^2\ .
\]
The $i$-th term in the sum represents a part of the contribution to the variance given by fluctuations of the passage time of the $i$-th edge $e_i$. The idea of \cite{NewmanPiza} is that if $e_i$ is in a geodesic from $0$ to $nx$ then if we lower its passage time (while keeping all other weights fixed), the variable $T$ will decrease linearly, and therefore $\tau_{e_i}$ will have influence on the fluctuations of $T$. They consequently argue (see Theorem \ref{thm:NPgen} below) that one has a lower bound for $\mathrm{Var}~T$ of
\begin{equation}\label{eq:eventsinNP}
C~\sum_{i=1}^\infty \P(F_i)^2\ ,
\end{equation}
where $F_i$ is the event that $e_i$ is in a geodesic from $0$ to $nx$.

Precisely, one can summarize the Newman-Piza lower bound method in the following statement.

Let $I$ be a countable set. Consider the probability space $(\Omega, \mathcal F, \mathbb P)$ where $\Omega = \mathbb{R}^I = \{\omega = (\omega_i~:~i \in I)\}$, $\mathcal{F}$ is the Borel sigma-field $\mathcal{B}^I$ and $\mathbb P$ is a probability measure. Suppose $T$ is a random variable with $\mathbb E(T^2)<\infty$. For $W \subseteq I$, write $\mathcal{F}(W)$ for the Borel sigma-field $\mathcal{B}^W$. Let $U_1, U_2, \ldots$ be disjoint subsets of $I$ and express $\omega$ for each $k$ as $(\omega^k,\hat\omega^k)$, where $\omega^k$ (resp. $\hat \omega^k$) is the restriction of $\omega$ to $U_k$ (resp. to $I\setminus U_k$). For each $k$, let $D_k^0$ and $D_k^1$ be disjoint events in $\mathcal{B}^{U_k}$. Define
\[
H_k(\omega) = T_k^1(\hat\omega^k) - T_k^0(\hat \omega^k)\ ,
\]
where 
\[
T_k^0(\hat\omega^k) = \sup_{\omega^k \in D_k^0} T((\omega^k,\hat\omega^k)) \text{ and } T_k^1(\hat\omega^k) = \inf_{\omega^k \in D_k^1} T((\omega^k,\hat\omega^k))\ .
\]
\begin{theorem}[Newman-Piza \cite{NewmanPiza}]\label{thm:NPgen}
Assume the setting just described and the following three hypotheses about $\mathbb P$, the $U_k$'s, the $D_k^\delta$'s and $T$:
\begin{enumerate}
\item Conditional on $\mathcal{F}(I \setminus \cup_k U_k)$, the $\mathcal{F}(U_k)$'s are mutually independent.
\item There exist $a,b>0$ such that, for any $k$,
\[
\mathbb P\big (\omega^k \in D_k^0~|~\mathcal{F}(U_k^c) \big) \geq a \text{ and } \mathbb P\big (\omega^k \in D_k^1~|~\mathcal{F}(U_k^c) \big) \geq b \text{ a.s. }
\]
\item For every $k$, $H_k \geq 0$ a.s.
\end{enumerate}
Suppose that, for some $\varepsilon>0$ and each $k$, $F_k \in \mathcal{F}$ is a subset of the event $\{H_k \geq \varepsilon\}$. Then
\begin{equation} \label{eq:dadsadasdassdsadasdasdsa}
\mathrm{Var}~ T \geq ab \varepsilon^2 \sum_k \mathbb P(F_k)^2\ .
\end{equation}
\end{theorem}

We will show in detail how \eqref{eq:dadsadasdassdsadasdasdsa} is used. We will use the approach of \cite{AD12}.  Before doing so, let us sketch how Newman-Piza proceeded in their original paper.

When  $F_i$ is the event that $e_i$ is in a geodesic from $0$ to $nx$, the sum on the right-side of \eqref{eq:dadsadasdassdsadasdasdsa} is exactly the expected overlap (the number of common edges) of two geodesics from $0$ to $nx$ sampled independently.  However, this overlap is difficult to control and Newman-Piza replace the right-hand side by an expression involving the partial sums:
\begin{equation}\label{eq:NewmanP}
\sum_{j=1}^{k} \Pro (F_j) = \E(\text{number of minimizing edges among the first } k \text{ edges}).
\end{equation}

 Once \eqref{eq:NewmanP} is established, the logarithmic lower bound follows after showing that \eqref{eq:NewmanP} is bounded below by $B\sqrt{k}$ for $k\leq Cn^2$. This step is where the $2$ dimensional assumption comes into play. They use the fact that within a box of $\ell_1$-distance $L$ to the origin, there are order $L^2$ edges.  Thus, the right side of \eqref{eq:NewmanP} is  bounded below by the expected passage time to the boundary of a box with radius $\sqrt{k}$, which is, by the shape theorem, of order $B\sqrt{k}$. 

A solution of the following question, combined with the Newman-Piza method, would provide an improvement on the lower bound of the variance. 
\begin{question}
Show that in $d=2$ there exists $a>0$ such that the expected overlap of two geodesics  from $0$ to $nx$ sampled independently is at least $n^a$.
\end{question}

The above technique works well if we are allowed to lower the edge weights along a geodesic. In the class $\mathcal{M}_p$, when the passage time distribution has an atom at $1$, the bottom of its support, this turns out to be a problem. Therefore if we are to use the same technique, we must show that geodesics use many edges with weight above $1$. It is not enough only to know this though; if one repeats the computations above, one finds only a lower bound of a constant (with no logarithm term). It is essential also to know information about the location of these non-one edges on the lattice. In particular, if they are heavily concentrated enough near the origin, we can extract a logarithmic bound. To do this, we need to know something about the geometry of geodesics (for instance, that they avoid certain regions of the plane). The proof that we will give will also allow us to extend Theorem \ref{thm:NPgen} as follows. 

Recall that $N_p$ is the edge of the percolation cone, and let $\theta_p$ be the unique angle such that the line segment connecting $0$ and $N_p$ has angle $\theta_p$ with the $e_{1}$-axis. Let $w_\theta$ be the vector 
\[
w_\theta = (\cos \theta, \sin \theta).
\]

\begin{theorem}\label{thm:unbounded}
Let $\mu \in \mathcal{M}_p$ for $p\in[\vec p_c,1)$ and $\theta \in [0,\theta_p)$. Suppose that $\E \tau_e^2<\infty$. Then there exists $C_\theta>0$ such that for all $n$,
\begin{equation}\label{eq:loglowerbound}
\mathrm{Var}~T(0,nw_\theta)\geq C_\theta \log n\ .
\end{equation}
\end{theorem}

Fix $y>1$ such that 
$$ q = \nu([y,\infty))>0.$$ Assume that there exists $\rho>0$ such that given $\e>0$ we may find $0<\zeta<1$ and $n_0 \in \mathbb{N}$ such that $n>n_0$ and $n^\zeta \leq M \leq n/2$ imply that with probability at least $1-\e/2$,
\begin{enumerate}
\item[(A)] every geodesic from $0$ to $nw_\theta$ contains at least $\rho M$ edges in the set $M \mathcal{B}_\nu$ with weights at least equal to $y$.
\end{enumerate}
Also assume that we can find $c>1$ and $L_0>0$ such that with probability at least $1-\e/2$,
\begin{enumerate}
\item[(B)] for all $L>L_0$ every geodesic from $0$ to $nw_\theta$ contains at most $cL$ edges in $L\mathcal{B}_\nu$.
\end{enumerate}

These assumptions are not difficult to verify given the information that we already know about the limit shape. Indeed, (B) is a straight-forward consequence of the shape theorem and the fact that $1$ is the infimum of the support. To verify (A), we use the following Lemma, taken from \cite[Lemma 9]{AD12}.

\begin{lemma}\label{lem:nononeedges}
Let $\mathcal C$ be a closed subset of $\partial \mathcal  B_\nu$ that does not intersect the flat edge $[M_p,N_p]$ or any of its reflections about the axes. Given $\varepsilon>0$, there exists $M_1$ and $\rho>0$ such that with probability at least $1-\varepsilon$, the following holds. For all $M>M_1$, $x \in M\mathcal C$ and for every geodesic $\Gamma$ from $0$ to $x$, at least $\rho M$ edges of $\Gamma$ have passage times $\geq y$.
\end{lemma}

 In the rest of this subsection, we explain how to prove Theorem~\ref{thm:unbounded} using Theorem \ref{thm:NPgen} and assumptions (A) and (B).

\begin{proof}[Proof of Theorem~\ref{thm:unbounded}]

Recall that $p=\nu(\{ 1\}).$ We use Theorem \ref{thm:NPgen} with $T = T(0,nw_\theta)$, $U_k = \{e_k\}$, $D_k^0$ the event that $\tau_{e_k} = 1$ and $D_k^1$ the event that $\tau_{e_k} \geq y$. Last, set $\varepsilon = y-1$, $a=p$ and $b=q$. The reader may verify that all of the hypotheses of Theorem~\ref{thm:NP} are satisfied. Therefore, setting
\[
F_k = \big \{\tau_{e_k} \geq y \text{ and } e_k \text{ is in a geodesic from } 0 \text{ to } nw_\theta \big\}\ ,
\]
we find that
\begin{equation}\label{eq:Fk}
\mathrm{Var}~T(0,nw_\theta) \geq pq (1-y)^2 \sum_k \P(F_k)^2\ .
\end{equation}

We now define a sequence of numbers $( M_i)$ by $ M_1 = n^\zeta$ and, setting $J=2c/\rho$, for $i=1, \ldots, \mathcal{I}$ (for $\mathcal{I}=\lfloor (1-\zeta)/\log J\rfloor \log (n/2)$),
\[
 M_{i+1} = J  M_i\ .
\]

From \eqref{eq:Fk} it follows that if we write $\mathcal{E}_i$, $i=0, \ldots, \mathcal{I}-1$ for the set of edges in $ M_{i+1} \mathcal B_\nu$ but having no endpoints in $ M_i \mathcal B_\nu$ (for $i=0$ we take all edges with an endpoint in $ M_i \mathcal B_\nu$), then
\[
\mathrm{Var}~T(0,nw_\theta) \geq pq(1-y)^2 \sum_{i=0}^{\mathcal{I}-1} \sum_{k:e_k \in \mathcal{E}_i} \P(F_k)^2\ .
\]

The idea of the above decomposition is that if the aspect ratio $J$ of the annuli is large enough then the inner sum will be always at least of order constant. Since there is a logarithmic number of such annuli, we will obtain the desired lower bound.

Indeed, using Jensen's inequality, we get a lower bound of
\begin{equation}\label{eq:sum1}
pq(1-y)^2 \sum_{i=0}^{\mathcal{I}-1} \frac{1}{|\mathcal{E}_i|}\left(\sum_{k:e_k \in \mathcal{E}_i} \P(F_k)\right)^2\ .
\end{equation}
We will now give a lower bound for the inner sum. Call $X_{i+1}$ the event that (A) and (B) hold for $M= M_{i+1}$. On this event the number of edges on any geodesic $\Gamma$ from $0$ to $nw_\theta$ in the set $ M_i\mathcal B_\nu$ is at most $(\rho/2)  M_{i+1}$ and the number of edges on $\Gamma$ in $ M_{i+1}\mathcal B_\nu$ with weight at least $y$ is at least $\rho  M_{i+1}$. From this it follows that on the event $X_{i+1}$, we have the lower bound
\[
\sum_{k:e_k \in \mathcal{E}_i} I(F_k) \geq (\rho/2)  M_{i+1}\ .
\]
Since $ \mathcal B_\nu$ is bounded, there exists $d>0$ such that for all $i$, $ M_i \geq d \sqrt{|\mathcal{E}_i|}$. So \eqref{eq:sum1} is bounded below by
\[
pq(1-y)^2\left[ d(\rho/2) (1-\e)\right]^2\mathcal{I} \geq  C \log n
\]
for $C$ independent of $n$. This completes the proof.

\end{proof}

\subsection{Concentration bounds}

\subsubsection{Subdiffusive concentration}

In addition to sub-linear variance bounds, there has been work to establish concentration inequalities for $T(0,x)$ on the scale $(\|x\|/\log \|x\|)^{1/2}$. These have so far only been exponential inequalities, not Gaussian ones.

As in the case of sub-linear variance, the first such exponential inequality was not for general distributions, only those in the ``nearly Gamma'' class. The result we present below is from \cite{DHS1}, and only moment conditions are needed. Note that the condition for the lower-tail inequality is weaker than that for the upper-tail inequality.

\begin{theorem}\label{thm: subdiffusive}
Let $d \geq 2$ and suppose that $\mathbb{P}(\tau_e=0) < p_c$. If $\mathbb{E}e^{\alpha \tau_e}<\infty$ for some $\alpha>0$ then there exist $c_1,c_2>0$ such that for all $x \in \mathbb{Z}^d$ with $\|x\|_1 > 1$,
\[
\mathbb{P}\left( T(0,x) - \mathbb{E}T(0,x)  \geq \lambda \sqrt{\frac{\|x\|_1}{\log \|x\|_{1}}} \right) \leq c_1 e^{-c_2\lambda} \text{ for } \lambda \geq 0.
\]
If $\mathbb{E}\tau_e^2(\log \tau_e)_+ < \infty$, then for all $x \in \mathbb{Z}^d$ with $\|x\|_1 > 1$,
\[
\mathbb{P}\left( T(0,x) - \mathbb{E}T(0,x) \leq - \lambda \sqrt{\frac{\|x\|_1}{\log \|x\|_{1}}} \right) \leq c_1 e^{-c_2\lambda} \text{ for } \lambda \geq 0.
\]
\end{theorem}

\medskip
\noindent
{\bf Problem.}  Is the bound on the right side optimal? 

The main strategy is again due to Bena\"im-Rossignol and follows the same lines as their proof of sub-linear variance. The proof of the general case in \cite{DHS1} again involves a Bernoulli encoding and estimating discrete derivatives after applying the two-point entropy estimate. Either way, one defines the averaged passage time
\[
F_m = \frac{1}{Z_m} \sum_{\|z\|_1 \leq m} T(z,z+x),
\]
where the sum is over integer sites only, $m = \left\lfloor \|x\|_1^{1/8} \right\rfloor$, and $Z_m$ is the number of terms in the sum. One can show \cite[Section~2.1]{DHS1} that it suffices to derive the concentration inequality for $F_m$.

The main idea is to obtain a variance estimate for an exponential function of $F_m$ analogous to the one obtained by $T$. By following the sub-linear variance strategy, with more technical difficulty, one obtains the following inequality: for some $C>0$,
\begin{equation}\label{eq: exponential_variance}
\mathrm{Var}~e^{\lambda F_m/2} \leq K\lambda^2 \mathbb{E}e^{\lambda F_m} < \infty \text{ for } |\lambda| < \frac{1}{2\sqrt{K}} \text{ and } \|x\|_1>1,
\end{equation}
where $K = \frac{C\|x\|_1}{\log \|x\|_1}$. Note the similarity to the entropy bound obtained in \eqref{eq: entropy_mgf} in the proof of Talagrand's theorem. The above inequality can be thought of as weaker, due to the presence of the variance instead of entropy, but stronger due to the logarithmic factor.

The above variance estimate is turned into an exponential concentration bound using the ``iteration method.'' The following comes from \cite[p. 70-71]{BLM}.
\begin{proposition}
If $\mathbb{E}Z=0$ and for some constants $B,C$ satisfying $0 < C \leq B$,
\[
\mathrm{Var}~e^{tZ/2} \leq C t^2 \mathbb{E}e^{tZ} < \infty \text{ for } t \in (0,B^{-1/2})\ ,
\]
then putting $\psi_Z(t) = \log \mathbb{E}e^{tZ}$, one has
\[
\psi_Z(t) \leq -2\log(1-Ct^2) \text{ for } t \in (0,B^{-1/2})\ .
\]
\end{proposition}
\begin{proof}
Beginning with
\[
\mathbb{E}e^{tZ} - \left( \mathbb{E}e^{tZ/2} \right)^2 = \mathrm{Var}~e^{tZ/2}  \leq C t^2 \mathbb{E}e^{tZ}\ ,
\]
we obtain
\[
\mathbb{E}e^{tZ} \leq \frac{1}{1-Ct^2} \left(\mathbb{E}e^{tZ/2}\right)^2\ ,
\]
or
\[
\psi_Z(t) \leq - \log(1-Ct^2) + 2\psi_Z(t/2)\ .
\]

By induction, for $n \geq 1$,
\[
\psi_Z(t) \leq -\sum_{k=1}^n 2^{k-1} \log(1-C(t/2^{k-1})^2) + 2^n\psi_Z(t/2^n)\ .
\]
Because $\mathbb{E}Z=0$,
\[
2^n \psi_Z(t/2^n) = t \frac{\psi_Z(t/2^n)-\psi_Z(0)}{t/2^n} \to t\psi_Z'(0) = 0\ ,
\]
so
\begin{equation}\label{eq: almostend}
\psi_Z(t) \leq - \sum_{k=1}^\infty 2^{k-1} \log(1-C(t/2^{k-1})^2)\ .
\end{equation}
To bound these terms, we use the fact that $-u^{-1}\log(1-u)$ is non-decreasing in $u \in (0,1)$. 

Therefore
\[
-C^{-1}(t/2^{k-1})^{-2} \log(1-C(t/2^{k-1})^2) \leq -C^{-1}t^{-2} \log(1-Ct^2) \text{ for } k \geq 1\ .
\]
Rewritten,
\[
-2^{k-1}\log(1-C(t/2^{k-1})^2) \leq -2^{-(k-1)} \log(1-Ct^2)\ .
\]
Placing this in \eqref{eq: almostend},
\[
\psi_Z(t) \leq - \log(1-Ct^2) \sum_{k=1}^\infty 2^{-(k-1)} = -2\log(1-Ct^2)\ .
\]
\end{proof}

To make the last bound more clear, we can use the inequality
\[
-\log(1-u) \leq \frac{u}{1-u} \text{ for } u \in [0,1)\ ,
\]
which follows from the mean value theorem.

So we obtain
\[
\psi_Z(t) \leq \frac{2Ct^2}{1-Ct^2} \leq \frac{4Ct^2}{2(1-\sqrt{C}t)} \text{ for } t \in (0,B^{-1/2}).
\]

In our case, $B=C=K$ in \eqref{eq: exponential_variance}, so we get
\[
\log \mathbb{E}e^{t|F_m - \mathbb{E}F_m|} \leq \frac{4Kt^2}{2(1-\sqrt{K}t)} \text{ for } t \in (0,K^{-1/2}).
\]
By using the bound
\[
\mathbb{P}(X \geq x) = \mathbb{P}(e^{\lambda X} \geq e^{\lambda x}) \leq e^{-\lambda x} \mathbb{E}e^{\lambda X},
\]
which is valid for $\lambda \geq 0$, and optimizing over $\lambda$, one can then complete the first inequality of Theorem~\ref{thm: subdiffusive}.

\subsubsection{Talagrand's theorem via the entropy method}\label{sec: talagrand}

Here we will give the concentration argument from \cite[Corollary~A.5]{DHS1}. The goal will be to give an exponential concentration inequality for the passage time about its mean assuming certain moment conditions for the edge weights. This is not the main result of that paper, but an auxiliary one used to obtain the main one. The lower tail inequality comes from \cite{DKubota}. Our aim will be to prove the following result. It was initially established by Talagrand in \cite{Talagrand_IHES}, using different methods. 
\begin{theorem}\label{thm:futurereference}
Let $d \geq 2$. Assuming $\mathbb{P}(\tau_e=0)<p_c$ and $\mathbb{E}e^{\alpha \tau_e}<\infty$ for some $\alpha>0$, there exist $C_1, C_2>0$ such that
\[
\mathbb{P}\left(T(0,x) - \mathbb{E}T(0,x) \geq t \sqrt{\|x\|_1}\right) \leq e^{-C_1t^2} \text{ for } t \in \left(0,C_2 \sqrt{\|x\|_1}\right)\ .
\]
If $\mathbb{P}(\tau_e=0)<p_c$ and $\mathbb{E}Y^2 < \infty$, where $Y$ is the minimum of $d$ i.i.d. copies of $\tau_e$, then also
\[
\mathbb{P}\left(T(0,x) - \mathbb{E}T(0,x) \leq -t \sqrt{\|x\|_1}\right) \leq e^{-C_1t^{2}} \text{ for all } t \geq 0.
\]
\end{theorem}

\begin{proof}
Write $T=T(0,x)$. We will show only the upper-tail inequality and we will assume $\mathbb{E}^{\alpha \tau_e}<\infty$ for all $\alpha>0$ for simplicity. The idea will be to set $\psi(\lambda) = \log \mathbb{E}e^{\lambda(T-\mathbb{E}T)}$ and to note that it suffices to show for some $C>0$ independent of $x$,
\begin{equation}\label{eq: mgf_bound}
\psi(\lambda) \leq C\|x\|_1 \lambda^2 \text{ for } \lambda \in (0,C).
\end{equation}
Indeed, one has for $\lambda>0$
\begin{align*}
\mathbb{P}(T - \mathbb{E}T \geq t \sqrt{\|x\|_1}) \leq \mathbb{P}(e^{\lambda(T-\mathbb{E}T)} \geq e^{t \lambda \sqrt{ \|x\|_1}}) &\leq e^{\psi(\lambda) - t\lambda \sqrt{\|x\|_1}} \\
&\leq e^{C \|x\|_1 \lambda^2 - t\lambda \sqrt{\|x\|_1}}.
\end{align*}
Choosing $\lambda = t/2C\sqrt{\| x \|_{1}}$ would then complete the proof.

To show the bound \eqref{eq: mgf_bound}, we will use the Herbst argument. Setting $$\mathrm{Ent} ~X = \mathbb{E}X\log X - \mathbb{E}X \log \mathbb{E}X$$ for a nonnegative random variable $X$, we will aim to show for some $C>0$ independent of $x$,
\begin{equation}\label{eq: entropy_mgf}
\mathrm{Ent}~e^{\lambda T} \leq C \|x\|_1\lambda^{2} \mathbb{E}e^{\lambda  T} \text{ for } \lambda \in (0,C).
\end{equation}
This implies that
\[
\frac{\text{d}}{\text{d}\lambda} \left( \frac{\psi(\lambda)}{\lambda} \right) \leq C\|x\|_1 \text{ for } \lambda \in (0,C),
\]
and this gives \eqref{eq: mgf_bound}.

So we focus on proving \eqref{eq: entropy_mgf}. Enumerating the edge variables (in any deterministic order) as $(\tau_{e_1}, \tau_{e_2}, \ldots)$, then
\begin{equation}\label{eq: ent_decomposition}
\mathrm{Ent}~e^{\lambda T} \leq \sum_{i=1}^\infty \mathbb{E} \; \mathrm{Ent_i}~e^{\lambda T}\ ,
\end{equation}
where $\mathrm{Ent}_i$ is entropy relative to only the edge-weight $\tau_{e_i}$. This is known as ``tensorization of entropy''. We now apply a modified Log-Sobolev inequality of Boucheron-Lugosi-Massart \cite[Theorem~6.15]{BLM}.
\begin{lemma}[Symmetrized modified LSI]
Let $q(x) = x(e^x-1)$. If $X$ is a random variable and $X'$ is an independent copy, then for all $t \in \mathbb{R}$,
\[
\mathrm{Ent}~e^{t X} \leq \mathbb{E}\left[ e^{t X} q(\lambda(X'-X)_+) \right]\ .
\]
\end{lemma}
Use the symmetrized LSI in \eqref{eq: ent_decomposition}:
\begin{equation}\label{eq: tensor_new}
\mathrm{Ent}~e^{\lambda T} \leq \sum_{i=1}^\infty \mathbb{E} e^{\lambda T} q(\lambda(T_i'-T)_+)\ .
\end{equation}
Here $T_i'$ is the passage time from $0$ to $x$ in the edge-weight configuration in which the weight $\tau_{e_i}$ is replaced by an independent copy $\tau_{e_i}'$, and all other edges remain the same (that is, they are equal to $\tau_e$). As in the FPP variance proof, we know that $T_i'-T$ is only positive if $e_i$ is in a geodesic from $0$ to $x$ in the original edge-weights $(\tau_e)$. We claim more here: that if we define $\mathcal{G}(0,x)$ as the collection of edges in the intersection of all geodesics from $0$ to $x$ (since there need not be a unique one), then
\[
T_i'-T > 0 \Rightarrow e_i \in \mathcal{G}(0,x) \text{ in the original weights } (\tau_e)\ .
\]
To argue this, assume that $T_i'-T > 0$. Then let $\gamma$ be a geodesic from $0$ to $x$ in the edge-weights $(\tau_{e})$. If $e_i \notin \gamma$ then $\gamma$ must have the same passage time in the new weights, since we replace only $\tau_{e_i}$ by an independent copy. In other words, all edge weights for edges on $\gamma$ have the same value in both configurations. Therefore if $T'$ denotes the passage time in the new weights (with only $\tau_{e_i}$ replaced),
\[
T_i' = T'(0,x) \leq T'(\gamma) = T(\gamma) = T(0,x) = T\ ,
\]
contradicting that $T_i' - T > 0$.

Returning to \eqref{eq: tensor_new}, we can give the upper bound
\[
\mathrm{Ent}~e^{\lambda T} \leq \sum_{i=1}^\infty \mathbb{E}e^{\lambda T} q(\lambda(T'_i-T)_+) \mathbf{1}_{\{e_i \in \mathcal{G}(0,x)\}}\ .
\]
The function $x \mapsto q(x)$ is monotone increasing for $x\geq 0$ so using the bound $T'_i-T \leq \tau_{e_i}'$ (which we established during the FPP variance bound) and independence,
\begin{equation}\label{eq: ent_near_the_end}
\mathrm{Ent}~e^{\lambda T} \leq \sum_{i=1}^\infty \mathbb{E}e^{\lambda T}q(\lambda \tau_{e_i}') \mathbf{1}_{\{e_i \in \mathcal{G}(0,x)\}} = \mathbb{E}q(\lambda \tau_e) ~\mathbb{E}\left[ e^{\lambda T}\#\mathcal{G}(0,x) \right]\ .
\end{equation}
To apply the Herbst argument we would love to decouple $e^{\lambda T}$ from $\#\mathcal{G}(0,x)$. Unfortunately, the variable $\#\mathcal{G}(0,x)$ is not bounded, so we cannot just pull it out. So we use a variational characterization of entropy:
\[
\mathrm{Ent}~X = \sup\{\mathbb{E}XY : \mathbb{E}e^Y\leq 1\},
\]
which implies that for $X \geq 0$ and any $Y$,
\[
\mathbb{E}XY \leq \mathrm{Ent}~X + \mathbb{E}X \log \mathbb{E}e^Y.
\]

We therefore write for arbitrary $a>0$,
\[
\mathbb{E}e^{\lambda T} \# \mathcal{G}(0,x) \leq a~\mathrm{Ent}~e^{\lambda T} + a~\mathbb{E}e^{\lambda T}\log \mathbb{E} \exp \left( \frac{\#\mathcal{G}(0,x)}{a} \right)\ .
\]
Combining with \eqref{eq: ent_near_the_end}, if $a\mathbb{E}q(\lambda \tau_e) < 1$,
\[
\mathrm{Ent}~e^{\lambda T} (1-a\mathbb{E}q(\lambda \tau_e)) \leq a\mathbb{E}q(\lambda \tau_e) \mathbb{E}e^{\lambda T} \log \mathbb{E} \exp \left( \frac{\#\mathcal{G}(0,x)}{a} \right)\ ,
\]
or
\[
\mathrm{Ent}~e^{\lambda T} \leq \frac{a\mathbb{E}q(\lambda \tau_e)}{1-a\mathbb{E}q(\lambda \tau_e)} \log \mathbb{E} \exp\left( \frac{\#\mathcal{G}(0,x)}{a} \right) \times \mathbb{E}e^{\lambda T}\ .
\]

To control these terms we will need a lemma from \cite{DHS1} which is a strong version of a geodesic length bound.
\begin{lemma}
Assuming $\mathbb{E}e^{\alpha \tau_e}<\infty$ for all $\alpha>0$, there exist $a,c_1>0$ such that
\[
\log \mathbb{E} \exp \left( \frac{\#\mathcal{G}(0,x)}{a} \right) \leq c_1 \|x\|_1 \text{ for all } x \in \mathbb{Z}^d\ .
\]
\end{lemma}
Applying the lemma,
\[
\mathrm{Ent}~e^{\lambda T} \leq c_1 \|x\|_1 \frac{a \mathbb{E}q(\lambda \tau_e)}{1-a\mathbb{E}q(\lambda \tau_e)}~\mathbb{E}e^{\lambda T} \text{ if } a\mathbb{E}q(\lambda \tau_e) < 1\ ,
\]
or
\[
\mathrm{Ent}~e^{\lambda T} \leq 2c_1 \|x\|_1 a\mathbb{E}q(\lambda \tau_e) \mathbb{E}e^{\lambda T} \text{ if } a \mathbb{E}q(\lambda \tau_e) < 1/2\ .
\]
Note that by dominated convergence,
\[
\lim_{\lambda \downarrow 0} \frac{\mathbb{E}q(\lambda \tau_e)}{\lambda^2} = \lim_{\lambda \downarrow 0} \mathbb{E} \left( \frac{\tau_e(e^{\lambda \tau_e}-1)}{\lambda} \right) = \mathbb{E}\tau_e^2\ .
\]
So we can find $c_2>0$ such that if $\lambda \in (0,c_2)$ then $\mathbb{E}q(\lambda \tau_e) \leq 2\lambda^2 \mathbb{E}\tau_e^2$. For such $\lambda$, note that $a \mathbb{E}q(\lambda \tau_e) < 1/2$ if $2a \lambda^2 \mathbb{E}\tau_e^2 < 1/2$, which occurs if $\lambda$ is smaller than some positive $c_3<c_2$. So we obtain the desired equation \eqref{eq: entropy_mgf}: for some $c_4>0$,
\[
\mathrm{Ent}~e^{\lambda T} \leq 4c_1 \|x\|_1 \lambda^2 \mathbb{E}\tau_e^2 \mathbb{E}e^{\lambda T} \leq \lambda^2 c_4 \|x\|_1 \mathbb{E}e^{\lambda T} \text{ for } \lambda \in (0,c_3)\ .
\]
\end{proof}

\subsection{Convergence of the mean for sub-additive ergodic processes}
We return to the perspective adopted in Section \ref{sec:visit}. Let $X$ be a subadditive ergodic process satisfying the hypotheses of Theorem \ref{subadditivity}. As discussed previously, since the leading order (shape theorem) limiting behavior of FPP is established using the general framework of the subadditive ergodic theorem, one would hope that general abstract arguments for subadditive sequences could be used to establish limit theorems and other sharper results. Since the class of processes satisfying Theorem \ref{subadditivity} is much too large to characterize the exponents and limiting behavior of FPP, we will have to impose additional axioms in order to derive useful results. Furthermore, many results will require us to go beyond considering subadditive sequences $X$ as above and instead consider the $d$-dimensional structure of the model.

We will restrict our attention here to the convergence of the mean $\E X_{0,n}/n$ of our subadditive processes to their limiting $\mu$ (henceforth, we assume that the sequence is ergodic). Recall that in the FPP setting, the fluctuations of $T(0,x) - \mu(x)$ can be written as:
\[
o(\|x\|_1) = \underbrace{T(0,x) - \mathbb{E}T(0,x)}_{random~fluctuations} + \underbrace{\mathbb{E}T(0,x) - \mu(x)}_{non-random~fluctuations}.
\]
Earlier in this section, we discussed bounds on the random fluctuations in FPP. We therefore shift to the problem of convergence of the mean in order to characterize the other error term in the convergence to the FPP limit shape.

\subsubsection{Non-random fluctuations for subadditive sequences \label{sec:subadseq}}

All of the existing methods for controlling the rate of convergence of the mean require also some control of the corresponding non-random fluctuations. 
As mentioned before, it is expected that the random fluctuations in FPP are governed by the fluctuation exponent $\chi$. It is reasonable to postulate the existence of a similar exponent $\gamma$ for the non-random fluctuations. Specifically, one expects that (for a suitable definition of ``$\sim$'')
\[T(0,x) - \mathbb{E}T(0,x) \sim \|x\|_{1}^{\chi},  \quad \mathbb{E}T(0,x) - \mu(x) \sim \|x\|_{1}^{\gamma}\]
for some exponents $\chi$ and $\gamma$. This motivates the following definitions from \cite{ADH2}, made for any subadditive ergodic sequence $(X_{m,n})$ satisfying the hypotheses of Theorem \ref{subadditivity}.

\begin{definition}\label{def:exponentgamma}
The exponents $\underline{\gamma}$ and $\overline{\gamma}$ are defined as
\[
\underline{\gamma} = \liminf_n \frac{\log (\mathbb{E}X_{0,n}-n \mu)}{\log n} \text{ and } \overline{\gamma} = \limsup_n \frac{\log (\E X_{0,n} - n\mu)}{\log n}\ .
\]
Here, recall $\mu$ is the almost-sure limit $\lim_{n}X_{0,n}/n$; we make the convention that $\log 0 ~=~ -\infty$.
\end{definition}

\begin{definition}\label{def:exponentxi}
For $p>0$, the fluctuation exponents $\underline{\chi}_p$ and $\overline{\chi}_p$ are defined as
\[\underline{\chi}_p = \liminf_{n} \frac{\log \|X_{0,n} - \E X_{0,n} \|_p}{\log n}\  \text{ and }\  \overline{\chi}_p = \limsup_{n} \frac{\log \|X_{0,n} - \mathbb{E}X_{0,n}\|_p}{\log n}\ .\]
\end{definition}

\noindent
Note that $\underline{\chi}_p \leq \overline{\chi}_p$ and by Jensen's inequality,
\[
\underline{\chi}_p \leq \underline{\chi}_q \text{ and } \overline{\chi}_p \leq \overline{\chi}_q \text{ if } p \leq q\ .
\]

\begin{remark} The definitions above only specify the first order growth of the fluctuation terms. For instance, in FPP with $X_{m,n} = T(me_1, n e_1)$, if there exist positive constants $c, C$ such that 
\[
c n^{2\chi} \leq \mathrm{Var} \; T(0,ne_{1}) \leq C n^{2\chi} \text{ for all large }n\ ,
\]
then $\underline{\chi}_2 = \overline{\chi}_2 = \chi$, and similarly for $p \neq 2$ and $\underline{\gamma}, \overline{\gamma}$.
\end{remark}

Our first result will be a lower bound of the form ``$\gamma \geq \chi$'' under some assumptions on the sequence which are natural in the case of FPP.  We explain these and the intuition. Recall that the $\chi$ exponent in FPP is, except in the one-dimensional case, expected to be strictly smaller than $1/2$. Suppose that we had a sub-additive ergodic sequence $X_{0,n}$ for which $\chi$ and $\gamma$ exponents existed in a moderately strong sense (i.e. $\overline \chi_{p} = \underline \chi_q$ for all $q,\,p \geq 1$ and $\overline \gamma = \underline \gamma$). Assume also that $\chi < 1/2$ but $\gamma < \chi$. Last, assume the sequence $(X_{Km, Kn})$ is weakly dependent for large $K$ in a sense which we leave imprecise.

For any positive integer $K$, we can write using subadditivity
\begin{align}
  [X_{0,n} - \E X_{0,n}] &\leq \sum_{i=1}^{n/K} [X_{(i-1)K, iK} - \E X_{0,K}] + (n/K) \E X_{0,K} - \E X_{0,n}\nonumber\\
&\leq \sum_{i=1}^{n/K} [X_{(i-1)K, iK} - \E X_{0,K}] + (n/K)[\E X_{0,K} - K\mu]\label{eq:subcltsketch}\ .
\end{align}

If the sequence $X_{m,n}$ is suitably weakly dependent, then we could expect that the sum on the right-hand side of \eqref{eq:subcltsketch} obeys a central limit theorem, and in particular frequently takes values on the order of $-n^{1/2} K^{\chi-1/2}$. The second term is of order $n K^{\gamma - 1}$. In particular, if we were to choose $K$ on the order of $n^{1-\delta}$ for $0<\delta \ll 1$, the first term would dominate, giving that $X_{0,n} - \E X_{0,n}$ typically has lower tail fluctuations of order $-n^{\chi + \delta(1/2 - \chi)}$. This is a contradiction to the definition of $\chi$ under the assumption that $\chi < 1/2$.

We are now ready to state some rigorous results following the line of reasoning above; these results appeared in \cite{ADH2}. In order to avoid precisely describing the weak dependence axioms needed, we restrict to the case of FPP on $\Z^d$ (here, the ``weak dependence'' axiom amounts to control of the diameter of geodesics -- see Theorem \ref{thm: geo_diameter_bound}). In the language of FPP, this result says that given existence of the fluctuation exponent $\chi$, the non-random fluctuations are at least of the same order as the standard deviation of the passage time.  

In the FPP results below, we consider exponents defined as in Definitions \ref{def:exponentgamma} and \ref{def:exponentxi} for some arbitrary fixed direction $x$ (that is, $\overline \gamma = \limsup [\E T(0,nx) - n \mu(x)]/\log n$, etc.).

\begin{theorem}\label{thm:main1ADH2}
Consider the case of FPP on $\Z^d$ for $d \geq 2$. Assume $\mathbb{E} \tau_{e}^{2+\delta}<\infty$ for some $\delta>0$.
\begin{enumerate}
\item If $\chi := \underline{\chi}_2 = \overline{\chi}_{2+\delta}$, then 
\[
\begin{cases}
\underline{\gamma} \geq \chi & \text{ if } \chi \neq 1/2 \\
\overline{\gamma} \geq \chi - \frac{1}{2}(2^{\beta^{-1}}-1) & \text{ if } \mathrm{Var}\; T(0,nx) = O\left( \frac{n}{(\log n)^\beta} \right) \text{ for some }\beta>0 \\
\end{cases}\ .
\]
\item If $\chi := \overline{\chi}_2=\overline{\chi}_{2+\delta}$, then if $\chi < 1/2$, $\overline{\gamma} \geq \chi$.
\end{enumerate}
\end{theorem}

\begin{remark}The result above should be compared with the classical example of a sum of i.i.d. random variables with finite mean and finite $2+\delta$ moment. In this case, $\chi=1/2$ and $\overline{\gamma}=\underline{\gamma}=-\infty$ and these exponents do  not satisfy the conclusions of the theorem above. Note that this example is equivalent to FPP on $\Z$ with the appropriate moment conditions.
\end{remark}

Another theorem of \cite{ADH2} gives the bound $\overline{\gamma} \geq -1/2$ under minimal assumptions. This result should be compared to \cite[Theorem~1]{Kesten}, where it is shown that $\underline{\gamma} \geq -1$.
\begin{theorem}\label{thm: second_fpp_thm}
Consider FPP on $\Z^d$ for $d \geq 2$.
Assume $F(0)<p_{c}(d)$, that the distribution of $\tau_e$ is not concentrated at a point, and that 
\begin{equation}\label{en}
\mathbb{E}e^{\alpha \tau_e} < \infty \text{ for some } \alpha>0\ .
\end{equation}
One has the bound $\overline{\gamma} \geq -1/2$: for any nonzero $x \in \mathbb{Z}^d$ and $\epsilon>0$,
\[
\mathbb{E}T(0,nx) - n\mu(x) \geq n^{-\frac{1}{2}-\epsilon} \text{ for infinitely many } n\ .
\]
\end{theorem}

One may be interested in the sharpness of the result $\gamma \geq \chi$. In the case of FPP on $\Z^d$ ($d \geq 2$), it is reasonable to expect the result is sharp -- that is, $\gamma =\chi$ -- as explained in the following remark.
\begin{remark}
Alexander \cite{AlexanderDirected} has remarked that if $\hat \chi$ is any number such that for some $a>0$,
\[
\mathbb{P}\left( |T(0,y) - \mathbb{E}T(0,y)| \geq \lambda \|y\|_1^{\hat \chi} \right) \leq e^{-a \lambda} \text{ for all } \lambda\geq 0,~ y \in \mathbb{Z}^d\ ,
\]
then $\overline{\gamma} \leq \hat \chi$. This follows from an extension of his method for controlling nonrandom fluctuations, discussed in Section \ref{sec:alexmethod} below. Note that if this exponential inequality holds for some $\hat \chi$, then $\overline{\chi}_p \leq \hat \chi$ for all $p>0$. Combining these observations with Theorem \ref{thm:main1ADH2} above, if $\hat \chi$ can be taken to be $\chi: = \hat \chi = \underline{\chi}_2$, then $$\gamma:= \underline{\gamma} = \overline{\gamma} = \chi$$ when $\chi<1/2$ and $\overline{\gamma} = \chi$ under the assumption $\mathrm{Var } \; T(0,nx) = O(n/(\log n)^\beta)$ for every $\beta>0$.
\end{remark}

\begin{question}
Show that  $\underline{\chi}_2 = \overline{\chi}_{2+\delta}$ for some $\delta>0$. 
\end{question}

A solution for the question above has the following consequence that improves Theorem \ref{thm: second_fpp_thm}.
\begin{remark} Combining Theorem \ref{thm:DHS222} with Theorem~\ref{thm:main1ADH2}, if $\underline{\chi}_2 = \overline{\chi}_{2+\delta}$ for some $\delta>0$ and $F(0)<p_{c}(d)$ holds, then $\overline{\gamma} \geq 0$.
\end{remark}

We now return to the general setting of subadditive ergodic sequences $(X_{m,n})$. Our goal is to find extensions to the axioms of Theorem \ref{subadditivity} which could guarantee upper bounds on the exponent $\gamma$, complementing the lower bounds given above.

Note that by subadditivity, for any $n$,
\[
\E X_{0,2n} \leq \mathbb{E}X_{0,n} + \mathbb{E} X_{n, 2n} = 2\E X_{0,n}.
\]
The magnitude of non-random fluctuations is related to the degree to which the above inequality is strict. Let us abbreviate $\E X_{0,n} =: H(n)$. As a first step, assume the above inequality were actually an equality; then
\[\mu = \lim_{k} H(2^kn)/(2^kn) = \lim_k H(n)/n = H(n)/n\ , \]
in which case we would have $H(n) = n \mu$.

In fact, we can give upper bounds on $H(n) - n\mu$ assuming just some weaker quantitative control on $H(2n) - 2 H(n)$. To illustrate this, assume that for some function $j:\mathbb{N} \to [0,\infty)$, one has
\begin{equation}\label{eq: j_condition}
H(2n) \geq 2H(n) - j(n).
\end{equation}
We then have the following lemma from \cite{HN01}.
\begin{lemma}\label{lem:hnjlem}
Suppose that $j$ satisfies $\Psi := \limsup_n j(2n)/j(n) < 2$, $j(n)/n \to 0$, and \eqref{eq: j_condition}. Then for all large $n$, one has $H(n) \leq n\mu + cj(n)$.
\end{lemma}
\begin{proof}
For $c>1/(2-\Psi)$, set $a(n) = H(n) - cj(n)$. Then for large $n$,
\[
a(2n) = H(2n) - cj(2n) \geq 2H(n) - 2cj(n)\left( \frac{1}{2c} +\frac{j(2n)}{2j(n)} \right) \geq 2a(n).
\]
So for $m \geq 1$,
\[
\frac{a(2^mn)}{2^mn} \geq \frac{a(n)}{n}.
\]
Taking $m \to \infty$, we obtain
\[
\frac{a(n)}{n} \leq \mu,
\]
or
\[
H(n) \leq n\mu + cj(n).
\]
\end{proof}

We would like to apply the above lemma to bound non-random fluctuations in FPP, taking $H(n) = h(nx)$ (recall $h(y) = \E T(0,y)$) and $\mu = \mu(x)$. By the last lemma, if we could show that $h(nx)$ is close to being a linear function of $n$, we can upper bound $h(nx)-n\mu(x)$.
This approach is most successful in the axis direction $e_1$ for symmetry reasons, and a version of this method is the  basis for several such bounds (see Zhang \cite{Zhang}, Rhee\cite{Rhee}, and related reflection trick in Alexander\cite{Alexander2}). One typically tries to use the fact that the passage time along geodesics is additive. Define the events
\[
A = \{T(0,ne_1) - h(ne_1) \geq -j(n)/3\} \cap \{T(ne_1,2ne_1) - h(ne_1) \geq -j(n)/3\}
\]
\[
B = \{T(0,2ne_1) \leq 2h(ne_1) + j(n)/3\}
\]
\[
C = \{\text{a geodesic from } 0 \text{ to } 2ne_1 \text{ passes through } ne_1\}.
\]
Then if $A \cap B \cap C$ has positive probability, on this event, one has $$T(0,2ne_1) = T(0,ne_1) + T(ne_1,2ne_1)$$ and therefore $$h(2ne_1) \geq h(ne_1) - j(n).$$

It is reasonable that $C$ has probability at least some power in $n$ given the heuristics for geodesic wandering, which would -- along with strong bounds on the probabilities of $A$ and $B$, provided by concentration results on the passage time -- complete the argument. Unfortunately, not enough is known about geodesic wandering to make this argument work in general directions. In practice, showing this positive probability statement has been done using a reflection trick, and unfortunately this typically does not work in non-axis directions.

\subsubsection{Asymptotic geodesicity and rate of convergence of the expected ball} 

Since the axioms of Lemma \ref{lem:hnjlem} are not easy to verify in FPP, one is led to generalize from subadditive sequences to some broader framework. Recalling that $h(x) : = \E T(0,x)$, it is obvious that $h$ has more than the structure of a subadditive sequence: it is subadditive as a function of $\Z^d$. Since the time constant and $h$ itself depend on the entire $d$-dimensional structure of the lattice, it is reasonable to hope that one could progress by considering the geometry induced on the lattice by $h$.

We can think of $h$ as a metric given by $h(x,y) = h(y-x)$, and although the metric space $(\mathbb{Z}^d, h)$ does not have geodesics, we would like to show that it has approximate midpoints. Recall that to implement the approach outlined  following Lemma \ref{lem:hnjlem}, we would like to show that $ne_1$ is nearly a midpoint between $0$ and $2ne_1$. A realization of Tessera is that we do not have to choose $ne_1$: it actually suffices to find skeletons of paths $0=x_0, x_1, \ldots, x_k = 2ne_1$ which are approximately geodesics. That is, we would like to give an upper bound for
\[
h(2ne_1) - \sum_{i=0}^{k-1} h(x_{i+1}-x_i),
\]
for some suitably chosen skeleton.

\begin{definition}
\label{def:saggin}
Let $N:\mathbb{R}_+ \to \mathbb{R}_+$ be an increasing function such that $\lim_{\alpha \to \infty} N(\alpha) = \infty$. A metric space $(X,d)$ is called SAG$(N)$ (strongly asymptotically geodesic) if there exists $\alpha_0 \geq 0$ such that for all integers $m \geq 1$, and for all $x,y \in X$ such that $d(x,y)/m \geq \alpha_0$, there exists a sequence $x=x_0, \ldots, x_m = y$ satisfying, for all $0 \leq i \leq m-1$,
\[
\frac{d(x,y)}{m} \left( 1 - \frac{1}{N(\alpha)} \right) \leq d(x_i,x_{i+1}) \leq \frac{d(x,y)}{m} \left( 1 + \frac{1}{N(\alpha)} \right),
\]
where $\alpha = d(x,y)/m$; and for all $x \in X$ and large enough $r$,
\[
B\left(x, \left(1 + \frac{1}{N(r)}r \right) \right) \subset [B(x,r)]_{\frac{6r}{N(r)}}.
\]
Here, $B(x,r)$ is the $r$-neighborhood of $x$ and for any $T>0$ and set $S \subset X$, $[S]_T$ is the $T$ neighborhood of $S$.
\end{definition}

We briefly try to motivate the above definition in the setting of $(\Z^d, h)$. Although $h$ is not geodesic, the passage time metric $T$ is. Given suitably strong concentration results, one can fix a realization of edge weights and hope to approximate the passage times between vertices in a geodesic by the expected passage time between them, up to a smaller order correction. Breaking this geodesic into a skeleton of $m$ roughly equally spaced vertices would then furnish the sequence $(x_i)$ in Definition \ref{def:saggin}. The second condition of Definition \ref{def:saggin} guarantees that there are no large ``jumps'' in $B(x,r)$ as one increases $r$, which is also natural to expect in FPP.

The SAG condition guarantees that the metric $d$ obeys a shape theorem with an explicit rate of convergence estimate, which we now describe. We say that an increasing function $\phi : \mathbb{R}_+ \to \mathbb{R}_+$ is sublinearly doubling if there exists a function $\eta : \mathbb{R}_+ \to \mathbb{R}_+$ such that $\lim_{\lambda \to \infty} \eta(\lambda)/\lambda = 0$ and for all $\lambda > 0$, $\phi(\lambda r) \leq \eta(\lambda)\phi(r)$. The following theorem (from \cite{Tessera}) gives an equivalence between asymptotic geodesicity in the form of $SAG(N)$ and a shape theorem with error term on the order $N^{-1}$:

\begin{theorem}
\label{thm:sagshape}
Let $\delta$ be some translation-invariant metric on $\mathbb{Z}^d$ and $\phi:\mathbb{R}_+\to [1,\infty)$ be an increasing, sublinearly doubling function. The following are equivalent.
\begin{enumerate}
\item There exists $c>0$ such that $\delta$ is SAG$(N)$ with
\[
N(\alpha) \geq c\phi(\alpha),
\]
for $\alpha$ large enough.
\item There exists a norm $\|\cdot\|$ on $\mathbb{R}^d$ and $C>0$ such that for all large enough $n$,
\[
B_{\|\cdot\|}\left(0, n - \frac{Cn}{\phi(n)} \right) \cap \mathbb{Z}^d \subset B_\delta(0,n) \subset B_{\|\cdot\|}\left(0, n + \frac{Cn}{\phi(n)} \right).
\]
\end{enumerate}
\end{theorem}
A version of SAG with exponent $1/2$ was shown  for FPP in \cite{Tessera}. Namely, if the edge weights have finite exponential moments and $F(0)<p_{c}(d)$ then $h$ satisfies SAG$(N)$ for $$N(\alpha) = c(\alpha/(\log \alpha))^{1/2},$$
for some $c>0$. Therefore there exist  $C>0$ and $n_{0}\geq 1$ such that for any $n \geq n_{0}$,
\begin{equation}
\label{eq:tesrate}
(n -C(n\log n)^{1/2})\mathcal B_{\nu}\cap \mathbb Z^{d} \subseteq \bar B(0,n)\subseteq (n + C(n\log n)^{1/2})\mathcal B_{\nu}, \end{equation}
where $\bar B(0,n) := \{ x \in \mathbb Z^{d}: h(x)\leq n \}.$

Let us conclude this subsection by giving some account of why Theorem \ref{thm:sagshape} should hold; we will freely make simplifying assumptions to avoid getting bogged down. We note that our simiplifications obscure some of the technical triumph of \cite{Tessera} and direct the reader there for a full account. We work in the metric $\delta$. Let $m$ and $n$ be arbitrary positive integers, and recall that the convergence in Theorem \ref{thm:sagshape} is convergence under the Hausdorff metric $d_H$, where for two compact sets $A$ and $B$
\[d_H(A,B) = \inf\{r > 0: A \subseteq [B]_r, \, B \subseteq [A]_r\}\ . \]
We will try to show that $d_H(n^{-1}B(0,n),\,(mn)^{-1}B(0,mn))$ is small; taking $m$ to infinity gives a corresponding bound for $d_H(n^{-1}B(0,n),\,B_{\|\cdot\|}(0,n)$.

A natural way to use the SAG assumption is to try to build the ball $B(0,mn)$ out of $m$ copies of the ball $B(0,n)$. This is because, given $x$ at distance $mn$ from $0$, we can find a skeleton sequence of vertices $(x_i)$ as in Definition \ref{def:saggin} each at distance $\sim n$ from each other. This reasoning may be made rigorous as in the following lemma, adapted from \cite[Lemma 5.1]{Tessera}:
\begin{lemma}
\label{lem:tessum}
  There exists a $C$ such that, for all $m$ and $n$ positive integers,
  \[d_H\left((mn)^{-1}B(0,n)^m,\, (mn)^{-1}B(0,mn) \right) \leq C / \phi(n)\ . \]
\end{lemma}
In the above, $AB = \{a + b: \, a \in A,\, b \in B \}$ and $A^2 = AA,$ etc. Using the fact that $d_H$ is a metric, we see
\begin{align}
d_H(n^{-1}B(0,n),\,(mn)^{-1}B(0,mn)) \leq &d_H(n^{-1}B(0,n),\,(mn)^{-1}B(0,n)^m) \nonumber\\
+ &d_H((mn)^{-1}B(0,n)^m,\,(mn)^{-1}B(0,mn))\ ,\label{eq:tessum2}
\end{align}
and the second term on the right-hand side is small by Lemma \ref{lem:tessum}.

In other words, we see that $B(0,mn)$ is approximately equal to $m$ copies of $B(0,n)$ strung together, and it remains to show that $m B(0,n)$ is also approximately equal to this set. Note that if $B(0,n)$ were a convex subset of $\mathbb{R}^n$, then $B(0,n)^m = m B(0,n)$, and we would be done. So the only way this line of attack could fail would be if $B(0,n)$ were irregular enough to be significantly different from its convex hull.

This complication is dealt with in \cite{Tessera} by inductively controlling the Hausdorff distance between balls $B(0,n)$ and their convex hulls $\widehat B(0,n)$. A crucial fact used is that in $\mathbb{R}^d$, Carath\'eodory's Theorem for convex hulls guarantees that $d_H(B(0,n)^m, \widehat B(0,n)^m)$ can be bounded by a uniform constant multiple of $d_H(B(0,n), \widehat B(0,n))$.

\subsubsection{Alexander's method \label{sec:alexmethod}}

The first upper bounds on non-random fluctuations go back to the work of Kesten \cite[Theorem 1]{Kesten}, who established the bound  
\[
\mathbb{E}T(0,ne_{1}) \leq n\mu(e_{1}) + C n^{5/6} \log n.
\] 
The right side has since been improved to $n\mu(e_1) + C(n\log n)^{1/2}$; see \eqref{eq:tesrate} and \eqref{nonrandomfluctuations}. 
An extraordinary contribution was the theory developed by Alexander in his famous paper \cite{Alexander}. It concerns the rate of convergence of a (deterministic) nonnegative function $h$ defined on the lattice $\mathbb Z^{d}$ to its limit $g$. That is, we assume 
$$h(x+y) \leq h(x) + h(y), \quad x,y \in \mathbb Z^{d},$$
and set 
$$g(x):= \lim_{n} h(nx)/n.$$
Clearly, $g(x)\leq h(x)$ for all $x\in \mathbb Z^{d}$.  Alexander's theory has a very geometric flavor and explores the fact that the function $h$ is defined on the whole space $\mathbb Z^{d}$; $h$ is not only a (one-dimensional) subadditive sequence. In his point of view, one way to guarantee an upper bound on the rate of convergence of $h(nx)/n$ is to request $h$ to have ``good skeleton paths'' towards $nx$ in $\mathbb Z^{d}$. It turns out that this condition is possible to verify in FPP and other models and suitable to bound the rate of convergence.     

To describe his approach,  let $0=v_{0}, v_{1}, v_{2}, \ldots, v_{k} =x$ be points in $\mathbb Z^{d}$. We will soon think of the $v_{i}'s$ as the skeleton of a path from $0$ to $x$. From subadditivity, we have
$$h(x)=h\bigg(\sum_{i=1}^{k} (v_{i}-v_{i-1})\bigg)\leq \sum_{i=1}^{k} h\bigg(v_{i}-v_{i-1}\bigg).$$
Imagine for a moment that we can choose these $v_{i}$'s in such a way that
\begin{equation}\label{eq:firsttryAlex}
h(v_{i}-v_{i-1}) = g(v_{i}-v_{i-1}) + o(|x|^{a}),
\end{equation}
 for some $a<1$. Well, this would lead to 
\begin{equation}\label{eq:thisisnotgood}
h(x)-g(x) \leq \sum_{i=1}^{k}\big(g(v_{i}-v_{i-1})\big) - g(x) + ko(|x|^{a}).
\end{equation}
 Unfortunately, \eqref{eq:thisisnotgood} is not enough for our purpose as we still need to control the difference $\sum_{i=1}^{k}\big(g(v_{i}-v_{i-1})\big) - g(x).$ Moreover, in a first look, it seems hard to actually find $v_{i}$'s satisfying \eqref{eq:firsttryAlex}.
 
To overcome these  problems, we will do two modifications. First, instead of requesting \eqref{eq:firsttryAlex}, we will replace $g$ by the projection $g_{x}$ of $g$ on the line that goes through $0$ and $x$. This projection is linear and satisfies $g_{x}(nx) = g(nx)$ for any $n\in \mathbb N$. In Alexander's description, the value $g_{x}(y)$ is the amount of progress toward $x$ made by a vector increment of $y$. The difference $h(y)-g_{x}(y)$ should be thought as the error or the inefficiency associated with such an increment. 
An increment $y$ will be called good if $y$ satisfies $h(y) = g_{x}(y) + C|x|^{a}$ for a given $C>0$.

This nice change from $g$ to $g_{x}$ helps us to get rid of the sum in the right side of \eqref{eq:thisisnotgood}. However, we still need to find the skeleton $v_{i}$'s of good increments. The main idea and contribution of \cite{Alexander} is to come up with an assumption on $h$ that allow us to roughly do so. It will suffice to verify the weaker condition that $x$ (divided by some constant larger than one) is in the convex hull of good increments.  This weaker condition can be verified in FPP and many other models. We will make this statement precise in the next definition.

 Let $\Phi$ denote the set of all positive nondecreasing functions on $(1,\infty)$ and fix $a<1$. For $\phi \in \Phi$, define $Q_{x}(a,\phi)$ be the set of good increments with exponent $a$, correction factor $\phi$ in direction $x$, that is,
 
$$ Q_{x}(a,\phi) := \bigg \{ y \in \mathbb Z^{d} : h(y) \leq g_{x}(y) + |x|^{a} \phi(|x|), \; g_{x}(y) \leq g(x) \bigg \}. $$
The condition $g_{x}(y) \leq g(x)$ is to avoid ``overshooting'' the increment. For a collection of points in $\mathbb R^{d}$, let $\text{Co}(\cdot)$ denote the convex hull. 

\begin{definition}[CHAP]\label{CHAP} The function $h$ satisfies the convex hull approximation property (or CHAP) with exponent $a$ and correction $\phi$ if there exists $L>1$ such that 
$$x/\alpha \in \text{Co}(Q_{x}(a,\phi)) \text{ for some } \alpha \in [1,L],$$
for all $x \in \mathbb Q^{d}$ with $|x|$ sufficiently large.
\end{definition}
 
We will discuss in a second how one would verify that a subadditive function $h$ satisfies CHAP and why it is natural to work with Definition \ref{CHAP}. The next definition is exactly the control on the rate of convergence of $h$.
 
\begin{definition}[GAP] For $a>0$ and $\phi \in \Phi$ we say that $h$ satisfies the general approximation property (GAP) with exponent $a$ and correction factor $\phi$ if there exist $M>1$ and $C>0$ such that for all $x\in \mathbb Z^{d}$ and $|x|\geq M$, 
$$g(x)\leq h(x) \leq g(x) + C |x|^{a} \phi(|x|).$$
\end{definition}

Alexander's main result is the following. 

\begin{theorem}[Theorem 1.8 \cite{Alexander}]\label{thm:AlexGAP} Suppose $h$ is a nonnegative subadditive function on $\mathbb Z^{d}$ with $h(x)\leq r |x|$ for all $x$ and some $r>0$. If $h$ satisfies CHAP with exponent $a>0$ and correction $\phi$, then it satisfies GAP with exponent $a$ and correction $\phi$.
\end{theorem}

We will now sketch the proof of Theorem \ref{thm:AlexGAP}. The following fixed point argument is the core of the proof.
\vspace{0.2cm}

\textsc{Step 1}. Suppose $h$ satisfies CHAP$(a,\phi)$ and GAP$(b,\phi)$ with $b \in (a,1]$. Then $h$ satisfies GAP$(b',\phi)$ with $b'=b/(1+b-a)<b$. 

\vspace{0.2cm}

\textit{Sketch of the proof of Step 1}. Suppose $|x|$ is large. Take $q \in \mathbb Q$ also large but smaller than $|x|$. Applying CHAP$(a,\phi)$ to $x/q$ we can express $x/(\alpha q)$ as a convex combination of $d+1$ good increments $y_{q,i}$, $1 \leq i \leq d+1$. Now, use the increments $y_{q,i}$'s to decompose 
$$x = \sum_{i=1}^{d+1} \alpha_{q,i} y_{q,i} + x_{(r)}$$ 
with $\alpha_{q,i} \in \mathbb N \cup \{0\}$ and some remainder $x_{(r)}$. By subadditivity and the fact that $g(x)=g_{x}(x)$, we have
\begin{equation}\label{eq:Alexanderequivalence}
\begin{split}
h(x)-g(x) &\leq \sum_{i=1}^{d+1}\alpha_{q,i}h(y_{q,i}) + h(x_{(r)}) - g_{x}(x-x_{(r)}) + g_{x}(x_{(r)})\\ &\leq \sum_{i=1}^{d+1}\alpha_{q,i} \big(h(y_{q,i}) - g_{x}(y_{q,i})\big) + h(x_{(r)})-g_{x}(x_{(r)}).
\end{split}
\end{equation}
Our choice of the $y_{q,i}$'s implies that the sum on the right side of  \eqref{eq:Alexanderequivalence} is bounded by 
$$ q^{1-a}|x|^{a}\phi(|x|),$$
while the remainder term $h(x_{(r)}) - g_{x}(x_{(r)})$ is controlled using GAP$(b,\phi)$. Optimizing over $q$ gives the desired result. (The optimal $q$ turns out to be of the order of a small power of $|x|$.) The details can be found in \cite[Proposition 2.1]{Alexander}.

\vspace{0.2cm}

\textsc{Step 2}.  By hypothesis, $h$ satisfies GAP$(1,\phi)$ with $\phi =1$. If $a\geq 1$, there is nothing to prove. If $a<1$, we can iterate \textsc{Step 1}  to obtain GAP$(a, \phi)$ as the fixed point of the function $x\to x/(1+x-a)$ is at $a$. 

Let us now end this subsection by saying a few words on why it is natural to expect CHAP and how one proves CHAP$(a,\phi)$ in FPP. The first ingredient is the following geometric lemma, that defines the skeleton $\{v_{i}\}$ discussed before. The proof is short, so we include it here.

\begin{lemma}[Lemma 1.6, \cite{Alexander}]\label{lem:whataniceone}
Let $h$ be a nonnegative subadditive function in $\mathbb Z^{d}$ and $c>1$. Suppose that for each 
$x\in \mathbb Q^{d}$, there exist $n\geq 1$, a lattice path $\Gamma$ from $0$ to  $nx$ and a sequence of sites $0= v_{0}, v_{1}, \ldots, v_{m}=nx$ in $\Gamma$ such that $m\leq cn$ and $v_{i}-v_{i-1} \in Q_{x}(a,\phi)$ for all $1 \leq i \leq m$. Then $h$ satisfies CHAP$(a,\phi)$.
\end{lemma}
\begin{proof}
 We need to find $\alpha \geq 1$ so that 
$$x/\alpha \in \mathrm{Co}(Q_{x}(a,\phi)).$$
Let $a(n,y)$ be the number of indices such that $v_{i}-v_{i-1} =y$.  Then, by hypothesis, 
$$ nx = \sum_{y\in Q_{x}(a,\phi)} a(n,y) y$$
and applying $g_{x}$ to the equation above
$$ ng(x) =   \sum_{y\in Q_{x}(a,\phi)} a(n,y) g_{x}(y) \leq g(x) \sum_{y\in Q_{x}(a,\phi)} a(n,y) $$
so 
$$ n \leq \sum_{y\in Q_{x}(a,\phi)} a(n,y). $$
On the other hand, by hypothesis, $$n^{-1}\sum_{y\in Q_{x}(a,\phi)} a(n,y) \leq c.$$ 
Taking $\alpha = n^{-1}\sum_{y\in Q_{x}(a,\phi)} a(n,y)$ and $L=c$ gives us CHAP.
\end{proof}

Now look at the $m+1$ sites in the Lemma above. It is easy to find inductively a sequence of marked sites for any path $\Gamma$ from $0$ to $nx$. One starts with $v_{0}=0$ and chooses $z_{i}$ as the first site in $\Gamma$ after $v_{i-1}$ such that $z_{i}-v_{i-1} \notin Q(a,\phi)$. Now take $v_{i}$ as the last vertex visited by $\Gamma$. We call the sequence of marked sites, obtained from a self-avoiding path in this way, the $Q_{x}(a,\phi)$ skeleton of $\Gamma$.  The difficulty is to control the number of vertices in a $Q_{x}$ skeleton. This is the role of the next proposition, which combined with Lemma \ref{lem:whataniceone} ends the proof of CHAP$(1/2,\log |x|)$. 

\begin{proposition}\label{prop:daAle}
Assume $F(0)<p_{c}(d)$ and $\mathbb E e^{\lambda \tau_{e}} <\infty$ for some $\lambda>0$. Then there exists $C>0$ so that if $|x|>C$ then for sufficiently large $n$ there exists  a lattice path from $0$ to $nx$ with $Q_{x}(1/2,\log |x|)$-skeleton of $2n+1$ or fewer vertices.
\end{proposition}

We will take one geodesic from $0$ to $nx$ as our lattice path. To prove Proposition \ref{prop:daAle}, Alexander takes use of one very useful estimate of Kesten \cite[(4.13)]{KestenAspects} and his skeleton construction. The argument goes as follows where here the constant $C$ will change from line to line. 

Consider a sequence of points $v_{i} \in \mathbb Z^{d}$ so that the increments $v_{i}-v_{i-1}$ are the $Q_{x}$ skeleton of some lattice path $\Gamma$. Let $$Y_{i}= \mathbb E T(v_{i-1},v_{i})-T(v_{i-1},v_{i}).$$ The hypothesis  $\mathbb E e^{\lambda \tau_{e}} <\infty$ is sufficient (see Theorem \ref{thm:futurereference}) to find a constant $C>0$ so that 
$$ \Pro\bigg(|Y_{i}| > \lambda |v_{i}-v_{i-1}|^{1/2}\bigg) \leq C \exp \bigg( - C \lambda \bigg).$$
In particular, this combined with the fact that $|v_{i}-v_{i-1}| \leq 2d|x|$ implies that the existence of a $\beta>0$ so that $\mathbb E \exp (\beta Y_{i} /|x|^{1/2})) <\infty$. Now, imagine for a second that the set of edges used by paths from $v_{i}$ to $v_{i+1}$ were disjoint and deterministic. Then $Y_{i}$ would independent. Moreover, taking exponentials and using Markov's inequality we obtain: 
\begin{equation}\label{eq:eqAlexr3}
\begin{split}
\Pro \bigg ( \sum_{i=0}^{m-1} \mathbb E T(v_{i},v_{i+1}) - T(0,v_{m}) \geq C |x|^{1/2}\log|x|\bigg) &\leq 
\Pro\bigg( \sum_{i=0}^{m-1} Y_{i} \geq C |x|^{1/2} \log |x|\bigg)\\ 
&\leq C\exp(-Cm\log |x|). 
\end{split}
\end{equation}
The fact that we assumed the use of disjoint edges and independence (illegally) can be handled by taking independent copies $Y_{i}'$ of $Y_{i}$ and using two extensions of the  BK inequality (this is done in Kesten \cite[(4.13)]{KestenAspects}), see also Theorem~\ref{eq:BK} below. So morally \eqref{eq:eqAlexr3} is correct. On the other hand, by geometric considerations, there are at most $C|x|^{md}$ ways to choose $v_{i}, i=1, \ldots, m$, with $|v_{i}-v_{i-1}| \leq 2d|x|$. Thus, union bound implies,
\begin{equation}\label{eq:faf55}
\begin{split}
\Pro\bigg( \sum_{i=0}^{m-1} \mathbb E T(v_{i},v_{i+1}) - T(0,v_{m}) \geq C |x|^{1/2} &\log |x| \text{ for some } m\geq 1 \text{ and some } Q_{x}\\
 \text{ skeleton}\text{ with }  &m \text{ vertices } \bigg) \leq C\exp(-C\log |x|). 
\end{split}
\end{equation}
Now take a $Q_{x}$-skeleton with $m$ vertices of a geodesic from $0$ to $nx$. By definition, it follows that $n\leq m$. The estimate \eqref{eq:faf55} above implies that we can choose $n$ large enough so that with probability going to one
$$ \sum_{i=0}^{m-1} \mathbb E T(v_{i},v_{i+1}) - T(0,nx) \leq Cm |x|^{1/2}\log |x|.$$
Hence, by the convergence to the time constant, 
$$ \sum_{i=0}^{m-1} \mathbb E T(v_{i},v_{i+1}) \leq n\mu(x)+ Cm |x|^{1/2}\log |x|.$$
Now the proof ends by arguing (see \cite[(3.16)]{Alexander}) that the large majority of the terms in the sum of the left-side are good increments, that is, they are of order $C\mu(x)$, which leads (when all constants are put in order) to $m\leq 2n$ for $n$ large and $|x|\geq C.$

We finish this section with a summary of the current state of the art on (a) non-random fluctuation upper bounds and (b) convergence rate to the limit shape. For (a), Alexander's methods were used in the low moment case in \cite{DKubota} along with a concentration inequality for the lower tail of $T(0,x)$ to obtain the following.

\begin{theorem}\label{thm:non-randomAlexanderFPP}
Assume $F(0)<p_c$ and $\mathbb{E}Y^2<\infty$, where $Y$ is the minimum of $d$ i.i.d. copies of $\tau_e$. There exists $C=C(F,d)$, such that for all $x \in \mathbb{Z}^d$ with $|x|>1$,
\begin{equation}\label{nonrandomfluctuations}
\mu(x) \leq \mathbb{E}T(0,x) \leq \mu(x) + C (|x| \log |x|)^{1/2}\ .
\end{equation}
\end{theorem}

For (b) above, the following version of convergence rate bounds under low moments also appears in \cite{DKubota}.
\begin{theorem}
Assume $\mathbb{P}(\tau_e=0)<p_c$. If $\mathbb{E}Y^2<\infty$, where $Y$ is the minimum of $d$ i.i.d. copies of $\tau_e$, then there exists $C>0$ such that almost surely,
\[
B(t) \subset (t+ C(t \log t)^{1/2}) \mathcal{B}_\nu \text{ for all large }t.
\]
If in addition $\mathbb{E}\tau_e^\alpha < \infty$ for some $\alpha > 1+1/d$, then almost surely
\[
(t-C(\log t)^4 t^{1/2})\mathcal{B}_\nu \subset B(t) \text{ for all large }t.
\]
\end{theorem}
The inner bound given here can be improved under finite exponential moments, as mentioned in \eqref{eq:tesrate}.

\subsection{Large deviations}

In this section we discuss some large deviation bounds for FPP. For given $\epsilon >0$, our first goal is to estimate the following probabilities
\begin{equation}
\begin{split}
 p^{u}_{n}(\epsilon) :&= \Pro \bigg( T(0,ne_{1}) > n (\mu(e_{1})+\epsilon) \bigg) \\
  p^{\ell}_{n}(\epsilon) :&= \Pro \bigg( T(0,ne_{1}) < n (\mu(e_{1}) -\epsilon) \bigg),
 \end{split}
\end{equation}
for $n$ large. If we assume \eqref{mincondition}, Theorem \ref{thm:timeconstant} implies that both $p^{u}_{n}(\epsilon)$ and $p^{\ell}_{n}(\epsilon)$ go to zero as $n\to \infty$. Stronger  assumptions will allow us to derive exponentially small upper bounds for these probabilities. 

Typically, these bounds will exhibit a strong asymmetry: the behavior of $p_{n}^{u}$ is  radically different from $p_{n}^{\ell}$. Indeed, If we want to increase the passage time $T(0,ne_{1})$, we may need to increase the passage time of every path from $0$ to $ne_{1}$. If the passage times are bounded, this roughly amounts to increasing almost every edge in a box with diameter of order $n$. Thus one should expect:  
\begin{equation}\label{upperUD}
p_{n}^{u}(\epsilon) \approx \exp (-n^{d}I_{u}(\epsilon)),
\end{equation}
where $a_{n} \approx b_{n}$ if $\log a_{n}/ \log b_{n} \to 1$. 

On the other hand, to decrease the passage time between two points, it suffices to force  the passage time of a single path to be very low. Thus,
\begin{equation}\label{upperLD}
p_{n}^{l}(\epsilon) \approx \exp (-n I_{\ell}(\epsilon)),
\end{equation}
as we should roughly decrease the passage time of order $n$ edges. 

The reader may have seen this asymmetry in large deviations in other areas of probability. In random matrix theory, for instance, the smallest eigenvalue of a $n\times n$ GOE matrix satisfies a large deviation principle of speed $n$ if forced out of the bulk; on the other hand, the probability that the smallest eigenvalue is  above the bottom edge of the equilibrium measure is of order $\exp(-cn^{2})$.

Before stating the results, let's make a remark. The boundedness assumption is important for \eqref{upperUD}. For example, suppose that $\tau_{e}$ is exponentially distributed with mean~$1$. We know that if all $2d$ edges connected to the origin take values larger than $n(\mu+\epsilon)$, then $T(0,ne_{1}) \geq n(\mu+\epsilon)$. Therefore, 

$$ p_{n}^{u}(\epsilon) \geq \Pro\bigg(\tau_{e} \geq n(\mu(e_{1}) +\epsilon) \bigg)^{2d}=
\exp\bigg(-2nd(\mu(e_{1})+\epsilon)\bigg),$$
so \eqref{upperLD} does not hold. 

We now state the main theorems of this section. We start with the large deviation lower bounds and we will assume that $F(0)<p_{c}(d)$ so that $\mu(e_{1}) >0$. Let $$\beta = \sup \bigg \{ x: F(\mu(e_{1})-x)>0 \bigg \}.$$

\begin{theorem}[Theorem 5.2 \cite{KestenAspects}]\label{thm:LLD}
For any $\epsilon>0$, there exist a positive constant $A(\epsilon)$ and  $I_{\ell}(\epsilon) \in (0,\infty]$ such that for any $n\geq 1$
$$p_{n}^{\ell}(\epsilon) \leq A(\epsilon)\exp(-nI_{\ell}(\epsilon)).$$
Furthermore,
\begin{equation}\label{eq:KestenLDPlb}
\lim_{n} \frac{1}{n} \log p_{n}^{\ell}(\epsilon) = -I_{\ell}(\epsilon),
\end{equation}
where
$0<I_{\ell} (\epsilon) < \infty$ for any $0<\epsilon < \beta$ and $I_{\ell} (\epsilon) = \infty$ for $\beta \leq \epsilon$. Setting $I_{\ell}(\epsilon) = 0$ for $\epsilon \leq 0$, the function $I_{\ell}:(-\infty,\beta) \to \mathbb R$ is convex on its domain and strictly increasing on $[0,\beta)$.
\end{theorem}

\begin{remark}
Theorem \ref{thm:LLD} requires no assumptions on the moments of the edge passage times. In particular, it holds even when \eqref{eq:conditionmomentsLS} is not satisfied and $\mu(e_{1})$ is not defined. In this case, $\mu(e_{1})$ in the definition of $p_{n}^{\ell}(\epsilon)$ is replaced by the shell time constant $\hat \mu(e_{1}) := \lim_{n} \hat T(0,ne_{1})/n$ from Section \ref{sec:shell}. 
\end{remark}
\begin{remark}
The function $I_{\ell}(\epsilon)$ should depend on $d$ and on the underlying distribution of passage times $F$.
\end{remark}

We now turn to upper large deviation bounds. As remarked before, it does not suffice to have finite exponential moments to derive \eqref{upperUD}. The most natural assumption is to require the passage times to be bounded. This was the assumption used by Kesten in \cite{KestenAspects}.

\begin{theorem}[Theorem 5.9 \cite{KestenAspects}] \label{thm:LDPUP}Assume that $\int e^{\lambda x}~\text{d}\nu <\infty$ for some $\lambda>0$. Then for each $\epsilon>0$ there exist constants $A_{1} = A_{1}(\epsilon)$ and $B_{1}=B_{1}(\epsilon)>0$ such that for all $n\geq 1$,
\[
\mathbb{P}\big(T(0,ne_{1}) > n(\mu+\epsilon)\big) \leq A_{1} e^{-B_{1}n}.
\]
\item If $\tau_e$ is bounded with probability one, then for all $\epsilon>0$ there exist constants $A_{2} = A_{2}(\epsilon)$ and $B_{2}=B_{2}(\epsilon)>0$ such that for all $n \geq 1$
\begin{equation}\label{eq:KestenUBldp}
\mathbb{P}\big(T(0,ne_{1}) > n(\mu+\epsilon)\big) \leq A_{2} e^{-B_{2}n^d}\ .
\end{equation}
\end{theorem}

Note that  Kesten's bound \eqref{eq:KestenUBldp} for bounded passage times is not of the form \eqref{upperUD}. Proving \eqref{upperUD} is still an open question.

\begin{question}\label{q:LDP}
Assume that the passage times are bounded and not concentrated at a single point. Show that there exists a convex function $I_{u}(\epsilon)$ such that the following limit holds:
$$ \lim_{n} \frac{1}{n^{d}} \log p_{n}^{u} (\epsilon) = I_{u}(\epsilon).$$
\end{question}

Theorem \eqref{thm:LDPUP} was extended by Cranston-Gauthier-Mountford. In \cite{CDM}, the authors found necessary and sufficient conditions on the distribution of edge weights so that $\limsup_{n} n^{-d}\log p_{n}^{u}(\epsilon) <0.$ Their result remains the most recent contribution concerning large deviation bounds of passage times.

\begin{theorem}[Theorem 1.3 in \cite{CDM}] Assume that for positive $M_{0}<\infty$ that there is a positive increasing function $f$ so that 
$$ \log \Pro \big( \tau_{e}>x\big)=-x^{d}f(x), \quad x>M_{0}.$$
Then for every sufficiently small $\epsilon>0$,
$$\limsup_{n} \frac{1}{n^{d}} \log \Pro\bigg(T(0,ne_{1}) \geq n(\mu+\epsilon)\bigg) <0$$ 
if and only if
$$ \sum_{n=1}^{\infty} \frac{1}{f(2^{n})^{1/(d-1)}}<\infty. $$
\end{theorem}

If instead of considering the sequence of point to point passage times $T(0,ne_{1})$, we take the box-to-plane passage times
$$ G(0,n) = \inf \bigg\{T(\Gamma) : \Gamma \text{ a path from } \{0\} \times [-n,n]^{d-1} \text{ to } \{n\} \times \mathbb Z^{d-1}\bigg\},$$
then \cite{GK2} $$\lim_{n} \frac{G(0,n)}{n} = \lim_{n} \frac{T(0,ne_{1})}{n}=\mu(e_{1}).$$ In this case, assuming $F(0)<p_{c}$, Chow and Zhang \cite{CZldp} showed the existence of convex functions $c(\epsilon)$ and $c'(\epsilon)$ so that
$$\lim_{n} \frac{1}{n^{d}} \log \Pro\bigg(G(0,n) \geq n(\mu+\epsilon)\bigg) =c(\epsilon) \text{ and }  \lim_{n} \frac{1}{n} \log \Pro\bigg(G(0,n) \leq n(\mu-\epsilon)\bigg) =c'(\epsilon). $$

We end this section with a brief remark on the proofs of Theorems \ref{thm:LLD} and \ref{thm:LDPUP}. The most important step in the proof of these theorems is Proposition \ref{Prop:KestenLDP} below. It estimates the distribution of $T(0,ne_{1})$ in terms of sums of independent passage times across large strips. These strips are still small with respect to $n$.  

Let $H_{k}$ be the hyperplane $\{ x=(x_{1},\ldots, x_{d}) : x_{1}=k \} \subset \mathbb R^{d}$. For a path $\Gamma$ we write 
$$ H_{k} < \Gamma < H_{l} $$ 
if all edges of $\Gamma$ lie strictly between $H_{k}$ and $H_{l}$. For $N\geq 1$, we define the $d-1$ dimensional box  
$$ H_{k}^{N}:= \bigg\{ x=(x_{1},x_{2}, \ldots, x_{d}): x_{1} = k,\; 0\leq x_{i} \leq N, i\geq 2 \bigg\}. $$
Note that $H_{k}^{N} \subset H_{k}$.  For $M\geq 1$ we also define the following cylinder passage time:

$$ s_{0,n}^{N} := \inf \bigg \{ T(\Gamma) : \Gamma \text{ a path from some point  in }  H_{0}^{N} \text{ to } H_{n} \text{ for which } H_{0}<\Gamma<H_{n} \bigg\}.  $$

\begin{proposition}[Proposition 5.23 in \cite{KestenAspects}]\label{Prop:KestenLDP}
Let $X_{i}(M,N)$ be independent random variables, each with distribution of $s_{0,M}^{N}$. Then for any $n\geq M \geq N \geq 1$, $t\geq0$ one has
\begin{equation}\label{eq:EstimateinAspects}
\Pro\bigg( T(0,ne_{1}) < t \bigg) \leq \sum_{\ell = n/(M+N) -1}^{\infty} (2d(15\frac{M}{N})^{d})^{\ell} \Pro\bigg(X_{1}(M,N)+\ldots+ X_{\ell}(M,N) < t \bigg).
\end{equation}
\end{proposition}
The proof of Proposition \ref{Prop:KestenLDP} relies on the van den Berg-Kesten (BK) inequality. In reality, it relies on the lemma below, one of the most useful illustrations of the BK inequality in FPP. Given sets of vertices of $\mathbb Z^{d}$, $A(i,j)$, $i\geq 1, 1\leq j \leq n(i)$, with $A(i,0) = \{ 0 \}$ and $A(i,n(i)+1)=\{x\}$, define
\begin{equation*}\begin{split}
\Pi_{i}(0,x,t):= \bigg \{ \Gamma: \Gamma &\text{ is a path from } 0 \text{ to } x \text{ that passes successively }\\ 
&\text{ through } A(i,1), \ldots A(i,n(i)) \text{ and has } T(\Gamma)<t \bigg\}.
\end{split}
\end{equation*}

\begin{lemma} [(4.13) in \cite{KestenAspects}]\label{eq:BK}For any $t>0$,
\begin{equation}
\Pro \bigg( \bigg \{ \bigcup_{i\geq 1} \Pi_{i}(0,x,t) \bigg\} \neq \emptyset \bigg) \leq \sum_{i\geq 1} \Pro\bigg( \sum_{j=0}^{n(i)} T'(i,j) < t \bigg), 
\end{equation}
where $T'(i,j)$ are all independent with $T'(i,j) \stackrel{d}{=} T(A(i,j), A(i,j+1))$.
\end{lemma}

\subsection{Cases where Gaussian fluctuations appear}
In this section, we discuss a few cases where Gaussian fluctuations are known or expected to appear. The first example is Critical FPP, where the edge weights are taken with $F(0)=p_{c}$, exactly at the percolation threshold. The second case, is when we look at FPP in a thin cylinder, where geodesics want to be thinner than usual.
\subsubsection{Critical first-passage percolation}

From Kesten's work, we know that 
\[
p:=\mathbb{P}(\tau_e = 0) \geq p_c \Leftrightarrow g(x) = 0 \text{ for all }x,
\]
where $p_c$ is the critical threshold for bond percolation on $\mathbb{Z}^d$. Most theorems in FPP have the assumption that the above probability is strictly less than $p_c$ for this reason. Namely, if $p > p_c$, then one can show that as $x\to \infty$, the family of passage times $T(0,x)$ is stochastically bounded above by a random variable $Z$ that does not depend on $x$. The reason is that to travel from $0$ to $x$, one simply needs to go from 0 to the infinite cluster of zero-weight edges, then from here to $x$. On the other hand, if $p=p_c$, then very little is known. This is due to the fact that one of the outstanding problems in probability theory is to show that there is no infinite cluster of open edges in bond percolation on $\mathbb{Z}^d$ at the critical point. Therefore, we do not know if vertices can take advantage of an infinite cluster of zero-weight edges when $p=p_c$, and so we do not know even the growth rate of $T(0,x)$ as $x \to \infty$.

The first result on the critical case came from Chayes \cite{Chayes}, who showed in the Bernoulli case:
\begin{theorem}[Chayes \cite{Chayes}]
Let $d \geq 3$ and $\tau_e$ satisfy $\mathbb{P}(\tau_e=0) = p_c = 1-\mathbb{P}(\tau_e=1)$. If $\epsilon>0$, then
\[
\lim_n \frac{T(0,ne_1)}{n^\epsilon} = 0 \text{ almost surely}.
\]
\end{theorem}
\noindent See \cite[Remark~3]{KestenAspects} also, where Kesten claims that Chayes's argument can be extended to $$T(0,ne_1) \leq \exp(C\sqrt{\log n})$$ for large $n$ almost surely.

Although this is the extent of knowledge on the critical case in general, much more progress has been made in two dimensions. In this case, Zhang \cite{Zhang} showed that different critical edge-weight distributions can display completely different asymptotic behavior: defining the passage time to infinity as
\begin{equation}\label{eq: zhangs_rho}
\rho = \lim_n T(0,\partial B(n)),
\end{equation}
where $B(n) = [-n,n]^2$, there are some critical distributions for which $\rho=\infty$ almost surely, and there are some for which $\rho<\infty$ almost surely. He did this by comparing critical FPP to a model called the CCD incipient infinite cluster\cite{CCD1}. This is an inhomogeneous bond percolation model which is tuned to resemble a large percolation cluster at criticality. A complete characterization was given recently in \cite{DLW}, using methods related to 2$d$ invasion percolation. For its statement, let $F$ be the distribution function of $\tau_e$ and set
\[
F^{-1}(y) = \inf\{x : F(x) \geq t\} \text{ for } t \geq 0.
\]
\begin{theorem}[Damron-Lam-Wang  \cite{DLW}]\label{thm: DLW}
For $d=2$, one has $\rho < \infty$ almost surely if and only if $\sum_{n=2}^\infty F^{-1}(p_c + 2^{-n}) < \infty$.
\end{theorem}
The theorem implies that if $F$ is concentrated enough near $0$, then the passage time to infinity can be finite. Otherwise, if there are too many large edges, the passage time to infinity will be infinite. For example, $\rho < \infty$ when $F$ decays to $p_c$ like a polynomial near 0, but can be infinity if $F$ decays to $p_c$ like an exponential (with large enough rate) near 0.

In the case that $\rho = \infty$, one can ask about the rate of growth of $T(0,\partial B(n))$ as $n \to \infty$. The first result of this type was due to Chayes-Chayes-Durrett\cite{CCD}, who showed that in the critical Bernoulli case, $$\mathbb{E}T(0,\partial B(n)) \asymp \log n.$$ The general case was given in \cite{DLW}. Let $Y$ be the minimum of $4$ i.i.d. copies of $\tau_e$.
\begin{theorem}[Damron-Lam-Wang  \cite{DLW}]
For $d=2$, suppose that $F(0)=p_c$ and $\mathbb{E}Y^\alpha < \infty$ for some $\alpha>1/4$. There exists $C>0$ such that
\[
\frac{1}{C} \mathbb{E} T(0, \partial B(2^n)) \leq \sum_{k=2}^n F^{-1}(p_c + 2^{-k}) \leq C \mathbb{E}T(0,\partial B(2^n)).
\]
If $\mathbb{E}Y^\alpha < \infty$ for some $\alpha > 1/2$, then there exists $C'>0$ such that
\[
\frac{1}{C'} \mathrm{Var}~T(0,\partial B(2^n)) \leq \sum_{k=2}^n \left( F^{-1}(p_c+2^{-k}) \right)^2 \leq C' \mathrm{Var}~T(0,\partial B(2^n)).
\]
\end{theorem}

For FPP on the sites of a triangular lattice, Yao \cite{Yao2014} proved the existence of a constant $\mu$ such that $(\log n)^{-1}T(0,\partial B(n))$ converges to $\mu$ almost surely.

Back in $\mathbb Z^{d}$, in the case $F(0)<p_c$, establishing limiting laws for $T(0,ne_1)$ is far off, except in some related exactly solvable models. Kesten and Zhang, however, showed that in the Bernoulli critical case in two dimensions, one has a Gaussian central limit theorem on scale $\sqrt{\log n}$. Their theorem applied to all distributions with a gap near zero; that is, ones which have $F(0)=p_c$ and $F(a) = p_c$ for some $a>0$.
 This reason for this condition is precisely the uncertain behavior of $\rho$ in the critical case.

\begin{theorem}[Kesten-Zhang \cite{KestenZhang}] 
Let $d=2$. Suppose that for some $C>0$, $F(0)=F(C)=1/2$. Also assume that $\E \tau_e^{4+\delta} < \infty$ for some $\delta >0$. Then there exist positive constants $C_1, C_2, \gamma_n$ with $$C_1 (\log n)^{1/2} \leq \gamma_n \leq C_2 (\log n)^{1/2}$$ such that for all $t \in \R$
$$ \lim_{n \to \infty} \Pro \big(\gamma_n^{-1}\big[T(0,ne_1) - \E T(0,ne_1)\big] \leq t\big) = \Pro (Z\leq t)$$
where $Z$ is a standard Gaussian random variable.
\end{theorem}

This result was extended to general distributions in \cite{DLW}. Note that if $\sum_k F^{-1}(p_c + 2^{-k}) = \infty$ but $\sum_k (F^{-1}(p_c+2^{-k}))^2 < \infty$, then the variance of the passage time converges, whereas the mean diverges. In this case, we do not get a Gaussian limit, but rather $T(0,\partial B(n)) - \mathbb{E}T(0,\partial B(n))$ converges to some other variable.
\begin{theorem}[Damron-Lam-Wang  \cite{DLW}]
Suppose that $F(0) = p_c$, and $\sum_{k=2}^\infty F^{-1}(p_c + 2^{-k}) = \infty$. Suppose further that $\mathbb{E}Y^\alpha<\infty$ for some $\alpha > 1/2$. 
\begin{enumerate}
\item If $\sum_{k=2}^\infty ( F^{-1}(p_c + 2^{-k}))^2 < \infty$, then there is a random variable $Z$ with $\mathbb{E}Z=0$ and $\mathbb{E}Z^2 < \infty$ such that as $n \to \infty$,
\[
T(0,\partial B(n)) - \mathbb{E}T(0, \partial B(n)) \to Z \text{ almost surely and in } L^2.
\]
\item If $\sum_{k=2}^\infty (F^{-1}(p_c + 2^{-k}))^2 = \infty$, then as $n \to \infty$,
\[
\frac{T(0,\partial B(n)) - \mathbb{E}T(0, \partial B(n))}{\sqrt{\mathrm{Var}~T(0,\partial B(n))}} \Rightarrow N(0,1).
\]
\end{enumerate}
\end{theorem}

In the case of site FPP on the triangular lattice with Bernoulli weights ($\mathbb{P}(t_v = 1) = 1/2 = \mathbb{P}(t_v=0)$), Yao \cite{Yao2} has recently used the conformal loop ensemble to find the exact values of the constants appearing the first-order asymptotics of the mean and variance of the passage time.
\begin{thm}[Yao \cite{Yao2}]
Consider site FPP on the triangular lattice with weights $(t_v)$ satisfying $\mathbb{P}(t_v=1)=1/2=\mathbb{P}(t_v=0)$. The following statements hold:
\begin{align*}
\lim_n \frac{T(0,\partial B(n))}{\log n} &= \frac{1}{2\sqrt{3}\pi} \text{ almost surely}, \\
\lim_n \frac{\mathbb{E}T(0,\partial B(n))}{\log n} &= \frac{1}{2\sqrt{3}\pi}, \\
\lim_n \frac{\mathrm{Var}~T(0,\partial B(n))}{\log n} &= \frac{4}{3\sqrt{3}\pi} - \frac{1}{\pi^2}.
\end{align*}
\end{thm}

There is a big difference between what is known in two dimensions and in three dimensions. This is due mostly to the lack of knowledge about critical percolation in three dimensions. A natural question is then:
\begin{question}
Let $d=3$ and $\mathbb{P}(\tau_e=0) = p_c = 1-\mathbb{P}(\tau_e=1)$. Is it true that 
\[
\mathbb{E}T(0,\partial B(n)) \asymp \log n?
\]
\end{question}
The two dimensional proof of this fact uses that with uniformly positive probability in $n$, every path which crosses $B(n) \setminus B(n/2)$ must contain at least one edge which does not have zero weight. This is no longer expected to be true in high dimensions (above, say, six \cite[Theorem 4]{Aiz}), so it is conceivable that for large $d$, one has $T(0, \partial B(n)) \ll \log n$ even in the Bernoulli case.

\subsubsection{Passage time in small cylinders}

A second case where Gaussian fluctuations are present is when we consider passage times of paths constrained in a thin cylinder. Precisely, for $h >0$ let $a_n(h)$ be the first passage time from $0$ to the point $ne_1$ in the graph $(\Z \times [-h,h]^{d-1}) \cap \Z^d$; that is, define
$$a_n(h) := \inf \{ T(\gamma) | \gamma \text{ is a path from } 0 \text { to } ne_1 \text{ in } \Z \times [-h,h]^{d-1} \}.$$

If $h$ is small enough, the constrained geodesic from $0$ to $ne_1$ should not coincide with the one in the full integer lattice. In this case, we have Gaussian fluctuations after proper centering and scaling.

\begin{theorem}[Chatterjee-Dey \cite{CD}] Suppose that $\E \tau_e^{2+\delta}<\infty$ for some $\delta >0$. Assume that $h=h(n)$ satisfies $h(n)=o(n^\alpha)$ with
\begin{equation}\label{eq.thin}
 \alpha < \frac{1}{d+1+2(d-1)/\delta}.
 \end{equation}
Then we have for all $t\in \R$
\begin{equation}\label{eq:thintheorem}
 \lim_{n \to \infty}  \Pro\bigg( \frac{a_n(h_n)-\E a_n(h_n)}{\sqrt{\mathrm{Var} (a_n(h_n))}} \leq t \bigg) = \Pro (Z\leq t)
\end{equation}
where $Z$ is a standard Gaussian random variable.
\end{theorem}

The theorem above gives rise to the following question, raised in \cite{CD}.

\begin{question}
Is \eqref{eq.thin} optimal? More explicitly, let 
$$\gamma_F(d):= \sup \bigg \{ \alpha:  \eqref{eq:thintheorem} \text{ holds for } h_n=n^\alpha  \bigg\}.$$ What is the value of $\gamma_F(d)?$
\end{question}

A partial answer was given by the same paper of  Chatterjee and Dey. If one has finite exponential moments for the passage time distribution and assumes (a) existence of both fluctuation exponent $\chi(d)$ and transversal exponent $\xi(d)$ (See Section \ref{sec:exponents}), (b) strict convexity of the limiting shape and (c)   $d=2$ or $d=3$  or $\xi(d), \chi(d)>0$ then $\gamma_F(d) = \xi(d)$.


\newpage

\section{Geodesics}\label{sec:geodesics}

Recall the definition of a geodesic from $x$ to $y$.
\begin{definition}A path $\gamma$ from $x$ to $y$ with $T(\gamma) = T(x,y)$ is called a {\it geodesic} from $x$ to~$y$. 
\end{definition}
Soon, we will deal with infinite geodesics, so we will sometimes use the term finite geodesic for a geodesic between two points. We will also require geodesics from points to sets: if $A$ is a subset of $\mathbb{Z}^d$, then a geodesic from $x$ to $A$ is a geodesic from $x$ to $y$, where $y$ is the vertex of $A$ minimizing $T(x, y)$. If $A$ is a subset of $\mathbb{R}^d$, then we identify each point of $A$ with the nearest vertex of $\mathbb{Z}^d$, and define $\hat A$ to be the union of all such vertices; the set of all geodesics from $x$ to $A$ is defined to be the set of all geodesics from $x$ to $\hat A$.

\subsection{Existence of finite geodesics, their sizes, and the geodesic tree}
The very first step in the study of geodesics is to determine if a geodesic from $x$ to $y$ exists with probability one. 
Hammersley and Welsh \cite{HW} showed almost sure existence of finite geodesics if the distribution of the passage times is bounded above and below away from zero.
They also conjectured that almost sure existence held for all distributions, with no moment conditions. Smythe and Wierman \cite{SW,SW3} verified this conjecture in two dimensions except for distributions with an atom at zero equal to the critical bond percolation probability.  In $\Z^2$, the conjecture was proven in full generality by Wierman and Reh \cite{RW78}, and it is now the state-of-the-art. It was stated in \cite[Corollary~1.3]{RW78} for geodesics between 0 and $ne_1$, but the same proof applies to any pair of points.

\begin{theorem}[Wierman-Reh \cite{RW78}] For any passage time distribution $F$, With probability one, there exists a geodesic between any two points of $\Z^2$. 
\end{theorem}

The theorem above was extended to connected, infinite subsets of $\Z^2$ with connected infinite complement in \cite{ADH1}. Existence under no assumptions on the distribution of passage times is currently open in $d$ dimensions if $d>2$.
\begin{question} Prove that under no assumptions on the distribution $F$ of passage times, geodesics exist in the $d$ dimensional lattice, $d\geq3$.
\end{question}

Next we will now explain how, as noted in \cite[(9.23)]{KestenAspects}, the existence of geodesics follows from standard estimates provided $F(0) < p_c$. In the case $F(0) > p_c$ with $\mathbb{E} \tau_e < \infty$, existence of geodesics was proved in \cite{Zhangsuper}.  The case $F(0) = p_c$ is still unresolved in general when $d > 2$. However, geodesics are known to  exist for distributions $F$ satisfying $F(0)= F(\delta) = p_c$ for some $\delta>0$ and in sufficiently high dimensions. See \cite[Theorem~8.1.7]{Zhangdoublebehavior}.

Recall the definition of the passage time to infinity \eqref{eq: zhangs_rho}, from \cite{Zhangdoublebehavior}:
\[
\rho = \rho(F) = \lim_n \left[ T(0,\partial B(n)) \right],
\]
where $B(n) = [-n,n]^d$. By the Komogorov 0-1 law, either $\rho = \infty$ almost surely, or $\rho < \infty$ almost surely.
\begin{lemma}
One has
\[
\rho = \inf\{T(\gamma) : \gamma \text{ is an infinite self-avoiding path from }0\}.
\]
\end{lemma}
\begin{proof}
Let $\gamma_n$ be a (self-avoiding) geodesic from 0 to $\partial B(n)$, with $\gamma$ a subsequential limit of $(\gamma_n)$ along some subsequence $(\gamma_{n_k})$. For a fixed $n$, if $\hat \gamma_n$ is the initial portion of $\gamma$ from 0 to $\partial B(n)$, then for all large $k$, the geodesic from 0 to $\partial B(n_k)$ contains $\hat \gamma_n$ as its first segment, so
\[
T(0, \partial B(n_k)) \geq T(\hat \gamma_n) \geq T(0, \partial B(n)).
\]
Letting $k \to \infty$, we obtain
\[
\rho \geq T(\hat \gamma_n) \geq T(0,\partial B(n)).
\]
As $n \to \infty$, $T(\hat \gamma_n) \to T(\gamma)$, so $T(\gamma) = \rho$, showing the inequality $\geq$ in the lemma. The other inequality holds because each infinite self-avoiding path from 0 contains an initial segment from 0 to $\partial B(n)$.
\end{proof}

From this lemma, we can show that $F(0) < p_c$ implies that geodesics exist.
\begin{proposition}
The following hold.
\begin{enumerate}
\item If $\rho = \infty$ for some configuration $(\tau_e)$, then for all $x,y \in \mathbb{Z}^d$, a geodesic exists from $x$ to $y$.
\item If $F(0)< p_c$, then almost surely, $\rho = \infty$.
\end{enumerate}
\end{proposition}
\begin{proof}
If $\rho = \infty$ in some passage time configuration $(\tau_e)$, pick any $x,y$, and fix a path $\gamma$ between them. Note that since $\rho = \infty$ and
\[
\left| T(x,\partial B(n)) - T(0,\partial B(n)) \right| \leq T(0,x) < \infty,
\]
one can choose $n$ large enough so that (a) $x,y$ and $\gamma$ are in $B(n)$ and (b) $T(x,\partial B(n)) > T(\gamma)$. If $\gamma'$ is any path from $x$ to $y$ that exits $B(n)$, then
\[
T(\gamma') \geq T(x,\partial B(n)) > T(\gamma),
\]
so the infimum in the definition of $T(x,y)$ is over a finite set of paths, and there must be a minimizer.

If $F(0)<p_c$, then choose $\delta>0$ such that $F(\delta) < p_c$. By the definition of $p_c$, almost surely there is no infinite self-avoiding path of edges $e$ with $\tau_e \leq \delta$, and in fact each such path $\gamma$ must have infinitely many edges with weight at least $\delta$. So by the characterization in the previous lemma, $\rho = \infty$.
\end{proof}

The condition $\rho = \infty$ is not necessary for existence of geodesics, as they exist when $F(0) > p_c$ (and therefore $\rho < \infty$). It is unknown in general dimensions if $\rho = \infty$ for a given $F$, but in two dimensions, Theorem~\ref{thm: DLW} gives necessary and sufficient conditions.

We conclude this section with estimates on the size of finite geodesics. Let $\Gamma(0,x)$ denote the set of geodesics from $0$ to $x$. The diameter of the set $\Gamma(0,x)$ is controlled by the following result that requires no moment condition on the distribution of $\tau_{e}$.

\begin{theorem}[Theorem 6.2 \cite{ADH2}]\label{thm: geo_diameter_bound} Assume $F(0)<p_{c}(d)$. Then there exist positive constants $M, C$ such that
\begin{equation*}
\Pro \bigg( \mathrm{diam}\; \Gamma(0,x) \geq M \| x \|_{\infty}\bigg) \leq \exp(-C\|x \|_{\infty}), \quad \text{ for all } x \in \mathbb Z^{d}.
\end{equation*}
\end{theorem}

Furthermore, we have the following linear bound on the number of edges in a geodesic.

\begin{theorem} \label{thm: new_geo_length_bound}
Assume $F(0) < p_c(d)$. Then there exist positive constants $M,C$ such that
\[
\mathbb{P}\left(m(x) \geq M\|x\|_\infty\right) \leq \exp\left( - C \|x\|_\infty^{1/d} \right) \text{ for all } x \in \mathbb{Z}^d,
\]
where $m(x) = \max\{|\sigma| : \sigma \text{ a self-avoiding geodesic from } 0 \text{ to } x\}$.
\end{theorem}

Theorem~\ref{thm: new_geo_length_bound} implies that if $F(0)<p_c$, then for each $\alpha>0$, there is $C_\alpha$ such that for all $x \neq 0$,
\[
\mathbb{E}m(x)^\alpha \leq C_\alpha \|x\|_\infty^\alpha.
\]
Under the additional condition $\mathbb{E}e^{\beta \tau_e} < \infty$ for some $\beta>0$, \cite[Corollary~1.4(2)]{DHS2} shows a similar bound for exponential moments: there exists $\alpha_1>0$ such that
\[
\sup_{0 \neq x \in \mathbb{Z}^d} \frac{\log \mathbb{E}e^{\alpha_1 m(x)}}{\|x\|_\infty} < \infty.
\]

The proofs of the above two theorems use almost the same method, and since the first has appeared, we give the second.
\begin{proof}[Proof of Theorem~\ref{thm: new_geo_length_bound}]

For $p\in [0,1]$ let $\mathbb{P}_p$ be the product measure on $\Omega = \{0,1\}^{\mathcal{E}^d}$ with marginal $\mathbb{P}_p(\omega(e) = 1) = p$, where $\omega$ is a typical element of $\Omega$. In a configuration $\omega$ we write $x \to y$ if there is a path from $x$ to $y$ with edges $e$ satisfying $\omega(e) = 1$. This gives a connectivity equivalence relation and the equivalence classes are called open clusters. It is known that for $p>p_c$, there is almost surely a unique infinite open cluster. Define $B(n) = \{x \in \mathbb{Z}^d : \|x\|_\infty \leq n\}$ and $\partial B(n) = \{x \in \mathbb{Z}^d : \|x\|_\infty = n\}$.

\begin{lemma}\label{lem: kesten}
Let $A_n$ be the event that every path from 0 to $\partial B(n)$ intersects the infinite open cluster. There exists $p_0 \in (p_c,1)$ such that if $p \in [p_0,1]$ then for some $C_1>0$,
\[
\mathbb{P}_p(A_n) \geq 1 - e^{-C_1n} \text{ for all }n.
\]
\end{lemma}
\begin{proof}
The proof is a slight modification of the result of Kesten \cite[Theorem~2.24]{KestenAspects}, and is given as \cite[Lemma~6.3]{ADH2}.
\end{proof}

We now recall the result of Antal-Pisztora and afterward prove the theorem.
\begin{lemma}[Theorem~1.1 \cite{AntalPisztora}]\label{lem: AP}
Let $p>p_c$. Then there exists a constant $\rho = \rho(p,d) \in [1,\infty)$ such that
\[
\limsup_{\|y\|_\infty \to \infty} \frac{1}{\|y\|_\infty} \log \mathbb{P}_p(d_I(0,y) > \rho \|y\|_\infty,~0,y \in I) < 0,
\]
where $I$ is the infinite open cluster and $d_I$ is the intrinsic distance in $I$.
\end{lemma}

We choose $K>0$ such that $p:= \mathbb{P}(\tau_e \leq K) > p_0$ and define a percolation configuration $(\eta_e)$ from the weights $(\tau_e)$ by
\[
\eta_e = \begin{cases}
1 & \text{ if } \tau_e \leq K \\
0 & \text{ if } \tau_e > K
\end{cases}.
\]
For $x \in \mathbb{Z}^d$ and an integer $M>0$, we first write $n_x = \lfloor \|x\|_\infty^{1/d} \rfloor$ and estimate
\begin{equation}\label{eq: nachos_bellegrande}
\mathbb{P}(m(x) \geq M\|x\|_\infty) \leq 2e^{-C_1 n_x} + \mathbb{P}(m(x) \geq M\|x\|_\infty,~A_{n_x},~B_{n_x}),
\end{equation}
where $A_{n_x}$ is written for the event in Lemma~\ref{lem: kesten} for the percolation configuration $(\eta_e)$ and $B_{n_x}$ is the same event with $0$ translated to $x$. From now on, we take $\|x\|_\infty \geq C'$ with $C'$ chosen so that $n_x \geq 1$ and $4n_x < \|x\|_\infty$. On the event on the right we can select a self-avoiding geodesic $P$ from 0 to $x$ which has length at least $M\|x\|_\infty$. Write $P_1$ for the portion of $P$ from 0 to its first intersection of $\partial B(n_x)$ and $P_2$ for the portion from its last intersection of $\partial B(x,n_x)$ to $x$ (here $B(x,n)$ is the translate of $B(n)$ centered at $x$). We can then choose $u,v$ vertices of $P_1$ and $P_2$ respectively such that $\|u\|_\infty \leq n_x$ and $\|v-x\|_\infty \leq n_x$ and both $u,v$ are in the infinite open cluster of $(\eta_e)$. By construction, the portion $P_3$ of $P$ from $u$ to $v$ is a self-avoiding geodesic that has at least
\[
M\|x\|_\infty - C_2 n_x^d \geq (M-C_2)\|x\|_\infty \text{ number of edges}.
\]
We now apply Lemma~\ref{lem: AP} to find $C_3,C_4>0$ such that
\[
\mathbb{P}_p(d_I(0,y) > \rho\|y\|_\infty,~0,y \in I) \leq C_4 e^{-C_3 \|y\|_\infty} \text{ for all } y \in \mathbb{Z}^d.
\]
So for $u \in B(n_x)$ and $v \in B(x,n_x)$, as $2\|x\|_\infty \geq \|u-v\|_\infty \geq \|x\|_\infty/4$,
\[
\mathbb{P}_p(d_I(u,v) > 2\rho\|x\|_\infty,~u,v \in I) \leq C_4 e^{-(C_3 /4)\|x\|_\infty}
\]
and by a union bound,
\[
\mathbb{P}_p(d_I(u,v) > 2\rho\|x\|_\infty \text{ for some } u \in B(n_x) \cap I,~v\in B(x,n_x) \cap I) \leq C_5 e^{-C_6 \|x\|_\infty}
\]
for some $C_5,C_6>0$.

On the complement of this event, each $u \in B(n_x) \cap I$ and $v \in B(x,n_x) \cap I$ have $d_I(u,v) \leq 2\rho\|x\|_\infty$ and so $T(u,v) \leq 2K\rho\|x\|_\infty$. Use this in the right side of \eqref{eq: nachos_bellegrande} to bound it above by
\begin{equation}\label{eq: nachos_headache}
C_7 e^{-C_8 \|x\|_\infty^{1/d}} + \mathbb{P}\left( \begin{array}{c}
\exists~u \in B(n_x),v \in B(x,n_x) \text{ with } m(u,v) \text{ at} \\
\text{least } (M-C_2) \|x\|_\infty \text{ but } T(u,v) \leq 2K\rho\|x\|_\infty
\end{array} \right).
\end{equation}
Here, $m(u,v)$ is the maximal number of edges in a self-avoiding geodesic from $u$ to $v$.

Last, we appeal to Kesten's result \cite[Proposition~5.8]{KestenAspects}, which states that under the assumption $F(0)<p_c$, there exist constants $a,C_9 >0$ such that for all $n \geq 1$,
\[
\mathbb{P}\left( \exists \text{ self-avoiding } \gamma \text{ starting at }0 \text{ with } \#\gamma \geq n \text{ but with } T(\gamma) < an \right) \leq e^{-C_9 n}.
\]
By a union bound, for all $n \geq 1$ and $x \in \mathbb{Z}^d$ with $\|x\|_\infty \geq C'$,
\[
\mathbb{P}\left( \exists \text{ self-avoiding }\gamma \text{ from } B(n_x) \text{ to } B(x,n_x) \text{ with } \#\gamma \geq n \text{ but } T(\gamma) < an \right) \leq C_{10} n_x^d e^{-{C_9 n}}.
\]
If there exist $u \in B(n_x)$ and $v \in B(x,n_x)$ with $m(u,v) \geq (M-C_2)\|x\|_\infty$, then we may select a self-avoiding geodesic between $u$ and $v$ that has at least this number of edges. So fixing any $M$ with $2K\rho/a + C_2 < M$, the expression in \eqref{eq: nachos_headache} is bounded by
\[
C_7 e^{-C_8 \|x\|_\infty^{1/d}} + C_{10} n_x^de^{-C_9 (M-C_2)\|x\|_\infty},
\]
and this is bounded by $C_{11}e^{-C_{12}\|x\|_\infty^{1/d}}$ for some $C_{11},C_{12} >0$. Increasing $M$ if necessary, we obtain $e^{-C_{13}\|x\|_{\infty}^{1/d}}$, completing the proof.
\end{proof}

In several occasions we will be interested in the case where $\Gamma(0,x)$ is a single path from $0$ to $x$.

\begin{definition}
A metric space is called uniquely geodesic if between any two points there exists a unique geodesic.
\end{definition}
A sufficient condition to be uniquely geodesic is to require the passage times to have a continuous distribution. Given the existence of finite geodesics, one can define the geodesic graph as follows. For each $x \in \Z^d$ we define the (directed) graph $\mathcal T(x)$ with vertex set $\Z^d$ and whose edge set is the collection of all directed edges $e$ that belong to some finite geodesic between $x$ and another point $y$, oriented as they are traversed moving from $x$ to $y$. This graph has the following properties:
\begin{enumerate}
\item $\mathcal T(x)$ is almost surely connected.
\item Every finite directed path in $\mathcal T(x)$ is a finite geodesic.
\item If the space is uniquely geodesic, then $\mathcal T(x)$ is a spanning tree of $\mathbb Z^d$. In this case, $\mathcal T(x)$ is called the {\it geodesic tree} or the tree of infection of $x$.
\end{enumerate}

Many of the geodesic questions discussed in this survey can be rephrased in the language of geodesic trees. Many of the ``large scale'' questions -- for instance, the number of ends of the tree $\mathcal T(x)$ -- will require the study of infinite geodesics, which we will define shortly. Before discussing these large scale questions, let us focus on a natural small scale question: how straight is the geodesic between $x$ and $y$?

\subsection{The wandering exponent}\label{sec:exponents}
Given a set $A \subseteq \mathbb{R}^d$, let $\Gamma(x, A)$ denote the geodesic from $x$ to $A$ (in case $A = \{y\}$ is a singleton, write $\Gamma(x,y)$). If the first-passage metric is not uniquely geodesic but geodesics exist, let $\Gamma(x,A)$ instead denote the union of all geodesics between $x$ and $A$ (similarly for $A =\{y\}$).

A natural asymptotic question for geodesics is: what does $\Gamma(0,x)$ ``typically'' look like for $\|x\|$ large? Since the first-passage metric is in some sense a random perturbation of the $\ell^1$ metric, one could perhaps imagine that $\Gamma(0,x)$ often stays order  $\|x\|$ distance from the midpoint $x / 2$. On the other hand, as discussed in Section \ref{sec:limitshape}, it is reasonably well-accepted that the limit shape should have uniformly positive curvature, suggesting that $\Gamma(0,x)$ should perhaps look more like the straight line between $0$ and $x$, as in the case of the $\ell^2$ metric.

Indeed, the $\ell^2$ intuition is believed to be the correct one, and $\Gamma(0,x)$ is believed to lie close to the line
\[L_x = \{mx:\, m \in \mathbb{R}\}. \]
It is easy to see very weak versions of this claim under the assumption that the limit shape is strictly convex. For instance, if $x = x(n) = (n, n, \ldots, n)$, and we denote by $\alpha_n$ the path which first takes $n$ steps down the $e_1$ axis to connect $0$ to $n e_1$, then moves $n$ steps up the $e_2$ axis, etc., we see by the shape theorem that
\[\lim_{n \rightarrow \infty} T(\alpha_n) / n = d \mu(e_1)\ . \]

On the other hand, $\lim_{n \rightarrow \infty} T(0, x(n)) / n = \mu(e_1 + \ldots + e_d)$, which is strictly smaller if the limit shape is strictly convex. In particular, $\alpha_n$ is not only not the geodesic to $x(n)$, but actually pays order $\|x\|$ more in passage time than the geodesic does. This simple argument only rules out a single path, however, and stronger arguments have been developed to simultaneously control all candidate geodesics.

Set $D(0,x)$ to be the maximal distance of $\Gamma(0,x)$ to $L_x$. An early argument showing that $\Gamma(0,x)$ lies close to $L_x$ (under assumptions on the limit shape) appeared in \cite[Lemma (9.10)]{KestenAspects}, where it was shown that, assuming that the limit shape has an exposed point in the $e_1$ direction,
\begin{equation}
\label{eq:kestenwander}
 \lim_{\varepsilon \rightarrow 0} \liminf_{n\to \infty} \mathbb{P}\left(\mathrm{dist}(D(0,ne_1) < \varepsilon n \right) = 1\ .\end{equation}
To discuss strengthening \eqref{eq:kestenwander}, we will introduce the ``wandering exponent'' $\xi$.

Roughly speaking, $\xi$ is the number such that the maximal distance between $\Gamma(0,x)$ and $L_x$ is typically of order $\|x\|^\xi$. A simple attempt at giving a concrete definition could be to set $\varepsilon = n^{\alpha - 1}$ in the left-hand side of \eqref{eq:kestenwander} and define $\alpha$ to be the smallest number such that \eqref{eq:kestenwander} still holds; however, there is no general agreement on the ``correct'' definition, and we will discuss several possible alternatives. The exponent $\xi$ is believed to be universal in the sense that, for all reasonable choices of continuous distributions $F$, one should find the same value of $\xi$, and $\xi$ should govern geodesic fluctuations in all directions. In general, however, $\xi$ should depend on the dimension $d$. It is widely believed that, for $d = 2$, the value of $\xi$ should be $2/3$ \cite{HH,HHF,K,KPZ}. (For distributions in $\mathcal{M}_p$ with $p >\vec{p}_c$, defined in Section~\ref{sec: flat_edge}, infinite oriented paths of $1$-edges give $\xi=1$ -- and $\chi=0$ -- for directions in the percolation cone.)

\begin{figure}[h]\label{DefXiExponent}
\centering
 \includegraphics[scale=1]{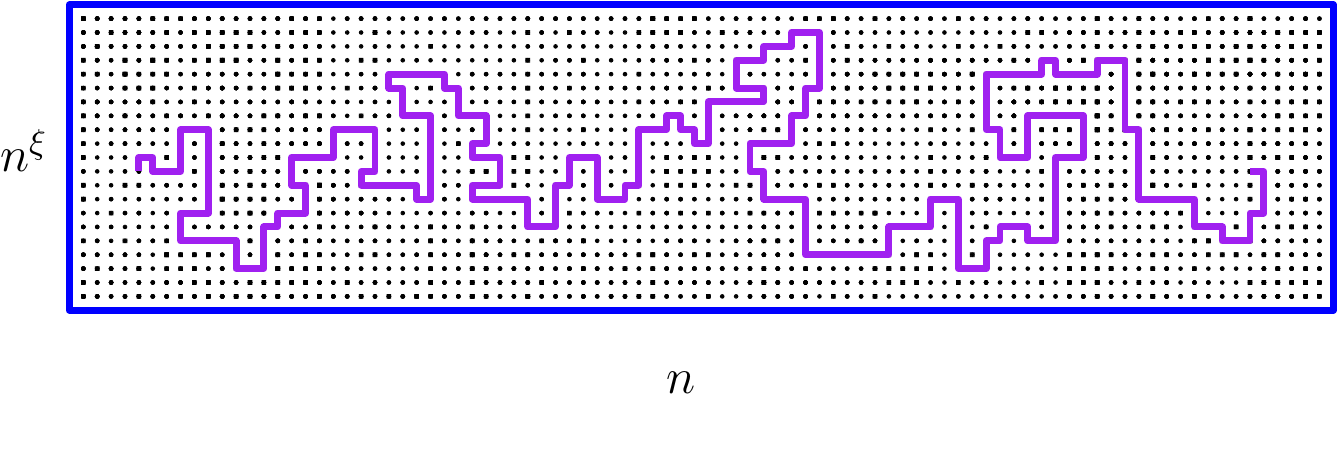}
 \caption{Representation of the wandering exponent $\xi$. The deviation of a geodesic from $0$ to $ne_1$ from the straight line that connects these points is of order $n^\xi$. }
\end{figure}
\subsubsection{Upper bounds on $\xi$}
We now discuss the most robust current upper bounds, from \cite{Newman, NewmanPiza}. 

Letting $\mathcal C(x, m)$ denote the cylinder with axis $L_x = \{mx: \, m \in \mathbb{R}\}$ and radius~$m$ -- i.e.,
\begin{equation}
\label{eq:cyldef}
\mathcal C(x,\,m) = \{z \in \mathbb{R}^d:\, \mathrm{dist}(z, \, L_x) \leq m\}\, , 
\end{equation}
then one expects $\mathcal C(x, \|x\|^\xi)$ is the ``right scale'' to contain $\Gamma(0,x)$. To be more precise: one expects
\begin{align}
\label{eq:xicylfirst}
\lim_{\|x\| \rightarrow \infty}\mathbb{P}\left(\Gamma(0,x) \subseteq \mathcal C(x, \|x\|^{\xi + \varepsilon}) \right) &= 1\ ,\\
\label{eq:xicylsecond}
\lim_{\|x\| \rightarrow \infty}\mathbb{P}\left(\Gamma(0,x) \subseteq \mathcal C(x, \|x\|^{\xi - \varepsilon}) \right) &= 0\ ,
\end{align}
for any $\varepsilon > 0$.

Note that \eqref{eq:xicylfirst} and \eqref{eq:xicylsecond} carry an implicit isotropy; if these equations are true, then the order of geodesic wandering must be the same in all directions. More should be true: if we replace $\Gamma(0,x)$ by, for instance, a geodesic to a hyperplane at distance $n$, then this geodesic should asymptotically be contained in cylinders of radius $n^{\xi + \varepsilon}$ (and not in cylinders of radius $n^{\xi - \varepsilon}$). Lastly, ``near-geodesics''should have the same behavior: if $\Gamma_M(x, y)$ is the set of all paths $\eta$ connecting $x$ to $y$ satisfying $T(\eta) \leq T(x, y) + M$, then $\Gamma_M$ should satisfy relations like \eqref{eq:xicylfirst} and \eqref{eq:xicylsecond}.

The article \cite{NewmanPiza}, in conjunction with its studies of $\chi$, provides a technique for bounding $\xi$ in terms of $\chi$ (a more detailed discussion of the relation between the exponents can be found in the next subsection). For a unit vector $x$, let 
$$
\xi_{x} = \inf \bigg\{ \alpha >0 : \liminf_n \Pro\big(\Gamma(0,nx) \subset \mathcal C(x,n^{\alpha})\big)> 0 \bigg \}.
$$  
We also define the fluctuation exponent
\begin{equation*}
\chi' = \inf \bigg \{\kappa: \, (t - t^\kappa) \mathcal B_\nu \subseteq B(t) \subseteq (t + t^\kappa) \mathcal B_\nu \text{ for large } t, \, \text{a.s.}\bigg \}
\end{equation*}
The following inequality is a version of one of the inequalities of the scaling relation $\chi = 2\xi-1$, discussed in Section~\ref{sec:scalingrelation} below. In their paper, Newman-Piza explain that their argument is essentially a rigorous version of the argument of Krug-Spohn \cite{KS} for this scaling relation.

\begin{theorem}[Theorem 6 \cite{NewmanPiza}]\label{thm:NPScaling} Assume $d \geq 2$ with $\mathbb{E} \tau_e^2 < \infty$ and $F(0)~<~p_c$. Assume that $x$ is a direction of curvature for the limit shape  (see Section \ref{sec:curvaturesection}). Then,
$$ \xi_{x} \leq \frac{1+\chi'}{2}.$$
\end{theorem}

A sketch of the argument that leads to Theorem \ref{thm:NPScaling} is given below, after Remark~\ref{rem: taco_head}. In the same paper, under the assumptions of Theorem \ref{thm:NP}, the authors also obtained the following bound, which they refer to as the Wehr-Aizenman bound, after the latter's work on phase transitions in disordered systems \cite{AW}:
\begin{equation}\label{eq:WehrAiz}
 \chi_{x} \geq \frac{1-(d-1)\xi_{x}}{2},
 \end{equation}
where $$\chi_{x} = \sup \bigg \{ \alpha>0 : \text{ for some } C>0, \; \mathrm{Var}\; T(0,nx) \geq Cn^{2\alpha} \text{ for all } n \bigg \}. $$

In dimension $2$, Theorem \ref{thm:NPScaling} can be strengthened when combined with the previous bound $\chi' \leq 1/2$, derived by Kesten \cite{Kesten} and Alexander \cite{Alexander}. One still assumes that $x$ is a direction of curvature of the limit shape. The result is the following.

\begin{theorem}[Theorem 7 \cite{NewmanPiza}]\label{thm:exponent_bounds} Assume $F(0)<p_c$, and that $\mathbb E \exp (\beta \tau_e)~<~\infty$ for some $\beta>0$. Then, if $x$ is a direction of curvature for the limit shape $\mathcal B_{\nu}$,
$$ \xi_x \leq 3/4, \quad \chi_x \geq 1/8,$$
the latter bound holding when $d=2$.
\end{theorem} 
\begin{remark} \label{rem: taco_head} Note that, for any limit shape $\mathcal B_{\nu}$, there is at least one direction of curvature. Indeed, for any nonempty bounded subset $B$ of $\mathbb R^d$, there is at least one point $z$ in $\partial B$ such that, for some closed ball $D$ with positive radius, $z \in \partial D$ and $\bar B \subset D$. If $B$ is not just the origin, just take a ball of radius $r$ large enough that contains $B$ and decrease $r$ to the first time it intersects $\partial B$.  
This  argument appeared in \cite{NewmanPiza}.
\end{remark}
Without the curvature assumption on the limit shape the best current bounds are still $0 \leq \xi \leq 1$, the upper bound following directly from the shape theorem.

We now give a sketch of the argument for Theorem~\ref{thm:NPScaling}. Rather than prove this inequality, we show a more intuitive Newman-Piza type argument under stronger assumptions, like those made in Chatterjee's work \cite{Sutav}. The argument also explains the inequality $\xi_x \leq 3/4$ from Theorem~\ref{thm:exponent_bounds}. Suppose that $\chi$ is any number for which there is $c>0$ such that
\begin{equation}\label{eq: exponential_assumption}
\mathbb{P}(|T(0,z) - \mu(z)| \geq \lambda \|z\|^{\chi}) \leq e^{-c\lambda^c}  \text{ for all } z \in \mathbb{Z}^d \text{ and } \lambda \geq 0.
\end{equation}
(In fact, it is only necessary here to make this assumption for deviations of $T(0,z)$ about its mean, not $\mu(z)$ because, as pointed out in \cite{Sutav}, one can run Alexander's argument under exponential bounds for $T(0,z) - \mathbb{E}T(0,z)$ to obtain upper bounds for non-random fluctuations on scale $\|z\|^{\chi}$, with logarithmic corrections.) We aim now to show that if $x$ is a direction of curvature, then
\begin{equation}\label{eq: easier_upper_bound}
\xi_x \leq \frac{1+\chi}{2}.
\end{equation}
Since the above exponential decay is known to hold for $\chi = 1/2+\delta$ for any $\delta>0$ (see Theorem~\ref{thm: subdiffusive}), one obtains $\xi_x \leq 3/4$.
\begin{proof}[Sketch of proof of \eqref{eq: easier_upper_bound}]
The version of the Newman-Piza argument we present here is from Chatterjee \cite{Sutav}. For simplicity, take $x=e_1$. Let $\epsilon>0$, $\hat \xi = \frac{1+\chi}{2} + \epsilon$, and note that since $\chi$ can be chosen arbitrarily close to $1/2$, for $\epsilon$ small one has $\hat \xi < 1$. Let $L$ be the $e_1$ axis, and for $n \geq 1$, let $L'$ the segment connecting $0$ and $ne_1$. Now define
\[
V = \{v \in \mathbb{Z}^d : \exists ~w \in L' \text{ such that } \|w-v\| \in [n^{\hat \xi},2n^{\hat \xi}]\},
\]
Then we must show that the probability that a geodesic from 0 to $ne_1$ enters $V$ goes to 0 with $n$. This will show that $\xi_x \leq \hat \xi$ and, taking $\epsilon \to 0$, we obtain the result.

By a union bound, the probability of this event is bounded by
\[
\sum_{v \in V} \mathbb{P}(v \text{ in a geodesic from } 0 \text{ to } ne_1) = \sum_{v \in V} \mathbb{P}(T(0,ne_1) = T(0,v) + T(v,ne_1)).
\]
The goal now is to show that the sum on the right is itself summable in $n$. We claim there is $c>0$ such that for all $n$ and all $v \in V$,
\begin{equation}\label{eq: to_show_burrito}
\mathbb{P}(T(0,ne_1) = T(0,v) + T(v,ne_1)) \leq (1/c)\exp\left( -cn^c \right).
\end{equation}
After summing over $v$ and then over $n$, we complete the proof.

The main step is to show that, due to the curvature assumption, there is $C$ such that for all $v \in V$,
\begin{equation}\label{eq: sutav_g_bound}
\mu(v) + \mu(ne_1-v) - \mu(ne_1) \geq Cn^{2\hat \xi - 1} = Cn^{\chi + 2\epsilon}.
\end{equation}
Note that this inequality implies that if we replace $T$ by $\mu$ in the event $\{T(0,v) + T(v,ne_1) = T(ne_1)\}$, then the event cannot occur and, furthermore, gives us large lower bound on the difference of both sides. Assuming this inequality for the moment, on the event $\{T(0,v) + T(v,ne_1) = T(ne_1)\}$, one has
\[
0 \geq [T(0,v) - \mu(v)] + [T(v,ne_1)-\mu(ne_1-v)] - [T(ne_1)-\mu(ne_1)] + Cn^{\chi + 2\epsilon},
\]
and so at least one of the first three terms on the right has absolute value at least $(C/3)n^{\chi+2\epsilon}$. Applying \eqref{eq: exponential_assumption} gives \eqref{eq: to_show_burrito} and would complete the proof.

To show \eqref{eq: sutav_g_bound} is actually quite simple in the case that $v = (n/2)e_1 + n^{\hat \xi}e_2$: for illustration, in this case, we use curvature and symmetry for
\begin{align*}
\mu(v) + \mu(ne_1-v) - \mu(ne_1) &= [\mu(v) - \mu(n/2 e_1)] + [\mu(ne_1-v)- \mu(n/2 e_1)] \\
& = 2[\mu(v) - \mu(n/2 e_1)] \\
&= n \left[ \mu\left( e_1 + \frac{2n^{\hat \xi}}{n}e_2\right) - \mu(e_1) \right] \\
&\geq Cn \left( \frac{n^{\hat \xi}}{n} \right)^2 \\
&= C n^{\chi + 2\epsilon}.
\end{align*}
To modify this argument for other $v$'s, we split into cases. Let $H$ be the hyperplane taken in the definition of curvature (which in the case $x=e_1$ is just orthogonal to $e_1$ by symmetry), and set $w$ to be the projection of $v$ onto $L$ along $H$. Note that convexity of the limit shape implies that $\mu(v) \geq \mu(w)$ and $\mu(ne_1-v) \geq \mu(ne_1-w)$. Furthermore, at least one of the two vectors $v$ or $ne_1-v$ has Euclidean norm at least $n/2$. 
\begin{enumerate}
\item Case 1: $w \in L'$. In this case, we may assume that $v$ is the vector with norm at least $n/2$, and use curvature as above:
\begin{align*}
\mu(v) - \mu(w) = \|w\|\left( \mu\left( \frac{w}{\|w\|} + \frac{v-w}{\|w\|}\right) - \mu\left( \frac{w}{\|w\|}\right) \right) &\geq C \|w\| \left\| \frac{v-w}{\|w\|} \right\|^2 \\
&\geq Cn^{2\hat \xi-1}.
\end{align*}
Therefore, since $w \in L'$, one has $\mu(ne_1-w)+\mu(w) = \mu(ne_1)$, and so
\begin{align*}
\mu(v) + \mu(ne_1-v) -\mu(ne_1) &= [\mu(v)-\mu(w)] + [\mu(ne_1-v)-\mu(ne_1-w)] \\
&\geq \mu(v)-\mu(w) \geq Cn^{2\hat \xi-1}.
\end{align*}
\item Case 2: $w \in L\setminus L'$, on the side closer to $ne_1$. In this case we still have as above
\[
\mu(v) - \mu(w) \geq C\|w\|\left\|\frac{v-w}{\|w\|} \right\|^2.
\]
Now
\[
n^{2\hat \xi} \leq \|v-ne_1\|^2 = (\|w\|-n)^2 + \|v-w\|^2 = (\|w\|-n)^2 + \|w\|^2 \left\|\frac{v-w}{\|w\|} \right\|^2,
\]
and $n \leq \|w\| \leq 3n$. Therefore either $\|w\|^2\left\|\frac{v-w}{\|w\|}\right\|^2 > n^{2\hat \xi}/2$, or $\|w - ne_1\| = \|w\|-n \geq n^{\hat \xi}/\sqrt{2}$. In the first case,
\begin{align*}
\mu(v)+\mu(ne_1-v)-\mu(ne_1) \geq \mu(v)-\mu(w) &\geq C\|w\|^2 \left\|\frac{v-w}{\|w\|}\right\|^2 \|w\|^{-1} \\
&\geq Cn^{2\hat \xi-1},
\end{align*}
and in the second case, since $\hat \xi < 1$,
\[
\mu(v)+\mu(ne_1-v)-\mu(ne_1) \geq \mu(w) - \mu(ne_1) \geq C\|w-ne_1\| \geq Cn^{\hat\xi} 
\]
is still at least $Cn^{2\hat \xi-1}$.
\item Case 3: $w \in L\setminus L'$, on the side closer to $0$. This case is similar to the previous one.
\end{enumerate}
These three cases together prove \eqref{eq: sutav_g_bound} and complete the proof.
\end{proof}

\subsubsection{Licea-Newman-Piza lower bound on $\xi$}
We turn to the question of lower bounds for the wandering exponent $\xi$, taken up in \cite{LNP}. The lower bounds proved in \cite{LNP} all depend on the particular definition of $\xi$ given.
We note a basic guideline for the possible lower bounds expected. As discussed earlier, one expects $\xi = 2/3$ when $d = 2$, and, in view of the scaling relation $\chi = 2\xi -1$ discussed below, one expects $\xi \geq 1/2$ for all $d$.

We are now ready to present the results of \cite{LNP}. They consider four different definitions of $\xi$ (in their words, proving these definitions are equivalent is  ``one of the open foundational problems of the subject''). Their first definition, $\xi^{(0)}$, is the only point-to-point exponent considered:
\[\xi^{(0)} := \sup\bigg \{\alpha \geq 0: \, \limsup_{\|x\| \rightarrow \infty}\, \mathbb{P}\left(\Gamma(0,x) \subseteq \mathcal C(x, \|x\|^\alpha)\right) < 1\bigg \}\ , \]
where $\mathcal C$ is as in \eqref{eq:cyldef}; compare \eqref{eq:xicylsecond}. The first result of \cite{LNP} is
\begin{theorem}\label{theorem:LNP0}
Let $d \geq 2$. Then
$\xi^{(0)} \geq 1/(d+1)$ .
\end{theorem}
The bound in Theorem \ref{theorem:LNP0} is extended to $\xi^{(1)}$, the second definition of $\xi$ made in \cite{LNP}; we omit it here for brevity. 

All the remaining bounds in \cite{LNP} are for point-to-hyperplane versions of $\xi$, so we will need notation for hyperplanes. For $\hat \theta$ a unit vector of $\mathbb{R}^d$, let
\[\Lambda(\hat \theta, \, L) := \{y \in \mathbb{Z}^d: \, y \cdot \hat \theta < L\}  \]
be the corresponding family of half-spaces. Let $\partial \Lambda (\hat \theta, L)$ be the set of $y \in \mathbb{Z}^d \setminus \Lambda(\hat \theta, \, L)$ adjacent to some vertex of $\Lambda(\hat \theta, \, L)$; the $\partial \Lambda$ are the families of hyperplanes to which we will consider geodesics.

For the next definition of $\xi$, let $\Gamma_M(x, A)$ be the set of ``near-geodesics'' from $x$ to a set $A$ as above.
Let $\mathrm{end}[\Gamma_M(x, A)]$ be the set of endpoints in $A$ of these near-geodesics.
\begin{align*}
\xi^{(2)}(M) = \sup \,\bigg \{&\alpha \geq 0: \, \exists  (\hat \theta_n)_n , \,(J_n)_n \text{ with } J_n \rightarrow \infty \text{ such that } \\
&\text{ there is no deterministic } A_n \text{ with } \mathrm{diam}(A_n) \leq J_n^\alpha\\
&\text{such that } \mathbb{P}\left(\mathrm{end}[\Gamma_M(0, \partial \Lambda(\hat \theta_n, J_n))] \subseteq A_n \right) \rightarrow 1  \bigg \}\ .
\end{align*}
Note that in some respects, this definition is considerably weaker than the definition of $\xi^{(0)}$: it concerns only near-geodesics, and it concedes a great deal of uniformity over direction and ``scale'' ($J_n$). On the other hand, it is stronger in the sense that we actually know where the wandering occurs (at the endpoints; in the defintion of $\xi^{(0)}$ it is not specified). The second theorem of \cite{LNP} is
\begin{theorem}
\label{theorem:LNP2}
Let $M > \inf \mathrm{supp}( \tau_e )$ and $d \geq 2$. Then $\xi^{(2)}(M) \geq 1/2$.
\end{theorem}

The last definition of $\xi$ considered has in some ways all the weaknesses of both $\xi^{(0)}$ and $\xi^{(2)}$:
\begin{align*}
\xi^{(3)}(M):= \sup \, \bigg \{&\alpha \geq 0: \, \exists  (\hat \theta_n)_n , \,(J_n)_n \text{ with } J_n \rightarrow \infty \text{ such that } \\
&\text{ there is no deterministic } x_n \text{ such that }\\
& \mathbb{P}\left(\Gamma_M(0, \partial \Lambda(\hat \theta_n, J_n)) \subseteq \mathcal{C}(x_n, \, J_n^\alpha) \right) \rightarrow 1 \bigg \}\ .
\end{align*}
\begin{theorem}\label{theorem:LNP3}
Let $d = 2$ and $M > \inf \mathrm{supp}(\tau_e)$. Then $\xi^{(3)}(M) \geq 3/5$.
\end{theorem}

We will attempt to justify the above results without providing complete proofs. In some cases, unessential details or complications will be elided for presentation; for complete arguments, see \cite{LNP}.

\begin{proof}[Proof of Theorem \ref{theorem:LNP0}]
The idea is an extension of the variance estimate of Newman-Piza \cite{NewmanPiza} discussed in Section \ref{sec:loweboundNP}. This idea was later used in W\"uthrich \cite{Wuthrich}, Johansson \cite{Joha1}, Auffinger-Damron \cite{AD14}. Let $\alpha > \xi^{(0)}$ and choose some sequence of vertices $x_n$ with $\|x_n\| \rightarrow \infty$ such that
\begin{equation}
\label{eq:LNPtrapped}
\mathbb{P}(\Gamma(0, x_n) \subseteq \mathcal C(x_n, \, \|x_n\|^\alpha)) \rightarrow 1\ .
\end{equation}
We will need to define a shifted version of $x_n$; let $x_n^{\perp}$ be a unit vector perpendicular to $x_n / \|x_n\|$ and define $0'$ and $x_n'$ to be the closest vertices to $0 + C \|x_n\|^\alpha x_n^\perp$ and $x_n + C \|x_n\|^\alpha x_n^\perp$. where $C$ is some large constant independent of $n$.

Define $\Delta T := T(0, x_n) - T(0', x_n')$; the proof is by analysis of $\Delta T$. Note that $\Delta T$ cannot be too large in magnitude. Indeed, there are paths $\Gamma_1$ and $\Gamma_2$ connecting $0$ to $0'$ and $x_n$ to $x_n'$ respectively, such that $\#\Gamma_i \leq C \|x_n\|^\alpha$. Therefore,
\begin{align}
\label{eq:LNPapriori}
|\Delta T| \leq T(\Gamma_1) + T(\Gamma_2) \lessapprox C \|x_n\|^\alpha 
\end{align} 
where the precise strength of the bound is a function of the tail behavior of $\tau_e$ (if $\tau_e$ is bounded, then the left-hand side is really $\leq C \|x_n\|^\alpha$).

On the other hand, we can get a lower bound on $\mathrm{Var}\, \Delta T$ using the machinery of Newman-Piza. Defining as before the sigma-algebra $\Sigma_j$ generated by the first $j$ edge weights in some enumeration $(e_j)$ of the edges, we have
\begin{align}
\label{eq:LNPdeltavar}
\mathrm{Var}\, \Delta T = \sum_{j=1}^\infty \mathbb{E}
\left[\mathbb{E}\left[\Delta T \mid \Sigma_j\right] - \mathbb{E}\left[\Delta T \mid \Sigma_{j-1}\right]  \right]^2\ .
\end{align}
As before, one can provide a lower-bound for \eqref{eq:LNPdeltavar} in terms of influences of individual edges. There is an extra complication: when decreasing an edge weight $\tau_{e_i}$, we are only guaranteed to decrease $\Delta T$ if $e_i$ simultaneously satisfies 
\begin{equation}
\label{eq:LNPgeoinf}
e_i \in \Gamma(0, x_n) \text{ and } e_i \notin \Gamma(0', x_n')\ . \end{equation}
Letting $\overline F_i$ be the event described in \eqref{eq:LNPgeoinf} (compare to \eqref{eq:eventsinNP}), we have similarly to before
\begin{equation}
\label{eq:LNPfbar}
\mathrm{Var}(\Delta T) \geq C \sum_{j=1}^\infty \mathbb{P}(\overline F_j)^2\ .
\end{equation}

For a further lower bound, we can restrict the sum in \eqref{eq:LNPfbar} to $j$ such that $e_j \in \mathcal C(x_n, \|x_n\|^\alpha)$. For this restricted set $\mathfrak{J}$ of indices,
\begin{align*}
\sum_{j \in \mathfrak{J}} \mathbb{P}(\overline F_j) &\geq \mathbb{E}\#\left( \Gamma(0, x_n) \cap \mathcal C(x_n, \|x_n\|^\alpha)\right) - \mathbb{E}\#\left(\Gamma(0', x_n') \cap \mathcal C(x_n, \|x_n\|^\alpha)\right)\\
& \geq c \|x_n\|_1
\end{align*}
for some $c > 0$ and all $n$ large (the last inequality holds using \eqref{eq:LNPtrapped}). As before, using the Cauchy-Schwarz inequality on \eqref{eq:LNPfbar}, we see
\begin{align}\label{eq:sadequation}
\mathrm{Var}\, \Delta T \geq \sum_{j \in \mathfrak{J}} \mathbb{P}(\overline F_j)^2  &\geq \frac{1}{\# \mathcal C(x_n, \|x_n\|^\alpha)}\left( \sum_{j \in \mathfrak{J}} \mathbb{P}(F_j)\right)^2\\
&\geq C \|x_n\|^{1 - \alpha(d-1)}\ \nonumber .
\end{align}
Comparing this with \eqref{eq:LNPapriori} (which gives a variance upper bound of $\sim \|x\|^{2 \alpha}$) completes the proof.
\end{proof}
\begin{proof}[Proof of Theorem \ref{theorem:LNP2}]
The proof of this theorem relies on a geometrical construction. Rather than trying to exactly give the parameters of this construction, we will outline the main idea; we leave it to the reader to fill in the details (and in particular to understand why $\xi^{(2)}$ is defined as it is). We first outline the argument in case $d = 2$, then explain the necessary changes for $d > 2$.

Assume $\xi^{(2)} < 1/2$, and choose $1/2 > \alpha > \xi^{(2)} + \varepsilon$. We choose a finite set of unit vectors $\{\hat \theta_i\}_{i=0}^N$ and construct a polygon $P_L$ (having the symmetries of $\mathbb{Z}^2$) whose sides are segments of lines perpendicular to $\hat \theta_i$:
\[P_L\, =\, \bigcap_i \Lambda(\hat \theta_i, L_i)\ , \text{ where } L_i \sim L\ . \]
Let $S_i$ be the side of $P_L$ corresponding to $\hat \theta_i$--that is, the side of $P_L$ which is a segment of the boundary of $\Lambda(\hat \theta_i, L_i)$. We will choose $\hat \theta_i$ and $L_i$ such that the lengths $|S_i| \sim L^{\alpha}$.
We will also ensure our choice of $\hat \theta_i$ guarantees geodesics to the lines defining $P_L$ are localized on sides of $P_L$.

\begin{figure}
\scalebox{1.0}{ 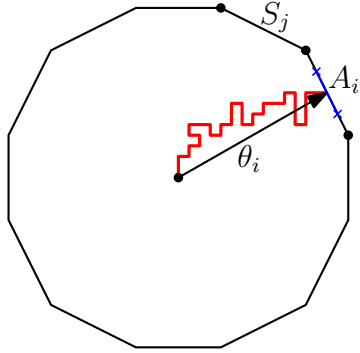}
 \centering
 \def \svgwidth{3000pt}
 \caption{The construction of the polygon $P_{L}$. The segments $A_{i}$ (in blue) are close to the midpoints of the sides $S_{i}$ and, with high probability, contain the end point of the geodesic $\Gamma(0,S_{i})$ (in red).}
 \label{fig:lnp}
\end{figure}

Let $A_i$ be a deterministic subset of $\partial \Lambda(\hat \theta_i, L)$ of diameter smaller than $L^{\xi^{(2)} + \varepsilon}$ in which $\mathrm{end}\, [\Gamma(0, \partial \Lambda(\hat \theta_i, L))]$ is localized with high probability; such a subset exists by the definition of $\xi^{(2)}$. Our main restriction in choosing $\hat \theta_i$ is that $A_i$ should be close to the midpoint of $S_i$, and in particular be at least distance $L^{\xi^{(2)} + \varepsilon}$ from $S_j, \, j \neq i$. This is done inductively.

To start the induction: we can choose $\hat \theta_0 = e_1$ and $L_0 = L$, with $|S_0| = L^{\alpha}$. Then by the symmetries of the lattice, we can choose $A_0$ symmetric about the $e_1$ axis (and in particular localized near the midpoint of $S_0$). Our last vector $\hat \theta_N$ will be the ``45 degree line'' $(\sqrt{2})^{-1}(e_1 + e_2)$; again, by the symmetries of the lattice, $A_N$ will be chosen symmetric about $\hat \theta_N$.
The remaining vectors will only be chosen to fill out the polygon in the arc of angles $[0, \pi / 4]$; once this is done, we fill the rest by symmetry. 

Assuming $\hat \theta_0, \, \ldots, \hat \theta_{j-1}$ and the corresponding $\{L_i\}$ have been chosen correctly, let $z_{j-1}$ be the counterclockwise endpoint of $S_{j-1}$. If $z_{j-1}$ is within distance $\leq  CL^{\alpha}$ from a point on the line $\{a\hat \theta_N : a \in \mathbb{R}\}$, then choose $\hat \theta_j = \hat \theta_N$ and $L_j$ is chosen so that $z_{j-1} \in \partial \Lambda(\hat \theta_{j}, \, L_j)$. Otherwise, we will choose $|S_j| = L^{\alpha}$, and we need to find a direction $\hat \theta_j$   with $\mathrm{arg}\, \hat \theta_j \in (\mathrm{arg}\,\hat \theta_{j-1}, \pi/4)$ such that $A_j$ is appropriately localized on $S_j$ (once we choose $\hat \theta_j$, then $L_j$ is fixed by the necessity of connecting to $S_{j-1}$).

The fact that such a choice is possible is somewhat subtle. If we were to try to choose $\hat \theta_j$ extremely close to $\hat \theta_{j-1}$, then we would fail. Indeed, $A_{j-1}$ is close to the midpoint of $S_{j-1}$, and as $\hat \theta_j \rightarrow \hat \theta_{j-1}$, the side $S_j$ is essentially a subsegment of the boundary of $\Lambda(\hat \theta_{j-1}, L_{j-1})$ and so $A_j$ should be identical to $A_{j-1}$, which lies order $L^\alpha$ distance {\it below} the midpoint of $S_j$. Similarly, taking $\hat \theta_j \rightarrow \hat \theta_N$ would give an $A_j$ lying too far {\it above} the midpoint. By a continuity argument, we can therefore find a $\hat \theta_j$ that meets our requirements.

After this somewhat lengthy construction, we can quickly conclude. Note that, because $|S_i| \leq C L^{\alpha}$, we have $$\arg(\hat \theta_i) - \arg(\hat \theta_{i-1}) \leq C L^{\alpha - 1}.$$ In particular, $\partial \Lambda(\hat \theta_{i-1}, L_{i-1})$ comes within distance $L^{2 \alpha - 1} = o(1)$ of the midpoint of $S_i$. But this means it takes only $O(1)$ edges to extend a time-minimizing path for $T(0, \partial \Lambda(\hat \theta_i, L_i))$ from $S_i$ to $\partial \Lambda(\hat \theta_{i-1}, L_{i-1})$--since this extension costs very little in the way of passage time, there is a ``near-geodesic" to $\partial \Lambda(\hat \theta_{i-1}, L_{i-1})$ lying on $S_i$, contradicting the placement of $A_{i-1}$. We conclude that $\alpha$ cannot be chosen $<1/2$; that is, $\xi^{(2)} \geq 1/2$. This argument assumes that $T(0,S_i) < T(0,S_{i-1})$, which, without loss in generality, holds with positive probability.

The case $d \geq 3$ is very similar. Now instead of using polygons $P_L$, we consider polygonal ``barrels'' of the form $P_L \times \mathbb{R}^{d-2}$, where $P_L$ is constructed similarly to before. The arguments above may be adapted to these barrels without substantially different ideas.

\end{proof}

\begin{proof}[Proof  of Theorem \ref{theorem:LNP3}]
The proof is in some ways heavily motivated by the proofs of both of the preceding theorems. 
Consider $\alpha > \xi^{(3)}$.
We construct a polygon $P_L$ with sides $\{S_j\}$ similarly to the arguments in the proof of Theorem~\ref{theorem:LNP2}. For fixed $j$, we consider $$\Delta T = T(0,\partial \Lambda(\hat \theta_{j+1}, L_{j+1})) - T(0, \partial \Lambda(\hat \theta_{j}, L_{j})).$$

As before, $\partial \Lambda(\hat \theta_j, L_j)$ comes within distance $\sim L^{2 \alpha - 1}$ of every point on $S_{j+1}$, so geodesics to $\partial \Lambda(\hat \theta_j, L_j)$ may be extended to $\partial \Lambda(\hat \theta_{j+1}, L_{j+1})$ at an increase in passage time $\lessapprox L^{2 \alpha - 1}$.
This (combined with a symmetric argument in the other direction) gives an a priori bound
\begin{equation}
\label{eq:lnpap2}
\mathrm{Var}(\Delta T) \lessapprox L^{4 \alpha - 2}\ .
\end{equation}

As in the proof of Theorem~\ref{theorem:LNP0}, we find an $\alpha$-dependent lower bound which contradicts this when $\alpha$ is too small. The geodesics to both hyperplanes are confined with high probability in cylinders with radius $L^{\alpha}$. These cylinders are largely disjoint for $L$ large, and so there are typically order $L$ edges which are on $\Gamma(0, \partial \Lambda(\hat \theta_j, L_j))$ but not on $\Gamma(0, \partial \Lambda(\hat \theta_{j+1}, L_{j+1}))$. 

With this in hand, we lower-bound $\mathrm{Var}(\Delta T)$ with a martingale expansion as in Theorem~\ref{theorem:LNP0}; we ultimately arrive at an analogue of \eqref{eq:sadequation}. As before, this yields
\[\mathrm{Var}(\Delta T) \geq  C L^{1 - \alpha}\ . \]
Comparing this bound with \eqref{eq:lnpap2} completes the contradiction.
\end{proof}

\subsection{The scaling relation $\chi = 2 \xi - 1$}\label{sec:scalingrelation}

Consider a geodesic from the origin to a point $x$ with passage time $T(0,x)$. One of the central questions in first-passage percolation (and in related models) is to establish the following statement.  There exists an intrinsic relation between the magnitude of deviation of $T(0,x)$ from its mean and the magnitude of deviation of the geodesic $T(0,x)$ from a straight line joining $0$ and $x$. This relation is \textit{universal}; that is, it is independent of the dimension $d$ and of the law of the weights (as long they satisfy certain moment assumptions).  

As discussed in Section \ref{sec:variancebounds} (see \eqref{eq:variancescaling}), the fluctuations of the passage time $T(0,x)$ about $\E T(0,x)$ should be of order $|x|^\chi$, where $\chi$ is the \textit{fluctuation exponent}. Analogously, the \textit{transversal or wandering exponent} $\xi$ studied in Section \ref{sec:exponents} measures the maximal Euclidean distance of a  geodesic from $0$ to $x$ from the straight line that joins $0$ to $x$. The intrinsic relation described above should be given as
\begin{equation}\label{ScalingRelation}
\chi = 2 \xi -1.
\end{equation}
As previously discussed, the existence and the `correct' definition of these exponents is part of the problem. Before we discuss the mathematical history behind \eqref{ScalingRelation} and current progress \cite{AD14, Sutav}, let us stress that, as of today, there is not a single distribution of passage times, where we know that \eqref{ScalingRelation} holds. Therefore proving equation \eqref{ScalingRelation} is still an open question. (As we will see in the pargraphs below, one will probably need to solve Question \ref{q:strictconvexity} first.)

\begin{question}
Find a distribution of passage times where \eqref{ScalingRelation} holds.
\end{question}
  
As shown in Theorem~\ref{thm:NPScaling}, a version of the inequality $ \chi \geq 2 \xi -1 $ was proved and understood in the 1995 work of Newman-Piza \cite{NewmanPiza}. A proof of the other inequality appeared first in a model of a Brownian particle in a Poissonian potential \cite{Wuthrich}; see also \cite{Joha1}.  
In FPP, Chatterjee \cite{Sutav} proposed a stronger definition of the exponents that allows a  proof of \eqref{ScalingRelation}. Unlike the definitions of $\xi$ and $\chi$ given in the previous sections, Chatterjee's exponents are not known to exist for any distribution of passage times. The proof in \cite{Sutav} relies on a construction similar to that in \cite{CD}. One first breaks a geodesic into smaller segments and then uses an approximation scheme to compare the passage time to a sum of nearly i.i.d. random variables. The proof is then a trade-off between minimizing the error and maximizing the variance of the passage time. Assuming that the distribution is `nearly Gamma' (see Section \ref{sec:logimprovementsUB}), the optimization can be achieved by choosing different parameters in the approximation.

The scaling relation should be valid only under the unproven assumption of uniform positive curvature of the limit shape. This is one of the main reasons for the introduction of a  strong definition of exponents. It is reasonable to believe that given a different curvature exponent $\kappa$ as defined in \eqref{def:curvatureexponent}, one would have the alternate scaling relation $\chi = \kappa \xi - (\kappa - 1)$; see \cite[Section~3]{AD14} for more details. The inequality $\leq$ comes from \eqref{eq:approxexpo} below, and the case $\kappa = 2$ reduces to $\chi = 2\xi - 1$.

We will now sketch how one could derive \eqref{ScalingRelation} assuming positive curvature. Here we follow the proof of \eqref{ScalingRelation} given in \cite{AD14}. It starts with a standard fact: if $X'$ is an independent copy of a random variable $X$ then 
$$ \text {Var } X = \frac{1}{2} \E (X-X')^2. $$
Thus, if we want to estimate $\chi$, it suffices to compare the difference of two independent copies of $T(0, ne_1)$ as 
\begin{equation}\label{eq:sushi}
n^{2\chi} \sim \E(T(0,ne_1) - T'(0,ne_1))^2.
\end{equation}

The exponent $\xi$ tells us exactly how to build such an independent copy $T'$. As in the proof of Theorem \ref{theorem:LNP0}, one just needs to consider passage times from starting points $0$ and $n^\xi e_2$, as the geodesics from there to $ne_1$ and $n^\xi e_2+ne_1$ will live in two disjoint cylinders. Therefore, one should expect that 

\begin{equation}\label{eq:dahi}
n^{2\chi} \lesssim \E (T'(0,ne_{1})-T(0,ne_{1}))^{2} \sim \E (T(n^\xi e_2, n^{\xi}e_{2}+ne_1)-T(0,ne_1))^{2}.
\end{equation}

Now, recall the definition of the curvature exponent given in Definition \ref{def:curvatureexponent}. If we are allowed to bound the passage time by its asymptotic value (as $T(0,x) \sim \mu(x)$), we would have

\begin{equation}\label{eq:approxexpo}
\begin{split}
T(n^\xi e_2,ne_1)-T(0,ne_1) &\leq  \mu(ne_1 - n^\xi e_2) -   \mu(ne_1) + \text{ error}\\&=  n\bigg( \mu (e_1 - n^{\xi-1} e_2) - \mu(e_1)\bigg) + \text{ error}\\
&\sim n n^{ \kappa(e_1) (\xi-1)} + \text{ error} \\
&= n^{2\xi-1} +\text{ error},
\end{split}
\end{equation}
if $\kappa(e_1)=2$. In the first equality above, we used the fact that $\mu$ is a norm. The above sequence of approximations tells us $2\xi-1$ comes from the assumption that the curvature exponent exists and it is equal to $2$.

Now we compare the left side of \eqref{eq:approxexpo} with the right side of \eqref{eq:dahi}. The difference in passage times in both equations is almost the same except that the ending points are not exactly equal. However, a crossing trick from \cite{AD14} and an argument from \cite{Sutav} using assumed exponential concentration on scale $n^\chi$ from the strong definition of $\chi$ allows us to take the same ending points, and thus we can replace \eqref{eq:approxexpo} in \eqref{eq:sushi} to get $\chi \leq 2\xi-1$. 

The above argument is made rigorous in \cite{AD14}. The assumptions on the exponents $\xi$ and $\chi$ are just to justify the approximations in \eqref{eq:approxexpo}.

\subsection{Infinite geodesics}
In this section we investigate infinite geodesics. We start with the following definitions. 

\begin{definition}We will say that an infinite self-avoiding nearest-neighbor path $\Gamma$ is an {\it infinite geodesic} (for a given edge-weight realization) if every finite subpath of $\Gamma$ is a finite geodesic.
\end{definition}
 Infinite geodesics come in two varieties:
 
 \begin{enumerate}
 \item Indexed by $\N$; that is, $\Gamma$'s edges in order are $e_1, e_2, \ldots.$ These are called {\it unigeodesics, singly infinite geodesics, geodesic rays} or (when it does not cause confusion) simply geodesics.
 \item Indexed by $\Z$. These are called {\it bigeodesics} or {\it geodesic lines}.
 \end{enumerate}
 
 Much work on geodesics in the model has focused on determining the number of geodesics of either type which exist, as well as their properties and relationships.
 
 \subsubsection{Existence of geodesic rays}
It is a simple argument to check that, given any $x \in \mathbb Z^d$ there exists (almost surely) at least one geodesic ray originating from $x$. The argument goes as follows. Suppose without loss of generality that $x$ is the origin $0$. Consider the sequence of geodesics $\Gamma_{n}=\Gamma(0,ne_{1})$. The first vertex of each $\Gamma_n$ is $0$. There are $2d$ possible choices for the second vertex of each $\Gamma_n$. Therefore, there must be some edge $e$ incident to the origin such that infinitely many $\Gamma_n$ have $e$ as their first edge. Repeating this argument on subsequences yields a singly infinite path which is seen to be a geodesic.
 
The natural question that follows is to determine if there exist multiple distinct geodesic rays with positive probability (here, distinct means sharing at most finitely many edges and vertices). A natural way to construct distinct geodesics would be to consider the subsequential limit obtained in the last construction by taking a different sequence of endpoints, such as $-ne_{1}$. Given the observations on the wandering exponent given above, it is reasonable to believe that one should be able to construct many infinite geodesics this way. However, this intuition is hard to make rigorous. 

In the two-dimensional case, with exponentially distributed edge weights, H\"aggstr\"om and Pemantle \cite{HP} showed that there are at least two distinct geodesic rays with positive probability. This result was extended to a wide range of first-passage distributions by Garet and Marchand \cite{GM} and Hoffman \cite{Hoffman1}. 

An important further advance was made in a paper of Hoffman \cite{Hoffman}. He demonstrated that it is in fact possible to derive the existence of more than two geodesic rays directly from properties of the limit shape. His theorem goes as follows. We write sides$(\mathcal B_\nu)$ for the number of sides of the limit shape $\mathcal B_\nu$ (that is, sides$(\mathcal B_\nu)$ is the number of extreme points of $\mathcal{B}_\nu$ and is finite if and only if $\mathcal B_\nu$ is a polygon). 

\begin{theorem}[Hoffman \cite{Hoffman}]\label{thm:Hoffman1} Assume $d=2$ and $\tau_e$ has continuous distribution with $\E \tau_e^{2+\alpha}~<~\infty$ for some $\alpha >0$. Define $G(x_1 , \ldots, x_k )$ to be the event that there exist distinct geodesic rays beginning at vertices $x_1,\ldots,x_k$. Then for any $k\leq$ sides$(\mathcal B_\nu)$ and any $\epsilon>0$, there exist $x_1,\ldots,x_k \in \Z^2$ such that
\begin{equation*}
\Pro (G(x_1,\ldots,x_k))>1-\epsilon.
\end{equation*}
\end{theorem}
This result also holds for a wide range of non-i.i.d. passage time distributions and was used by Hoffman to establish existence of $4$ distinct geodesic rays. In light of Hoffman's results, the following question could be seen as a weak version of the uniform curvature conjecture, since uniform curvature implies sides$(\mathcal B_{\nu})=\infty$. We will discuss this theorem (and give some idea of its proof) in another guise below as Theorem \ref{Hoof}.

\begin{question}
Show that for continuous distributions, there are infinitely many geodesic rays.
\end{question} 

It is expected that sides$(\mathcal B_{\nu}) = \infty$ for any distribution that has more than one point in its support. As we saw in Section \ref{sec:limitshape}, this fact is only known for distributions in $\mathcal{M}_p$ with $p \geq \vec{p}_c$.   

\begin{question}
Find a distribution that is not in $\mathcal M_{p}$ that has an infinite number of geodesic rays.
\end{question}
The last observation combined with Hoffman's result allows us to conclude: 

\begin{theorem}[Auffinger-Damron \cite{AD12}] Assume $d=2$. For any measure $\nu \in \mathcal M_p$, $p \geq \vec p_c$, almost surely, there exist infinitely many geodesic rays.
\end{theorem} 
It is important to note that when $p>\vec{p}_c$, for $\nu \in \mathcal M_p$, one can trivially find infinitely many geodesic rays by choosing them to be oriented up-right and contain only edges $e$ with $t_e=1$. In the above result, the geodesics constructed can be chosen to contain a positive fraction of edges $e$ with $t_e>1$. Each of these rays will be directed in a sector that is disjoint from the Durrett-Liggett flat edge of the limit shape.
  
   The argument that we provided at the beginning of this section for the existence of a geodesic ray leads to another important question.  If $\Gamma_n$ is a sequence of finite geodesics, we say that $\Gamma_n \to \Gamma$ for some path $\Gamma$ if for all $N>0$, the first $N$ steps of $\Gamma_n$ equal those of  $\Gamma$ for all $n$ large. 
\begin{question}
Is it true that the sequence of geodesics $\Gamma(0,ne_{1})$ converges to some $\Gamma$?
\end{question} 
By the definition of convergence, such $\Gamma$ is clearly a geodesic ray. The question above was solved if one replaces $\mathbb Z^2$ by the upper-half plane or any other infinite connected subgraph of $\mathbb Z^{2}$ with infinite connected complement in \cite{ADH1} (in the latter case, $ne_{1}$ is replaced by a sequence of points on the boundary of the domain). 
The following similar open question (a version of the ``BKS midpoint problem'' posed in \cite{BKS}) is related to the existence of geodesic lines, which is the topic of Subsection \ref{sec:bigeodesics}.
\begin{question}
Assume that the space is uniquely geodesic. Show that in $d=2$, there is no fixed vertex that is in infinitely many of the geodesics between $-ne_1$ and $ne_1$. 
\end{question} 

Both of the last two questions are solved if one assumes that the limit shape boundary is differentiable both in direction $e_1$, and at $\theta_1,\theta_2$, where the $\theta_i$'s are the endpoints of the sector of angles of contact of the unique tangent line to the limit shape in direction $e_1$. See Theorem~\ref{thm: new_DH} and Remark~\ref{rem: BKS_midpoint}. 
  
 \subsubsection{Directions and coalescence}
 
 In this section, we investigate questions of whether there exist geodesic rays that have an asymptotic direction, and if the geodesic in a given direction is unique. Much of the work described here was done in the '90s by Newman and co-authors; for a complete pedagogical account of these results we refer to the book of Newman \cite{Newmanbook}. (Also see Section~\ref{sec:Busemann} on Busemann functions, where geodesics are analyzed without a curvature assumption.)
 For $\theta \in \mathbb{R}^d$ with $\|\theta\|_2 = 1$ (identified with $[0,2\pi)$ when $d=2$), we will say that a self-avoiding singly infinite path $(v_1, e_1, v_2, \ldots)$ has direction $\theta$ if $\arg v_n := v_n/\|v_n\|_2 \to \theta$ as $n$ goes to infinity. Under the assumption of uniform positive curvature, the following theorem shows that infinite geodesics cannot wind around the origin; they must have a direction.
 
 \begin{theorem}[Newman \cite{Newman}]\label{theorem:Newman}
 For $d \geq 2$, assume that $F(0)<p_c$, there exists $\alpha>0$ such that $\E \exp(\alpha \tau_e)~<~\infty$ and that $\mathcal{B}_\nu$ is uniformly curved (see Definition \ref{def:unifcurved}). Then with probability one,
 \begin{enumerate}
\item Every geodesic ray has a direction.
\item For every $\theta$ with $\|\theta\|_2=1$, there exists a geodesic ray with direction $\theta$.
 \end{enumerate}
  \end{theorem}

Newman's theorem from \cite[Theorem~2.1]{Newman} is stated for continuous distributions, in which both items above are deduced not for geodesic rays, but for infinite branches in $\mathcal{T}(0)$, but the same proofs yield the above extension.

\begin{proof}
We will need to control geodesics from all points at once, and we will control them by ensuring that they exit cones along certain parts of their boundaries. For this, we need definitions of cone sectors: for $x \in \mathbb{Z}^d$ and $\epsilon>0$, define the set
\[
C_x = \{z \in \mathbb{Z}^d : \|z\| \in [\|x\|/2, 2\|x\|],~ \|\arg x - \arg z\| \leq \|x\|^{-1/4 + \epsilon}\},
\]
its ``forward boundary''
\[
\partial_f C_x = \{z \in \mathbb{Z}^d : \exists ~w \in C_x \text{ such that } \|z-w\| = 1 \text{ and } \|z\| > 2\|x\|\},
\]
and the rest of the boundary is $\partial'C_x = \{z \in \mathbb{Z}^d \setminus C_x: \exists~w \in C_x \text{ such that } \|z-w\| = 1\} \setminus \partial_fC_x$. Define the ``out tree'' of $x$ as
\[
\text{out}(x) = \{z \in \mathbb{Z}^d : T(0,z) = T(0,x) + T(x,z)\}
\]
(the set of vertices whose geodesics from 0 contain $x$) and the event
\[
G_x = \{\text{out}(x) \cap \partial'C_x \neq \emptyset\}.
\]
A union bound gives
\[
\mathbb{P}(G_x) \leq \sum_{z \in \partial' C_x} \mathbb{P}(T(0,z) = T(0,x) + T(x,z)).
\]
Nearly the same argument as in the Newman-Piza bound (summing a similar bound to that given in \eqref{eq: to_show_burrito}) shows that there is $c>0$ such that for all $x \in \mathbb{Z}^d$,
\[
\mathbb{P}(G_x) \leq (1/c) e^{-c\|x\|^c}.
\]
(Here it is important that the aperture of the cone $C_x$ be at least $\|x\|^{-1/4+\epsilon}$, so that the distance between $x$ and the side boundaries of $C_x$ be larger order than $\|x\|^{3/4}$. The exponent $3/4$ is the bound for the wandering exponent from Theorem~\ref{thm:exponent_bounds}.) 

By Borel-Cantelli, we can almost surely find a random $M$ such that for $\|x\|\geq M$, the event $G_x^c$ occurs. Now fix any such configuration; we claim that there is a $C>0$ such that the following geodesic-direction estimate holds. Let $\gamma$ be a geodesic from $0$ with vertices $0=x_0, x_1, x_2, \ldots, x_r$. So long as $n \in [0,r]$ is such that $\|x_n\| \geq M$, then
\begin{equation}\label{eq: argument_bound}
\|\arg x_r - \arg x_n\| \leq C\|x_n\|^{-1/4+\epsilon},
\end{equation}
where $C = \sum_{i=1}^\infty 2^{-i(1/4-\epsilon)}$. To prove this inequality, define a sequence of points $y_1, \ldots$ inductively as follows. Let $y_1 = x_n$ and for $i=2, \ldots$, define $y_i$ to be the first point on $\gamma$ after $y_{i-1}$ and after the first exit from $C_{y_{i-1}}$ (as long as there is one). Then there is $I$ such that $x_r \in C_{y_I}$, and we can bound
\begin{align*}
\|\arg x_r - \arg x_n\| &\leq \sum_{i=1}^{I-1} \|\arg y_i - \arg y_{i+1}\| + \|\arg y_I - \arg x_r\| \\
&\leq \sum_{i=1}^I \|y_i\|^{-1/4+\epsilon} \\
&\leq \sum_{i=1}^\infty (2^i \|y_1\|)^{-1/4+\epsilon} \\
&\leq C\|x_n\|^{-1/4+\epsilon}.
\end{align*}
In the second to last inequality, we have used that $G_{y_i}^c$ occurs, and thus the geodesic exits $C_{y_i}$ along the forward boundary. This proves \eqref{eq: argument_bound}.

Now we can argue the two points of Newman's theorem. For any passage time configuration as above, let $\gamma$ be an infinite geodesic. We may assume it starts from 0 and label its vertices $x_0, x_1, \ldots$. Pick $n$ such that $\|x_n\| \geq M$ and note now that by \eqref{eq: argument_bound}, one has $\| \arg x_n - \arg x_m\| \leq \|x_n\|^{-1/4+\epsilon}$ for all $m \geq n$. This means $(\arg x_n)$ is a Cauchy sequence and thus converges, proving that $\gamma$ has a direction.

For the second point, for our given passage time configuration, pick any direction $\theta$ and let $(y_n)$ be any sequence of vertices with $\arg y_n \to \theta$. Let $\gamma_n$ be a self-avoiding geodesic from 0 to $y_n$ and let $(\gamma_{n_k})$ be a subsequence that converges to some $\gamma$. We claim that $\gamma$, which is a self-avoiding infinite geodesic, has direction $\theta$. Indeed, if its vertices are $x_0, x_1, x_2, \ldots$, then for any $n$, one must have for all large $k$ that $\gamma_{n_k}$ has as its initial segment the portion of $\gamma$ from 0 to $x_n$. Thus if $\|x_n\| \geq M$, by \eqref{eq: argument_bound},
\[
\|\arg y_{n_k} - \arg x_n\| \leq \|x_n\|^{-1/4+\epsilon}.
\]
Taking $k \to \infty$, we obtain $\|\theta - \arg x_n\| \leq \|x_n\|^{-1/4+\epsilon}$, and this means $\arg x_n \to \theta$.
\end{proof}

Under the unproved assumption of uniform curvature, one could now ask how many geodesic rays with direction $\theta$ exist. Is it possible that there is more than one such ray, and does the answer change if we insist that they share the same initial vertex?

In this direction, the following two results are due to Licea and Newman \cite{LN96}. The first one roughly says that geodesic rays cannot bifurcate in the same deterministic direction. We note, however, that birfucations are expected in random directions, and indeed such a statement has been proved in a related last-passage model \cite{coupier}. See Remark~\ref{rem: multiple_geos} below.
\begin{theorem}[Licea-Newman \cite{LN96}, Newman \cite{Newman}] \label{theorem:LN} Let $d=2$ and assume that $F$ is continuous. Fix $x \in \mathbb Z^2$. Then there exists some set $D\subseteq [0, 2\pi)$ of full Lebesgue measure such that if $\theta \in D$, then there is zero probability that there exist distinct geodesic rays starting at $x$ with direction $\theta$.
\end{theorem}

The proof of Theorem \ref{theorem:LN} is surprisingly simple, and we will provide a sketch of the main idea below. Let us first mention that not much is known about the set $D$, although it is expected that $D=[0,2 \pi).$ Establishing the question below would be a step towards one of the main problems (absence of geodesic lines) discussed in Section \ref{sec:bigeodesics}.
\begin{question}\label{question:setD}
Show that $D = [0,2\pi).$ 
\end{question}
In \cite[Theorem 1.5]{Newmanbook}, it was shown by Zerner that $D$ has at most a countably infinite complement. But we do not even know that a particular angle - for instance, $0$ - belongs to $D$. 
\begin{question}
Show that $0\in D.$ 
\end{question}
These two questions have been recently solved by Damron-Hanson under additional assumptions on the limit shape boundary \cite{DHanson2}. Namely, the first is solved if one assumes that the limit shape boundary is differentiable everywhere. For the second, the condition is differentiability of the boundary in the direction $e_1$ and in directions $\theta_1$ and $\theta_2$, where the $\theta_i$'s are the endpoints of the sector of angles at which the unique tangent line to the limit shape in direction $e_1$ contacts the limit shape. See Theorem~\ref{thm: new_DH}.

\begin{proof}[Proof of Theorem \ref{theorem:LN}]  Suppose there are two different geodesic rays $r_1, r_2$ starting at some $x$ with the same direction $\theta$. As they are different, they will have to bifurcate at a vertex $u \in \mathbb Z^2$, one taking the directed edge $(u,v_{1})$, the other the edge $(u, v_{2})$, with $v_2 \neq v_1$. As the space is uniquely geodesic, $r_{1}$ and $r_{2}$ do not intersect after $u$. The construction that follows buids a geodesic between $r_{1}$ and $r_{2}$ by exploring the geodesic tree $\mathcal T(x)$ between these two rays. By planarity, any other geodesic ray that sits between these two geodesics must also have direction $\theta$. 

\begin{figure}[htbp]
\begin{center}
\def\svgwidth{12cm}
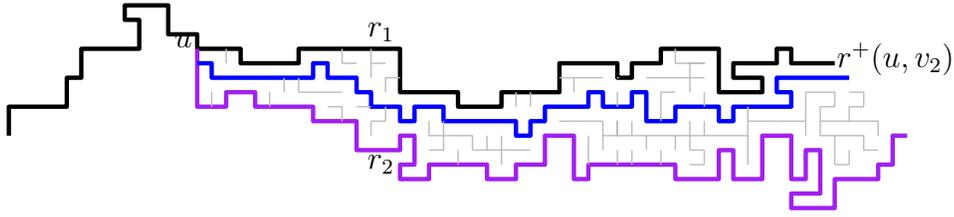
\caption{Construction of the path $r^{+}(u,v_{2}).$ Path $r_{1}$ depicted in black, $r_{2}$ in purple and $r^{+}(u,v_{2})$ in blue.}
\label{default}
\end{center}
\end{figure}

Assume without loss of generality that $r_1$ is asymptotically counter-clockwise to $r_2$. Now construct the geodesic ray $r^+(u,v_2)$ that starts with the edge $( u,v_2)$ and always chooses the most counter-clockwise edge possible in $\mathcal{T}(x)$ (in particular this edge must belong to at least one geodesic ray). Roughly speaking, the ray $r^{+}(u,v_{2})$ is the `closest' geodesic ray to $r_{1}$ that takes the edge $( u,v_{2})$ in the region delimited by $r_{1}$ and $r_{2}$. As $r_2$ is a geodesic ray, $r^+(u,v_2)$ sits between $r_2$ and $r_1$ and must have direction $\theta$. Through this construction, it is not difficult to see that to any direction $\theta$ with multiple geodesic rays from $x$ corresponds at least one edge $e$  such that $r^{+}(e)$ has direction $\theta$. Let $G(e,\theta)$ be the event that $r^{+}(e)$ has direction $\theta$.

We thus get the estimate
$$ \Pro\big(\text{there exist at least two geodesic rays from }x \text{ with direction } \theta\big) \leq \sum_{e} \Pro( G(e,\theta)).$$
Now, for each edge $e$, $r^+(e)$ cannot be a geodesic ray with direction $\theta$ for more than one $\theta$; thus Fubini's theorem implies
$$ \int \Pro(G(e,\theta)) d \theta = \int \int \indi_{G(e,\theta)}(\omega) d \theta d \Pro  = 0. $$
(One must check that $\indi_{G(e,\theta)}(\omega)$ is jointly measurable in $\theta, \omega$.) Combining the last two displays ends the proof of the theorem.
\end{proof}

\begin{remark}\label{rem: multiple_geos}
It is expected (and, as mentioned above, proved under the assumption of differentiability of $\partial \mathcal{B}_\nu$) that for any fixed $\theta$, one cannot find disjoint geodesic rays with direction $\theta$. However, there should exist (and in fact, it is provable if one assumes existence of geodesics in any deterministic direction) a random, dense, countably infinite set of $\theta$'s for which there are disjoint geodesic rays with direction $\theta$. Such a statement is proved in related exactly solvable last-passage percolation model by Coupier \cite{coupier}, along with the statement that almost surely, there exist no $\theta$ such that three geodesic rays from 0 have direction $\theta$. Although this last result does not yet have a counterpart in non-exactly solvable models, progress toward the first under limit shape assumptions was made in \cite{GRS}.
\end{remark}

The second theorem in this direction establishes that for all $\theta \in D$ any two geodesic rays with direction $\theta$ (starting from possibly different points) must coalesce. 

\begin{theorem}[Licea-Newman \cite{LN96}, Newman \cite{Newman}]\label{theorem:Newman2}
For $d=2$ and $\theta \in D$ with $F$ continuous, there is zero probability that there are disjoint geodesic rays in the direction $\theta$.
   \end{theorem}
   
The proof of the above theorem is two-dimensional. It is unclear if the result holds for $d>2$. 

\begin{question}
Decide whether Theorem \ref{theorem:Newman2} holds in arbitrary dimension.
\end{question}

Ending this subsection, we quote an apparently still open question raised in the end of the original paper of Hammersley and Welsh.

\begin{question}
Let $\Gamma(v)$ be a geodesic from $0$ to $v$. Consider the collection of all edges that belong to an infinite number of $\Gamma(v)$ and let $f(r)$ be the number of such edges which intersect the circle $x^2 + y^2 = r^2$. Does $f(r) \to \infty$ as $r \to \infty$; and, if so, how fast? 
\end{question}

\begin{proof}[Sketch of the proof of Theorem \ref{theorem:Newman2}]
We sketch the proof in the case where the weights are unbounded. The bounded weight case requires a little modification of the arguments, and the details can be found in the original paper \cite[Theorem 1]{LN96}. 

Consider all geodesic rays that have asymptotic direction $\theta\in D$. Construct the geodesic graph $G$ composed of the union of all these geodesic rays. By Theorem \ref{theorem:LN}, if two such rays meet, they must coalesce. Thus, the geodesic graph must be a forest with $N\geq 1$ distinct trees. The proof in \cite{LN96} has $3$ steps.
\vspace{0.1cm}

\textsc{Step 1}. \textit{ If $\Pro(N\geq 2) >0$ then $\Pro(N \geq 3)>0$}. 
\vspace{0.1cm}

\noindent To see this, without loss of generality, assume that the direction $\theta$ has positive $e_{1}$-coordinate. Also suppose that the two different geodesics touch the $e_{2}$ axis at two  points $y_1 \neq y_2$ and are contained in the half-plane $x \geq 0$. By vertical translation invariance, it is possible to find with positive probability two different points $y_3$, $y_4$ that also have two such disjoint geodesic rays. By planarity, at least three of these rays must be distinct (see figure). 
In fact, if $\Pro(N\geq 2) >0$ then there exist $y_1<y_2<y_3$ in $\mathbb Z$ such that $\mathbb{P}(E(y_1,y_2,y_3))>0$, where $E(y_1,y_2,y_3)$ is the event that there are three disjoint geodesic rays with direction $\theta$ starting from $(0,y_i)$ and contained in the half-plane $\{(x,y) \in \mathbb Z^2: x>0\}$ except their initial points.

\begin{figure}[htbp]
\begin{center}
\def\svgwidth{12cm}
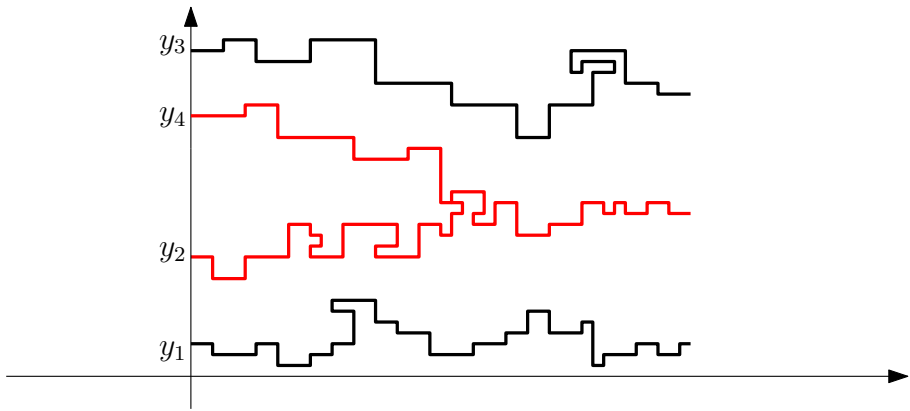
\caption{Construction of the three disjoint geodesics starting from $y_{1},y_{2},y_{3}.$ The geodesic starting from $y_{4}$ may intersect the geodesic emanating from $y_{2}$ but all others are disjoint.}
\label{fig:32}
\end{center}
\end{figure}

Let $F$ be the subevent of $E(y_1,y_2,y_3)$ on which the tree in $G$ containing the ``middle'' point $y_2$ does not intersect the half-plane $\{(x,y) \in \mathbb Z^2: x< 0\}$. The second step of the proof is the following. 

\vspace{0.1cm}

\textsc{Step 2}. If $\Pro(N\geq 3) >0$ then $\Pro(F)>0.$ 

\vspace{0.1cm}
\noindent By the arguments in step 1, with positive probability there are three disjoint geodesic rays $r_1, r_2, r_3$ starting respectively from the points $(0,y_1), (0,y_2), (0,y_3) \in L$. We now set $y^* = \max\{|y_1|,|y_3|\}$ and increase the values of the passage times on edges between $u=(-1,y)$ and $v=(0,y)$, $-y^* \leq y \leq y^*$. If we increase this finite collection of passage times enough, we create a barrier through which no geodesic can pass. (Specifically, if $\mathcal{F}$ is this set of edges, one needs that for all $e \in \mathcal{F}$, $\tau_e$ is greater than $T_e(x,y)$, where $x,y$ are the endpoints of $e$, and $T_e$ is the passage time among all paths which do not use $e$.) Note that, by our construction, this change also does not alter the $r_i's$. Furthermore, as we have unbounded passage times, one can construct such a  barrier with positive probability. Now, since there is zero probability that two geodesic rays with direction $\theta$  meet and do not coalesce, this barrier attached to $L$ together with $r_1$ and $r_3$ implies that the tree that contains $r_2$ is the desired tree.

\vspace{0.1cm}

\textsc{Step 3}. $\Pro(F)>0$ is impossible. 
\vspace{0.1cm}

Consider all possible translates $L^u$ of $L$ by points of the form $3uy^*$ for $u \in \mathbb Z^2$ and the corresponding translated events $F^u$ of $F$. If $F^u$ and $F^v$ both occur, then, by definition, the corresponding ``middle'' geodesics must be disjoint, as they are blocked from each other by the ``outer'' geodesics constructed in the event $F$. Now, let $C_M$ be the number of $L^u$'s contained in the box $[0,M] \times [0,M]$ and $N_M$ be the random number of the corresponding $F^u$'s which occur. By translation invariance, and the assumption that $\Pro(F)>0$, 
\begin{equation}\label{eq:sec4.22}
\E N_M = C_M \Pro(F) \geq cM^{2},
\end{equation}
for some $c>0$. But $N_M$ is bounded above by the number of disjoint trees that touch $[0,M] \times [0,M]$. Since each tree is infinite, this number cannot exceed the number of boundary points in that box, which is at most  $c'M$ for some $c'$. As $c'M < cM^2$ for $M$ large, this contradicts \eqref{eq:sec4.22}, showing that $\Pro(F)=0$.
\end{proof}

 \subsection{Absence of geodesic lines and connection to the two-dimensional Ising ferromagnet }
\label{sec:bigeodesics}

A final (and one of the most) important question in the study of infinite geodesics is the following:
\begin{question}\label{q:bi}
Do geodesic lines exist?
\end{question}

\subsubsection{Heuristic argument}
We start with a heuristic argument that there should (a.s.) exist no geodesic lines (bigeodesics), at least in dimensions $d$ where $\xi > 1/2$. In particular, the argument gives a plausible reason to disbelieve in bigeodesics in two dimensions. It is important to note that this argument is far from rigorous even under usual unproven assumptions (e.g. curvature) on the model.

Indeed, a main feature of the argument is that it uses various meanings of $\xi$ interchangably. For instance, it is assumed that the geodesics $\Gamma(0, n e_1)$ and $\Gamma(m^\xi, n e_1)$ ($m \ll n$) merge after a distance of order $m$, in addition to $\Gamma(0, n e_1)$ having transversal wandering typically of order $n^\xi$.

This argument was presented by C. Newman at the workshop ``First-passage percolation and related models'' at AIM in Summer, 2015.

\bigskip
\noindent
{\bf Setup.} Consider the box $[-N, N]^d$ for $N$ large; we denote the boundary of this box by the symbol $\partial_N$. Assume that there were positive probability that a bigeodesic existed (and hence that this event had probability one). Then this bigeodesic would have to pass through the origin, at least with positive probability.

Such a bigeodesic must pass through a vertex of $\partial_N$ before passing through $0$ and another vertex of $\partial_N$ afterwards. In particular, the event that a geodesic between a pair of vertices of $\partial_N$ passes through $0$ must have uniformly (in $N$) positive probability. To rule out bigeodesics, it therefore suffices to show
\begin{equation}
  \label{eq:ruleoutbig}
  \lim_{N \rightarrow \infty}\Pro\big(\exists \; x, y \in \partial_N \text{ such that } 0 \in \Gamma(x,y)\big) = 0\ .
\end{equation}
The key idea to prove \eqref{eq:ruleoutbig} is to use the wandering property of individual geodesics to show that they are likely to avoid $0$.

Unfortunately, using a union bound on the events $\{0 \in \Gamma(x,y)\}_{x, y \in \partial_N}$ is unlikely to give \eqref{eq:ruleoutbig}, since there are many such pairs $(x,y)$. To overcome this difficulty, we will group geodesics together, tiling $\partial_N$ with $(d-1)$-dimensional cubes $B_j$ of side length $\sim N^\xi$. For a particular cube $B_j$, let $B_{-j}$ denote the corresponding cube on the opposite side of $\partial_N$.

\bigskip
\noindent
{\bf Reducing to opposite blocks.}
We will need to make assumptions on the structure of geodesics between vertices in blocks $B_i, B_j$. The first is that we can neglect the contribution of the probability
\[\mathbb{P}\left(\exists x \in B_i, y \in B_j \text{ for some } i \neq -j \text{ such that } 0 \in \Gamma(x,y)\right) \]
to \eqref{eq:ruleoutbig}.

This is a not unreasonable assumption, since the straight line segment between such $x$ and $y$ will typically lie distance $\gg \|x-y\|^\xi$ from $0$ (perhaps requiring the diameter of $B_j$ to be at least $N^{\xi + \varepsilon}$). This of course is likely to necessitate a strong concentration bound on the variable
\[D(0, ne_1) / n^\xi \,  \]
where $D$ was defined above \eqref{eq:kestenwander}. In any case, once such a result is established, it allows us to reduce the problem of showing \eqref{eq:ruleoutbig} to showing
\begin{equation}
  \label{eq:ruleoutbig2}
\lim_{N \rightarrow \infty}\Pro\big(\exists j, \, x \in B_j, \, y \in B_{-j} \text{ such that } 0 \in \Gamma(x,y)\big) = 0\ .
\end{equation}

\bigskip
\noindent
{\bf Reducing within blocks.}
We have (at least assuming $\xi < 1$) succeeded in reducing the number of pairs of geodesics under consideration by a power by reducing to opposite pairs of blocks. Since geodesics emanating from vertices within the same $B_j$ are likely to be highly correlated, we can reduce further by grouping together such geodesics. In particular, we will assume that all geodesics originating in one block $B_j$ merge before coming close to the origin.

For each block $B_j$, choose a particular vertex $x_j$ near the center of $B_j$. As discussed earlier, it is likely that there is ``geodesic merging on scale $\xi$''. In particular, for any $y \in B_j$, $y' \in B_{-j}$, it is plausible that the event $\{0 \in \Gamma(y, y')\}$ has extremely small symmetric difference with the event $\{0 \in \Gamma(x_j, x_{-j})\}$. Thus, for fixed $j$ we can approximate
\begin{equation}
  \label{eq:blockapprock}
  \Pro\big(\exists\, x \in B_j, \, y \in B_{-j} \text{ such that } 0 \in \Gamma(x,y)\big) \approx \mathbb{P}\big(0 \in \Gamma(x_j, x_{-j}) \big)\ .
\end{equation}

With this approximation established, we have reduced the problem sufficiently to take a union bound:
\begin{align}
\label{eq:ruleoutbig3}
\eqref{eq:ruleoutbig2} \leq \mathbb{P}\big(\exists\, j \text{ such that } 0 \in \Gamma(x_j, x_{-j}) \big) \leq \#(\text{boxes } B_j)\  \mathbb{P}\big(0 \in \Gamma(x_1, x_{-1})\big) \ . 
\end{align}

\bigskip
\noindent
{\bf Individual geodesic wandering.}
We now need to control the probability $\mathbb{P}(0 \in \Gamma(x_1, x_{-1}))$, which will require a further assumption on geodesic wandering. In some sense, what is required here is a lower bound on fluctuations, that $\Gamma(x_1, x_{-1})$ is ``spread out'' on scale $N^\xi$ far from its endpoints. 

To be more precise, if $L$ is the line between $x_1$ and $x_{-1}$ and $H$ denotes the hyperplane passing through $0$ to which $L$ is normal, we assume that $\Gamma(x_1, x_{-1})$ passes through $\sim 1$ vertices of $H$ and every vertex in a block of diameter $N^\xi$ is equally likely to be passed through. In particular,
\[\mathbb{P}(0 \in \Gamma(x_1, x_{-1})) \sim N^{-\xi(d-1)}\ . \]

With this information, we can conclude by bounding the right-hand side of \eqref{eq:ruleoutbig3}. Note that the number of boxes $B_i$ is of the order of
\[\frac{N^{d-1}}{N^{\xi(d-1)}} = N^{(1-\xi)(d-1)}\ , \]
giving
\begin{align*}
\eqref{eq:ruleoutbig3} \lessapprox \frac{N^{(1-\xi)(d-1)}}{N^{\xi(d-1)}} =  \big(N^{1 - 2 \xi} \big)^{d-1}\ .
\end{align*}
In particular, if $\xi > 1/2$, then \eqref{eq:ruleoutbig} holds, giving a.s. absence of bigeodesics.

\subsubsection{Rigorous results}
In this section, we move to rigorous results, outlining the work of Licea-Newman and Wehr from the '90s. In Section~\ref{sec: bigeo_new}, we give the most recent results, ruling out bigeodesics in any deterministic direction under assumptions on the limit shape boundary. 

The first partial answer to the bigeodesic question comes from Theorem \ref{LN96theorem}. We say that a geodesic line has asymptotic directions $(\theta_1,\theta_2)$ if its two rays have directions $\theta_1$, $\theta_2$.

\begin{theorem}[Licea-Newman \cite{LN96}]\label{LN96theorem}
For $d=2$, $\theta_1, \theta_2 \in D$, and $F$ continuous, there is zero probability that there exists a geodesic line that has directions $(\theta_1, \theta_2)$.
\end{theorem}
It turns out that the proof of Theorem \ref{LN96theorem} is not difficult once we have Theorem \ref{theorem:Newman2}.
\begin{proof}[Proof of Theorem \ref{LN96theorem}] If there are two distinct geodesic lines with directions $(\theta_1,\theta_2)$, then by Theorem \ref{theorem:Newman2} they must coalesce in both directions. This cannot happen, as we assumed the model is uniquely geodesic. Thus, there must be at most one geodesic line with directions $(\theta_1,\theta_2)$. Let $E$ be the event that there exists a unique geodesic line with directions  $(\theta_1,\theta_2)$. By the ergodic theorem, the probability of the event $E$ is either $0$ or $1$. 

Suppose $\Pro(E) = 1$. Then it is possible to find $L$ large enough such that the geodesic line intersects the box $B_L:=[0,L]\times[0,L]$ with probability at least $2/3$. Choosing a sequence of points $z_n \in \mathbb Z^2$ with $z_n/|z_n| \to \alpha \neq \theta_1$ or $\theta_2$, the geodesic line passes through only finitely many  shifted boxes $z_n + B_L$. Letting $E_S$ be the event that the geodesic line intersects the set $S \subseteq \mathbb Z^2$, this implies 
$$ \Pro( E_{B_L} \cap E_{z_n + B_L} ) \to 0 \text{ as } n\to \infty.$$
On the other hand, by translation invariance, 
$$\Pro( E_{B_L} \cap E_{z_n + B_L} ) \geq 1 - \Pro(E_{B_L}^c) - \Pro(E_{z_n + B_L}^c) \geq 1/3,$$
a contradiction.
\end{proof}

We note that it is possible to replace this proof with one that proceeds along the lines of Wehr's proof that the number of bigeodesics in all of $\mathbb{Z}^d$ must be zero or infinity:
\begin{theorem}[Wehr \cite{Wehr}]\label{thm: wehr_one_infty}
Assume $F$ is continuous, $d \geq 2$, and $\tau_e$ has finite mean. The number $\mathcal{N}$ of bigeodesics in a configuration $(\tau_e)$ is almost surely constant, and $\mathcal{N} \in \{0,\infty\}$.
\end{theorem}
The idea of the proof is that if there are only finitely many bigeodesics, then by translation invariance, in a large box, at least one of the bigeodesics must take a positive density of edges. So the passage time along this bigeodesic from its first entry to its last exit of the box must be of the order of the volume of the box. This contradicts the fact that the Euclidean distance between these two points is of the order of the length of the box, and so the passage time between them is also of this same order, with high probability. This theorem has been extended to some non-i.i.d. times by Boivin-Derrien \cite{BoivinDerrien}. In that paper, they also construct non-i.i.d. but translation invariant distributions on $(\tau_e)$ for which there exist bigeodesics.

A positive answer to Question \ref{question:setD}, combined with the above result, would exclude the possibility of geodesic lines with deterministic asymptotic directions. It would not solve Question \ref{q:bi}, as one could still have geodesic lines in random directions.  

Question \ref{q:bi} was answered negatively in the upper half-plane $\mathcal H$ of $\Z^2$ under passage times with finite mean and continuous distribution by Wehr and Woo \cite{WW98}. 
\begin{theorem}[Theorem 1 \cite{WW98}]\label{WWTheorem} Assume $F$ is continuous and $\tau_e$ has finite mean. With probability $1$, there exist no geodesic lines in $\mathcal H$.
\end{theorem}
We outline the proof of Theorem \ref{WWTheorem} at the end of this section. The assumption of finite mean was relaxed in \cite{ADH1}, where one just requires uniqueness of finite geodesics. That work also included a range of non-i.i.d. times.

Before sketching Wehr-Woo's argument, let us stress that the importance of determining non-existence of geodesic lines goes beyond the field of FPP. Apparently, this problem was first posed by H. Furstenberg in a different context. As far as we know, the first mathematical reference is in \cite{KestenAspects}. Question \ref{q:bi} is equivalent to the problem of existence of nontrivial (that is, non-constant) ground states in the two-dimensional Ising ferromagnet with random exchange constants. Ground states of higher-dimensional random ferromagnets are similarly related to hypersurfaces with minimal random weights \cite{KestemMaxflow, Wehr}. 

More precisely, consider the lattice dual to $\Z^2$, defined by 
$$ (\Z_*^2, \mathcal E_*^2) = (\mathbb Z^2, \mathcal E^2) + \frac{1}{2} (e_1 + e_2),$$
where $\mathcal{E}^2$ is the set of nearest-neighbor edges of $\mathbb{Z}^2$, and define a spin configuration as an element $\sigma$ of $\{\pm1\}^{\Z^2_*}$. Let $(J_{x,y})_{(x,y) \in \mathcal E_*^2}$ be a collection of i.i.d. positive random variables. For any configuration and any finite set $S\subset  \Z_*^2$ define the (random) energy 
$$H_S(\sigma) = - \sum_{\stackrel{\{x,y\} \in \mathcal E_*^2}{x\in S}} J_{x,y} \sigma_x \sigma_y.$$

We will call $\sigma$ a ground-state for $(J_{x,y})$ if, for each configuration $\rho$ such that $\rho_x = \sigma_x$ for all $x$ outside some finite set, we have
$$H_S(\sigma) \leq H_S(\rho) \text{ for all finite } S\subset  \Z_*^2. $$
It is an open problem to describe the set of ground states for this ferromagnetic model. In particular, it is not known how many ground states there are for a given $(J_{x,y})$, although it is conjectured (see, for instance, the discussion in \cite[Chapter 1]{Newmanbook}) that, if the law of $J$ is continuous, there should be only two almost surely.

These two are the constant configurations $\sigma = \pm1$, which are clearly almost surely ground states. If any non-constant ground states $\sigma$ exist, they cannot have finite regions of disagreement; that is, there can be no finite $S$ such that $\sigma_x =+1$ for all $x \in S$ and $\sigma_y =-1$ for all $y \in \partial S$, or vice-versa. Therefore, any non-constant ground state must have a two-sided (and circuitless) infinite (original lattice) path of edges dual to edges $\{x,y\}$ such that $\sigma_x =-\sigma_y$.

Now consider FPP with $\tau_e = J_{(x,y)}$ if $\{x,y\}$ is the dual of the edge $e$. If such a first-passage configuration had a geodesic line, then the configuration $\sigma$ which takes the value $+1$ on one side of the geodesic line and $-1$ on the other would be a non-constant ground state for the associated spin model. Therefore in $d=2$, based on the analogy above with the Ising ferromagnet, it is believed that one should have a negative answer for Question \ref{q:bi}. 

Partial results in this direction were obtained by Newman under the assumption of uniform curvature. In \cite{Newman}, it was shown that there cannot exist geodesic lines other that those that have directions on both ends. As we have seen also, Theorem \ref{LN96theorem} ruled out almost all possible pairs of deterministic directions and Theorem~\ref{thm: wehr_one_infty} showed that the number of geodesic lines is 0 or infinity. There are also further arguments in the physics literature \cite{FLN} against the existence of non-constant ground states.

Let us end this subsection by presenting the argument of Wehr-Woo for absence of geodesic lines in the upper half-plane $\mathcal H$. For simplicity, we assume that the passage time $\tau_e$ has finite mean and continuous distribution so that finite geodesics are unique almost surely. Let
\[
E = \{\text{there exists a geodesic line in } \mathcal H\}.
\]
By horizontal translation ergodicity, $\mathbb{P}(E)$ is zero or one; so let us assume for a contradiction that $\mathbb{P}(E)=1$.

Any geodesic line $\Gamma$ divides $\mathbb{R}^2 \setminus \Gamma$ into two components, say $R^+=R^+(\Gamma)$ and $R^- = R^-(\Gamma)$; that is,
\begin{align*}
R^+(\Gamma) \cap R^-(\Gamma) &= \emptyset\ , \\
R^+(\Gamma) \cup R^-(\Gamma) &= \mathbb{R}^2 \setminus \Gamma\ , \\
\partial R^+ &= \partial R^- = \Gamma\ ,
\end{align*}
where $R^-$ is a region that contains $(0,-1)$ and where $\partial A$ denotes the usual boundary of a set $A \subset \mathbb{R}^2$. Hence, by uniqueness of passage times, for any points $x,y \in R^-(\Gamma)$, no edge $b$ belonging to the finite geodesic $\Gamma(x,y)$ can intersect $R^+(\Gamma)$. 
\vspace{0.1cm}

\textsc{Step 1}. The first step of the proof  is to establish the following fact \cite[Proposition~4]{WW98}. If a geodesic line exists with probability one, then there is a lowest one.

\begin{proposition}\label{prop: lowest_def}
Assume that $\Pro (E) =1$. Consider the sequence of geodesics $\Gamma_n$ from $(-n,0)$ to $(n,0)$ for $n \in \mathbb{N}$. With probability 1, this sequence has a limit:
\[
\Gamma_0 = \lim_{n \to \infty} \Gamma_n \ .
\]
Moreover, $\Gamma_0$ is a geodesic line, and for any geodesic line $\Gamma$,
\[
\Gamma_0 \subset \left[ R^-(\Gamma) \cup \Gamma\right]\ .
\]
\end{proposition}
\begin{proof}
Suppose that $\Gamma$ is a geodesic line. By uniqueness of passage times, for every $n$, $\Gamma_n$ must lie in $R^{-}(\Gamma) \cup \Gamma$. Now,  compactness (that is, the same argument used to establish existence of geodesic rays) implies that (passing through a subsequence) $\Gamma_{n_k}$ must converge. For instance, there is only a finite number of possible edges for the intersection of $\Gamma_{n}$ with the $e_{2}$, axis and this intersection is nonempty for every $n\geq 1$.  However, the geodesics $\Gamma_n$ are increasing in the sense that for $n>m$, $\Gamma_n$ must lie above $\Gamma_m$. Thus, they converge.
\end{proof}

\vspace{0.25cm}

\textsc{Step 2}. If $\Gamma_0$ intersects the strip $S_n= \{ (x,y) \in \mathbb Z^2 : 0\leq y < n\}$ then $R^+(\Gamma_0) \cap S_n$ is a non-empty union of bounded connected sets. See Figure \ref{Fig:WW}.

\begin{figure}
\scalebox{1.0}{ 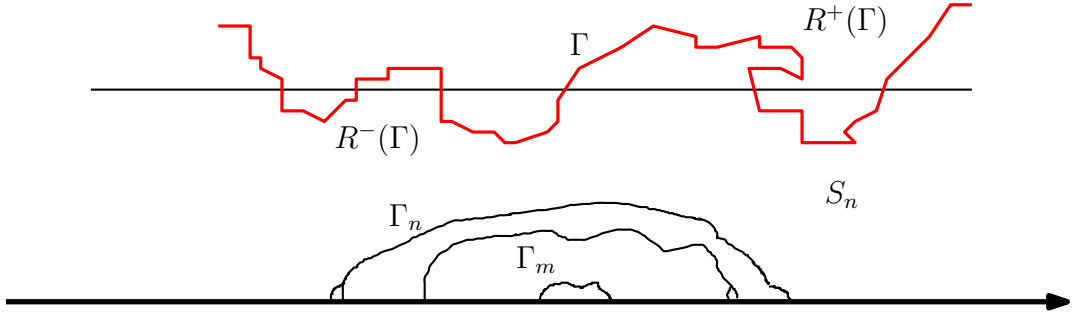}
 \centering
 \def \svgwidth{3000pt}
 \caption{ The geodesic line $\Gamma_0$ (in red) and the regions $R^{+}(\Gamma_0)$ and $R^{-}(\Gamma_0)$. The intersection of $\Gamma$ and the strip $S_{n}$ is a union of bounded sets. The geodesics $\Gamma_{n}$ monotonically converge to $\Gamma_0$.}
 \label{Fig:WW}
\end{figure}

\vspace{0.1cm}
This step is the main tool to prove:
\vspace{0.25cm}

\textsc{Step 3}. Uniquely geodesic implies that $\Gamma_0$ must intersect any large box with probability bounded below uniformly in the position of the box. 

Precisely, for $l \in \mathbb{N}$, write $B=B(l) = [-l,l] \times [0,2l]$ and let $K$ be the event that at least one geodesic line intersects $B$. Define, for $L \in \mathbb{N}$, translations of $B$ by
\[
B_{i,j} = B_{i,j}(l,L) = B + (iL,jL) \text{ for } (i,j) \in \mathbb Z^2, j \geq 0 \ .
\]
For $L > 2l$, the $B_{i,j}$ are mutually disjoint.
\begin{prop}\label{prop: hit_all_boxes}
Let $\delta = 1-\mathbb{P}(K)$. Then

\begin{equation}\label{eq:WW1}
\mathbb{P}\bigg(B_{i,j} \subset R^+(\Gamma_0)\bigg) \leq \delta, \quad \quad  \mathbb{P}\bigg(B_{i,j} \subset R^-(\Gamma_0)\bigg) \leq \delta\  \text{ and }
\end{equation}
\begin{equation}\label{eq:WW2}
\mathbb{P}\bigg(\Gamma_0 \cap  B_{i,j} \neq \emptyset \text{ for all } 0\leq i, j<k\bigg) \geq 1 -2k^2\delta.
\end{equation}
\end{prop}

To prove \eqref{eq:WW1}, one uses \textsc{Step 1} and \textsc{Step 2}. Consider  the reflected half-planes
\[
\mathcal H_{j}^{r}=\big \{ (x,y) \in \mathbb R^{2}: y \leq jL+2l \big \}
\]
and set the event
\[
K_{i,j}:= \big \{ \text{at least one } \mathcal H_{j}^{r}-\text{geodesic line intersects } B_{i,j}  \big \}.
\]
Since the model is invariant under translations and under rotations by $180$ degrees, 
\begin{equation}\label{eq:reflected}
\Pro(K)=\Pro(K_{i,j}).
\end{equation}
 
If $B_{i,j} \subset R^+(\Gamma_0)$ then, by \textsc{Step 2}, $B_{i,j}$ is contained in exactly one of the bounded components of $R^+(\Gamma_0) \cap \mathcal H_{j}^{r}$. 
Hence, there is no $\mathcal H_{j}^{r}$-geodesic line intersecting $B_{i,j}$, since otherwise two different finite $\mathcal H_{j}^{r} \cap \mathcal H$-geodesics would meet more than once. 
Thus the event $\{B_{i,j} \subset R^+(\Gamma_0) \}$ is contained in the event  $K_{i,j}^c$, which combined with \eqref{eq:reflected} leads to the first part of \eqref{eq:WW1}. 

The second equation in \eqref{eq:WW1} uses \textsc{Step 1} and the fact that the box $B_{i,j}$ would be trapped by the union of the geodesics $\Gamma_n$ for $n$ large and the boundary of the upper-half plane. Therefore no geodesic line in the shifted half-plane $\mathcal{H}_j =\{(x,y) \in \mathbb{R}^2: y \geq jL\}$ can intersect $B_{i,j}$. To get \eqref{eq:WW2}, one uses \eqref{eq:WW1} to obtain 
\[
\Pro( \Gamma_0 \cap  B_{i,j} \neq \emptyset ) \geq 1 - 2 \delta
\]
for any $i,j$. Thus, 
\[
\Pro\bigg( \bigcap_{0\leq i,j < k} \left\{ \Gamma_0 \cap  B_{i,j} \neq \emptyset\right\}  \bigg) \geq 1 -2k^{2}\delta.
\]

\vspace{0.1cm}

\textsc{Step 4}. Equation \eqref{eq:WW2}  contradicts $\Pro(E)=1$. Thus, $\Pro(E)=0$.

\noindent To see this, assume  $\Pro(E)=1$ and for any $k\geq1$, choose $l=l(k)$ large enough so that $B=B(l)$ satisfies 
\[
\Pro(\Gamma_0 \cap B \neq \emptyset) = \Pro(K) >1-1/4k^2.
\]
By \eqref{eq:WW2}, this leads to
\begin{equation}\label{eq:WWcontradiction} 
\Pro\bigg( \bigcap_{0\leq i,j < k} \left\{ \Gamma_0 \cap  B_{i,j} \neq \emptyset \right\} \bigg) \geq 1/2 .
\end{equation}
The proof now consists of showing that \eqref{eq:WWcontradiction} cannot hold for $k$ large if $L$, the distance of the shift, is large enough.

To see this, choose $L>2l$ so all shifted boxes $B_{i,j}$ are disjoint. If the event in \eqref{eq:WWcontradiction} happens, it is possible to find $k^{2}$ different points in the geodesic line $\Gamma_{0}$, each one inside a different box $B_{i,j}$. Let $x_{m}, 1\leq m \leq k^{2}$ be these points, ordered as they appear in $\Gamma_{0}$. As the $\ell^{1}$-diameter of the union of the $B_{i,j}$ is at most $2Lk$, one has 
\[
L-2l < \|x_{1} - x_{k^{2}}\|_1 < 2Lk \; \text{ and } |\Gamma(x_{1},x_{k^{2}})| \geq k^{2}L/4,
\]
where $|\Gamma(x_{1},x_{k^{2}})|$ is the number of edges in the geodesic $\Gamma(x_{1},x_{k^{2}})$. There at most  $4l^2k^{2}$ choices for each $x_{1}$ and $x_{k^{2}}$
so an union bound leads to 
\[
\Pro \bigg(\bigcap_{0\leq i,j < k} \left\{\Gamma_0 \cap  B_{i,j} \neq \emptyset\right\} \bigg) \leq 4l^{2} k^{4} \sup_{\|x-y\|_1\leq 2Lk} \Pro\bigg(|\Gamma(x,y)|>k^{2}L/4\bigg).
\]
If $k$ (and therefore $l$) is fixed large enough then, by Theorem \ref{thm: new_geo_length_bound}, the supremum on the right side is exponentially small in $L$. This contradicts \eqref{eq:WWcontradiction} for $L$ large, ending the proof of Theorem \ref{WWTheorem}.


\newpage
\section{Busemann functions}\label{sec:Busemann}

The modern approach to infinite geodesics involves Busemann functions. These functions were first exploited in FPP in the important papers of Hoffman \cite{Hoffman, Hoffman1}, and they allow one to obtain some information about geodesics without the unproven assumptions (for example, positive curvature of the limiting shape) of Newman.

\subsection{Basics of Busemann functions}

Long before FPP existed, H. Busemann invented various tools to study the geometry of geodesics in certain metric spaces \cite{Busemann}. Busemann functions grew out of the attempt to understand parallelism between geodesics. We can give some basic properties in the setting of FPP. Consider an infinite geodesic $R$ (which we will call a geodesic ray), and list its vertices in order as
\[
V(R) = \{x_0, x_1, \ldots\}.
\]
Here, $x_0$ is taken as the initial vertex of the ray. 

\begin{definition}
The Busemann function associated to the ray $R$ is $f_R:\mathbb{Z}^d \to \mathbb{R}$, given by
\[
f_R(x) = \lim_n \left[ T(x,x_n) - T(x_0,x_n) \right].
\]
\end{definition}
This limit exists since the terms are bounded ($|T(x,x_n) - T(x_0,x_n)| \leq T(x_0,x)$) and monotone:
\begin{align*}
T(x_0,x_{n+1})-T(x,x_{n+1}) &= T(x_0,x_n) + T(x_n,x_{n+1}) - T(x,x_{n+1}) \\
&\geq T(x_0,x_n) - T(x,x_n).
\end{align*}

The Busemann function for $R$ measures asymptotically how far behind $x_0$ is relative to $x$ when both points attempt to travel down the ray $R$. The next lemma collects a few properties of this function. 
\begin{lemma}\label{lem:busproperties} The following holds.

\begin{enumerate}
\item[$(a)$] For $m < n$, one has $f_R(x_m) - f_R(x_n) = T(x_m,x_n)$.
\item[$(b)$] For all $x$, one has $|f_R(x)| \leq T(x,x_0)$.
\item[$(c)$] If $R_1$ and $R_2$ are geodesic rays that coalesce (they have finite symmetric difference) and have initial points $x_0$ and $y_0$, then
\[
f_{R_1}(x) = f_{R_2}(x) - f_{R_2}(x_0) \text{ for all } x \in \mathbb{Z}^d.
\]
\end{enumerate}
\end{lemma}
\begin{proof} Proofs of items $(a)$ and $(b)$ follow directly from the definition. As for $(c)$, write the vertices of $R_1$ as $x_0, x_1, \ldots$ and the vertices of $R_2$ as $y_0, y_1, \ldots$. By coalescence, we can find $N \in \mathbb{Z}$ such that for $n$ large, one has $x_n = y_{n+N}$. For such $n$,
\begin{align*}
T(x_0,x_n) - T(x,x_n) &= T(x_0,y_{n+N}) - T(y_0,y_{n+N}) + T(y_0,y_{n+N}) - T(x,y_{n+N}) \\
&\to f_{R_2}(x_0) - f_{R_2}(x). 
\end{align*}
\end{proof}

We will often use a Busemann function that does not depend on the initial point. If $R$ is a geodesic ray, then set
\[
f_R(x,y) = f_R(x)-f_R(y).
\]
We note the following properties of this Busemann function:
\begin{enumerate}
\item $f_R$ is additive in the following sense:
\[
f_R(x,y) = f_R(x,z) + f_R(z,y) \text{ for } x,y,z \in \mathbb{Z}^d.
\]
\item If $R_1$ and $R_2$ are geodesic rays that coalesce, then $f_{R_1}(x,y) = f_{R_2}(x,y)$ for all $x,y \in \mathbb{Z}^d$.
\begin{proof}
Using item $(c)$ in Lemma \ref{lem:busproperties} and writing $x_0$ for the initial point of $R_1$, one has
\begin{align*}
f_{R_1}(x) - f_{R_1}(y) &= (f_{R_2}(x)-f_{R_2}(x_0)) - (f_{R_2}(y) - f_{R_2}(x_0)) \\
&= f_{R_2}(x) - f_{R_2}(y).
\end{align*}
\end{proof}
\item (Translation covariance) Let $\theta$ be a translation of the lattice by an integer vector. Then
\[
f_R(x,y)(\omega) = f_{\theta R}(\theta(x),\theta(y))(\theta(\omega)).
\]
Here, the translated weight-configuration $\theta(\omega)$ is defined as $t_e(\theta(\omega)) = t_{\theta^{-1}e}(\omega)$ for any edge $e$. Furthermore $\theta R$ is the translated geodesic ray.
\end{enumerate}

\subsection{Hoffman's argument for multiple geodesics}
In this section we present Hoffman's \cite{Hoffman} argument that there exist at least two infinite geodesics almost surely. This result was also proved by Garet-Marchand \cite{GM} in less generality, but Hoffman's techniques, involving Busemann functions, led to many other results. We will use the assumptions of Hoffman:
\begin{enumerate}
\item $\mathbb{P}$ is ergodic under lattice translations.
\item $\mathbb{P}$ has unique geodesics.
\item $\mathbb{E}t_e^{2+\delta}<\infty$ for some $\delta>0$.
\item The limit shape for $\mathbb{P}$ is bounded.
\end{enumerate}

We say that two geodesic rays $R_{1}$ and $R_{2}$ are disjoint if the cardinality of both sets $R_{1} \setminus R_{2}$ and $R_{2} \setminus R_{1}$ are infinite. Let $\mathcal{N} = \mathcal{N}(\omega)$ be the maximal number of disjoint geodesic rays in the edge-weight configuration $\omega$.  It can be shown that $\mathcal{N}$ is a measurable function (relative to the product sigma-algebra) and is invariant under lattice translation. So by ergodicity, it is almost surely constant.
\begin{theorem}
The number $\mathcal{N}$ is at least two.
\end{theorem}
\begin{proof}
By taking a sub-sequential limit of the geodesics from $0$ to $ne_1$, we see that $\mathcal{N} > 0$. Assume for a contradiction that $\mathcal{N}=1$. Let $R = R(\omega)$ be any geodesic ray and define $f:\mathbb{Z}^d \times \mathbb{Z}^d \to \mathbb{R}$ by
\[
f(x,y) = f_R(x,y) = f_{R(\omega)}(x,y)(\omega).
\]
On the probability 1 set on which $\mathcal{N}=1$, this function is well-defined and is independent of the choice of $R$. Indeed, if $R' = R'(\omega)$ is another geodesic ray, then it cannot be disjoint from $R$. Due to unique geodesics, $R'$ must coalesce with $R$ and therefore $f_R = f_{R'}$.

One can also check that $f$ is measurable. Due to translation covariance of $f_R$, the function $f$ is also translation covariant:
\begin{equation}\label{eq: ergodic_meat}
f(x,y)(\omega) = f\big(\theta (x), \theta (y)\big)(\theta (\omega)) \text{ almost surely}
\end{equation}
for any lattice translation $\theta$. This, combined with additivity and the ergodic theorem, implies for each $x \in \mathbb{Z}^d$,
\[
\frac{1}{n} f(0,nx) = \frac{1}{n} \sum_{k=1}^n f((k-1)x,kx) \to \mathbb{E}f(0,x) \text{ almost surely and in }L^1.
\]

To find this limit, we use symmetry: for $k=1, \ldots d$,
\[
0 = \mathbb{E}f(0,e_k) + \mathbb{E}f(e_k,0)  = \mathbb{E}f(0,e_k) + \mathbb{E}f(-e_k,0)
\]
and by translation invariance, this equals
\[
\mathbb{E}f(0,e_k) + \mathbb{E}f(0,e_k).
\]
We then obtain
\[
\mathbb{E}f(0,me_k) = m \mathbb{E}f(0,e_k) = 0
\]
for all integers $m$ and then by additivity,
\[
\mathbb{E}f(0,x) = 0 \text{ for all } x \in \mathbb{Z}^d.
\]
Combining with \eqref{eq: ergodic_meat}, one has
\[
\frac{1}{n} f(0,nx) \to 0 \text{ almost surely and in }L^1.
\]

To upgrade this convergence to an omni-directional statement, one can use the condition $\mathbb{E}\tau_e^{2+\delta}<\infty$ along with the bound $|f(x,y)| \leq T(x,y)$ to prove a type of shape theorem for the function $f$. That is, for each $\epsilon>0$,
\[
\mathbb{P}\left( |f(0,x)| < \epsilon \|x\|_1 \text{ for all } x \in \mathbb{Z}^d \text{ with } \|x\|_1 \text{ large enough} \right)=1.
\]
In particular, if we define $R$ to be the (almost surely unique) limit of geodesics from $0$ to $ne_1$, and we write its vertices in order as $x_0, x_1, \ldots$, then one has
\[
f(0,x_n)/\|x_n\|_1 \to 0 \text{ almost surely}.
\]
On the other hand, by property 1 of $f_R$, one has $f(0,x_n) = T(0,x_n)$. Therefore
\[
T(0,x_n)/\|x_n\|_1 \to 0 \text{ almost surely}.
\]
This contradicts the fact that the limit shape is bounded and therefore shows $\mathcal{N}>1$.
\end{proof}

\subsection{Directions of geodesics via Busemann functions}

In \cite{Newman}, Newman introduced the following limits, in an attempt to study the local geometry of the boundary of the growing ball $B(t)$:
\[
\lim_n [T(x,x_n) - T(y,x_n)] \text{ for } x,y \in \mathbb{Z}^2,
\]
where $(x_n)$ is a deterministic sequence diverging to infinity in a direction $\theta$. He was able to show that under a strong (unproven) curvature assumption on the limit shape, the limit exists almost surely for Lebesgue-almost every direction. 

In 2008, Hoffman \cite{Hoffman} defined Busemann-type functions similar to the above, and introduced the idea of using Busemann functions to prove properties of geodesics. These ideas were continued by Damron-Hanson in 2013 \cite{DHanson}. In the next two sections, we will describe some of the ideas therein. For simplicity, we will assume that the $\tau_e$'s are i.i.d., with continuous distribution having finite mean, although this is far from necessary. 

Let $H_n$ be the hyperplane $\{x \in \mathbb{R}^d : x \cdot e_1 = n\}$ and define the function
\[
B_n(x,y) = T(x,H_n) - T(y,H_n) \text{ for } x,y \in \mathbb{Z}^d.
\]
For purposes of illustration, we will in this section assume that
\begin{equation}\label{eq: busemann_assumption}
B(x,y) := \lim_n B_n(x,y) \text{ exists almost surely}
\end{equation}
for $x,y \in \mathbb{Z}^d$. 
This assumption should be compared to assumption \eqref{Busemanningeneral}.

\begin{lemma}\label{lem: expected_busemann_simple}
Assume \eqref{eq: busemann_assumption}. One has
\[
\mathbb{E}B(x,y) = (x-y) \cdot \rho \text{ for } x,y \in \mathbb{Z}^d,
\]
where $\rho = e_1 g(e_1)$.
\end{lemma}
\begin{proof}
By translation invariance,
\[
\mathbb{E}B(x,y) = \mathbb{E}B(\theta(x),\theta(y))
\]
for any lattice translation $\theta$. As $B$ is additive and satisfies $\mathbb{E}B(0,-x) = - \mathbb{E}B(0,x)$, we find that $\mathbb{E}B(x,y)$ is a linear function of $x-y$. To find this function, it suffices to find $\mathbb{E}B(0,e_i)$ for $i=1, \ldots, d$. By symmetry, one has $\mathbb{E}B(0,e_i)=0$ for $i > 1$. To find the other expectation, we use an averaging trick introduced by Garet-Marchand and Hoffman (and also appears in some form in Kingman \cite[Eq.~(26)]{Kingman}). By translating by $e_1$, we obtain
\begin{align*}
\mathbb{E}T(0,L_n) &= \sum_{k=1}^n \mathbb{E}[T((k-1)e_1,H_n) - T(ke_1,H_n)] \\
&= \sum_{k=1}^n \mathbb{E}[T(0,H_{n-k+1}) - T(e_1, H_{n-k+1})] \\
&= \sum_{k=1}^n \mathbb{E}B_k(0,e_1).
\end{align*}
Using the bound $|B_k(x,y)| \leq T(x,y)$ and dominated convergence, one has
\[
\mathbb{E}B(0,e_1) = \lim_n \frac{1}{n} \mathbb{E}T(0,H_n) =\mu(e_1).
\]
The last equality follows from the shape theorem.
\end{proof}

Using the lemma, along with the ergodic theorem, one has almost surely and in $L^1$,
\[
\frac{1}{n} B(0,nx) = \frac{1}{n} \sum_{k=1}^n B((k-1)x,kx) \to x \cdot \rho \text{ for } x \in \mathbb{Z}^d.
\]
It is not difficult to upgrade this to a sort of shape theorem for the Busemann function. See, for example, Damron-Hanson \cite[Section 4]{DHanson}, where this is done for a ``reconstructed Busemann function.''
\begin{lemma}\label{lem: busemann_shape}
Assume \eqref{eq: busemann_assumption} and that $\mathbb{E}T(x,y) < \infty$ for all $x,y$. For each $\epsilon>0$,
\[
\mathbb{P}\left( |B(0,x) - x\cdot \rho| > \epsilon \|x\|_1 \text{ for only finitely many }x \right) = 1,
\]
where $\rho = e_1 g(e_1)$.
\end{lemma}

We now show how assumption \eqref{eq: busemann_assumption}, along with a mild curvature assumption, implies existence of geodesic rays with asymptotic direction $e_1$. Let $S$ be the set
\[
S = \partial \mathcal{B} \cap \{w : \rho \cdot w = 1\}.
\]
Because the latter set is a supporting hyperplane for the limit shape at the point $e_1/g(e_1)$, the set $S$ is a portion of the boundary $\partial \mathcal{B}$ containing $e_1/g(e_1)$. We will say that a geodesic ray with vertices $x_0,x_1, \ldots$ is asymptotically directed in $S$ if each limit point of $\{x_n/g(x_n)\}$ is contained in $S$. If $S$ contains only one point, we say that this ray has an asymptotic direction. If the point $e_1/g(e_1)$ is exposed (there is a hyperplane for $\mathcal{B}$ that touches $\mathcal{B}$ only at $e_1/g(e_1)$), then by symmetry, $S = \{e_1/g(e_1)\}$. In that case, the following theorem implies that almost surely, there is an infinite geodesic from 0 that has asymptotic direction $e_1$. It is an open problem to show without assumption \eqref{eq: busemann_assumption} that with positive probability, there is an geodesic ray with an asymptotic direction. The following theorem is similar to \cite[Theorem~5.3]{DHanson}.


\begin{theorem}
Assume \eqref{eq: busemann_assumption}. With probability one, every subsequential limit of geodesics from 0 to $H_n$ is asymptotically directed in $S$.
\end{theorem}
\begin{proof}
Let $\gamma$ be any subsequential limit of geodesics from 0 to $H_n$ and label its vertices $0=x_0,x_1,x_2, \ldots$. First note that for $m \geq 0$, one has
\[
B(0,x_n) = T(0,x_n).
\]
Indeed, choose $(n_k)$ a subsequence such that some geodesics $\gamma_{n_k}$ from $0$ to $H_{n_k}$ converge to $\gamma$. For $k$ large, the first $n$ vertices of $\gamma_{n_k}$ coincide with those of $\gamma$, so
\[
B(0,x_n) = \lim_N \left[ T(0,H_N) - T(x_n,H_N) \right] = \lim_k [ T(0,H_{n_k}) - T(x_n,H_{n_k}) ] = T(0,x_n).
\]

Choose $z$ to be any limit point of $\{x_n/g(x_n)\}$ so that for some subsequence $(n_j)$, one has $x_{n_j}/g(x_{n_j}) \to z$. Now by the Busemann shape theorem, almost surely for any choice of $\gamma$ and $(n_j)$,
\[
\lim_j \frac{B(0,x_{n_j})}{g(x_{n_j})} = \lim_j \rho \cdot (x_{n_j}/g(x_{n_j})) = \rho \cdot z.
\]
On the other hand, by the shape theorem and the above remarks, this equals
\[
\lim_j \frac{T(0,x_{n_j})}{g(x_{n_j})} = 1.
\]
Therefore $\rho \cdot z = 1$. Since $z \in \partial \mathcal{B}$, this means $z \in S$.
\end{proof}

\begin{cor}
Assume \eqref{eq: busemann_assumption} and that $e_1/g(e_1)$ is an exposed point of $\mathcal{B}$. With probability one, every subsequential limit of geodesics from 0 to $H_n$ has asymptotic direction $e_1$.
\end{cor}

\subsection{Busemann increment distributions and geodesic graphs}\label{sec: busemann_increments}

It is not known how to show that limits like \eqref{eq: busemann_assumption} exist, and this is one of the first and main obstacles when dealing with Busemann functions.
\begin{question}\label{q22}
Prove that under some conditions on the edge weights the limit \eqref{eq: busemann_assumption} exists.  
\end{question}
\begin{remark}
For the generalized Busemann functions introduced in Section \ref{sec:visit}, the question above translates to Question \ref{q10}. One may also want to consider Busemann limits to points rather than lines. 
\end{remark}

Here we will explain one method of getting around Question \ref{q22} to obtain partial results about directions of geodesic rays. This method shares some similarities with the plan of attack used in the original paper of Kingman, described in Section \ref{sec:visit}. (We make this relation precise below in Remark~\ref{rem: busemann}.) Instead of trying to establish the limit \eqref{eq: busemann_assumption}, we will focus on a weak limit by considering translational averages. We follow Damron-Hanson \cite{DHanson}. Their results were stated only for dimension $d=2$, but the existence of geodesics directed in sectors (the first theorem below) also holds for general dimensions, so we outline this more general argument.

We will focus on i.i.d. measures, although the theorems apply to a wide class of translation-ergodic measures. Consider the following conditions on the distribution $\mathbb{P}$ of the edge-weights:
\begin{enumerate}
\item[I.] $\mathbb{P}$ is i.i.d. with $\mathbb{E}Y^d < \infty$, where $Y$ is the minimum of $2d$ i.i.d. edge-weights. Furthermore, $\mathbb{P}(\tau_e=0)<p_c$, the bond percolation threshold for $d$ dimensions.
\item[II.] $\mathbb{P}$ has continuous marginals.
\end{enumerate}
Recall that a supporting hyperplane $H$ for $\mathcal{B}$ at a point $z_0 \in \partial \mathcal{B}$ is a hyperplane that contains $H$ and such that $\mathcal{B}$ does not intersect both components of $H^c$. If there is only one such hyperplane, we say that $\mathcal{B}$ is differentiable at $z_0$. In this case, we write $H_{z_0}$ for this hyperplane. Write
\[
S_{z_0} = H_{z_0} \cap \partial \mathcal{B}.
\]


\begin{theorem}\label{theorem: directions}
Assume (I). If $\partial \mathcal B$ is differentiable at $z_0$, then with probability one, there is a geodesic ray $\gamma$ containing the origin which is asymptotically directed in $S_{z_0}$. This means every limit point of $\{x/g(x) : x \in \gamma\}$ is contained in $S_{z_0}$.
\end{theorem}

Note that if $z_0$ is an exposed point of differentiability, then there is a geodesic ray with asymptotic direction $z_0$, since in that case, $S_{z_0} = \{z_0\}$. This is an improvement on Newman's theorem because we have replaced the global curvature condition with a local, directional condition.

The next result is related to coalescence of geodesic rays.
\begin{theorem}\label{theorem: coalesce_increments}
Take dimension $d=2$, assume (I) and (II) and that $\partial \mathcal{B}$ is differentiable at $z_0$. With probability one, there exists a collection of geodesic rays $(\Gamma_x : x \in \mathbb{Z}^2)$ satisfying the following properties:
\begin{enumerate}
\item Each $x \in \mathbb{Z}^2$ is a vertex of $\Gamma_x$.
\item Each $\Gamma_x$ is asymptotically directed in $S_{z_0}$.
\item For all $x,y \in \mathbb{Z}^2$, the paths $\Gamma_x$ and $\Gamma_y$ coalesce.
\item Each $x \in \mathbb{Z}^2$ is on $\Gamma_y$ for only finitely many $y \in \mathbb{Z}^2$.
\end{enumerate}
\end{theorem}

As discussed in Section \ref{sec:exponents},  Licea-Newman \cite{LN96} (along with an improvement by Zerner \cite{Newmanbook}) have shown (under a global curvature condition) that there exists a deterministic set $D \subset [0,2\pi)$ with countable complement such that, for each $\theta \in D$, the following holds almost surely.
\begin{enumerate}
\item (Existence) There exists a collection of infinite geodesics $(\Gamma_x : x \in \mathbb{Z}^2)$ such that each $\Gamma_x$ starts from $x$, has asymptotic direction $\theta$, and each $\Gamma_x$ and $\Gamma_y$ coalesce.
\item (Uniqueness) For each $x$, $\Gamma_x$ above is the only infinite geodesic starting from $x$ with asymptotic direction $\theta$.
\end{enumerate}
The above theorem shows that one can take $D = [0,2\pi)$ for existence, but does not address uniqueness.

\begin{proof}[Idea of proof]

The idea is to work with sub-sequential Busemann limits in distribution. First we set $\Omega_1 = [0,\infty)^{\mathcal{E}^d}$ to be a copy of our passage time space and let $\Omega_2 = \mathbb{R}^{\mathbb{Z}^d \times \mathbb{Z}^d}$ be our space for recording Busemann increments. Last, we have a space $\Omega_3 = \{0,1\}^{\vec{\mathcal{E}}^d}$, where $\vec{\mathcal{E}}^d$ is the set of oriented nearest-neighbor edges of $\mathbb{Z}^d$, which will record geodesic graphs. Put $\widetilde \Omega = \prod_{i=1}^3 \Omega_i$.

Let $H$ be any supporting hyperplane for $\mathcal{B}$ and let $\rho$ be the vector with $H = \{w : w \cdot \rho = 1\}$. Define $H_\alpha = \{w : w \cdot \rho = \alpha\}$ for $\alpha \in \mathbb{R}$. We now define Busemann increments and geodesic graphs toward $H_\alpha$. For an edge-weight configuration $\omega \in \Omega_1$, set
\[
B_\alpha(\omega) = \left( B_\alpha(x,y) : x,y \in \mathbb{Z}^d \right) \in \Omega_2,
\]
where $B_\alpha(x,y) = T(x,H_\alpha) - T(y,H_\alpha)$. Furthermore define the geodesic graph configuration $\eta_\alpha(\omega)$ by
\[
\eta_\alpha(\omega)(\langle x,y \rangle) = \begin{cases}
1 & \text{ if } \{x,y\} \in G_\alpha(z) \text{ for some }z \text{ and } B_\alpha(x,y) \geq 0 \\
0 & \text{ otherwise}
\end{cases}.
\]
Here, $G_\alpha(z)$ is the union of all edges in geodesics from $z$ to $H_\alpha$. Last, define $\Phi_\alpha : \Omega_1 \to \widetilde \Omega$ by
\[
\Phi_\alpha(\omega) = (\omega, B_\alpha(\omega), \eta_\alpha(\omega))
\]
and $\mu_\alpha$ the push-forward of $\mu$ through $\Phi_\alpha$. The measure $\mu_\alpha$ is a Borel probability measure on $\widetilde{\Omega}$.

A generic element of $\widetilde{\Omega}$ we write as 
\[
\widetilde \omega = ((t_e), B, \eta).
\]
$B$ is called the reconstructed Busemann function. We would like to take $\alpha$ to infinity and sample a configuration $\widetilde \omega$ from the limit of $(\mu_\alpha)$. Such a measure would be the distribution of Busemann increments and geodesic graphs ``at infinity'' in a direction related to $H$. However, we do not know how to show that this sequence converges, so we settle for sub-sequential limits of an average of $\mu_\alpha$ over $\alpha$. The averaging is done to ensure that any limit is translation invariant. That is, set
\[
\mu_k^* = \frac{1}{k} \int_0^k \mu_\alpha~\text{d}\alpha
\]
and let $\mu^{*}$ be any sub-sequential limit of $(\mu_k^*)$. (Tightness is not difficult to verify, due to the condition $|B_\alpha(x,y)| \leq T(x,y)$.) The measure $\mu_k^*$ can be thought of as representing first sampling a hyperplane $H_\alpha$ uniformly at random for $\alpha \in [0,k]$ and then sampling the configuration $\Phi_\alpha(\omega)$.


The idea is then to try to approximate the proof of last section to (a) find asymptotics of the function $B(0,x)$ as $x \to \infty$ and (b) use them to control geodesics. First, we give the mean of $B$, writing $\mathbb{E}_{\mu^{*}}$ for expectation relative to $\mu^{*}$. Recall that $\rho$ is the vector with $H_\alpha = \{w : w\cdot \rho = \alpha\}$.
\begin{lemma}\label{lem:todayisaniceday}
Assume (I). The mean of the reconstructed Busemann function is
\[
\mathbb{E}_{\mu^{*}} B(x,y) = \rho \cdot (y-x) \text{ for } x,y \in \mathbb{Z}^d.
\]
\end{lemma}
The idea of this lemma is that the Busemann function on average measures distances traveled perpendicular to $H$, but from a sort of stationary state, where the velocity is given by the time constant in direction $z_0$ (which is the same time constant corresponding to the line $H$). The reader should compare to Lemma~\ref{lem: expected_busemann_simple} in the previous section, where the mean of the Busemann function is similarly given by a projection along the vertical hyperplane $H$ used to create the function.
\begin{proof}
The proof uses a version of the averaging trick (inspired by Gou\'er\'e \cite{Gouere}) from last section. For $k \geq 1$ we can use stationarity to write
\begin{align*}
\mathbb{E}_{\mu_k^*} B(-x,0) &= \frac{1}{k} \left[ \int_0^k \mathbb{E}T(-x,H_\alpha)~\text{d}\alpha - \int_0^k \mathbb{E}T(0,H_\alpha)~\text{d}\alpha \right] \\
&= \frac{1}{k} \left[ \int_k^{k+\rho \cdot x} \mathbb{E}T(0,H_\alpha) ~\text{d}\alpha - \int_0^{\rho \cdot x}\mathbb{E}T(0,H_\alpha)~\text{d}\alpha \right].
\end{align*}
Choosing $k_j$ so that $\mu_{k_j}^* \to \mu^{*}$, one can show using translation invariance and the bound $\mathbb{E}_{\mu_{{k}_{j}}^*}B(-x,0)^2 \leq \mathbb{E}T(-x,0)^2$ that the left side converges to $\mathbb{E}_{\mu^{*}} B(0,x)$. So we obtain the formula
\[
\mathbb{E}_{\mu^{*}} B(0,x) = \lim_{j \to \infty} \int_0^{\rho \cdot x} \frac{\mathbb{E}T(0,H_{\alpha + k_j})}{k_j}~\text{d}\alpha.
\]
By the shape theorem, the integrand converges to 1, so by dominated convergence, we obtain $\rho \cdot x$. By translation invariance, then, of $\mu^*$, one has
\[
\mathbb{E}_{\mu^*} B(x,y) = \mathbb{E}_{\mu^*}B(0,y-x) = \rho \cdot (y-x).
\]
\end{proof}

As usual, we can prove asymptotics for the reconstructed Busemann function in the form of a shape theorem. However, due to the extra randomness introduced through $\alpha$, we may have lost ergodicity for $\mu$ under translations, and therefore obtain a random ``shape.''
\begin{lemma}[Shape theorem for $B$]\label{thm:Busemannlimitshape}
Assume (I). There exists a random vector $\varrho \in \mathbb{R}^2$ such that for any $\epsilon>0$,
\[
\mu^{*}\left( |B(0,x) - x \cdot \varrho| > \epsilon \|x\|_1 \text{ for infinitely many } x\in \mathbb{Z}^d \right) = 0.
\]
The vector $\varrho$ satisfies the following conditions:
\begin{enumerate}
\item $\mu^{*}$-almost surely, the hyperplane 
\[
H_\varrho := \{w \in \mathbb{R}^d : w \cdot \varrho = 1\}
\]
is a supporting line for $\mathcal{B}$ at $z_0$.
\item The mean of $\varrho$ under $\mu^*$ is $\rho$.
\end{enumerate}
\end{lemma}
The proof of the above result establishes first radial limits of the form $\lim_n B(0,nx)/n$ and patches them together using shape theorem arguments. A key tool is that the vector $\varrho$ is invariant under translating the edge-weights, Busemann increments, and geodesic graphs. The mean of $\varrho$ follows directly from the previous lemma, whereas the fact that $H_\varrho$ is a supporting line for $\mathcal{B}$ follows from the statements (a) $x \cdot \varrho \leq 1$ for all $x \in \mathcal{B}$ and (b) $z_0 \cdot \varrho = 1$. The claim (a) is shown by noting that for $x \in \mathcal{B}$, one has
\[
x \cdot \varrho = \lim_n \frac{B(0,nx)}{n} \leq \lim_n \frac{T(0,nx)}{n} = \mu(x) \leq 1.
\]
On the other hand, using the mean of $\varrho$, one has
\[
\mu^{*}(\varrho \cdot z_0 \leq 1)=1 \text{ and } \mathbb{E}_{\mu^{*}}(\varrho \cdot z_0) = \rho \cdot z_0 = 1.
\]
This implies that $\varrho \cdot z_0 = 1$ with $\mu^{*}$-probability one.

At this point, we may identify the asymptotics of the reconstructed Busemann function if $\partial \mathcal{B}$ is differentiable at $z_0$. In this case, there is a unique supporting hyperplane for $\mathcal{B}$ at $z_0$, so this must be $H_\varrho$. In other words, one has
\begin{corollary}
Assume (I). If $\partial \mathcal{B}$ is differentiable at $z_0$, then
\[
\mu^{*}(\varrho = \rho) = 1.
\]
\end{corollary}
In the presence of differentiability, the asymptotics of the Busemann function are given exactly by projection onto the line $H$ used to create the function. This result gives us a major piece needed to establish directional properties of geodesic rays sampled from $\mu^{*}$. To complete the proofs of the directional results, we simply need to argue as in the last section, combining asymptotics of $B$ with the shape theorem for $T$. 

To do this, we consider the configuration $\eta \in \Omega_3$ sampled from $\mu^{*}$ and build a directed graph from it. Specifically, we set $\mathbb{G} = \mathbb{G}(\eta)$ to be the directed graph with vertex set equal to $\mathbb{Z}^d$ and with edge set equal to $\{\langle x,y \rangle : \eta(\langle x,y \rangle) = 1\}$. Simple properties of geodesics carry over to the weak limit. For example: with $\mu^{*}$-probability one, for all $x,y \in \mathbb{Z}^d$,
\begin{itemize}
\item Each directed path in $\mathbb{G}$ is a geodesic and from each $x$ there is an infinite self-avoiding directed path.
\item If $x \to y$ in $\mathbb{G}$ (there is a directed path from $x$ to $y$ in $\mathbb{G}$), then $f(x,y) = T(x,y)$.
\item Under assumption (II), viewed as an undirected graph, $\mathbb{G}$ has no circuits.
\item Each vertex $x$ has out-degree at least 1 in $\mathbb{G}$. Under assumption (II), the out-degree is exactly 1, and thus emanating from each $x$ is a unique infinite directed path $\Gamma_x$.
\end{itemize}

Given these properties, the next result follows as in the last section.
\begin{theorem}
Assume (I). With $\mu^{*}$ probability 1, each directed infinite path from 0 in $\mathbb{G}$ is asymptotically directed in $S_\varrho = \partial \mathcal{B} \cap H_\varrho$.
\end{theorem}
The ``asymptotic shape'' for the Busemann function has boundary which is the line $H_\varrho$, whereas the asymptotic shape for $T$ is simply the limit shape for the FPP model. The picture here is that where these two make contact are angles in which geodesic rays are created. Therefore if $\mathcal{B}$ has an exposed point in $H_\varrho$, then we can produce geodesic rays with asymptotic directions. To use this theorem, one should know the exact line $H_\varrho$, so we impose a condition of differentiability. By pulling this result back to the space $\Omega_1$, we can obtain Theorem~\ref{theorem: directions}.

The proof of the coalescence statement of Theorem~\ref{theorem: coalesce_increments} follows the argument of Licea-Newman outlined in Section~\ref{sec:exponents}. The main difference is that one must be careful to construct barrier events on the original space $\Omega_1$ and port them over to $\widetilde \Omega$, and this creates a considerable headache. We refer the reader to \cite[Section 6]{DH}. As for absence of infinite backward paths, we have
\begin{theorem}\label{thm: no_bigeo_BK}
Assume (I) and (II) in dimension $d=2$. With $\mu^{*}$ probability 1, the set $\{y : 0 \in \Gamma_y\}$ is finite.
\end{theorem}
\begin{proof}
The idea is a modification of that of Burton-Keane \cite{BK89}, from their proof of uniqueness of the infinite occupied cluster in Bernoulli percolation. For $x \in \mathbb{Z}^2$, define the event $A_x$ that in $\mathbb{G}$, $x$ has an infinite forward path and two disjoint infinite backward paths. A counting argument shows that in the box $B(n) = [-n,n]^2$, one has surely
\[
\# \{x \in B(n) : A_x \text{ occurs}\} \leq Cn,
\]
as the number of points for which $A_x$ occurs is bounded above by the number of points on the boundary of $B(n)$. On the other hand, if we show that $\mathbb{P}(A_x) > 0$, then translation invariance gives
\[
\mathbb{E}_{\mu^{*}} \#\{x \in B(n) : A_x \text{ occurs}\} = \#B(n) \mu^{*}(A_x),
\]
and this would be a contradiction for large $n$.

Assume then that with positive probability, the set $\{y : 0 \in \Gamma_y\}$ is infinite; we will show that this implies $\mu^{*}(A_x)>0$. By directedness of paths in $\mathbb{G}$, we may find a line $L$ intersecting $\mathbb{Z}^2$ such that each $\Gamma_y$ eventually stays on one fixed side of $L$. Choose $y$ on the other side of $L$ and follow $\Gamma_y$ on the event that $\{z : y \in \Gamma_z\}$ is infinite until its last intersection with $L$. Such a point has an infinite forward path in $\mathbb{G}$ that does not intersect $L$ except at its initial point, and has an infinite backward path. By the ergodic theorem, we can then find $x_1,x_2 \in L$ such that with positive $\mu^{*}$ probability, each of $x_1$ and $x_2$ has an infinite forward path that does not touch $L$ except at the initial point and has an infinite backward path. Because all forward paths in $\mathbb{G}$ coalesce, we may choose $x$ to be the coalescence point of $\Gamma_{x_1}$ and $\Gamma_{x_2}$ and so $A_x$ occurs. By translation invariance, $\mu^{*}(A_x)>0$.
\end{proof}

Combining the previous theorems, this ends the proof of Theorem \ref{theorem: coalesce_increments}.
\end{proof}

\begin{remark}\label{rem: busemann}
Here we explain the relation between the above method and Kingman's approach. The analogue of Kingman's variables defined in \eqref{eq:metoohahaha} and \eqref{eq:sohumble} would be, respectively
\[
f_k = \frac{1}{k} \sum_{j=1}^k(T(0,H_j) - T(e_1,H_j)) = \frac{1}{k} \sum_{j=1}^k B_j(0,e_1)
\]
and the replacement for
\[
f_k + Tf_k + \ldots + T^{n-1}f_k
\]
would be
\[
\frac{1}{k} \sum_{j=1}^k B_j(0,ne_1).
\]
(These last two are not quite equal, but the average of Busemann increments is more natural in our setup.) Note that $f_k$ is similar to the projection of our averaged measure $\mu_k^*$ on the Busemann field coordinate, but it is one level lower (at the level of random variables rather than measures). To make this precise, suppose that $f_k$ converges almost surely to some random variable $f$ and that $\mu_k^*$ converges to a measure $\mu$. Then one would have
\[
f = \mathbb{E}_\mu \left[ B(0,e_1) \mid (t_e) \right],
\]
where we have conditioned on the edge-weight configuration $(t_e)$. In other words, $f$ would be an average of the reconstructed Busemann function $B(0,e_1)$ over the additional randomness we introduced into $\mu$ (the uniformly random hyperplane $H_n$).

In fact, exactly this approach is taken in Liggett's proof of his improved subadditive ergodic theorem \cite{LiggettST}. In that result, Kingman's assumptions are weakened, and the question of building a decomposition into an additive and nonnegative subadditive process is handled using weak limits, instead of employing weak-$*$ compactness to find a limit point for the sequence $(f_k)$ in $(L^1)^{**}$, as Kingman did. Precisely, for a subadditive ergodic sequence $(X_{m,n})$, Liggett defined an independent uniform $\{1, \ldots, n\}$ random variable $U_n$ and set
\[
Y_i^{(n)} = X_{0,i+U_n}-X_{0,i+U_n-1}.
\]
Note the distribution of $Y_i^{(n)}$ is the same as that given by averaging the distributions of $X_{i,i+j}-X_{i+1,i+j}$ for $j=1, \ldots, n$. Then he sets $(Y_1, Y_2, \ldots)$ to be any subsequential limit in distribution of the sequences $(Y_1^{(n)}, Y_2^{(n)}, \ldots)$ and uses the distributional monotonicity
\[
(Y_1, Y_{1}+Y_2, Y_{1}+Y_{2}+Y_{3} \ldots) \leq_{st} (X_{0,1}, X_{0,2}, X_{0,3} \ldots)
\]
in place of Kingman's $f \leq X_{0,1}$. If one wants to recover a pointwise monotonicity, one can then simply couple these sequences together with a standard coupling.

One could ask whether the approach used in this section could work by only considering quantities at the level of the random variable $f_k$ rather than at the level of $\mu_k^*$. This may not be possible in general translation-ergodic environments. Suppose that geodesics from $0$ to $L_n$ and $e_1$ to $L_n$ have exactly two pairs of subsequential limiting geodesics $\gamma_0^{(1)}, \gamma_0^{(2)}$ and $\gamma_{e_1}^{(1)}, \gamma_{e_1}^{(2)}$. Then if $\gamma_0^{(i)}$ and $\gamma_{e_1}^{(i)}$ coalesce for each fixed $i$, one would expect the Busemann functions constructed from these limiting geodesics to be different. However it is still possible that the limit $f$ of $f_k$ exists, assuming the subsequences on which we took limits are regular enough. Then $f$ is simply an average of these two Busemann functions corresponding to the different geodesics and for this reason, it will no longer necessarily have property (a) from Lemma~\ref{lem:busproperties}. That is, if only one of the two geodesics passes over the edge from $0$ to $e_1$, then it need not be true that $f = \tau_{\{0,e_1\}}$. This property is essential for deriving directional properties of geodesics from their Busemann functions.
\end{remark}

\subsection{Busemann functions along boundaries in $\mathbb{Z}^2$}
In one case, Busemann limits such as \eqref{eq: busemann_assumption} can be shown to exist surely, and this implies that limits exist almost surely for sequences of finite geodesics. We explain in this section the results of Auffinger-Damron-Hanson \cite{ADH1}, where these statements are proved when Busemann limits are taken to points on boundaries of subsets in $\mathbb{Z}^2$. The existence of a boundary will allow for a ``paths-crossing'' trick due to Alm-Wierman \cite{AW99} (first introduced in \cite{Alm} in a special case), and existence of limits follow from this.

For simplicity, the results below will be stated for the half plane with vertices 
\[
V_H = \{(x,y) \in \mathbb{Z}^2 : y \geq 0\}
\]
and edges $E_H = \{\{x,y\} : x,y \in V_H\}$, although we will remark about extending them to general subsets. For $x,y \in V_H$, let $T_H(x,y)$ be the minimum passage time among all paths with all vertices in $V_H$ from $x$ to $y$.
\begin{theorem}
Let $x_n = (n,0)$ and let $(\tau_e)$ be any edge-weight configuration in $[0,\infty)^{E_H}$. For all $x,y \in V_H$, the Busemann limit to $x_n$ exists:
\begin{equation}\label{eq: busemann_sequence_pizza}
B_H(x,y) = \lim_n \left[ T_H(x,x_n) - T_H(y,x_n) \right].
\end{equation}
\end{theorem}

As usual, existence of Busemann limits gives us quite a bit of control on geodesics. Using this theorem, we can prove existence of limiting geodesic graphs. Formally, we represent geodesic graphs as elements of a directed graph space, as in the last section. Let $\vec{E}_H$ be the set of directed edges of $V_H$
\[
\vec{E}_H = \{\langle x,y \rangle : x,y \in V_H\}
\]
and write $\eta$ for an arbitrary element of $\{0,1\}^{\vec{E}_H}$. Build the graph $\mathbb{G} = \mathbb{G}(\eta)$ as follows: the vertices are all the vertices of $V_H$ and an edge $e \in \vec{E}_H$ is present in the graph if and only if $\eta(e) =1$. For a sequence of vertices $(x_n)$, we let $\eta_n(e) = 1$ if $e=\langle x,y \rangle$ is in a geodesic from some point to $x_n$ and $T_H(x,x_n) \geq T_H(y,x_n)$; we then set $\mathbb{G}_n = \mathbb{G}(\eta_n)$. The graphs $\mathbb{G}_n$ converge to a graph $\mathbb{G} = \mathbb{G}(\eta)$ if for each $e \in \vec{E}_H$, one has $\eta_n(e) \to \eta(e)$.

\begin{theorem}\label{theorem: limiting_geodesics}
Let $\mathbb{P}$ be a probability measure on $[0,\infty)^{E_H}$ such that
\begin{equation}\label{eq: geodesics_exist}
\mathbb{P}(\exists \text{ geodesic between } x,y  \text{ for all } x,y \in V_H) = 1.
\end{equation}
Then with probability one, $(\mathbb{G}_n)$ converges to a graph $\mathbb{G}$. Each directed path in $\mathbb{G}$ is a geodesic.
\end{theorem}

\begin{remark}
The above two theorems are valid in the following more general setting. Let $V$ be any subset of $\mathbb{Z}^2$ that is connected and infinite and has infinite connected complement. (The reader can think of slit planes, sectors, etc.) There is a unique doubly infinite path in $\mathbb{Z}^2 + (1/2,1/2)$ which separates $V$ and $V^c$. Enumerate the vertices in $V$ that are adjacent to this dual path as $\ldots, x_{-1},x_0,x_1, \ldots$. If we define passage times between vertices $x,y$ in $V$ to be $T_V(x,y)$, the minimum passage time among all paths with all vertices in $V$ connecting $x$ and $y$, then the Busemann limits $B_V(x,y) = \lim_n \left[ T_V(x,x_n) - T_V(y,x_n) \right]$ exist for all edge-weight configurations in $[0,\infty)^{E_V}$, and limiting geodesic graphs exist almost surely for any measure on $[0,\infty)^{E_V}$ that has the geodesic property \eqref{eq: geodesics_exist}. Here, $E_V$ is the set of nearest-neighbor edges with both endpoints in $V$.
\end{remark}

In the case of the half-plane, one can say more about the structure of the limiting graph $\mathbb{G}$. We will state the theorem, but we refer to \cite{ADH2} for the complete proof.
\begin{theorem}
Let $\mathbb{P}$ be a product measure on $[0,\infty)^{E_H}$ with continuous marginals. The limiting geodesic graph $\mathbb{G}$ from Theorem~\ref{theorem: limiting_geodesics} satisfies the following almost surely.
\begin{enumerate}
\item Each vertex in $V_H$ has out-degree 1. Therefore from each $x \in V_H$ emanates a unique infinite directed path $\Gamma_x$.
\item Viewed as an undirected graph, $\mathbb{G}$ has no circuits.
\item For each $x\in V_H$, the backward cluster $B_x = \{y \in V_H : x \in \Gamma_y\}$ is finite.
\item For all $x,y \in V_H$, $\Gamma_x$ and $\Gamma_y$ coalesce. That is, their edge symmetric difference is finite.
\end{enumerate}
\end{theorem}

\begin{proof}
We give the ideas of the proofs of the first two items only. Busemann limits are proved as a consequence of the ``paths crossing'' trick of Alm and Wierman \cite{AW99}. We begin by showing the limit for $x = m_1e_1, y = m_2 e_1$ with $m_1 < m_2 \in \mathbb{Z}$. For simplicity, let us assume that for each $w,z \in V_H$, there is a geodesic between $w$ and $z$. If this is not the case, then the geodesics can be replaced by paths that have passage time within $\epsilon$ of the infimum.

Let $n_2 > n_1 > m_2$ and let $\sigma_1$ be a geodesic from $x$ to $n_1e_1$, with $\sigma_2$ a geodesic from $y$ to $n_2e_1$. Note that by planarity, the paths $\sigma_1$ and $\sigma_2$ must share a vertex $z$. Define the path $\hat{\sigma}_1$ by traversing $\sigma_1$ from $x$ to $z$ and then $\sigma_2$ from $z$ to $n_2e_1$. Define $\hat{\sigma}_2$ by traversing $\sigma_2$ from $y$ to $z$ and then $\sigma_1$ from $z$ to $n_1e_1$. Then (see Figure \ref{fig:AWtrick}),
\begin{align*}
T_H(x,n_1e_1) + T_H(y,n_2e_1) = T_H(\sigma_1) + T_H(\sigma_2) &= T_H(\hat{\sigma}_1) + T_H(\hat{\sigma}_2) \\
&\geq T_H(x,n_2e_1) + T_H(y,n_1e_1).
\end{align*}

\begin{figure}[h]
\scalebox{1.4}{ 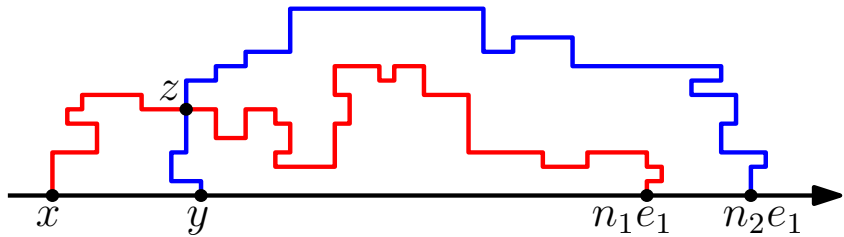}
 \centering
 \def \svgwidth{3000pt}
 \caption{``Paths crossing'' trick: the red path $\sigma_{1}$  and the blue path $\sigma_{2}$ must intersect at a point $z$.  }
 \label{fig:AWtrick}
\end{figure}

Rearranging this, we obtain
\[
T_H(x,n_1e_1) - T_H(y,n_1e_1) \geq T_H(x,n_2e_1) - T_H(y,n_2e_1).
\]
Therefore the sequence in \eqref{eq: busemann_sequence_pizza} is monotone and the limit exists.

To generalize this off the axis, we take $x \in V_H$ be of the form $(x_1,x_2)$ with $x_2>0$ and $x' \in V_H$. Let $B$ be a set of vertices of the type $[-k,k] \times [0,k]$ for some $k$ large enough so that $B$ contains both $x$ and $x'$. The set $V' := V_H \setminus B$ is connected and infinite with complement that is also connected and infinite. So we define for $y,z \in V'$ the restricted passage time $T'(y,z)$ to be the minimum of passage time of all paths from $y$ to $z$ which have only vertices in $V'$. One may then repeat the above paths crossing argument to see that if $y,z$ are on $\partial V'$, the inner boundary of $V'$ (they have a $\mathbb{Z}^2$-neighbor which is in $B$), then the limit $$B'(y,z) = \lim_n \left[ T'(y,x_n) - T'(z,x_n)\right]$$ exists. Furthermore, for all such $y,z$, the sequence defining $B'(y,z)$ is monotone.

Now the idea is to use existence of Busemann limits in the restricted region to prove it for $x$ and any other point. The crucial point is that for large $n$,
\[
T(x,x_n) = \min \{T(x,y) + T'(y,x_n) : y \in \partial V'\}.
\]
For $z \in \partial V'$ fixed and $y \in \partial V'$ variable, we set $\psi_n(y) = T(x,y) + T'(y,x_n)-T'(z,x_n)$. We find then that $\psi_n(y)$ has a finite limit $\psi(y)$ for each $y$ and furthermore
\[
\lim_n \left[ T(x,x_n) - T'(z,x_n) \right] = \min \{ \psi(y) : y \in \partial V'\}.
\]
The same argument works if we replace $x$ by $x'$, so $\lim_n \left[ T(x,x_n)-T(x',x_n) \right]$ exists. 

To prove existence of limiting geodesic graphs, we take $e = \langle x,y \rangle \in \vec{E}_H$ and again let $B$ be a box of the form $[-k,k]\times [0,k]$ which is large enough to contain $x$ and $y$. If $N$ is large enough so that $x_n \notin B$ for all $n \geq N$, then if we again put $V' = V_H \setminus B$ one has for all $y \in \partial V'$,
\[
y \text{ is on a geodesic from }x \text{ to }x_n \Leftrightarrow \psi_n(y) = \min\{\psi_n(u) : u \in \partial V'\},
\]
where $\psi_n(y)$ was defined as above: $$T(x,y) + T'(y,x_n) - T'(z,x_n)$$ for the fixed $z \in \partial V'$. However, due to monotonicity of the differences $T'(y,x_n) - T'(z,x_n)$ in $n$ as $y$ ranges over $\partial V'$, the set of minimizers $\mathfrak{m}_n$ of $\psi_n$ is eventually constant for $n$ large. Thus, the set of points in $\partial V'$ which are in geodesics from $x$ to $x_n$ is eventually constant in $n$. 

The above work readily implies that $\mathbb{G}_n$ converges. Indeed, one can check that $\eta_n(e)=1$ if and only if $e$ is in a geodesic from $x$ to a vertex of $\mathfrak{m}_n$. Since $\mathfrak{m}_n$ is eventually constant, the value of $\eta_n(e)$ must also be eventually constant.
\end{proof}

\subsection{Nonexistence of Bigeodesics in fixed directions}\label{sec: bigeo_new}

Recently, Busemann functions have been used in last-passage percolation models by Sepp\"al\"ainen, Rassoul-Agha, and Georgiou to prove analogues (and improved versions) of the directional results from \cite{DHanson} in FPP. This is a big advance because most work on LPP has assumed exponential or geometric weights, where random matrix and queueing theory analysis can be used. 

A main tool in the general LPP models is the directedness of paths, which allows one to use various forms of the ``paths crossing'' argument of Alm and Wierman. Applied to Busemann functions, this gives a certain monotonicity of directional Busemann functions. In FPP, paths are not directed, and this creates a fundamental obstacle to obtaining similar improvements. 

In this section, we describe recent work of Damron-Hanson \cite{DHanson2} which shows that the LPP results are also valid in FPP. These theorems address the ``uniqueness'' issue in Newman's conjectures about infinite geodesics, and consequently rule out bigeodesics in fixed directions. In that paper, the following theorem is proved for general passage-time distributions considered in the first work of Damron-Hanson \cite{DHanson}, but we will focus again on i.i.d. weights, assuming I and II from Section~\ref{sec: busemann_increments}.

We will again make some assumptions on the limiting shape: let $\theta \in [0,2\pi)$ and let $v_\theta$ be the point on the boundary of the limit shape in direction $\theta$. When $v_\theta$ is a point of differentiability of the boundary of the limit shape, let $L_\theta$ be the unique tangent line and let $S_\theta$ be the sector of angles of contact of $L_\theta$ with $\partial \mathcal{B}$. Let $\theta_1$ and $\theta_2$ be the endpoints of $S_\theta$.

\begin{theorem}[Damron-Hanson \cite{DHanson2}]\label{thm: new_DH}
Assume I, II, and that the limit shape boundary is differentiable at $\theta, \theta_1$, and $\theta_2$. The following hold with probability one.
\begin{enumerate}\label{thm: uniqueness_upgrade}
\item For each $x \in \mathbb{Z}^2$, there is an infinite geodesic $\Gamma_x$ that is directed in $S_\theta$ such that for any (random) sequence $(x_n)$ of points directed in $S_\theta$,
\[
\Gamma_x = \lim_n \Gamma(x,x_n).
\]
\item For each $x,y \in \mathbb{Z}^2$, the geodesics $\Gamma_x$ and $\Gamma_y$ coalesce.
\item There are no bigeodesics with one end directed in $S_\theta$.
\end{enumerate}
\end{theorem}
Notice that this result reduces Licea-Newman's ``exceptional set'' $D^c$ (see Theorem~\ref{theorem:LN}) of directions to be empty. Furthermore, their result rules out bigeodesics with both ends directed off the exceptional set, whereas the above theorem only requires directedness of one end (in a sector).

From this result we obtain the following progress on the bigeodesic conjecture.
\begin{corollary}
Assume I, II and that the limit shape boundary is differentiable. Then for each $\theta$,
\[
\mathbb{P}(\text{there is a bigeodesic with one end in direction }\theta)=0.
\]
\end{corollary}

\begin{remark}
Since geodesic lines with fixed directions cannot exist, one should ask if infinite geodesics are even required to have directions. One can show using planarity and the results of \cite{DHanson} that if $\partial B$ is differentiable and I and II hold, then the following statements are true with probability one:

\begin{enumerate}
\item for all $\theta$, there is an infinite geodesic starting from $0$ directed in $S_{\theta}$ and
 \item every infinite geodesic is directed in $S_{\theta}$ for some $\theta$.
 \end{enumerate}
\end{remark}

As a consequence of the theorem and arguments like those presented elsewhere in this section, one has
\begin{theorem}
Assume I, II, and that the limit shape boundary is differentiable at $\theta, \theta_1$, and $\theta_2$. With probability one, for each $x,y \in \mathbb{Z}^2$ and sequence $(x_n)$ directed in $S_\theta$, the limit
\[
B(x,y) = \lim_n[T(x,x_n) - T(y,x_n)]
\]
exists. Furthermore, letting $\rho$ be the unique vector in $\mathbb{R}^2$ such that $\{r \in \mathbb{R}^2 : r \cdot \rho = 1\}$ is the tangent line to $\mathcal{B}$ in direction $\theta$, one has:
\begin{enumerate}
\item $\mathbb{E}B(0,x) = \rho \cdot x$ for all $x \in \mathbb{Z}^2$.
\item For each $\epsilon>0$, the set of $x \in \mathbb{Z}^2$ such that $|B(0,x) - \rho \cdot x| > \epsilon \|x\|_1$ is finite.
\end{enumerate}
\end{theorem}

In these theorems, one may ask the degree to which the differentiability assumption is necessary. The main obstacle to removing differentiability is the result of H\"aggstrom-Meester (Remark~\ref{rem: HM} and Theorem~\ref{thm:HM}), which we recall states that given any compact, convex subset $\mathcal{C}$ of $\mathbb{R}^d$ that has the symmetries of $\mathbb{Z}^d$ that fix the origin, there exists a translation-ergodic distribution of edge-weights whose limit shape is $\mathcal{C}$. In particular, there are models of FPP whose limit shapes are polygons. In these situations, it is reasonable to believe that one can construct models in which the only infinite geodesics are ones which do not have asymptotic directions -- they wander across the sectors corresponding to the sides of the limit shape. This leads one to consider directedness in sectors. However, it is reasonable to expect some FPP models which have polygonal limit shapes, and for which there are infinite geodesics directed along the corners. In this case, uniqueness of infinite geodesics in sectors corresponding to these sides will not hold. This is prevented by making a differentiability assumption.

\begin{remark}\label{rem: BKS_midpoint}
The above results give progress toward solving the ``BKS midpoint problem'' introduced in \cite{BKS}. The question is: is it true that the probability that $(n/2)e_1$ is in a geodesic from 0 to $ne_1$ goes to zero at $n \to \infty$? By a translation, if the answer is no, then 
\[
\mathbb{P}\left(0 \text{ is in a geodesic from } (-n/2)e_1 \text{ to } (n/2)e_2 \text{ for infinitely many }n\right)>0.
\]
On this event, 0 is in a bigeodesic, and under the differentiability assumptions of Theorem~\ref{thm: new_DH}, one of its ends is directed in the sector corresponding to $\theta=0$. This is impossible by the theorem.
\end{remark}

\begin{proof}[Sketch of proof of Theorem~\ref{thm: uniqueness_upgrade}] 

The proof we give is more similar to the proof in LPP \cite{GRS} because, for ease of exposition, we will omit the many applications of the Jordan curve theorem needed to deal with undirected paths.

For simplicity, we will take $\theta = \pi/2$ and $\pi > \theta_1 \geq \theta \geq \theta_2 > 0$. We will first reduce to the half-plane, where it is easier to order infinite geodesics. We will skip this step, although it is the most work, because it does not involve Busemann functions, but just topological arguments. However we mention that this reduction is done using the fact that every path that is directed in $S$ must have at most finitely many intersections with $L_0 = \{(x_1,0) : x_1 \in \mathbb{Z}\}$. So consider the upper half-plane $\mathbb{H}$ with vertices $V_H$ and nearest-neighbor edges $E_H$, and define passage times $T_H(x,y)$ for vertices $x,y \in V_H$ by considering only paths in the upper half-plane. We will content ourselves with proving the theorem for geodesics constructed in the upper half-plane using $T_H$.

Let $\mathcal{G}_H(x)$ be the set of infinite geodesics starting from $x$ in the half-plane which are directed in $S = S_\theta$. By ordering in the half-plane, for each $x \in L_0$, there is a left-most infinite geodesic in $\mathcal{G}_H(x)$, written as $\Gamma_{x,H}^L$. Similarly there is a right-most infinite geodesic written as $\Gamma_{x,H}^R$. The main step is to argue that that
\begin{equation}\label{eq: left_right_equal}
\Gamma_{0,H}^L = \Gamma_{0,H}^R \text{ with probability one}.
\end{equation}
Once we show this, then part 1 of the theorem follows from the following lemma, whose proof uses the assumption that $\partial \mathcal{B}$ is differentiable at the endpoints $\theta_1,\theta_2$.
\begin{lemma}\label{lem: directed_in_S}
With probability one, the following holds. If $(x_n)$ is any (random) sequence of points directed in $S$, and $\Gamma$ is a subsequential limit of half-plane geodesics from $0$ to $x_n$, then $\Gamma$ is directed in $S$.
\end{lemma}
The idea of the proof of this lemma is to use trapping. Since $\partial \mathcal{B}$ is differentiable at $\theta_1$, one can find infinitely many extreme points of $\mathcal{B}$ which converge to $v_{\theta_1}$ outside of $S$. One then derives a half-plane version of the existence of infinite geodesics directed in sectors from \cite{DHanson}, and uses this to construct infinitely many geodesics which are directed in distinct sectors whose endpoints converge to $v_{\theta_1}$ from outside $S$. The subsequential limit $\Gamma$ cannot intersect any of these geodesics more than once, due to unique passage times, and so it must be directed ``to the right'' of $\theta_{1}$. A similar argument works for $\theta_2$ and traps $\Gamma$ in $S$.

If we assume \eqref{eq: left_right_equal} and use the previous lemma, we find that any subsequential limit of geodesics from 0 to a (random) sequence $(x_n)$ directed in $S$ must itself be directed in $S$, and must then sit between $\Gamma_{0,H}^L$ and $\Gamma_{0,H}^R$. This means all three geodesics are equal and this proves item 1.

Now we argue for \eqref{eq: left_right_equal}. Similarly to the definition of left and right-most geodesics in the half-plane, one can define left-most and right-most infinite geodesics from $x$ in the full-plane; we write these as $\Gamma_x^L$ and $\Gamma_x^R$. One can show that with positive probability, $\Gamma_{x,H}^L = \Gamma_x^L$ and similarly for $\Gamma_{x,H}^R = \Gamma_x^R$. This is argued using the fact that any path that is directed in $S$ and starts on $L_0$ has a last intersection with $L_0$. One can then use a Licea-Newman style argument to show that each $\Gamma_x^L$ and $\Gamma_y^L$ coalesce (and similarly for right-most) and then deduce this same statement for the half-plane geodesics. This allows one to define Busemann functions
\[
B_H^*(x,y) = \lim_n \left[ T_H(x,x_n) - T_H(y,x_n) \right]
\]
for $*=L$ or $R$, where $(x_n)$ is the sequence of vertices on $\Gamma_{0,H}^*$. A similar statement holds for full-plane geodesics and in this case, one may define the Busemann functions on all of $\mathbb{Z}^2$.

By the ergodic theorem, if we set $\Delta_H(x,y)$ to be $B_H^L(x,y) - B_H^R(x,y)$ (and similarly for $\Delta(x,y)$ as the difference of full-plane Busemann functions), then there is a $c$ such that
\[
\frac{1}{n} \Delta_H(0,ne_1) \to c \text{ and } \frac{1}{n} \Delta_H(0,-ne_1) \to -c
\]
almost surely. Furthermore, one can use equality of full-plane and half-plane geodesics (with positive probability) to show that these limits exist for $\Delta$ (the full-plane difference) and are equal to $c$ and $-c$, respectively.

Now \eqref{eq: left_right_equal} follows from two results:
\begin{lemma}
Under our differentiability assumption on $\partial \mathcal{B}$, $c=0$.
\end{lemma}

\begin{proof}
For $*=L,R$ the function $x \mapsto \mathbb{E}B^*(0,x)$ is linear, and so there is a vector $\rho^*$ such that
\[
\mathbb{E}B^*(x,y) = \rho^* \cdot (y-x) \text{ for } x,y \in \mathbb{Z}^2.
\]
Once again, one may upgrade this to a shape theorem: for each $\epsilon>0$, 
\[
\mathbb{P}\left( |B^*(0,x) - \rho^* \cdot x| > \epsilon \|x\|_1 \text{ for finitely many }x \in \mathbb{Z}^2\right) = 1.
\]
Label the last intersections of $\Gamma_0^*$ with $L_0, L_1, L_2, \ldots$ as $x_0^*, x_1^*, x_2^*, \ldots$. (Here, $L_k = L_0 + ke_2$.) Then choose a subsequence $(x_{n_k}^*)$ such that $x_{n_k}^* / \|x_{n_k}^*\|_1$ converges to some vector $z$. By the shape theorem,
\[
B^*(0,x_n^*)/\|x_n^*\|_1 = T(0,x_n^*)/\|x_n^*\|_1 \to \mu(z).
\]
On the other hand, this converges to $\rho^*\cdot z$, so
\[
\rho^* \cdot z = \mu(z).
\]
Furthermore, as usual, $\rho^* \cdot r \leq \mu(r)$ for all $r \in \mathbb{R}^2$, using the fact that Busemann functions are bounded by the passage time. As $z$ must have argument in $S$, one concludes that
\[
L^* = \{r \in \mathbb{R}^2 : r \cdot \rho^* = 1\}
\]
is a supporting line for the limit shape at $z$. By differentiability, there is only one such supporting line, and so we conclude that $\mathbb{E}B^*(0,e_1) = 0$ for both $*=L,R$. The ergodic theorem completes the proof.
\end{proof}

\begin{lemma}
With probability one, $\Delta_H(0,e_1) \geq 0$. If $\Gamma_{0,H}^L \neq \Gamma_{0,H}^R$ with positive probability, then $c>0$.
\end{lemma}
\begin{proof}
If $\Gamma_{0,H}^L \neq \Gamma_{0,H}^R$, then by coalescence, the same must be true for geodesics starting from any point $x$ on $L_0$. Since the edge-weights are continuous, one then deduces that $\Delta(0,e_1) \neq 0$. So if we can show that $\Delta(0,e_1) \geq 0$ almost surely, one then has $\mathbb{E}\Delta(0,e_1) > 0$, and the ergodic theorem will prove the lemma.

To show the almost sure inequality, we use the ``paths crossing'' trick. We claim that $\Gamma_{0,H}^R$ and $\Gamma_{e_1,H}^L$ must share a vertex. If $\Gamma_{0,H}^L = \Gamma_{0,H}^R$, then by coalescence, this is obvious. Otherwise if follows from the fact that $\Gamma_{0,H}^R$ coalesces with $\Gamma_{e_1,H}^R$, which is to the right of $\Gamma_{e_1,H}^L$. Let $y$ be a vertex on $\Gamma_{e_1,H}^L$ beyond $z$ and let $y'$ be a vertex on $\Gamma_{0,H}^R$ beyond $z$. Then
\begin{align*}
T_H(0,y') + T_H(e_1,y) &= T_H(0,z) + T_H(z,y') + T_H(e_1,z) + T_H(y,z) \\
&\geq T_H(0,y) + T_H(e_1,y').
\end{align*}
Rearranging and taking $y,y' \to \infty$ along their respective geodesics, we obtain $B_H^L(0,e_1) \geq B_H^R(0,e_1)$.
\end{proof}

Now that $\Gamma_{0,H}^L = \Gamma_{0,H}^R$ almost surely, coalescence of left and right-most geodesics implies item 2 of the theorem. For the last item, we use the argument of Theorem~\ref{thm: no_bigeo_BK}. To recall, this result states that in general cases in which from each point $x$, there is an infinite geodesic $\Gamma_x$ such that (a) for $x,y \in \mathbb{Z}^2$, $\Gamma_x$ and $\Gamma_y$ coalesce, and (b) there is a line $L$ such that almost surely, each $\Gamma_x$ intersects this line finitely often, the tree of geodesics $(\Gamma_x : x \in \mathbb{Z}^2)$ cannot have infinite backward paths. In other words, with probability one, the set $\{y : 0 \in \Gamma_y\}$ is finite. If there were a bigeodesic with one end directed in $S$, then by uniqueness, it would have to be the forward infinite path of the tree of geodesics directed in $S$. However this tree cannot have an infinite backward path, which would be a contradiction.

\end{proof}

\newpage

\section{Growth and infection models, competition interface, and more }\label{sec:growth}

\subsection{Eden Model and the limit shape in high dimensions}

The Eden model was introduced by Eden \cite{Eden} in 1961 with the goal of analyzing a simple question in cell reproduction: 

`{\it Starting from a single cell which may divide, its daughters cells divide again and again, what are the structural properties of the resulting colony of cells and how do various possible constraints effect the architecture? }'

We assume that each cell is identical to every other cell, that each cell is connected to at least one other cell and that each of their locations is specified by a node in a lattice. Eden considered the lattice to be two-dimensional, and reasoned that the close biological counterpart to this process was the growth of bacterial cells or tissue cultures that are constrained from moving. One example given is {\it Ulva latuca} \cite{Burrows}, the common sea lettuce, that grows as a sheet of two cells thick and apparently grows only at its periphery. Although sometimes (wrongly) attributed as a possible model for cancer, it is important to point out that there is strong {\it in vitro} and {\it in vivo} evidence that Eden model does not explain any possible tumor growth (see for instance \cite{BioPeople} and the references therein).

The model can be described as follows. We start with a cell at the origin. Next, we adjoin one of the four neighbor vertices with equal probability. This two-celled configuration has six adjacent nodes.  We attach one of these six neighbor vertices with equal probability $1/6$. We continue this procedure indefinitely, each time adjoining a single cell uniformly from the collection of neighbor vertices. The construction is similar in any dimension.

The model just described can be seen as a time change of a site first passage percolation model in $\mathbb Z^d$ with exponential mean one passage times (on the vertices). Indeed, let $\mathcal A(0) = \{ 0\}$ and, for $n\geq 1$, define $\mathcal A(n)$ as the subset of $\mathbb Z^{d}$ after we attach $n$ cells (vertices) to $\mathcal A(0)$ according to the above construction. Let $\sigma_{0}=0$ and 
$$\sigma_k = \inf \bigg\{t \geq 0: B(t) \text{ contains } k+1 \text{ points of } \Z^d \bigg\},$$
where $B(t):= \{ x \in \mathbb Z^{d}: T(0,x) \leq t\}$ and the vertex weights have distribution $F(t)=\Pro(\tau_{e}\leq t) = 1-e^{-t}$. Then, Richardson \cite{Richardson} noticed that
\begin{equation}\label{eq:Eden}
\big \{ \mathcal{A}(k) : k\geq 1 \big \} \text{ and } \big \{ B(\sigma_{k}) : k\geq 1 \big \} \text{ have the same distribution.}
\end{equation}

One of the most important consequences of \eqref{eq:Eden} is that, contrary to other vertex weight distributions, the evolution of the ball $B(t)$ when the weights are exponentially distributed is Markovian. Precisely, in the Eden model, the process $$B(\sigma_1), B(\sigma_2), B(\sigma_3), \ldots$$ is Markovian on the space of possible shape configurations. Indeed, one only needs to know the (boundary) of the ball at time $\sigma_i$ to know its possible transitions at time $\sigma_{i+1}$.

\begin{figure}
\centering
\def \svgwidth{14cm}
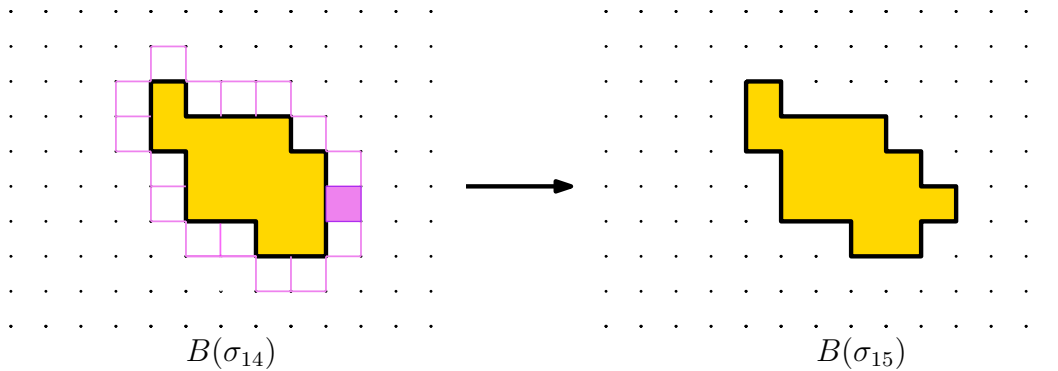
 \caption{The evolution of the Eden model. On the left, $B(\sigma_{14})$ in yellow and the possible $16$ neighbor squares that can be added. We choose one (shaded in pink) uniformly with probability $1/16$. On the right, $B(\sigma_{15})$ is shown after the attachment. Note that the order in which we attached previous boxes does not matter for the decision, only the knowledge of the boundary is needed to establish the next transition. The vertices in the picture are vertices of the dual lattice $\mathbb Z^2_*$ and a box represents a vertex of the lattice $\Z^2$.}
 \label{figEden}
\end{figure}

The Markovian property of the growing ball allows different tools that are not available for general passage times. This has been explored since \cite{Richardson}. The reader may think that this powerful property would open doors to a better knowledge of the limit shape and its properties. The truth is, however, distressing. As of today, there are almost no results that successfully explored the Markovian nature of $B$ to describe limit shape and geodesic results. The state of the art of the Eden model is essentially the same as for any FPP model. In Section \ref{LPP}, we discuss the last-passage analogue of the Eden model, where through combinatorics and explicitly formulae (achieved through the RSK correspondence), one indeed can say much more. We stress however that the methods of that section are not available in FPP. Connections between first-passage percolation and growth processes are further explored in Section \ref{growth}.

We end this section with the few theorems that explicitly use the Markovian nature of the Eden model. The first one is the determination of the growth of the time constant as the dimension diverges.

\begin{theorem}[Dhar \cite{Dhar}] For exponential passage times,

\begin{equation*}
\lim_{d \rightarrow \infty} \mu(d) \frac{d}{\log d} = \frac{1}{2}
\end{equation*}

\end{theorem}

The next result is the only one which rules out certain shapes as possible limit shapes. However it was only proved in astronomically large dimensions.

\begin{theorem}[Kesten \cite{KestenAspects}] If $d>650,000$ then the limit shape of the Eden model is not an Euclidean ball.

\end{theorem}

The theorem above was later improved by Couronn\'{e}, Enriquez, Gerin \cite{CEG11} to $d \geq 35$. Eden and Richardson observed by simulation that the limit shape looks circular in dimension two. However, one of the most fascinating long-standing questions remains:

\begin{question}
Show that for $d \geq 2$, the limit shape of the Eden model is not an Euclidean ball.
\end{question}

\subsection{Growth and competition models}\label{growth}

First-passage percolation is closely related to certain growth and competition models. In fact, the original version of the shape theorem was proved by Richardson \cite{Richardson} for a simplified growth model now known as the original 1-type Richardson model. In this model on $\mathbb{Z}^d$, we suppose that the origin houses an infection at time $t=0$ and all other sites are healthy. At each subsequent time $t=1, 2, \ldots$, any healthy site with at least one infected neighbor becomes infected independently of all other sites with probability $p \in (0,1)$. Richardson proved a shape theorem for the infected region at time $t$ as $t \to \infty$ and believed, on the basis of computer simulations, that as $p$ varies from $1$ to $0$, the limit shape varies from a ``diamond to a disk.'' This model was shown to be equivalent to an FPP model with i.i.d. weights on sites with a geometric distribution \cite{DurrettLiggett}. It is also common to describe the infection as species that are trying to colonize all sites of the space $\mathbb Z^d$.

We can build an infection/species model based on edge FPP in an analogous manner. Let $(\tau_e)$ be a realization of passage times. We infect the origin at time $0$ and the infection spreads at unit speed across edges, taking time $\tau_e$ to cross the edge $e$. In the case that the edge weights are exponential, the memoryless property of the distribution implies that the growth of the infected/colonized region is equal in distribution to a time change of a type of Eden process (discussed in the previous section). For general distributions, this growth process is called a 1-type Richardson model.

Building on the above definition, we may describe first-passage competition models, first introduced in the plane by H\"aggstr\"om and Pemantle \cite{HP} with two competing infections (species). We will consider the same model in $\mathbb Z^d$ with $k \in \mathbb N \cup \{ \infty\}$ different species. Fix $k$ vertices $x_{1},\ldots,x_{k}\in\mathbb{Z}^{d}$ and at time $t=0$, infect site $x_i$ by an infection of type $i$.  Each species spreads at unit speed as before, taking time $\tau_{e}$
to cross an edge $e$. An uninhabited site is exclusively
and permanently colonized by the first species that reaches it; that is,
$y\in\mathbb{Z}^{d}$ is occupied at time $t$ by the $i$-th species
if $\tau(y,x_{i})\leq t$ and $\tau(y,x_{i})<\tau(y,x_{j})$ for all
$j\neq i$.  If multiple $x_i$'s try to colonize a site at the same time (an event that happens with probability $0$ if the passage time distribution is continuous), we use some rule to break ties.

Now, consider the set colonized
by the $i$-th species starting at vertex $x_i$:
$$C(x_i)=\{y\in\mathbb{Z}^{d}\,:\, y\mbox{ is eventually occupied by }i\}.$$
We are interested in describing the geometry of these random sets.  In particular, much of the research so far has been driven to answer the following questions:
 
\begin{enumerate}
\item Can multiple species simultaneously succeed in invading an infinite subset of the lattice? In other words, does coexistence occur?
\item Assuming coexistence of at least two species, what can we say about the boundary interface? Does it have an asymptotic direction? A possible scaling limit?  
\item What happens if one species moves with faster speed than the others? Can we show that coexistence is impossible?
\end{enumerate}
These questions will be discussed in the following sections.

\subsection{Competition with same speed}\label{sec:growthr}

Before describing previous work, let's be precise and put the above questions on solid ground. We start with the first question on coexistence. Recall that $\nu$ denotes the probability measure of the passage times. One says that $\nu$ \emph{admits coexistence of $k$ species} if for some choice of initial sites
$x_{1},\ldots,x_{k}$, 
\begin{equation}\label{eq:coex}
\mathbb{P}\bigg(|C(x_i)|=\infty\mbox{ for all }i=1,\ldots,k\bigg)>0.
\end{equation}
The first point to make is that the choice of the initial sites is not important. One just needs to make sure that with positive probability, a new set of initial sites $y_1, y_2, \ldots, y_k$ are respectively infected by the initial sites $x_1, x_2, \ldots, x_k$. 
When the passage times are exponentially distributed, this can be easily done using the Markov property (see \cite[Propositon 1.1]{HP}, when $d=k=2$). When the passage times are not exponentially distributed, this modification argument follows the same lines as those described in Section $3$ and can be found in \cite[Section 5 and 6]{GM}. One therefore has the following proposition:

\begin{proposition}\label{prop:initialsites} Assume that the measure $\nu$ is non-atomic and the $\text{supp} \; \nu = \mathbb R^+$. Suppose that $x_1, \ldots, x_k$ and $y_1, \ldots, y_k$ are points in $\mathbb Z^d$ such that there exists disjoint paths from $x_i$ to $y_i$. Then 
$$\mathbb{P}\bigg(|C(x_i)|=\infty\mbox{ for all }i=1,\ldots,k\bigg)>0 \Leftrightarrow \mathbb{P}\bigg(|C(y_i)|=\infty\mbox{ for all }i=1,\ldots,k\bigg)>0.$$
\end{proposition}

The proposition above deals with choosing the initial sites. One has a similar result if starting with $k$ different infected sets instead of single points \cite{DH06}. However, the key and difficult step to prove coexistence is to establish that there indeed exist sources such that \eqref{eq:coex} occurs. Heuristically, the shape theorem gives the intuition that the larger the distance between the sources, the harder it is for one infection to surround the other one. If one hopes to prove that \eqref{eq:coex} holds, then the best strategy is to choose the sources far from each other and hope that their infections will grow at least in different directions. This gives an indication that coexistence is strongly related to the existence of distinct geodesic rays with different asymptotic directions. We will explore this avenue in a few moments. First, let us briefly say what is the state of the art on Question 1 above. 

When $\nu$ is the exponential distribution, H\"{a}ggstr\"{o}m and Pemantle \cite{HP} proved coexistence of 2 species (one could also see \cite{Blair-Stahn} for a review of results on Richardson models, focused on exponential passage times). Shortly thereafter, Garet and Marchand \cite{GM} and Hoffman \cite{Hoffman1} independently extended these results to prove coexistence of 2 species for a broad class of translation-invariant measures, including some non-i.i.d. ones.  Later, Hoffman \cite{Hoffman} demonstrated coexistence of $4$ species in $d=2$ for a similarly broad class of measures by establishing a relation with the number of
sides of the limit shape in the associated FPP. Tying the question of coexistence with the number of sides of the limit shape is a very neat argument that we try to sketch below. Recall that the number of sides of the limit shape $\mathcal{B}_\nu$ is equal to $k$ if $\partial \mathcal{B}_\nu$ is a polygon with $k$ sides. If $\partial \mathcal{B}_\nu$ is not a polygon, the number of sides is defined to be infinity.

\begin{theorem}[Hoffman \cite{Hoffman}]\label{Hoof}
Suppose that $\mathcal{B}_\nu$ is bounded, $\mathbb E \tau_e^{2+\eta} <\infty$ for some $\eta >0$, and that $\nu$ is uniquely geodesic. If the number of sides of $\mathcal{B}_\nu$ is at least $k$ then for any $\epsilon >0$ there exist $x_{1},\ldots,x_{k}$ such that 
\begin{equation}\label{eq:lda32}
\mathbb{P}\bigg(|C(x_i)|=\infty\mbox{ for all }i=1,\ldots,k\bigg)>1-\epsilon.
\end{equation}
\end{theorem}

\begin{remark} J.-B. Gou\'er\'e has extended these methods to give a geometric description of invaded regions; see \cite{Gouere}.
\end{remark}

We now sketch the proof of the above theorem.

\textsc{Step 1} - \textsl{ Moving to Busemann functions.}

We start the proof by some geometric considerations. Let the set $V$ consist of all $v\in \partial \mathcal{B}_\nu$ such that there is a unique line $L_v$ which is tangent to $\mathcal{B}_\nu$ through $v$. For such a $v$, let $w(v)$ be a unit vector parallel to $L_v$. Let $L_{n,v}$ be the line through $nv$ in the direction of $w(v)$. It is not difficult to see that if the number of sides of $\partial \mathcal{B}_\nu$ is at least $k$ then one can find points $v_1, \ldots, v_k \in V$ such that the lines $L_{v_i}$ are distinct for all $i$. Fix such $v_i$'s from now on. The first step is to establish the following lemma. Define the Busemann-type function $B_S(x,y)$ for $x,y \in \mathbb{Z}^2$ and $S \subset \mathbb{R}^2$ as
\[
B_S(x,y) = T(x,S) - T(y,S).
\]
From the subadditivity of $T$ and the definition of $B_S$, we have $B_S(x,y) \leq T(x,y)$ and 
\begin{equation}\label{cocycle}
B_S(x,y)+B_S(y,z)=B_S(x,z).
\end{equation} These Busemann functions are useful because of the following fact.

\begin{lemma} Assume that the number of sides of $\mathcal{B}_\nu$ is at least $k$ and let $v_i, 1 \leq i \leq k$ be defined as above. If there exist $\epsilon > 0$, and $x_1,\ldots,x_k$ such that 
\begin{equation}\label{eq:Busagain}
\mathbb P \bigg( B_{L_{n,v_i}}(x_j,x_i)>0 \; \forall i \neq j \bigg) \geq 1-\epsilon
\end{equation}
for infinitely many $n$, then \eqref{eq:lda32} holds.
\end{lemma}
\begin{proof}
It is easy to be convinced that such lemma must be true. Recall that geometrically $B_{L_{n,v_i}}(x,y)$ is roughly given by the orthogonal projection of the vector $y-x$ on the line orthogonal to $L_{n,v_i}$. If for all $j\neq 1$, $B_{L_{n,v_1}}(x_j,x_1) >0$ then roughly speaking $x_1$ is closer than $x_j$ is to the line $L_{n,v_1}$ and the species $1$ will infect in the direction orthogonal to $L_{n,v_1}$. 

Precisely, if for all  $j\neq 1$ $$B_{L_{n,v_1}}(x_j,x_1) >0,$$ then one can find $z_n \in L_{n,v_1}$ such that $T(z_n,x_1) < T(z_n,x_j)$ for all $j \neq 1$. Thus the species $1$ will infect infinitely many sites. Repeating the same argument for the other species, we obtain exactly what we want.
\end{proof}
\textsc{Step 2} - \textsl{ The density argument.}
Now, to end the proof of Theorem \ref{Hoof}, one  needs to construct points $x_i, 1 \leq i \leq k$ such that \eqref{eq:Busagain} holds. The idea is to choose $x_i = M v_i$ with $M$ large enough. The proof then boils down to proving the following stronger statement.

\begin{lemma} There exists $c>0$ such that for all $\epsilon>0$ there exists $M$ such that the density of $n$ with 
$$ \Pro \bigg( B_{L_{n,v_i}}(Mv_j,Mv_i) > cM \quad \forall i \neq j\bigg) > 1-\epsilon$$
is at least $1-\epsilon$.
\end{lemma}
\begin{proof} We sketch the argument here. The idea is that $B_{L_{n,v_i}}(Mv_j,Mv_{i})$ for large $M$ behaves as $f(M(v_{j}-v_{i}))$ where $f$ is a linear functional with null space in direction $w(v_{i})$ such that $f(v_{i}) = \mu(v_{i})$. Precisely, for each pair $(v_i, v_j)$ with $i \neq j$ we will find $w \in \mathbb Z^2$  and $c>0$ (depending on the angle between the points $v_{i}$ and $v_{j}$) such that for all large enough $M$ the events: 
\begin{equation}\label{eq:dadfafdsfsdfsdfsdfdsffasds}
 B_{L_{n,v_i}}(Mv_j,Mw) > Mc(1-\epsilon)
\end{equation}
and 
\begin{equation}\label{eq:dajidsfdsfdfsdajo}
 B_{L_{n,v_i}}(Mw,Mv_i) > -M\epsilon
\end{equation}
happen with probability $1- \epsilon$ in a set of values of $n$ with density at least $1-\epsilon$. Then \eqref{cocycle} implies $$B_{L_{n,v_i}}(Mv_j,Mv_i) > M c (1-2\epsilon).$$
Establishing \eqref{eq:dadfafdsfsdfsdfsdfdsffasds} and \eqref{eq:dajidsfdsfdfsdajo} was the central part Hoffman's paper (see Lemma 4.4, 4.5 in \cite{Hoffman}) and these have much in common with finding the mean of the Busemann function (see Lemma \ref{lem:todayisaniceday}). For \eqref{eq:dadfafdsfsdfsdfsdfdsffasds}, Hoffman developed his averaging trick, which was explained in Lemma \ref{lem: expected_busemann_simple} in Section \ref{sec:Busemann}. The proof of \eqref{eq:dajidsfdsfdfsdajo} relies on the assumption that the limit shape is differentiable at $v_{i}$. 
\end{proof}

\vspace{1cm}

 Coexistence of infinitely many species is defined similarly and clearly implies coexistence of $k$ species for any $k$.   In \cite{DH}, Damron and Hochman extended the above Theorem to include the case $k=\infty$ and used it to construct a first example of a passage time distribution that admits infinite coexistence. The current state of the art is given by a combination of Theorem \ref{Hoof} and Theorem \ref{thm: diffll}. It establishes that infinite coexistence is not only possible but is necessary for all measures in $\mathcal{M}_p$:
\begin{theorem}\label{thm:infinitecoexistence}
If $\nu \in \mathcal{M}_p$ for $p\in[\vec p_c,1)$ then $\nu$ admits coexistence of infinitely many species.
\end{theorem}

\subsection{Competition with different growth speeds}

If two species have two different growth rules, with one faster than the other, one may expect that non-coexistence occurs with probability $1$. We try to make this precise in this section. We first consider the original case treated in \cite{HP}, where the passage times are exponentially distributed. We restrict ourselves to the case of $2$ infections.

The model now has two parameters, denoted by $\lambda_1$ and $\lambda_2$, indicating the intensities of the infections. We assume that $\lambda_1 < \lambda_2$ so species $1$ spreads faster than species $2$. The infection rule is the same as in the last section. We start with one infected site of each species. An uninfected site becomes type $i$ infected at a rate proportional to the number of type $i$ infected nearest neighbors and then stays type $i$ for the rest of its life.

The best result to date in this case is still from the original paper \cite{HP2}. With one of the intensities fixed, coexistence has probability $0$ for all but at most countably many values of the other intensity.
It is however expected that this should be true for all values of $\lambda_1 \neq \lambda_2$. 

\begin{question}
Show that for all values of $\lambda_1 \neq \lambda_2$ and all $x_1 \neq x_2$,
$$\mathbb{P}\bigg(|C(x_1)|=|C(x_2)| = \infty \bigg)=0. $$
\end{question}

It is worth mentioning that at this point, we are still left in the embarrassing situation where we know that for almost every value of $\lambda_2$, coexistence does not occur, but we are unable to exhibit any single value of $\lambda_2$ such that coexistence actually does not happen. A possible way to answer this conjecture is to show:

\begin{question} Fix $\lambda_1$. Show that the event $\{|C(x_1)|=|C(x_2)| = \infty \}$ is monotone in $\lambda_2$. 
\end{question}

Monotonicity is believed to hold on $\mathbb Z^d$, although there are examples of graphs where such a statement does not hold \cite{DefHag2}. As in Proposition \ref{prop:initialsites}, the initial placement of the infections does not matter and one could start with any bounded initial sets. Certain special cases where one of the initial sets is unbounded (a half-space against a point, for instance) are treated in \cite{DefHag} and extended in \cite{AntunovicProcaccia}.

As in Section \ref{sec:growthr}, one can define competition with different speeds using more general passage times. This was studied by Garet and Marchand in \cite{GM3}, where they extended the results of \cite{HP} when the passage time distributions for the two species are stochastically comparable. Under certain conditions on the passage times, in this setting, they established in any dimension $d$ that if the slow species survives, the fast species cannot occupy a very high density of space. For $d = 2$, they show that almost surely, one species must finally occupy a set of full density in the plane while the other species occupies only a set of null density.

\subsection{The competition interface}
 
Before ending the chapter, it is worth mentioning that in the past twenty years, there has also been a growing interest in questions related to the asymptotic shape of infected regions. 

In the physics literature, Derrida and Dickman \cite{DD} describe simulations for the Eden model in which two clusters grow into a vacant sector of the plane with angle $\theta$. They obtain values for the roughening exponents of the competition interface depending on the angle.
Two species competition with mutation and selection was also studied numerically in the Eden model in a paper by Kuhr, Leisner and Frey \cite{KLF} (see the references therein for other variations of the model).  

However, precise results are scarce, except for solvable models of Last Passage Percolation, where there is a coupling between the competition interface and the second class particle in TASEP \cite{FMPim2009, FPimentel2005, Pimentel}. One of the main difficulties (again!) in the case of FPP is the lack of knowledge of the limit shape $\mathcal B_{\nu}$.

\begin{figure}
\centering
\def \svgwidth{14cm}
\input{Competition.pdf_tex}
 \caption{Competition interface between two species. On the left, the starting configuration, with initial angle $\theta > 180^{\circ}$. In blue, the region infected by species $1$ while, in red, the region infected by species $2$. In black, the random curve defining the competition interface.}
 \label{figCompetition}
\end{figure}
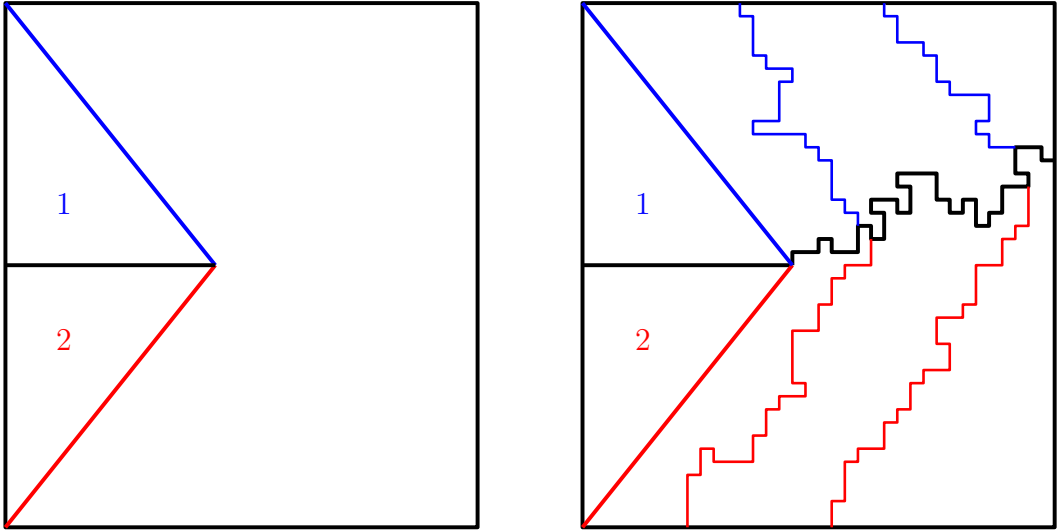

One way to define the competition interface is as follows (see Figure \ref{figCompetition}). For a given $\theta \in (0,2\pi),$ consider the sets $$\mathrm{Blue} = \{ (x,y) \in \mathbb Z^{2} : \frac{\theta}{2}< \arctan (x/y) < \pi \}$$ and $$\mathrm{Red} = \{ (x,y) \in \mathbb Z^{2} : \pi < \arctan (x/y) < 2 \pi - \frac{\theta}{2}  \}. $$
Blue and Red are the initial infected sites. We now color a vertex $v$ in $\mathbb Z^{2} \setminus (\mathrm{Blue} \cup \mathrm{Red})$ blue if 
$$T(v,\mathrm{Blue}) < T(v, \mathrm{Red}),$$
and red otherwise. If the passage time distribution is continuous, almost surely there will be no ties and a color will be assigned to each vertex of $\mathbb Z^2$. The competition interface is the unique path on the dual lattice that sits between the blue and red regions. Now for $v \in \mathbb Z^{2} \setminus (\mathrm{Blue} \cup \mathrm{Red})$ define 
$$p(v) = \mathbb P (v \text{ is } blue)$$
and let the width of the interface at absissa $x$ be defined as $$\omega(x)= \sum_{v:v=(x,y)} p(v)(1-p(v)).$$
For $x$ large, \cite{DD} predicted that the function $\omega(x)$ has power law behavior with exponent $\omega(x) \sim x^{\zeta}$ where 
\begin{equation*}
\zeta = \begin{cases}
&\frac{1}{3}, \text{ if } \theta < 180^\circ \\ 
&\frac{2}{3}, \text{ if } \theta = 180^\circ \\
&1, \text{ if } \theta > 180^\circ.
\end{cases}
\end{equation*}
As far as we know, none of these questions were investigated except in the case of LPP with exponential weights. Some recent result in this exactly solvable case were also obtained in \cite{GRS}.

\newpage

\section{Related models and questions}\label{sec:alltheotherthings}

In this section, we discuss a sequence of related questions that were not addressed in the previous sections. They all share something in common: they are all variants of a first-passage percolation model. Here we try to summarize the most well-studied connections and variants for the reader; however, we refer to the papers for their proofs and methodologies.

\subsection{The maximum flow}
The maximum flow problem can be considered a higher-dimensional version of first-passage percolation, where the passage times $\tau(e)$ play the role of the maximal amount of fluid that can go through an edge $e$ per unit of time.
Consider an open bounded connected subset $\Upsilon$ of $\mathbb R^d$ such that the boundary $\partial \Upsilon$ is  piecewise of class $\mathcal C^1$. Let $\Lambda_1, \Lambda_2$ be two open disjoint subsets of $\partial \Upsilon$. For simplicity, let's consider the case where $\Upsilon$ is a $d$-dimensional cube of diameter $n$ and $\Lambda_1$ and $\Lambda_2$ are two opposite faces. $\Lambda_1$ will play the role of the source of the network, while $\Lambda_2$ denotes the sink. A flow $g$ from $\Lambda_1$ to $\Lambda_2$ is an assignment of non-negative values $g(e)$ and a direction to each edge $e$ in $\Upsilon$ such that $$0 \leq g(e) \leq \tau_e, \quad \text{ for all } e $$ 
and such that for each vertex $v \in \Upsilon \setminus (\Lambda_1 \cup \Lambda_2)$, the total inflow equals the total outflow; that is,
$$ \sum_{e \in v^+} g(e) = \sum_{e \in v^-} g(e)$$
where $v^+$ ($v^-$) is the set of edges incident to $v$ and directed towards $v$ (away from $v$). Furthermore, ``inflow" should be equal to the ``outflow"; that is,
$$ \sum_{v\in \Lambda_1}  \sum_{ x \notin \Lambda_1, x \sim v} g(\langle v,x \rangle ) = \sum_{z \in \Lambda_2}\sum_{ x \notin \Lambda_2, x \sim z}  g(\langle x,z \rangle), $$
where $\langle v,x \rangle$ denotes the oriented edge that starts at $v$ and ends at $x$ and $x\sim v$ means $x$ and $v$ are neighboring vertices.  The strength of the flow $g$ is now defined as $$S(g) =  \sum_{v\in \Lambda_1}  \sum_{ x \notin \Lambda_1, x \sim v} g(\langle v,x \rangle). $$

Clearly, if we choose $g(e)=0$ at all edges, we have a flow with $S(g) = 0$. Intuitively, as we increase the strength, it becomes harder to find flows. The main question is therefore: 

\vspace{0.3cm}
{\it Given the passage times, the source $\Lambda_1$, the sink $\Lambda_2$ and the network $\Upsilon$, what is the maximum strength flow from $\Lambda_1$ to $\Lambda_2$? How does its strength scale as $\Lambda_1$ and $\Lambda_2$ are moved far apart? }
\vspace{0.3cm}

Precisely, we define the max-flow from $\Lambda_1$ to $\Lambda_2$ as 
 
 $$ \Phi(\Lambda_1, \Lambda_2)= \max_{g: g \text{ flow from } \Lambda_1 \text{ to }\Lambda_2} S(g).$$
Our goal is to study the behavior of $ \Phi(\Lambda_1, \Lambda_2)$ as $n$ (the diameter of the $d$ dimensional cube) goes to infinity.

The connection with first-passage percolation is not only in the construction of the network with the passage times $\tau_e$. The maximum flow can be represented in a different way, as the next theorem shows. We need a few definitions. We say that a set of edges $E$ separates $\Lambda_1$ from $\Lambda_2$ if there is no path in $\Upsilon \setminus E$ from $\Lambda_1$ to $\Lambda_2$. We call $E$ a cut of $(\Lambda_1, \Lambda_2)$ if $E$ separates $\Lambda_1$ from $\Lambda_2$ and if no proper subset of $E$ separates $\Lambda_1$ from $\Lambda_2$. Now the max-flow and min-cut theorem (see~\cite{B79}) states that

 \begin{theorem} The following equality holds:
 \begin{equation*}
 \Phi(\Lambda_1, \Lambda_2) = \min \big\{ \sum_{e\in E} \tau_e : E \text{ is a cut of  }(\Lambda_1, \Lambda_2) \big\}.
 \end{equation*}
 \end{theorem}
 
 Thus, the max-flow can be viewed as the first-passage time from $\Lambda_1$ to $\Lambda_2$ over all possible cuts. In two dimensions, a cut $E$ reduces to a path and we can clearly see the connection. In higher dimensions, $E$ is a hypersurface and thus the claim is that the maximum flow is the analogue of FPP in higher dimensions.
 
 Now, let $d=3$ and write $\Lambda = [0,k] \times [0,l] \times [0,m]$. The following result is due to Kesten,~\cite{KestemMaxflow}. 
 \begin{theorem}[Kesten \cite{KestemMaxflow}]\label{Kestenmax} Assume that $m(k,l)\rightarrow \infty$ as $k\geq l \rightarrow \infty$ and that for some $\delta>0$,
 \begin{equation}\label{eq:Kestenmaxflow}
 k^{-1+\delta} \log m(k,l) \rightarrow \infty.
 \end{equation}
 Then there exists $p^*$, with $1/27 \leq p^*\leq p_c$ such that for any distribution $F$ with $F(0)<p^*$  and finite exponential moments
 \begin{equation}\label{eq:maxflow}
 \lim_{k,l,m \rightarrow \infty} \frac{1}{kl}  \Phi(\Lambda_1, \Lambda_2) = \heartsuit \quad \text{ a.s.  and in } L_1
 \end{equation}
  for some $\heartsuit>0$.
\end{theorem}
In the same paper, Kesten showed that if $F(0)>1-p_c$ the limit above exists and it is equal to zero. The critical case when $F(0)=1-p_c$ was also  considered by Zhang \cite{Zhang3}.
 
 \begin{theorem}[Zhang \cite{Zhang3}] If $F(0)=1-p_c$ and $F(0^-)=0$ and if in addition $\E \tau < \infty$ then for any $l>0$, 
 $$ \lim_{k,m\rightarrow \infty} \frac{1}{kl} \Phi(\Lambda_1, \Lambda_2) =0$$
and 
$$\lim_{k,l,m\rightarrow \infty} \frac{1}{kl} \Phi(\Lambda_1, \Lambda_2) =0,$$ where $k,l,m$ go to infinity without any restriction (in particular wiithout assuming \eqref{eq:Kestenmaxflow}).
\end{theorem}

\begin{remark} The arguments in \cite{Zhang3} are also valid for dimensions $d\geq 2$.
\end{remark}

Assuming finite exponential moments, Zhang \cite{Zhang3} improved the above results and proved the following theorem.

\begin{theorem}[Zhang \cite{Zhang3}] Let $d\geq 2$. Assume finite exponential moments and that for all $m=m(k)$ that goes to infinity when all the length of the sides $k_i$, $i=1,\ldots, d-1$ go to infinity in such a way that 
$$\exists \; \delta \in (0,1), \quad \log m(k) \leq \max_{i=1,\ldots, d-1} k_i^{1-\delta}.$$
Then we have
$$ \lim_{k_1, \ldots, k_{d-1} \rightarrow \infty} \frac{1}{\prod_{i=1}^{d-1} k_i}\Phi(\Lambda_1,\Lambda_2) = \heartsuit(d) \quad \text{ a.s.  and in } L_1.$$
Moreover, this limit is positive if and only if $F(0) < 1-p_c(d).$
\end{theorem}

After the work of Kesten and Zhang, remarkable progress has been made over a period of years. Flows between general subsets of $\mathbb Z^d$ were studied in the papers of Garet \cite{GaretMax} and Rossignol-Theret \cite{RossignolTheret10, RossignolTheret13} in dimension 2. The study culminates in a series of papers of Cerf and Th\'{e}ret \cite{CT1,CT2, CT3}. Under finite exponential moments for the passage times, they establish the analogue of Theorem \ref{Kestenmax} for very general networks and derive lower and upper large deviations (of surface order $n^{d-1}$, and volume order $n^d$, respectively). 

\begin{question}\label{oq:maxflow}
Determine if the limit in \eqref{eq:maxflow} exists without the exponential moment assumption.
\end{question}

\subsection{Variants}

\subsubsection{Different edge weights}

Over the years, some attention has been devoted to cases of FPP with non-i.i.d edge weights. Boivin \cite{B90} considers the general stationary and ergodic case and proved the analogue of Theorem~\ref{thm:timeconstant} and Theorem~\ref{thm:limitshape}. (See Bjorklund \cite{Bjorklund} for the most general known version.) In the special case where the $\tau_e$'s are also bounded, this was first noted by Derrienic (see \cite[page 259]{Kestensurvey}). A remarkable contribution came in the work of H\"aggstr\"om and Meester \cite{HM}, who classify all possible limit shapes in the stationary, ergodic case:

\begin{theorem}[Haggstrom-Meester \cite{HM}]\label{thm:HM} Let $\mathcal C$ be the set of all subsets of $\R^d$ that are compact, convex, and symmetric about the axes with nonempty interior. Let $\mathcal C^*$ be the set of all compact subsets of $\R^d$ with nonempty interior that can arise as limiting shapes for stationary first-passage percolation. Then $\mathcal C = \mathcal C^*$.
\end{theorem}

The proof of the theorem above given by  H\"aggstr\"om and Meester is constructive. That is, given a set $B \in \mathcal C^*$, they construct a stationary measure $\mathbb{P}$ whose limit shape is $B$. The  measure $\mathbb{P}$ has bounded support and trivial tail $\sigma$-fields. For any set $B$, their construction does not lead to an i.i.d. environment. The following open question is a rephrasing of Question \ref{q1}.

\begin{question}
Find a convex, compact set $B$ with non-empty interior and symmetry about the axes, not equal to the $\ell^1$ unit ball, and an i.i.d. first passage model that has $B$ as its limit shape. 
\end{question}

Other types of dependent passage times were considered in the literature. Fontes and Newman \cite{FNewman}  and Chayes and Winfield \cite{Chayes92} considered dependent first-passage weights coming from random colorings of vertices of $\mathbb Z^d$. Van den Berg and Kiss \cite{vdBergkiss12} extended the BKS sublinear variance theorem to a class of dependent weights that includes passage times that correspond with the minimal number of sign changes in a subcritical Ising landscape, a model investigated by Higuchi and Zhang \cite{HiguchiZhang}.

\begin{figure}\label{fig:Michael}
\centering
 \includegraphics[scale=0.6]{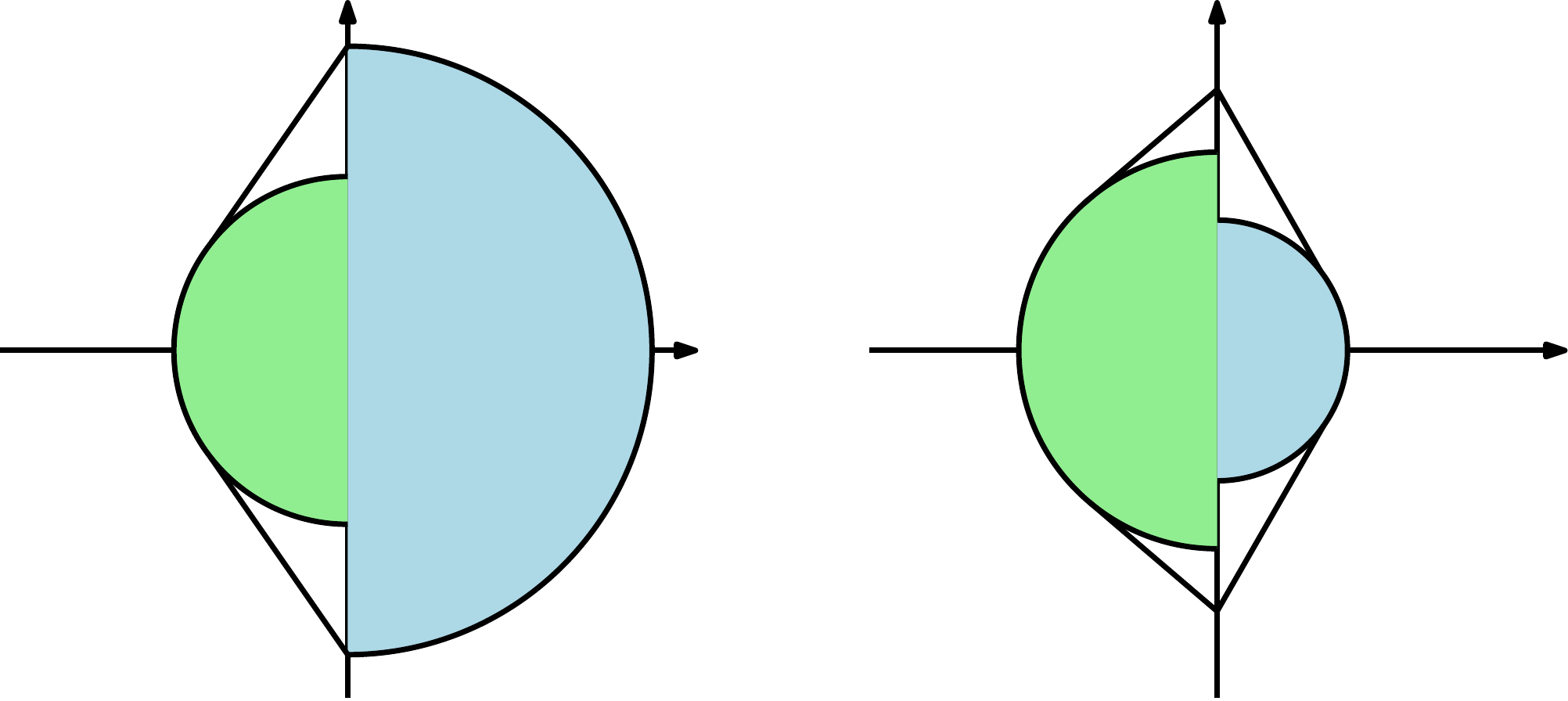}
 \caption{Representation of the limit shape in the left-right inhomogeneous model. On the left, the edge distribution $F_+$ is more variable than $F_-$. On the right, the case where no such relation is present.}
\end{figure}

 The following variation was also considered recently in \cite{ADS13}. Consider an inhomogeneous version of FPP where edges in the left and right half-planes of the $\mathbb Z^2$
lattice are assigned weights independently according to distributions $F_-$ and $F_+$.  Complementing the classical approach of the proof of Theorem \ref{thm:limitshape} with large deviation estimates for half-plane passage times as in \cite{Daniel1}, the authors proved a shape theorem. They also show how the limit shape behaves as a function of the limit shapes of regular (homogeneous) FPP with weights from $F_-$ and $F_+$. If either of $F_-$ and $F_+$ is more variable than the other, then the asymptotic shape is the convex hull of the restriction to respective half-planes of the asymptotic shapes for $F_-$ and $F_+$. When no such relation is present, the asymptotic
shape equals the convex hull of the two half-shapes and a potentially wider additional line
segment along the vertical axis (see Figure~\ref{fig:Michael}). The question introduced in that work is:

\begin{question}
Find necessary and sufficient conditions on $F_-$ and $F_+$ so that the asymptotic speed in the vertical direction is strictly greater than that given by either $F_-$ or $F_+$ in the standard homogeneous model.
\end{question} 

\subsubsection{FPP on different graphs}

FPP on graphs different from $\mathbb Z^d$ has also appeared in the recent literature. In the past decade, several authors were interested in FPP on complete graphs, on the hypercube and also on random graphs such as Erdos-Renyi and the configuration model. As most of these examples are finite graphs, the questions are slightly different from those considered here. We describe them below.

Given a finite graph $G=(V,E)$ and weights $(\tau_e)_{e \in E}$, the passage time between two vertices $u$ and $v$ is defined analogously to \eqref{definition:passagetime} as the minimum passage time among paths between $u$ and $v$. A few new and old quantities emerge as the main object of study. First, one would like to understand the typical passage time and the number of edges of a geodesic between two vertices. That is, choosing $U$ and $V$ uniformly at random from the vertices of $G$, we want to understand the behavior of the passage time $T(U,V)$ and the length of a geodesic $|\Gamma(U,V)|$. Further objects of interest include the flooding time $\max_{v \in V} T(u,v)$ and the the weight diameter $\max_{u,v \in V} T(u,v)$. 

For results in this direction we refer the reader to the beautiful notes of Remco van der Hofstad \cite[Chapter 8]{Remco} and the references therein.

\begin{figure}\label{fig:directed}
\centering
 \includegraphics[scale=0.6]{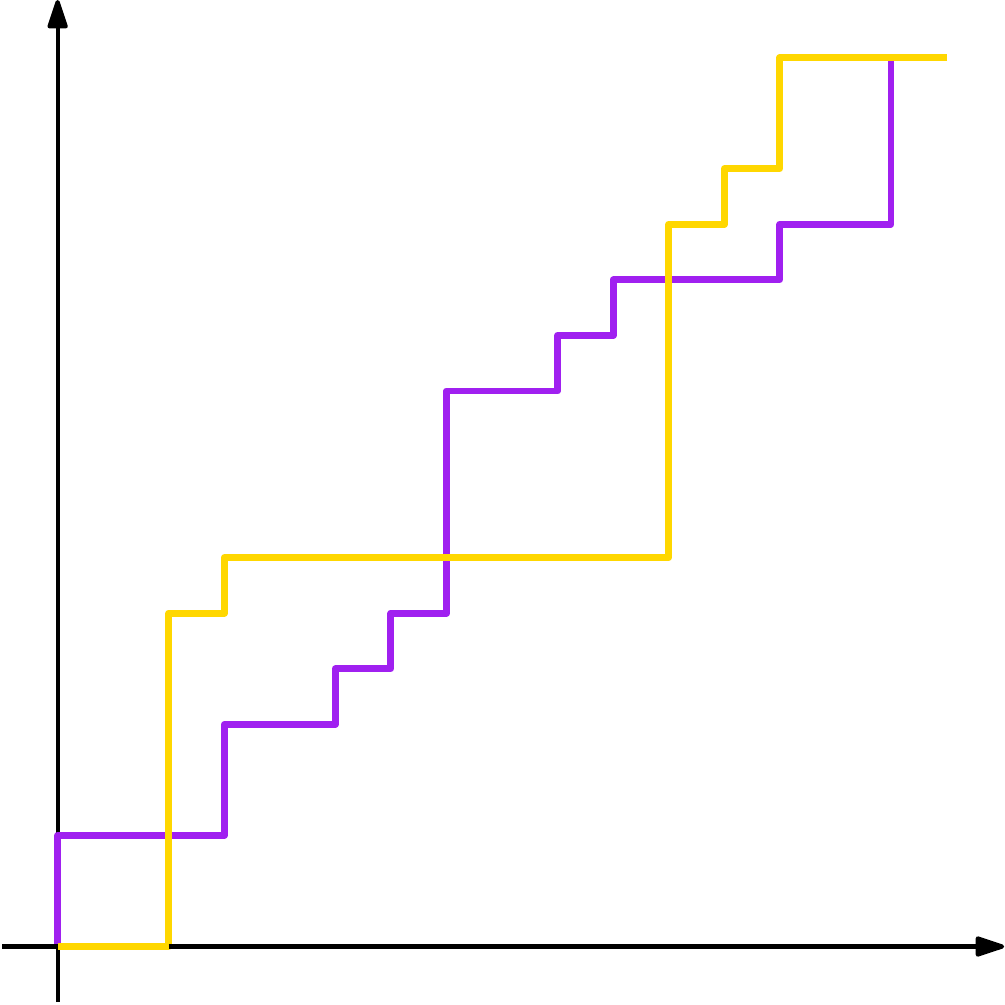}
 \caption{Examples of directed paths from the origin to $(n,n)$. These paths are used in oriented first passage and last passage percolation models.}
\end{figure}

\subsubsection{Last Passage Percolation}\label{LPP}

In last-passage percolation (LPP), paths are constrained to have non-decreasing coordinates. Usually the weights are placed on vertices instead of the edges and the passage time between two points is given by 
\begin{equation*}
L(u,v) = \max_{\Gamma :u\rightarrow v} T(\Gamma )\ .
\end{equation*}
Existence of time constants and the limit shape theorem are proved as in Section \ref{sec:limitshape}. The proof follows the same strategy with the difference that the process is now superadditive and $B_\nu$ is concave instead of convex. The directness of paths, however, requires that vertices $u$, $v$, $w$ be ordered in a non-decreasing fashion so that $L(u,w) \geq L(u,v) + L(v,w)$. The details for the proof of the limit shape can be found in  \cite{Martin04}, including a necessary extra step to handle the continuity of $\mathcal{B}_\nu$ up to the boundary of the upper right quadrant.  Here, we will discuss the model in two dimensions and in this subsection, we will use the notation $\mu(x,y)= \mu(xe_1+ye_2)$ for the time constant.

The importance of this small variation of FPP is that in dimension $2$ there are very natural correspondences between LPP models and certain queueing networks, systems of queues in tandem, and interacting particle systems. This connection reaches a deeper level as very precise scaling laws have been obtained for very particular distributions. If the passage times are exponentially distributed with mean $1$, then the time constant can be explicitly computed:
$$\mu(x,y) = (\sqrt{x} + \sqrt{y})^2,$$
as it was first shown by Rost \cite{Rost}. If $\nu$ is the probability measure of a geometric random variable with parameter $p$ then \cite{Joha}
$$\mu(x,y) = \frac{1}{p} (x+y+ 2\sqrt{xy(1-p)})\ .$$
In both cases, finer asymptotics are available \cite{Joha}, as the distribution of
 \begin{equation}\label{Joha}
 \frac {L(0,n(x,y)) - n\mu(x,y)}{n^{1/3}}
\end{equation} 
converges to a non-degenerate limit as $n$ goes to infinity. The proof of \eqref{Joha} goes through the following special identity that identifies the law of the passage time with the law of the largest eigenvalue of the Laguerre Unitary ensemble. Let $A$ be an $n\times n$ matrix with entries that are complex Gaussian random variables with mean zero and variance $1/2$.

\begin{theorem}[Johansson \cite{Joha}]\label{thm:joha} If $\nu$ is exponentially distributed with mean $1$
then for all $t >0$,
\begin{equation}
\Pro( L(0,(n,n)) \leq t ) = \Pro (\lambda_n \leq t), 
\end{equation}
where $\lambda_n$ is the largest eigenvalue of the $n \times n$ matrix $AA^*$.
\end{theorem}
The identity above is achieved by considering independent geometric passage times with parameter $p$ first. In this case, one can condition on the sum of the weights in the box $[0,N-1] \times [0,N-1]$ to discover:
\begin{enumerate}
\item The conditional law is uniform on the space $\mathcal M_N(s)$ of all possible non-negative integer valued $N \times N$ matrices with fixed sum $s$.
\end{enumerate} Indeed, letting $A(i,j)$ be the passage time attached to the vertex $(i,j)$, we have for the matrix $A=(A(i,j))$, $$\Pro \bigg(A= (a_{ij})\bigg) = \prod_{1 \leq i, j \leq N}(1-p)^{a_{ij}}p = p^{N^2}(1-p)^{\sum_{ij} a_{ij}}.$$
\begin{enumerate}
\item[2.] The space $\mathcal M_N(s)$ is in one-to-one correspondence with the space of generalized permutations of length $s$ on the set $\{ 1, \ldots, N\}$. The last-passage time is exactly the length of the longest non-decreasing subsequence. 
\end{enumerate}

This allowed Johansson to use the RSK correspondence. At the end of the day, it is possible to connect the probability of the event $L(0,(n,n)) \leq t $ to the number of pairs of semi-standard Young tableaux with the size of the first row. Determining the distribution of the passage times is thus a combinatorial problem and it has an explicit expression for any finite $N$. By approximating an exponential random variable by geometric, one derives Theorem \ref{thm:joha}. We refer the reader to the very readable Section 2 of \cite{Joha} for the details.

The law of $\lambda_n$ is well-known to be explicit and amenable to asymptotic analysis through the asymptotics of Laguerre orthogonal polynomials (see \cite{AGZbook} and the references therein). 
In particular, \eqref{Joha} is a simple combination of the theorem above and the fact that \cite{Joha}

\begin{equation*}
Z_n:=  \frac{\lambda_n - 4n}{2^{4/3}n^{1/3}} \stackrel{\mathcal D}{\rightarrow} W_2 \sim F_2,\
 \end{equation*}
where 
$$F_2(s) = \exp\bigg(-\int_{s}^\infty (x-s) q(x)^2 d x\bigg),$$ 
with $q$ the solution of the Painlev\'e II differential equation
\begin{eqnarray*}
q''(x)= xq(x) +2q^3(x)\\
q(x)\sim Ai(x) \text{ as } x \rightarrow +\infty,
\end{eqnarray*}
and $Ai(x)$ denotes the Airy function. The $F_2$ distribution is called the Tracy-Widom distribution due to the famous paper \cite{TW}. In the case of geometric passage times, the asymptotic analysis was also carried out in \cite{Joha} using the asymptotics of Meixner polynomials. Once proper rescaled the passage times fluctuate according to $F_2$.

As mentioned before, the explicit formula for the limit shape and the connection with TASEP was also explored to understand the asymptotic behavior of infinite geodesics and the competition interface.  The reader interested in this direction should look at the work of Ferrari, Martin and Pimentel \cite{FMPim2009} and the references therein.

For passage times that are not geometric or exponentials, the results in LPP are somewhat limited to same stage as those in FPP. Some recent progress was obtained in the sequence of papers \cite{GRS,GRS2, Seppa2} where, using results from queueing theory, the authors manage to (1) derive variational formulas for the time constant, similar to those in Section \ref{sec:3243}, (2) obtain asymptotic results for infinite geodesics and Busemann functions under mild assumptions on the limit shape.




\newpage

\section{Summary of open questions} \label{sec:open}

Here, just for convenience, we catalogue all questions that appear in the manuscript. They are certainly not exhaustive, but we believe that this set provides a good scope of the open and important areas of the field.

\begin{questionend}
Find a non-trivial explicit distribution for  which we can actually determine $\mu(e_1)$. 
\end{questionend}

\begin{questionend} Suppose that the support of the distribution of $\tau_e$ equals $\mathbb R^+$. 
Let $$X_n = \max \{ \tau_e: e \text{ is in a geodesic from } 0 \text{ to } ne_1 \}.$$ How does  $X_n$ scale with $n$?
\end{questionend}

\begin{questionend}
Extend the comparision theorem to the case $d>2$, $r>0$ and $\vec p_c(d)\leq F(r)$.
\end{questionend}

\begin{questionend}
Show that if $F$ is a continuous distribution then the limit shape is strictly convex.
\end{questionend}

\begin{questionend}
Show that the $d$-dimensional cube is not a possible limit shape for a FPP model with independent, identically distributed passage times.
\end{questionend}

\begin{questionend}
Show that for any measure $\nu \in \mathcal M_p$ the boundary of the limit shape is not flat outside the percolation cone. 
\end{questionend}

\begin{questionend}
Show that if the limit shape of a measure $\nu$ has a flat piece then $\nu \in \mathcal M_p$, for $p > p_{c}$ and the flat piece is delimited by the percolation cone. 
\end{questionend}

\begin{questionend}
Show that for any measure $\nu$, in direction $e_1$ the boundary of the boundary of the limit shape does not contain any segment parallel to the $e_{2}$ axis.  
\end{questionend}

\begin{questionend}
Find a non-trivial example of a measure $\nu$ such that in direction $e_1$ the limit shape boundary is not parallel to the $e_{2}$ axis. 
\end{questionend}

\begin{questionend}
Find conditions that guarantee the existence of the limit \eqref{Busemanningeneral}.
\end{questionend}

\begin{questionend}
Show that for continuous distributions of passage times, the limit shape is uniformly curved.
\end{questionend}

\begin{questionend}
Assume $F(0)=0$. Show that $\mu \geq \mathbb E_{4} \tau_{e}$.
\end{questionend}

\begin{questionend}
Show that for any $d\geq 2$, under suitable conditions on $F$, 
$\chi<1/2$. 
If $d=2$, show that $\chi=1/3$.
\end{questionend}

\begin{questionend}
Show that $$\lim_{d \to \infty} \chi(d) = 0.$$
\end{questionend}

\begin{questionend} For suitable $d$, show that $\chi >0$.
\end{questionend}

\begin{questionend}
Show that in $d=2$ there exists $a>0$ such that the expected overlap of two geodesics  from $0$ to $nx$ sampled independently is at least $n^a$.
\end{questionend}

\begin{questionend}
Show that  $\underline{\chi}_2 = \overline{\chi}_{2+\delta}$ for some $\delta>0$. 
\end{questionend}


\begin{questionend}
Assume that the passage times are bounded and not concentrated at a single point. Show that there exists a convex function $I_{u}(\epsilon)$ such that the following limit holds:
$$ \lim_{n} \frac{1}{n^{d}} \log p_{n}^{u} (\epsilon) = I_{u}(\epsilon).$$
\end{questionend}

\begin{questionend}
Let $d=3$ and $\mathbb{P}(\tau_e=0) = p_c = 1-\mathbb{P}(\tau_e=1)$. Is it true that 
\[
\mathbb{E}T(0,\partial B(n)) \asymp \log n?
\]
\end{questionend}

\begin{questionend}
Is \eqref{eq.thin} optimal? More explicitly, let 
$$\gamma_F(d):= \sup \bigg \{ \alpha:  \eqref{eq:thintheorem} \text{ holds for } h_n=n^\alpha  \bigg\}.$$ What is the value of $\gamma_F(d)?$
\end{questionend}

\begin{questionend} Prove that under no assumptions on the distribution $F$ of passage times, geodesics exist in the $d$ dimensional lattice, $d\geq3$.
\end{questionend}

\begin{questionend}
Find a distribution of passage times where \eqref{ScalingRelation} holds.
\end{questionend}

\begin{questionend}
Show that for continuous distributions, there are infinitely many geodesic rays.
\end{questionend} 

\begin{questionend}
Find a distribution that is not in $\mathcal M_{p}$ that has an infinite number of geodesic rays.
\end{questionend}

\begin{questionend}
Is it true that the sequence of geodesics $\Gamma(0,ne_{1})$ converges to some $\Gamma$?
\end{questionend} 

\begin{questionend}
Assume that the space is uniquely geodesic. Show that in $d=2$, there is no fixed vertex that is in infinitely many of the geodesics between $-ne_1$ and $ne_1$. 
\end{questionend} 

\begin{questionend}
Show that $D = [0,2\pi).$ 
\end{questionend}

\begin{questionend}
Show that $0\in D.$ 
\end{questionend}

\begin{questionend}
Decide whether Theorem \ref{theorem:Newman2} holds in arbitrary dimension.
\end{questionend} 

\begin{questionend}
Let $\Gamma(v)$ be a geodesic from $0$ to $v$. Consider the collection of all edges that belong to an infinite number of $\Gamma(v)$ and let $f(r)$ be the number of such edges which intersect the circle $x^2 + y^2 = r^2$. Does $f(r) \to \infty$ as $r \to \infty$; and, if so, how fast? 
\end{questionend}

\begin{questionend}
Do geodesic lines exist?
\end{questionend}

\begin{questionend}
Prove that under some conditions on the edge weights the limit \eqref{eq: busemann_assumption} exists.  
\end{questionend}

\begin{questionend}
Show that for $d \geq 2$, the limit shape of the Eden model is not an Euclidean ball.
\end{questionend}

\begin{questionend}
Show that for all values of $\lambda_1 \neq \lambda_2$ and all $x_1 \neq x_2$,
$$\mathbb{P}\bigg(|C(x_1)|=|C(x_2)| = \infty \bigg)=0. $$
\end{questionend}

\begin{questionend} Fix $\lambda_1$. Show that the event $\{|C(x_1)|=|C(x_2)| = \infty \}$ is monotone in $\lambda_2$. 
\end{questionend}

\begin{questionend}
Determine if the limit in \eqref{eq:maxflow} exists without the exponential moment assumption.
\end{questionend}

\begin{questionend}
Find a convex, compact set $B$ with non-empty interior and symmetry about the axes, not equal to the $\ell^1$ unit ball, and an i.i.d. first passage model that has $B$ as its limit shape. 
\end{questionend}

\begin{questionend}
Find necessary and sufficient conditions on $F_-$ and $F_+$ so that the asymptotic speed in the vertical direction is strictly greater than that given by either $F_-$ or $F_+$ in the standard homogeneous model.
\end{questionend}

\end{document}

%% file: Triangle.pdf_tex
\begingroup%
  \makeatletter%
  \providecommand\color[2][]{%
    \errmessage{(Inkscape) Color is used for the text in Inkscape, but the package 'color.sty' is not loaded}%
    \renewcommand\color[2][]{}%
  }%
  \providecommand\transparent[1]{%
    \errmessage{(Inkscape) Transparency is used (non-zero) for the text in Inkscape, but the package 'transparent.sty' is not loaded}%
    \renewcommand\transparent[1]{}%
  }%
  \providecommand\rotatebox[2]{#2}%
  \ifx\svgwidth\undefined%
    \setlength{\unitlength}{215.80866699bp}%
    \ifx\svgscale\undefined%
      \relax%
    \else%
      \setlength{\unitlength}{\unitlength * \real{\svgscale}}%
    \fi%
  \else%
    \setlength{\unitlength}{\svgwidth}%
  \fi%
  \global\let\svgwidth\undefined%
  \global\let\svgscale\undefined%
  \makeatother%
  \begin{picture}(1,0.45781294)%
    \put(0,0){\includegraphics[width=\unitlength,page=1]{Triangle.pdf}}%
    \put(0.06112496,0.03813616){\color[rgb]{0,0,0}\makebox(0,0)[lb]{\smash{}}}%
    \put(0.03,0.2){\color[rgb]{0,0,0}\makebox(0,0)[lb]{\smash{$\Gamma(x,y)$}}}%
    \put(0.03,0.0){\color[rgb]{0,0,0}\makebox(0,0)[lb]{\smash{$x$}}}%
    \put(0.95,0.03){\color[rgb]{0,0,0}\makebox(0,0)[lb]{\smash{$z$}}}%
    \put(0.61285595,0.35){\color[rgb]{0,0,0}\makebox(0,0)[lb]{\smash{$y$}}}%
    \put(0.78603823,0.35){\color[rgb]{0,0,0}\makebox(0,0)[lb]{\smash{$\Gamma(y,z)$}}}%
  \end{picture}%
\endgroup%

%% file: lnp.pdf_tex
\begingroup%
  \makeatletter%
  \providecommand\color[2][]{%
    \errmessage{(Inkscape) Color is used for the text in Inkscape, but the package 'color.sty' is not loaded}%
    \renewcommand\color[2][]{}%
  }%
  \providecommand\transparent[1]{%
    \errmessage{(Inkscape) Transparency is used (non-zero) for the text in Inkscape, but the package 'transparent.sty' is not loaded}%
    \renewcommand\transparent[1]{}%
  }%
  \providecommand\rotatebox[2]{#2}%
  \ifx\svgwidth\undefined%
    \setlength{\unitlength}{136.79883427bp}%
    \ifx\svgscale\undefined%
      \relax%
    \else%
      \setlength{\unitlength}{\unitlength * \real{\svgscale}}%
    \fi%
  \else%
    \setlength{\unitlength}{\svgwidth}%
  \fi%
  \global\let\svgwidth\undefined%
  \global\let\svgscale\undefined%
  \makeatother%
  \begin{picture}(1,0.9517625)%
    \put(0,0){\includegraphics[width=\unitlength,page=1]{lnp.pdf}}%
    \put(0.88373239,0.71716108){\color[rgb]{0,0,0}\makebox(0,0)[lb]{\smash{$A_i$}}}%
    \put(0.69542673,0.89523283){\color[rgb]{0,0,0}\makebox(0,0)[lb]{\smash{$S_j$}}}%
    \put(0.63402263,0.50020017){\color[rgb]{0,0,0}\makebox(0,0)[lb]{\smash{$\theta_i$}}}%
  \end{picture}%
\endgroup%

%% file: Thm43.eps_tex
\begingroup%
  \makeatletter%
  \providecommand\color[2][]{%
    \errmessage{(Inkscape) Color is used for the text in Inkscape, but the package 'color.sty' is not loaded}%
    \renewcommand\color[2][]{}%
  }%
  \providecommand\transparent[1]{%
    \errmessage{(Inkscape) Transparency is used (non-zero) for the text in Inkscape, but the package 'transparent.sty' is not loaded}%
    \renewcommand\transparent[1]{}%
  }%
  \providecommand\rotatebox[2]{#2}%
  \ifx\svgwidth\undefined%
    \setlength{\unitlength}{249.2bp}%
    \ifx\svgscale\undefined%
      \relax%
    \else%
      \setlength{\unitlength}{\unitlength * \real{\svgscale}}%
    \fi%
  \else%
    \setlength{\unitlength}{\svgwidth}%
  \fi%
  \global\let\svgwidth\undefined%
  \global\let\svgscale\undefined%
  \makeatother%
  \begin{picture}(1,0.24398074)%
    \put(0,0){\includegraphics[width=\unitlength]{Thm43.eps}}
    \put(0.186,0.195){$u$}
    \put(0.4,0.207){$r_{1}$}
    \put(0.4,0.06){$r_{2}$}
    \put(0.92,0.173){$r^{+}(u,v_{2})$}
  \end{picture}%
\endgroup%

%% file: Step1.eps_tex
\begingroup%
  \makeatletter%
  \providecommand\color[2][]{%
    \errmessage{(Inkscape) Color is used for the text in Inkscape, but the package 'color.sty' is not loaded}%
    \renewcommand\color[2][]{}%
  }%
  \providecommand\transparent[1]{%
    \errmessage{(Inkscape) Transparency is used (non-zero) for the text in Inkscape, but the package 'transparent.sty' is not loaded}%
    \renewcommand\transparent[1]{}%
  }%
  \providecommand\rotatebox[2]{#2}%
  \ifx\svgwidth\undefined%
    \setlength{\unitlength}{332.4bp}%
    \ifx\svgscale\undefined%
      \relax%
    \else%
      \setlength{\unitlength}{\unitlength * \real{\svgscale}}%
    \fi%
  \else%
    \setlength{\unitlength}{\svgwidth}%
  \fi%
  \global\let\svgwidth\undefined%
  \global\let\svgscale\undefined%
  \makeatother%
  \begin{picture}(1,0.44645006)%
    \put(0,0){\includegraphics[width=\unitlength]{Step1.eps}}%
    \put(0.17,0.06){$y_{1}$}
    \put(0.17,0.17){$y_{2}$}
    \put(0.17,0.32){$y_{4}$}
    \put(0.17,0.4){$y_{3}$}
  \end{picture}%
\endgroup%

%% file: WWfig.pdf_tex
\begingroup%
  \makeatletter%
  \providecommand\color[2][]{%
    \errmessage{(Inkscape) Color is used for the text in Inkscape, but the package 'color.sty' is not loaded}%
    \renewcommand\color[2][]{}%
  }%
  \providecommand\transparent[1]{%
    \errmessage{(Inkscape) Transparency is used (non-zero) for the text in Inkscape, but the package 'transparent.sty' is not loaded}%
    \renewcommand\transparent[1]{}%
  }%
  \providecommand\rotatebox[2]{#2}%
  \ifx\svgwidth\undefined%
    \setlength{\unitlength}{400.99978364bp}%
    \ifx\svgscale\undefined%
      \relax%
    \else%
      \setlength{\unitlength}{\unitlength * \real{\svgscale}}%
    \fi%
  \else%
    \setlength{\unitlength}{\svgwidth}%
  \fi%
  \global\let\svgwidth\undefined%
  \global\let\svgscale\undefined%
  \makeatother%
  \begin{picture}(1,0.2891048)%
    \put(0,0){\includegraphics[width=\unitlength,page=1]{WWfig.pdf}}%
    \put(0.76961964,0.10121326){\color[rgb]{0,0,0}\makebox(0,0)[lb]{\smash{$S_n$}}}%
    \put(0.53,0.24){\color[rgb]{0,0,0}\makebox(0,0)[lb]{\smash{$\Gamma$}}}%
    \put(0.35994643,0.08){\color[rgb]{0,0,0}\makebox(0,0)[lb]{\smash{$\Gamma_n$}}}%
    \put(0.47947464,0.04){\color[rgb]{0,0,0}\makebox(0,0)[lb]{\smash{$\Gamma_m$}}}%
    \put(0.74841303,0.26411859){\color[rgb]{0,0,0}\makebox(0,0)[lb]{\smash{$R^+(\Gamma)$}}}%
    \put(0.30885781,0.15326585){\color[rgb]{0,0,0}\makebox(0,0)[lb]{\smash{$R^-(\Gamma)$}}}%
    \put(0,0){\includegraphics[width=\unitlength,page=2]{WWfig.pdf}}%
  \end{picture}%
\endgroup%

%% file: AWfirst.pdf_tex
\begingroup%
  \makeatletter%
  \providecommand\color[2][]{%
    \errmessage{(Inkscape) Color is used for the text in Inkscape, but the package 'color.sty' is not loaded}%
    \renewcommand\color[2][]{}%
  }%
  \providecommand\transparent[1]{%
    \errmessage{(Inkscape) Transparency is used (non-zero) for the text in Inkscape, but the package 'transparent.sty' is not loaded}%
    \renewcommand\transparent[1]{}%
  }%
  \providecommand\rotatebox[2]{#2}%
  \ifx\svgwidth\undefined%
    \setlength{\unitlength}{224.61366061bp}%
    \ifx\svgscale\undefined%
      \relax%
    \else%
      \setlength{\unitlength}{\unitlength * \real{\svgscale}}%
    \fi%
  \else%
    \setlength{\unitlength}{\svgwidth}%
  \fi%
  \global\let\svgwidth\undefined%
  \global\let\svgscale\undefined%
  \makeatother%
  \begin{picture}(1,0.26814115)%
    \put(0,0){\includegraphics[width=\unitlength,page=1]{AWfirst.pdf}}%
    \put(0.03335511,0.005){\color[rgb]{0,0,0}\makebox(0,0)[lb]{\smash{$x$}}}%
    \put(0.21421129,0.005){\color[rgb]{0,0,0}\makebox(0,0)[lb]{\smash{$y$}}}%
    \put(0.7,0.005){\color[rgb]{0,0,0}\makebox(0,0)[lb]{\smash{$n_1e_1$}}}%
    \put(0.85701071,0.005){\color[rgb]{0,0,0}\makebox(0,0)[lb]{\smash{$n_2e_1$}}}%
    \put(0.18,0.15614618){\color[rgb]{0,0,0}\makebox(0,0)[lb]{\smash{$z$}}}%
  \end{picture}%
\endgroup%

%% file: Eden.pdf_tex
\begingroup%
  \makeatletter%
  \providecommand\color[2][]{%
    \errmessage{(Inkscape) Color is used for the text in Inkscape, but the package 'color.sty' is not loaded}%
    \renewcommand\color[2][]{}%
  }%
  \providecommand\transparent[1]{%
    \errmessage{(Inkscape) Transparency is used (non-zero) for the text in Inkscape, but the package 'transparent.sty' is not loaded}%
    \renewcommand\transparent[1]{}%
  }%
  \providecommand\rotatebox[2]{#2}%
  \ifx\svgwidth\undefined%
    \setlength{\unitlength}{481.32bp}%
    \ifx\svgscale\undefined%
      \relax%
    \else%
      \setlength{\unitlength}{\unitlength * \real{\svgscale}}%
    \fi%
  \else%
    \setlength{\unitlength}{\svgwidth}%
  \fi%
  \global\let\svgwidth\undefined%
  \global\let\svgscale\undefined%
  \makeatother%
  \begin{picture}(1,0.63284314)%
    \put(0,0){\includegraphics[width=\unitlength,page=1]{Eden.pdf}}%
    \put(0.16758082,0.10055065){$B(\sigma_{14})$}%
    \put(0.76593534,0.10055065){$B(\sigma_{15})$}%
  \put(0,0){\includegraphics[width=\unitlength,page=2]{Eden.pdf}}%
  \end{picture}%
\endgroup%

%% file: Competition.pdf_tex
\begingroup%
  \makeatletter%
  \providecommand\color[2][]{%
    \errmessage{(Inkscape) Color is used for the text in Inkscape, but the package 'color.sty' is not loaded}%
    \renewcommand\color[2][]{}%
  }%
  \providecommand\transparent[1]{%
    \errmessage{(Inkscape) Transparency is used (non-zero) for the text in Inkscape, but the package 'transparent.sty' is not loaded}%
    \renewcommand\transparent[1]{}%
  }%
  \providecommand\rotatebox[2]{#2}%
  \ifx\svgwidth\undefined%
    \setlength{\unitlength}{321.20000005bp}%
    \ifx\svgscale\undefined%
      \relax%
    \else%
      \setlength{\unitlength}{\unitlength * \real{\svgscale}}%
    \fi%
  \else%
    \setlength{\unitlength}{\svgwidth}%
  \fi%
  \global\let\svgwidth\undefined%
  \global\let\svgscale\undefined%
  \makeatother%
  \begin{picture}(1,0.501868)%
    \put(0,0){\includegraphics[width=\unitlength,page=1]{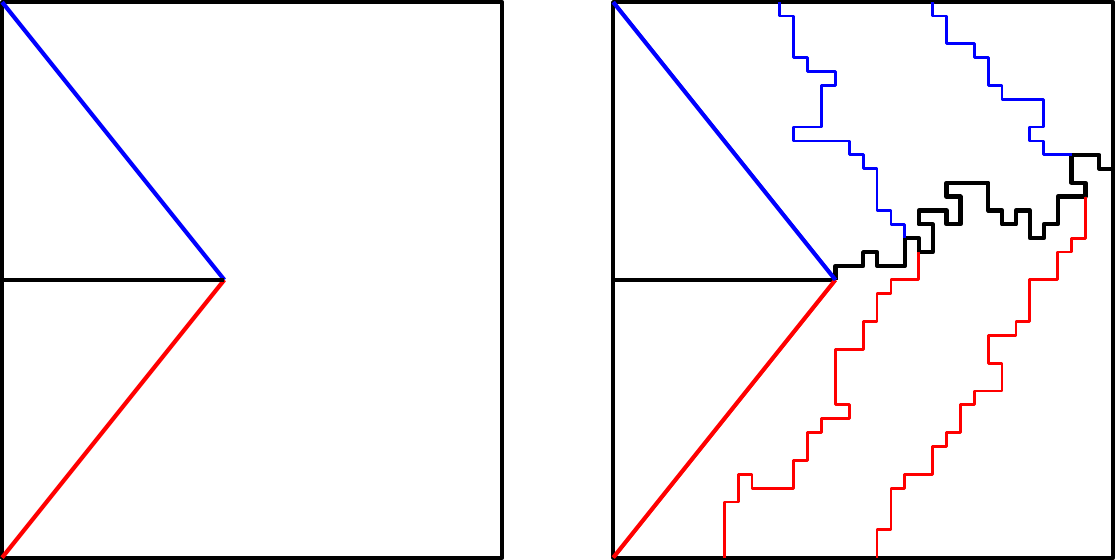}}%
    \put(0.05,0.3){\textcolor{blue}{$1$}}
        \put(0.6,0.3){\textcolor{blue}{$1$}}
            \put(0.05,0.17){\textcolor{red}{$2$}}
        \put(0.6,0.17){\textcolor{red}{$2$}}
  \end{picture}%
\endgroup%